\documentclass[a4paper,10pt]{book}

\usepackage{amsfonts}
\usepackage{amsmath}
\usepackage{mathrsfs}
\usepackage{amssymb,color}
\usepackage{fancyhdr}

\renewcommand{\chaptermark}[1]{%
	\markboth{#1}{\textbf{Characterization of constant sign Green's function of a two point boundary value problem by  means of spectral theory.}}}
\renewcommand{\sectionmark}[1]{%
	\markright{\textbf{Characterization of constant sign Green's function of a two point boundary value problem by  means of spectral theory.}}}
\fancypagestyle{plain}{%
	\fancyhead{} % elimina cabeceras en páginas "plain"
}
\usepackage{}
\numberwithin{section}{chapter}
\numberwithin{equation}{chapter}

\usepackage[latin1]{inputenc}
\usepackage{amsthm}
\usepackage{graphicx}
\usepackage{enumerate}
\newtheorem{theorem}{Theorem}[chapter]
\newtheorem{proposition}[theorem]{Proposition}

\newtheorem{corollary}[theorem]{Corollary}

\newtheorem{exemplo}[theorem]{Example}

\newtheorem{notation}[theorem]{Notation}

\newtheorem{lemma}[theorem]{Lemma}
\newtheorem{definition}[theorem]{Definition}

\newtheorem{remark}[theorem]{Remark}

\newcommand{\R}{{{\mathbb R}}}

\usepackage[original]{imakeidx}
\makeindex
\begin{document}
	
	\frontmatter
	
	\title{\textbf{Characterization of constant sign Green's function of a two point boundary value problem by  means of spectral theory.\footnote{Partially supported by Ministerio de Economía y competividad, Spain and FEDER, project MTM2013-43014-P}}}
	
	%    Remove any unused author tags.
	
	%    author one information
	\author{Alberto Cabada\footnote{alberto.cabada@usc.es}, Lorena Saavedra\footnote{lorena.saavedra@usc.es}\footnote{Supported by FPU scholarship, Ministerio de Educaci\'on, Cultura y Deporte, Spain.}\\Instituto de Matem\'aticas\\ Facultade de Matem\'aticas\\Universidade de Santiago de Compostela\\ Santiago de Compostela \\ Galicia \\	Spain}
%	\address{}
%	\curraddr{}
%	\email{alberto.cabada@usc.es}
%	\thanks{}
%	
%	%    author two information
%	\author{Lorena Saavedra}
%%	\address{Instituto de Matem\'aticas\\ Facultade de Matem\'aticas\\Universidade de Santiago de Compostela\\ Santiago de Compostela \\ Galicia \\	Spain}
%%	\curraddr{}
%%	\email{lorena.saavedra@usc.es}
%	\thanks{}
%	
%	%    \date is required; it is the date received by the editor.
\date{}

	\maketitle

\chapter*{Abstract}\addcontentsline{toc}{chapter}{Abstract}
	This paper is devoted to the study of the parameter's set where the Green's function related to a general linear $n^{\rm th}$-order operator, depending on a real parameter, $T_n[M]$, coupled with many different two point boundary value conditions, is of constant sign. This constant sign is equivalent to the strongly inverse positive (negative) character of the related operator on suitable spaces related to the boundary conditions.
		
	This characterization is based on spectral theory, in fact the extremes of the obtained interval are given by suitable eigenvalues of the differential operator with different boundary conditions.
	
	Moreover, we also obtain a characterization of the strongly inverse positive (negative) character on some sets, where non homogeneous boundary conditions are considered.

	In order to see the applicability of the obtained results, some examples are given along the paper. This method avoids the explicit calculation of the related Green's function.\\[1cm]
%\footnote{2010 Mathematics Subject Classification. Primary  34B05, 34B08, 34B27, 34L05. Secondary  34B15, 34B18, 47A05}
%%	%    Recognition of the 2010 edition of the Mathematics Subject
%%	%    Classification requires a version of amsbook.cls from July 2009
%%	%    or later.  If "2010" is not recognized, please upgrade.
%%	
%\footnote{\textbf{Keywords}: Green's functions, spectral theory, boundary value problems}

\noindent \textbf{2010 Mathematics Subject Classification:\\}
\textbf{ Primary:}  34B05, 34B08, 34B27, 34L05.\\
\textbf{Secondary:}  34B15, 34B18, 47A05\\

\noindent \textbf{Keywords}: Green's functions, spectral theory, boundary value problems

	\tableofcontents
\chapter*{Introduction}\addcontentsline{toc}{chapter}{Introduction}
\markboth{Introduction}{Introduction}
The study of the qualitative properties of the solution of a nonlinear two-point  differential equation has been widely developed in the literature. 

This work is devoted to the study of the strongly inverse positive (negative) character of the general $n^{\mathrm{th}}$- order operator coupled with different boundary conditions. In order to do that, we  characterize the parameter's set where the related Green's function is of constant sign. Our characterization is based on the spectral theory, actually, the extremes of the parameter's interval, where the Green's function is of constant sign, are just the eigenvalues of the given operator with suitable related boundary conditions.

This result avoids the requirement of obtaining the explicit expression of the Green's function, which can be very complicate to work with. In a wider class of situations, specially in the non constant coefficients case, it is not possible to obtain such an expression. Moreover, a slight change on the operator or in the boundary conditions may produce a big change on the expression of the related Green's function and its behavior. So, it is very useful to give a direct and easy way to characterize its sign.

It is well-known that the constant sign of the Green's function related to the linear part of a nonlinear problem is equivalent to  the validity of the method of lower and upper solutions, coupled with monotone iterative techniques, that allows to deduce the existence of solution of such a problem, see for instance \cite{Cab,Cab2,CoHa,LaLaVa}.

Moreover, by using the constant sign of the related Green's function, nonexistence, existence and multiplicity results for nonlinear boundary value problems are derived, by means of the well-known Kranosel'ski\u{\i} contraction/expansion fixed point theorem \cite{Kra}, from the construction of suitable cones on Banach spaces \cite{AnHo,CaCi,CaPrSaTe,GrKoWa,To}.

The combination of these two methods has also been proved as a useful tool to ensure the existence of solution \cite{CaCi2,CaCiIn,CiFrMi,FrInPe,Pe}.

It is important to point out that the study of the constant sign of the related Green's function has been widely developed along the literature, by means of studying its expression \cite{CaCiSa,CaEn,CaFe,MaLu}. In all of them the expression of the Green's function has been obtained in order to prove the optimality of the previously obtained bounds. This work generalizes the ones given in \cite{CabSaa} for the problems with the so-called $(k,n-k)$ boundary conditions and in \cite{CabSaa2} for a fourth order problem with the simply supported beam boundary conditions. 

In this paper, we study a huge number of different boundary conditions including the previously mentioned.

 First, we introduce two sets of index which  describe de boundary conditions in each case.

Let $k\in\{1,\dots, n-1\}$ and consider the following sets of index $\{\sigma_1,\dots, \sigma_k\}\subset\{0,\dots,n-1\}$ and  $\{\varepsilon_1,\dots, \varepsilon_{n-k}\}\subset\{0,\dots,n-1\}$, such that
\[ 0\leq \sigma_1<\sigma_2<\cdots<\sigma_k\leq n-1\,,\quad  0\leq \varepsilon_1<\varepsilon_2<\cdots<\varepsilon_{n-k}\leq n-1\,.\]

\begin{definition}\label{D::Na}
	Let us say that $\{\sigma_1,\dots, \sigma_k\}-\{\varepsilon_1,\dots, \varepsilon_{n-k}\}$ satisfy property $(N_a)$ if, 
\begin{equation}\label{Ec::Na}
	\sum_{\sigma_j< h}1+\sum_{\varepsilon_j< h}1\geq h\,,\quad \forall h\in\{1,\dots, n-1\}\,.\end{equation}	
\end{definition}
	
\begin{notation}
	
	Let us denote $\alpha$, $\beta\in\{0,\dots, n-1\}$, such that
	\begin{align}
	\label{Ec::alpha} &\alpha\notin \{\sigma_1,\dots, \sigma_k\}\,,\quad \text{and if } \alpha\neq 0\,,\quad \{0,\dots, \alpha-1\}\subset\{\sigma_1,\dots, \sigma_k\}\,,\\\label{Ec::beta}
	 &\beta\notin \{\varepsilon_1,\dots, \varepsilon_{n-k}\}\,,\quad \text{and if } \beta\neq 0\,,\quad \{0,\dots, \beta-1\}\subset\{\varepsilon_1,\dots, \varepsilon_{n-k}\}\,.
	\end{align}
	\end{notation}
	
	Realize that $\alpha\leq k$ and $\beta\leq n-k$.
	
	Let us define the following family of $n^{\rm th}$-order linear differential equations
	\begin{equation}\label{Ec::T_n[M]}
	T_n[M]\,u(t)=u^{(n)}(t)+p_1(t)\,u^{(n-1)}(t)+\cdots+(p_n(t)+M)\,u(t)=0\,,\ t\in I\,,
	\end{equation}
	where $I=[a,b]$ is a real fixed interval, $M\in\mathbb{R}$ a parameter and $p_j\in C^{n-j}(I)$ are given functions.
	
Realize that this equation represents a general $n$ order equation, this is due to the fact that if we denote as $\tilde p_n$ the coefficient of $u$, we have
$$\tilde p_n(t)=p_n(t)+\displaystyle \frac{1}{b-a} \int_a^b{\tilde p_n(s)\, ds} \equiv p_n(t)+M,\quad t \in I.$$

So, if $p_n$ is a function of average equals to zero, the parameter $M$ represents the average of the coefficient of $u$ and, as consequence, the problem of finding the values of $M$ for which the Green's function has constant sign is equivalent to look for the values of the average of such a coefficient.

We study \eqref{Ec::T_n[M]}, coupled with the following boundary conditions:
\begin{eqnarray}
\label{Ec::cfa} u^{(\sigma_1)}(a)=\cdots=u^{(\sigma_k)}(a)&=&0\,,\\
\label{Ec::cfb} u^{(\varepsilon_1)}(b)=\cdots=u^{(\varepsilon_{n-k})}(b)&=&0\,.
\end{eqnarray}

This boundary conditions cover many different problems. As an example, we can consider $n=4$,  $\{\sigma_1,\sigma_2\}=\{0,2\}$ and $\{\varepsilon_1,\varepsilon_2\}=\{0,2\}$ which correspond to the simply supported beam boundary conditions. 

Realize that, in the second order case, the Neumman conditions do not satisfy property $(N_a)$. However, Dirichlet and Mixed conditions are included.

Along the paper, we will illustrate the different obtained results with an example based on the choice $\{\sigma_1,\sigma_2\}=\{0,2\}$ and $\{\varepsilon_1,\varepsilon_2\}=\{1,2\}$.

	We consider the following space of  definition related to the boundary conditions \eqref{Ec::cfa}-\eqref{Ec::cfb}:
	{\normalsize \begin{equation}	\label{Ec::X_se}\begin{split}
	X_{\{\sigma_1,\dots, \sigma_k\}}^{\{\varepsilon_1,\dots, \varepsilon_{n-k}\}}=&\left\lbrace u\in C^n(I)\ \mid\ u^{(\sigma_1)}(a)=\cdots=u^{(\sigma_k)}(a)=u^{(\varepsilon_1)}(b)\right. \\&\left. =\cdots=u^{(\varepsilon_{n-k})}(b)=0\right\rbrace.\end{split}
	\end{equation}}

\begin{remark}
	Along the paper we consider different choices of boundary conditions. Sometimes, we do not know the relative position of the given index which define the spaces of definition. In particular, if we consider the following boundary conditions
	\[\begin{split}
	u^{(\sigma_1)}(a)=\cdots=u^{(\sigma_{k-1})}(a)&=0\,,\\
	u^{(\alpha)}(a)&=0\,,\\
	u^{(\varepsilon_1)}(b)=\cdots=u^{(\varepsilon_{n-k})}(b)&=0\,,
	\end{split}\]
	with $\alpha$ defined in \eqref{Ec::alpha}.
	 
	 In order to point out this setting of the index we use the following notation:
	 	{\normalsize \begin{equation*}	\begin{split}
	 	X_{\{\sigma_1,\dots, \sigma_{k-1}|\alpha\}}^{\{\varepsilon_1,\dots, \varepsilon_{n-k}\}}=&\left\lbrace u\in C^n(I)\ \mid\ u^{(\sigma_1)}(a)=\cdots=u^{(\sigma_{k-1})}(a)=0\,,\ u^{(\alpha)}(a)=0\,,\right. \\&\left. u^{(\varepsilon_1)}(b)=\cdots=u^{(\varepsilon_{n-k})}(b)=0\right\rbrace.\end{split}
	 	\end{equation*}}
 	
 	For instance, if $n=4$, $\sigma_1=0$, $\sigma_2=2$, $\varepsilon_1=0$ and $\varepsilon_2=1$, then $X_{\{\sigma_1|\alpha\}}^{\{\varepsilon_1,\varepsilon_2\}}=X_{\{0,1\}}^{\{0,1\}}$, where $\sigma_1=0<\alpha=1$. On another hand, if $\sigma_1=2$, $\sigma_2=3$, $\varepsilon_1=0$ and $\varepsilon_2=1$, then $X_{\{\sigma_1|\alpha\}}^{\{\varepsilon_1,\varepsilon_2\}}=X_{\{0,2\}}^{\{0,1\}}$ where $\alpha=0<\sigma_1=2$.
\end{remark}
So, we are interested into characterize the parameter's set for which the operator $T_n[M]$ is either strongly inverse positive or negative on $X_{\{\sigma_1,\dots, \sigma_k\}}^{\{\varepsilon_1,\dots, \varepsilon_{n-k}\}}$.

Moreover, once we have obtained such a characterization with the homogeneous boundary conditions \eqref{Ec::cfa}-\eqref{Ec::cfb}, we study its strongly inverse positive (negative) character on related spaces with non homogeneous boundary conditions.

The work is structured as follows, at first, in order to make the paper more readable, we introduce a preliminary chapter where some previous known results are shown. After that, in the next chapter, we introduce the hypotheses that both the operator $T_n[M]$ and the boundary conditions should satisfy to our results be applied. In Chapter 3 we obtain the expression of the related adjoint operator and boundary conditions. We deduce suitable properties of them. Next chapter is devoted to the study of operator $T_n[M]$ for a given $M=\bar M$ that  verifies  some suitable previously introduced hypothesis. After that, in the two next chapters, we study the existence and properties of the related eigenvalues of the operator and its adjoint, respectively, together to additional properties of the associated eigenfunctions. Chapter 7 is devoted to prove the main result of the work, where the characterization of the interval of parameters, where the Green's function has constant sign, is attained. At the end of such a chapter, some examples are shown. In Chapter 8, we obtain a necessary condition that $M$ should verify,  in order to allow $T_n[M]$ to be strongly inverse negative (positive) on the non considered cases on Chapter 7. At the end of such a chapter, we prove that this necessary condition can give an optimal interval in some cases. Once  we have worked with the homogeneous boundary conditions, we obtain a characterization for a particular case of non homogeneous boundary conditions. Finally, we study a class of operators that satisfy the imposed hypotheses. Moreover, for this kind of operators, we obtain a  characterization for more general non homogeneous boundary conditions. The chapter finishes by showing some examples of this class of operators.

\mainmatter
\chapter{Preliminaries}

In this chapter, for the convenience of the reader, we introduce the fundamental tools in the theory of disconjugacy and Green's functions that will be used in the development of further sections.

\begin{definition}
	Let $p_k\in C^{n-k}(I)$ for $k=1,\dots,n$. The $n^{\rm th}$-order linear differential equation (\ref{Ec::T_n[M]}) is said to be disconjugate on  $I$ if every non trivial solution has less than $n$ zeros on $I$, multiple zeros being counted according to their multiplicity.
\end{definition}
\begin{definition}
	The functions $u_1,\dots, u_n \in C^n(I)$ are said to form a Markov system on the interval $I$ if the $n$ Wronskians
	\begin{equation}
	W(u_1,\dots,u_k)=\left| \begin{array}{ccc}
	u_1&\cdots&  u_k\\
	\vdots&\cdots&\vdots\\
	u_1^{(k-1)}&\cdots&u_k^{(k-1)}\end{array}\right| \,,\quad k=1,\dots,n \,,
	\end{equation}
	are positive throughout $I$.
\end{definition}

The following results about this concept are collected on \cite[Chapter 3]{Cop}.

\begin{theorem}\label{T::4}
	The linear differential equation (\ref{Ec::T_n[M]}) has a Markov fundamental system of solutions on the compact interval $I$ if, and only if, it is disconjugate on $I$.
	
\end{theorem}

\begin{theorem}\label{T::3}
	The linear differential equation (\ref{Ec::T_n[M]}) has a Markov system of solutions if, and only if, the operator $T_n[M]$ has a representation
	\begin{equation}
	\label{e-descomp}
	T_n[M]\,y\equiv v_1 \,v_2\,\dots\,v_n \dfrac{d}{dt}\left( \dfrac{1}{v_n}\,\dfrac{d}{dt}\left( \cdots \dfrac{d}{dt}\left( \dfrac{1}{v_2} \dfrac{d}{dt}\left( \dfrac{1}{v_1}\,y\right) \right) \right) \right),	
	\end{equation}
	where $v_k>0$  on $I$ and $v_k\in C^{n-k+1}(I)$ for $k=1,\dots,n$.
\end{theorem}

In order to introduce the concept of Green's function related to the $n^{\rm th}$ - order scalar problem \eqref{Ec::T_n[M]}--\eqref{Ec::cfb}, we consider the following equivalent first order vectorial problem:
\begin{equation}\label{Ec::vec}
x'(t)=A(t)\, x(t)\,,\ t\in I\,,\quad 
B\,x(a)+C\,x(b)=0,
\end{equation}
with $x(t) \in \R^n$, $A(t), \, B,\,\ C\in \mathcal{M}_{n\times n}$, defined by
\[x(t)=\left( \begin{array}{c} 
u(t)\\
u'(t)\\
\vdots\\
u^{(n-1)}(t) \end{array}\right),\, \quad A(t)=\left( \begin{array}{c|c}
&\\
0&\quad I_{n-1}\quad\\
&\\
\hline
-(p_n(t)+M)&-p_{n-1}(t)\cdots-p_1(t) \end{array}\right), \]

\begin{equation}\label{Ec::Cf}
B=\left( \begin{array}{ccc} b_{11}&\cdots&b_{1n}\\
\vdots&\ddots&\vdots\\
b_{n1}&\cdots&b_{nn}\end{array}\right), \;\quad C=\left( \begin{array}{ccc} c_{11}&\cdots&c_{1n}\\
\vdots&\ddots&\vdots\\
c_{n1}&\cdots&c_{nn}\end{array}\right),
\end{equation}
where $b_{j,1+\sigma_j}=1$ for $j=1,\dots,k$ and $c_{j+k,1+\varepsilon_j}=1$ for $j=1,\dots,n-k$; otherwise, $b_{ij}=0$ and $c_{ij}=0$.

\begin{definition} \label{Def::G}
	We say that $G$ is a Green's function for problem \eqref{Ec::vec} if it satisfies the following properties:
	\begin{itemize}
\item[$\mathrm{(G1)}$] $G\equiv (G_{i,j})_{i,j\in\{1,\dots,n\}}\colon (I\times I)\backslash \left\lbrace (t,t)\,,\ t\in I\right\rbrace \rightarrow \mathcal{M}_{n\times n}$.

\item[$\mathrm{(G2)}$] $G$ is a $C^{1}$ function on the triangles $\left\lbrace (t,s)\in \mathbb{R}^2\,,\ a\leq s<t\leq b\right\rbrace $ and $\left\lbrace (t,s)\in \mathbb{R}^2\,,\ a\leq t < s\leq b\right\rbrace $.

\item[$\mathrm{(G3)}$] For all $i\neq j$ the scalar functions $G_{i,j}$ have a continuous extension to $I\times I$.

\item[$\mathrm{(G4)}$] For all $s\in(a,b)$, the following equality holds:
\[\dfrac{\partial }{\partial t}\, G(t,s)=A(t)\,G(t,s),\quad \text{for all } t\in I\backslash \left\lbrace s\right\rbrace .\]

\item[$\mathrm{(G5)}$] For all $s\in(a,b)$ and $i\in\left\lbrace 1,\dots, n\right\rbrace $, the following equalities are fulfilled:
\[\lim_{t\rightarrow s^+}G_{i,i}(t,s)=\lim_{t\rightarrow s^-}G_{i,i}(s,t)=1+\lim_{t\rightarrow s^+}G_{i,i}(s,t)=1+\lim_{t\rightarrow s^-}G_{i,i}(t,s).\]

	\item[$\mathrm{(G6)}$] For all $s\in(a,b)$, the function $t\rightarrow G(t,s)$ satisfies the boundary conditions
	\[B\,G(a,s)+C\,G(b,s)=0.\]
	\end{itemize}
\end{definition}

\begin{remark}\label{R:2.5}
	On the previous definition, item $\mathrm{(G5)}$ can be modified to obtain the characterization of the lateral limits for $s=a$ and $s=b$ as follows:	
	\[\lim_{t\rightarrow a^+}G_{i,i}(t,a)=1+\lim_{t\rightarrow a^+}G_{i,i}(a,t),\quad\text{and}\quad \lim_{t\rightarrow b^-}G_{i,i}(b,t)=1+\lim_{t\rightarrow b^-}G_{i,i}(t,b).\]
\end{remark}

It is very well known that Green's function related to this problem is given by the following expression \cite[Section 1.4]{Cab}
{\scriptsize \begin{equation}\label{Ec:MG} G(t,s)=\left( \begin{array}{ccccc}
g_1(t,s)&g_2(t,s)&\cdots&g_{n-1}(t,s)&g_M(t,s)\\&&&&\\
\dfrac{\partial }{\partial t}\,g_1(t,s)& \dfrac{\partial }{\partial t}\,g_2(t,s)&\cdots&\dfrac{\partial }{\partial t}\,g_{n-1}(t,s)& \dfrac{\partial }{\partial t}\,g_M(t,s)\\
\vdots&\vdots&\cdots&\vdots&\vdots\\
\dfrac{\partial^{n-1} }{\partial t^{n-1}}\,g_1(t,s)&\dfrac{\partial^{n-1} }{\partial t^{n-1}}\,g_2(t,s)&\cdots&\dfrac{\partial^{n-1} }{\partial t^{n-1}}\,g_{n-1}(t,s)&\dfrac{\partial^{n-1}} {\partial t^{n-1}}\,g_M(t,s)\end{array} \right) ,\end{equation}}where $g_M(t,s)$ is the scalar Green's function related to operator $T_n[M]$ in $X_{\{\sigma_1,\dots,\sigma_k\}}^{\{\varepsilon_1,\dots,\varepsilon_{n-k}\}}$. 

Using  Definition \ref{Def::G} we can deduce the properties fulfilled by $g_M(t,s)$.
In particular, $g_M\in C^{n-2}(I \times I)$ and it is a $C^n$ function on the triangles $a\le s < t \le b$ and $a\le t < s \le b$. Moreover it satisfies, as a function of $t$, the two-point boundary value conditions \eqref{Ec::cfa}-\eqref{Ec::cfb} and solves equation \eqref{Ec::T_n[M]} whenever $t \neq s$.

In \cite{CabSaa} $g_{n-j}(t,s)$ are expressed as functions of $g_M(t,s)$ for all $j=1,\dots,n-1$ as follows:\begin{equation}
\label{Ec::gj} g_{n-j}(t,s)=(-1)^j\dfrac{\partial^j}{\partial\,s^j}\,g_M(t,s)+\sum_{i=0}^{j-1}\alpha_i^j(s)\,\dfrac{\partial^i}{\partial s^i}g_M(t,s)\,,
\end{equation}where $\alpha_i^j(s)$ are functions of $p_1(s)\,,\dots,\ p_j(s)$ and of its derivatives until order $(j-1)$ and  follow the recurrence formula
\begin{align}
\label{r2}
\alpha_0^0(s)&=0\,,\\
\alpha_i^{j+1}(s)&=0\,,\quad i\geq j+1\geq1 ,\\
\label{r1}
\alpha_0^{j+1}(s)&=p_{j+1}(s)-\left( \alpha_0^j\right) '(s), \quad j \geq 0,\\\label{r4}
\alpha_i^{j+1}(s)&=-\left( \alpha_{i-1}^j(s)+\left( \alpha_i^j\right) '(s)\right), \quad 1\leq i\leq j.
\end{align}

The adjoint of the operator $T_n[M]$ is given by the following expression, see for details \cite[Section 1.4]{Cab} or \cite[Chapter 3, Section 5]{Cop},
\begin{equation}\label{EC::Ad}
T_n^*[M]v (t)\equiv (-1)^n \,v^{(n)}(t)+\sum_{j=1}^{n-1}(-1)^j\,\left(p_{n-j}\,v\right)^{(j)}(t)+(p_n(t)+M)\,v(t)\,,
\end{equation}
and its domain of definition is
{\footnotesize \begin{align} \nonumber
 D(T_n^*[M])=&\left\lbrace v\in C^n(I)\ \mid \sum_{j=1}^{n}\sum_{i=0}^{j-1} (-1)^{j-1-i} (p_{n-j}\,v)^{(j-1-i)}(b)\,u^{(i)}(b)\right. \\\label{Ec::cfad}
& \left. \ =\sum_{j=1}^{n}\sum_{i=0}^{j-1} (-1)^{j-1-i} (p_{n-j}\,v)^{(j-1-i)}(a)\,u^{(i)}(a) \ \text{ (with $p_0=1$)}\,, \forall u\in D(T_n[M])\right\rbrace .
\end{align}}

Next result appears in \cite[Chapter 3, Theorem 9]{Cop}
\begin{theorem}\label{T::2}
	The equation (\ref{Ec::T_n[M]}) is disconjugate on an interval $I$ if, and only if, the adjoint equation, $T_n^*[M]\,y(t)=0$ is disconjugate on $I$.
\end{theorem}

We denote $g_M^*(t,s)$ as the Green's function related to the adjoint operator, $T_n^*[M]$.

In \cite[Section 1.4]{Cab} it is proved the following relationship
\begin{equation} \label{Ec::gg}
g^*_M(t,s)=g_M(s,t)\,.
\end{equation}

Now, let us define the following operator
\begin{equation}\label{Ec::Tg}
\widehat{T}_n[(-1)^n\,M]:=(-1)^n T_n^*[M] \,,
\end{equation}
we deduce, from the previous expressions, that
\begin{equation}\label{Ec::gg1}
\widehat{g}_{(-1)^n\,M}(t,s)=(-1)^n\,g_M^*(t,s)=(-1)^n\,g_{M}(s,t)\,,
\end{equation}
where $\widehat g_{(-1)^nM}(t,s)$ is the scalar Green's function related to operator $\widehat T_n[(-1)^nM]$ in $D\left( T_n^*[M]\right)$. 

Obviously, Theorem \ref{T::2} remains true for operator $\widehat{T}_n[(-1)^n\,M]$.

\begin{definition}
\label{d-IP}
Operator $T_n[M]$ is said to be inverse positive (negative) on $X_{\{\sigma_1,\dots,\sigma_k\}}^{\{\varepsilon_1,\dots,\varepsilon_{n-k}\}}$ if every function $u \in X_{\{\sigma_1,\dots,\sigma_k\}}^{\{\varepsilon_1,\dots,\varepsilon_{n-k}\}}$ such that $T_n[M]\, u \ge 0$ on $I$, satisfies $u\geq 0$ ($u\leq 0$) on $I$.
\end{definition}

Next results are consequence of the ones proved on \cite[Section 1.6, Section 1.8]{Cab} for several two-point $n-$order operators.

\begin{theorem}\label{T::in1}
	Operator $T_n[M]$ is inverse positive (negative) in $X_{\{\sigma_1,\dots,\sigma_k\}}^{\{\varepsilon_1,\dots,\varepsilon_{n-k}\}}$ if, and only if, Green's function related to problem \eqref{Ec::T_n[M]}--\eqref{Ec::cfb} is non-negative (non-positive) on its square of definition.
\end{theorem}

\begin{theorem}\label{T::d1}
	Let $M_1$, $M_2\in\mathbb{R}$ and suppose that operators $T_n[M_j]$, $j=1,2$, are invertible in $X_{\{\sigma_1,\dots,\sigma_k\}}^{\{\varepsilon_1,\dots,\varepsilon_{n-k}\}}.$
	Let $g_j$, $j=1,2$, be Green's functions related to  operators $T_n[M_j]$ and suppose that both functions have the same constant sign on $I \times I$. Then, if $M_1<M_2$, it is satisfied that $g_2\leq g_1$ on $I \times I$.
\end{theorem}

	\begin{theorem}\label{T::int}
		Let $M_1<\bar{M}<M_2$ be three real constants. Suppose that operator $T_n[M]$ is invertible in $X_{\{\sigma_1,\dots,\sigma_k\}}^{\{\varepsilon_1,\dots,\varepsilon_{n-k}\}}$  for $M=M_j$, $j=1,2$ and that the corresponding Green's function satisfies $g_2\leq g_1\leq 0$ (resp. $0\leq g_2\leq g_1$) on $I\times I$. Then the operator $T_n[\bar{M}]$ is invertible in $X_{\{\sigma_1,\dots,\sigma_k\}}^{\{\varepsilon_1,\dots,\varepsilon_{n-k}\}}$  and the related Green's function $\bar{g}$ satisfies $g_2\leq \bar{g}\leq g_1\leq 0$ ($0\leq g_2\leq \bar{g}\leq g_1$) on $I\times I$.
	\end{theorem}
Now, we introduce a stronger concept of inverse positive (negative) character.

\begin{definition}
	\label{d-SIP}
	Operator $T_n[M]$ is said to be strongly inverse  positive  in $X_{\{\sigma_1,\dots,\sigma_k\}}^{\{\varepsilon_1,\dots,\varepsilon_{n-k}\}}$ if every function $u \in X_{\{\sigma_1,\dots,\sigma_k\}}^{\{\varepsilon_1,\dots,\varepsilon_{n-k}\}}$ such that $T_n[M]\, u \gneqq 0$ on $I$, must verify $u> 0$  on $(a,b)$ and, moreover, $u^{(\alpha)}(a)>0$ and $u^{(\beta)}(b)>0$ if $\beta$ is even, $u^{(\beta)}(b)<0$ if $\beta$ is odd, where $\alpha$ and $\beta$ are defined in \eqref{Ec::alpha} and \eqref{Ec::beta}, respectively.
\end{definition}
\begin{definition}
	\label{d-SIN}
	Operator $T_n[M]$ is said to be strongly inverse  negative in $X_{\{\sigma_1,\dots,\sigma_k\}}^{\{\varepsilon_1,\dots,\varepsilon_{n-k}\}}$ if every function $u \in X_{\{\sigma_1,\dots,\sigma_k\}}^{\{\varepsilon_1,\dots,\varepsilon_{n-k}\}}$ such that $T_n[M]\, u \gneqq 0$ on $I$, must verify $u<0$ on $(a,b)$ and, moreover, $u^{(\alpha)}(a)<0$ and $u^{(\beta)}(b)<0$ if $\beta$ is even, $u^{(\beta)}(b)>0$ if $\beta$ is odd, where $\alpha$ and $\beta$ are defined in \eqref{Ec::alpha} and \eqref{Ec::beta}, respectively.
\end{definition}

Analogously to Theorem \ref{T::in1}, the following ones can be shown:
\begin{theorem}\label{T::in2}
	Operator $T_n[M]$ is strongly inverse positive  in $X_{\{\sigma_1,\dots,\sigma_k\}}^{\{\varepsilon_1,\dots,\varepsilon_{n-k}\}}$ if, and only if, Green's function related to problem \eqref{Ec::T_n[M]}--\eqref{Ec::cfb}, $g_M(t,s)$, satisfies the following properties:
	\begin{itemize}
		\item $g_M(t,s)>0$  a.e. on $(a,b)\times (a,b)$.
		\item $\dfrac{\partial^\alpha}{\partial t^\alpha}g_M(t,s)_{\mid t=a}>0$  for a.e. $s\in(a,b)$.
		\item  $\dfrac{\partial^\beta}{\partial t^\beta}g_M(t,s)_{\mid t=b}>0$   if $\beta$ is even and $\dfrac{\partial^\beta}{\partial t^\beta}g_M(t,s)_{\mid t=b}<0$  if $\beta$ is odd for a.e. $s\in(a,b)$.
	\end{itemize}
\end{theorem}
 \begin{theorem}\label{T::in21}
 	Operator $T_n[M]$ is strongly inverse  negative in $X_{\{\sigma_1,\dots,\sigma_k\}}^{\{\varepsilon_1,\dots,\varepsilon_{n-k}\}}$ if, and only if, Green's function related to problem \eqref{Ec::T_n[M]}--\eqref{Ec::cfb}, $g_M(t,s)$, satisfies the following properties:
 	\begin{itemize}
 		\item $g_M(t,s)<0$ a.e. on $(a,b)\times (a,b)$.
 		\item $ \dfrac{\partial^\alpha}{\partial t^\alpha}g_M(t,s)_{\mid t=a}<0$ for a.e. $s\in(a,b)$.
 		\item  $ \dfrac{\partial^\beta}{\partial t^\beta}g_M(t,s)_{\mid t=b}<0$ if $\beta$ is even and $ \dfrac{\partial^\beta}{\partial t^\beta}g_M(t,s)_{\mid t=b}>0$ if $\beta$ is odd for a.e. $s\in(a,b)$.
 	\end{itemize}
 \end{theorem}

In the sequel, we introduce two conditions on $g_M(t,s)$ that will be used along the paper.
\begin{itemize}
	\item[$(P_g$)] Suppose that there is a continuous function $\phi(t)>0$ for all $t\in (a,b)$ and $k_1,\ k_2\in \mathcal{L}^1(I)$, such that $0<k_1(s)<k_2(s)$ for a.e. $s\in I$, satisfying
	\[\phi(t)\,k_1(s)\leq g_M(t,s)\leq \phi(t)\, k_2(s)\,,\quad \text{for a.e. } (t,s)\in I \times I \,.\]
	\item[($N_g$)] Suppose that there is a continuous function $\phi(t)>0$ for all $t\in (a,b)$ and $k_1,\ k_2\in \mathcal{L}^1(I)$, such that $k_1(s)<k_2(s)<0$ for a.e. $s\in I$, satisfying
	\[\phi(t)\,k_1(s)\leq g_M(t,s)\leq \phi(t)\, k_2(s)\,,\quad \text{for a.e. }(t,s)\in I \times I\,.\]
\end{itemize}

Finally, we introduce the following sets that characterize where the Green's function is of constant sign,
\begin{align}
P_T&= \left\lbrace M\in \mathbb{R}\,,\ \mid \quad g_M(t,s)\geq 0\quad \forall (t,s)\in I\times I\right\rbrace, \\
N_T&= \left\lbrace M\in \mathbb{R}\,,\ \mid \quad g_M(t,s)\leq 0\quad \forall (t,s)\in I\times I\right\rbrace.
\end{align}

Realize that, using Theorem \ref{T::d1}, we can affirm that the two previous sets are real intervals (which can be empty in some situations).

Next results describe one of the extremes of the two previous intervals (see \cite[Theorems 1.8.31 and 1.8.23]{Cab}).

\begin{theorem}\label{T::6}
	Let $\bar{M}\in \mathbb{R}$ be fixed. If  $T_n[\bar{M}]$ is an invertible operator in $X_{\{\sigma_1,\dots,\sigma_k\}}^{\{\varepsilon_1,\dots,\varepsilon_{n-k}\}}$ and its related Green's function satisfies condition $(P_g)$, then the following statements hold:
	\begin{itemize}
		
		\item There is $\lambda_1>0$, the least eigenvalue in absolute value of operator $T_n[\bar{M}]$ in $X_{\{\sigma_1,\dots,\sigma_k\}}^{\{\varepsilon_1,\dots,\varepsilon_{n-k}\}}.$ Moreover, there exists a nontrivial constant sign eigenfunction corresponding to the eigenvalue $\lambda_1$.
		\item Green's function related to operator $T_n[M]$ is nonnegative on $I\times I$ for all $M\in(\bar{M}-\lambda_1,\bar{M}]$.
		
		\item Green's function related to operator $T_n[M]$ cannot be nonnegative on $I\times I$ for all $M<\bar{M}-\lambda_1$.
		\item If there is $M\in \mathbb{R}$ for which Green's function related to operator $T_n[M]$ is non-positive on $I\times I$, then $M<\bar{M}-\lambda_1$.
	\end{itemize}
\end{theorem}	

\begin{theorem}\label{T::7}
	Let $\bar{M}\in \mathbb{R}$ be fixed. If $T_n[\bar{M}]$ is an invertible operator in $X_{\{\sigma_1,\dots,\sigma_k\}}^{\{\varepsilon_1,\dots,\varepsilon_{n-k}\}}$ and its related Green's function satisfies condition $(N_g)$, then the following statements hold:
	\begin{itemize}
		\item There is $\lambda_2<0$, the least eigenvalue in absolute value of operator $T_n[\bar{M}]$ in $X_{\{\sigma_1,\dots,\sigma_k\}}^{\{\varepsilon_1,\dots,\varepsilon_{n-k}\}}.$ Moreover, there exists a nontrivial constant sign eigenfunction corresponding to the eigenvalue $\lambda_2$.
		\item Green's function related to operator $T_n[M]$ is non-positive on $I\times I$ for all $M\in[\bar{M},\bar{M}-\lambda_2)$.
		\item Green's function related to operator $T_n[M]$ cannot be non-positive on $I\times I$ for all $M>\bar{M}-\lambda_2$.
		\item If there is $M\in \mathbb{R}$ for which Green's function related to operator $T_n[M]$ is nonnegative on $I\times I$, then $M>\bar{M}-\lambda_2$.
	\end{itemize}
\end{theorem}

Next results give some relevant  properties of the intervals $N_T$ and $P_T$:

\begin{theorem}\label{T::8}
	Let $\bar{M}\in \mathbb{R}$ be fixed. If $T_n[\bar{M}]$ is an  invertible  operator in $X_{\{\sigma_1,\dots,\sigma_k\}}^{\{\varepsilon_1,\dots,\varepsilon_{n-k}\}}$ and its related Green's function satisfies condition $(P_g)$; then if the interval $N_T\neq \emptyset$, then $\sup(N_T)=\inf(P_T)$.
\end{theorem}	

\begin{theorem}\label{T::9}
	Let $\bar{M}\in \mathbb{R}$ be fixed. If $T_n[\bar{M}]$ is an invertible  operator  in $X_{\{\sigma_1,\dots,\sigma_k\}}^{\{\varepsilon_1,\dots,\varepsilon_{n-k}\}}$ and its related Green's function satisfies condition $(N_g)$; then if the interval $P_T\neq \emptyset$, then $\sup(N_T)=\inf(P_T)$.
\end{theorem}	

By using these results, we know that one of the extremes of the interval of constant sign of the Green's function, if it is not empty, is characterized by its first eigenvalue. So, the rest of the paper is devoted to characterize the other extreme of the interval, provided that it is bounded.

\chapter{Hypotheses on operator $T_n[M]$}
As we have mentioned at the introduction, the aim of this work is to generalize the results given in \cite{CabSaa} and \cite{CabSaa2}.

\vspace{0.3cm}

In \cite{CabSaa}, the problems studied are the so-called $(k,n-k)$ boundary conditions which correspond to $\{\sigma_1,\dots,\sigma_k\}=\{0,\dots,k-1\}$ and $\{\varepsilon_1,\dots,\varepsilon_{n-k}\}=\{0,\dots,n-k-1\}$. We will characterize the parameter's set where the Green's function has constant sign, by assuming that the boundary conditions satisfy $(N_a)$, which, clearly, holds for $(k,n-k)$.

By using Theorems \ref{T::4} and \ref{T::3}, under the hypothesis that \eqref{Ec::T_n[M]} is disconjugate on $[a,b]$, it is proved in \cite{CabSaa} the existence of a decomposition as follows:
\begin{equation*}
 T_0\,u(t)=u(t)\,,\quad T_k\,u(t)=\dfrac{d}{dt}\left( \dfrac{T_{k-1}\,u(t)}{v_k(t)}\right) \,,\ k = 1,\dots, n\,,
\end{equation*}
where $v_k>0$, $v_k\in C^{n}(I)$ such that
\begin{equation*} T_n[M]\,u(t)=v_1(t)\,\dots\,v_n(t)\,T_n\,u(t)\,,\ t\in I\,,\end{equation*}
and, moreover, this decomposition satisfies, for every $u\in  X_{\{0,\dots,k-1\}}^{\{0,\dots,n-k-1\}}$:
\begin{eqnarray}
\nonumber T_0\,u(a)=\cdots=T_{k-1}\,u(a)&=&0\,,\\
\nonumber T_0\,u(b)=\cdots=T_{n-k-1}\,u(b)&=&0\,.
\end{eqnarray}

In \cite{CabSaa2}, it is studied a fourth order problem coupled with the simply supported beam boundary conditions, that is, $\{\sigma_1,\sigma_2\}=\{\varepsilon_1,\varepsilon_2\}=\{0,2\}$. It is also obtained a decomposition as follows:
\begin{equation*}
T_0\,u(t)=u(t)\,,\quad T_k\,u(t)=\dfrac{d}{dt}\left( \dfrac{T_{k-1}\,u(t)}{v_k(t)}\right) \,,\ k = 1,\dots, 4\,,
\end{equation*}
where $v_k>0$, $v_k\in C^{4}(I)$ such that
\begin{equation*} T_4[M]\,u(t)=v_1(t)\,\dots\,v_4(t)\,T_4\,u(t)\,,\ t\in I\,,\end{equation*}
and, moreover, this decomposition satisfies, for every $u\in  X_{\{0,2\}}^{\{0,2\}}$:
\begin{eqnarray}
\nonumber T_0\,u(a)=T_{2}\,u(a)&=&0\,,\\
\nonumber T_0\,u(b)=T_{2}\,u(b)&=&0\,.
\end{eqnarray}

Furthermore, the simplest $n^{th}$-order operator which we can study is $T_n[0]\,u(t)=u^{(n)}(t)$. It is obvious that such an operator satisfies:
\begin{equation*}
T_0\,u(t)=u(t)\,,\quad T_k\,u(t)=\dfrac{d}{dt}\left( \dfrac{T_{k-1}\,u(t)}{v_k(t)}\right) \,,\ k = 1,\dots, n\,,
\end{equation*}
where $v_k\equiv 1$ on $I$ and
\begin{equation*}T_n[0]\,u(t)=v_1(t)\,\dots\,v_n(t)\,T_n\,u(t)\,,\ t\in I\,,\end{equation*}
and, moreover, this decomposition satisfies, for every $u\in  X_{\{\sigma_1,\dots,\sigma_k\}}^{\{\varepsilon_1,\dots,\varepsilon_{n-k}\}}$:
\begin{align}
\nonumber T_{\sigma_1}\,u(a)=u^{(\sigma_1)}(a)=0\,,&\dots\,, T_{\sigma_k}\,u(a)=u^{(\sigma_k)}=0\,,\\
\nonumber T_{\varepsilon_1}\,u(b)=u^{(\varepsilon_1)}(b)=0\,,&\dots\,,T_{\varepsilon_{n-k}}\,u(b)=u^{(\varepsilon_{n-k})}=0\,.
\end{align}

Thus, it is natural to impose that the operator $T_n[M]$ satisfies the following property.
\begin{definition}	
	Let us say that the operator $T_n[M]$ satisfies the property $(T_d)$ in $X_{\{\sigma_1,\dots,\sigma_k\}}^{\{\varepsilon_1,\dots,\varepsilon_{n-k}\}}$ if, and only if, there exists the following decomposition:
\begin{equation}
\label{Ec::Td1} T_0\,u(t)=u(t)\,,\quad T_k\,u(t)=\dfrac{d}{dt}\left( \dfrac{T_{k-1}\,u(t)}{v_k(t)}\right) \,,\ k = 1,\dots, n\,,
\end{equation}
where $v_k>0$, $v_k\in C^{n}(I)$ such that
\begin{equation}\label{Ec::Td2} T_n[M]\,u(t)=v_1(t)\,\dots\,v_n(t)\,T_n\,u(t)\,,\ t\in I\,,\end{equation}
and, moreover, such a decomposition satisfies, for every $u\in  X_{\{\sigma_1,\dots,\sigma_k\}}^{\{\varepsilon_1,\dots,\varepsilon_{n-k}\}}$:
\begin{eqnarray}
\nonumber T_{\sigma_1}\,u(a)=\cdots=T_{\sigma_k}\,u(a)&=&0\,,\\
\nonumber T_{\varepsilon_1}\,u(b)=\cdots=T_{\varepsilon_{n-k}}\,u(b)&=&0\,.
\end{eqnarray}
\end{definition}

As we have shown above, the operator $T_n[M]\,u(t)\equiv u^{(n)}(t)+M\,u(t)$ satisfies property $(T_d)$ for $M=0$. Indeed, the existence of such a decomposition for $M=\bar M$ allows to express the operator $T_n[\bar M]$ as a composition of operators of order $1$ verifying the boundary conditions given on \eqref{Ec::cfa}-\eqref{Ec::cfb}. That is, in order to study the oscillation, we can think that the operator   $T_n[\bar M]\,u(t)$ has an analogous behavior to $u^{(n)}(t)$.

\begin{remark}
	Realize that, due to Theorems \ref{T::4} and \ref{T::3}, the disconjugacy of the linear differential equation \eqref{Ec::T_n[M]} on $I$ is a necessary condition for the operator $T_n[M]$ to satisfy property $(T_d)$ on $X_{\{\sigma_1,\dots,\sigma_k\}}^{\{\varepsilon_1,\dots,\varepsilon_{n-k}\}}$.
	
	Furthermore, as it has been proved in \cite{CabSaa}, the disconjugacy hypothesis is also a sufficient condition for the operator $T_n[M]$ to satisfy property $(T_d)$ on $X_{\{0,\dots,k-1\}}^{\{0,\dots,n-k-1\}}$.
\end{remark}
	
	\begin{remark}
		Realize that there may exist different decompositions \eqref{Ec::Td1} depending on the choice of $v_k$ for $k=1,\dots,n$.
		
		Moreover, even if we are not able to obtain such a decomposition, we cannot ensure that it does not exist, unless we prove that the linear differential equation \eqref{Ec::T_n[M]} is not disconjugate.
	\end{remark}
	
	In \cite{CabSaa}, it is shown that 
	\begin{equation}\label{Ec::Tl}T_\ell u(t)=\dfrac{1}{v_1(t)\dots v_\ell(t)}\,u^{(\ell)}(t)+p_{\ell_1}(t)\,u^{(\ell-1)}(t)+\cdots p_{\ell_\ell}(t)\,u(t)\,,\end{equation}
	where $p_{\ell_i}\in C^{n-\ell}(I)$, for every $i=1,\dots,\ell$, and $\ell=0,\dots,n$.
	
	Now, let us see that 
	\begin{align}
	\label{Ec::pl1}
	p_{\ell_1}(t)=&\,a_1^1(v_1(t),\dots,v_\ell(t))\,v_1'(t)+\cdots+a_1^\ell(v_1(t),\dots,v_\ell(t))\,v_\ell'(t)\,,\\\nonumber\\\label{Ec::pl2} p_{\ell_2}(t)=&\,a_2^1(v_1(t),\dots,v_\ell(t))\,v_1''(t)+\cdots+a_2^{\ell-1}(v_1(t),\dots,v_\ell(t))\,v_{\ell-1}''(t) \\\nonumber &+f_2(v_1(t),\dots,v_\ell(t),v_1'(t),\dots,v_\ell'(t))\,,\\\nonumber\\\label{Ec::pl3} p_{\ell_3}(t)=&\,a_3^1(v_1(t),\dots,v_\ell(t))\,v_1'''(t)+\cdots+a_3^{\ell-2}(v_1(t),\dots,v_\ell(t))\,v_{\ell-2}'''(t) \\\nonumber &+f_3(v_1(t),\dots,v_\ell(t),v_1'(t),\dots,v_\ell'(t),v_1''(t),\dots,v_{\ell-1}''(t))\,,\\\nonumber
	&\vdots\\\label{Ec::pll} p_{\ell_\ell}(t)=&\,a_\ell^1(v_1(t),\dots,v_\ell(t))\,v_1^{(\ell)}(t)+ \\\nonumber &+f_\ell(v_1(t),\dots,v_\ell(t),v_1'(t),\dots,v_\ell'(t),\dots,v_1^{(\ell-1)}(t),v_2^{(\ell-1)}(t))\,,
	\end{align}
	where $a_i^j\in C^{\infty}\left( (0,+\infty)^\ell\right)$, $f_i\in C^{\infty}\left((0,\infty)^\ell\times \mathbb{R}^{\frac{\ell-i+2}{2}}\right) $ for all $\ell=0,\dots,n$, $i=1,\dots,\ell$ and $j=1,\dots,\ell-i+1$.
	
	We can see that for $\ell=1$ the result is true:
	\begin{equation}\label{Ec::T1}T_1\,u(t)=\dfrac{d}{dt}\left( \dfrac{u(t)}{v_1(t)}\right) =\dfrac{u'(t)}{v_1(t)}-\dfrac{v_1'(t)}{v_1^2(t)}\,u(t)\,,\end{equation}
	hence $a_1^1(x)=-\dfrac{1}{x^2}$.
	
	Suppose, by induction hypothesis, that the result is true for a given $\ell\geq 1$. Then, let us see what happens for $\ell+1$.
	\[T_{\ell+1}u(t)=\dfrac{d}{dt}\left(\dfrac{1}{v_1(t)\dots v_{\ell+1}(t)}\,u^{(\ell)}(t)+\dfrac{p_{\ell_1}(t)}{v_{\ell+1}(t)}\,u^{(\ell-1)}(t)+\cdots +\dfrac{p_{\ell_\ell}(t)}{v_{\ell+1}(t)}\,u(t)\right)\,,\]
	or, which is the same,
	\[T_{\ell+1}u(t)=\dfrac{1}{v_1(t)\dots v_{\ell+1}(t)}\,u^{(\ell+1)}(t)+p_{{\ell+1}_1}(t)\,u^{(\ell)}(t)+\cdots+ p_{{\ell+1}_{\ell+1}}(t)\,u(t)\,,\]
	where
	\[\begin{split}
p_{\ell+1_1}(t)=&\,\dfrac{d}{dt}\left( \dfrac{1}{v_1(t)\dots v_{\ell+1}(t)}\right) +\dfrac{p_{\ell_1}(t)}{v_{\ell+1}(t)}\,,\\\nonumber\\\nonumber
	p_{\ell+1_j}(t)=&\,\dfrac{d}{dt}\left( \dfrac{p_{\ell_{j-1}}(t)}{v_{\ell+1}(t)}\right) +\dfrac{p_{\ell_j}(t)}{v_{\ell+1}(t)}\,,\quad 2\leq j\leq \ell\,,\\\nonumber\\\nonumber
	p_{\ell+1_{\ell+1}}(t)=&\,\dfrac{d}{dt}\left( \dfrac{p_{\ell_\ell}(t)}{v_{\ell+1}(t)}\right) \,,
	\end{split}\]
	which clearly satisfy \eqref{Ec::pl1}--\eqref{Ec::pll} for $\ell+1$.
	
	\begin{exemplo}\label{Ex::1}
		Now, let us show, as an example, the expression of $T_2u(t)$:
		{\small \begin{equation*}
				\begin{split}
			T_2u(t)=\dfrac{d}{dt}\left( \dfrac{T_1u(t)}{v_2(t)}\right) =&\dfrac{u''(t)}{v_1(t)v_2(t)}-u'(t)\dfrac{2v_2(t)v_1'(t)+v_1(t)v_2'(t)}{v_1^2(t)v_2^2(t)}\\
			&+u(t)\dfrac{v_1(t)v_1'(t)v_2'(t)+v_2(t)\left( 2{v_1'}^2(t)-v_1(t)v_1''(t)\right) }{v_1^3(t)v_2^2(t)}\,.
			\end{split}
			\end{equation*}}
		
		Thus, in this case, $a_1^1(x,y)=-\dfrac{2}{x^2\,y}$, $a_1^2(x,y)=-\dfrac{1}{x\,y^2}$, $a_2^1(x,y)=-\dfrac{1}{x^2\,y}$ and $f(x,y,z,t)=\dfrac{x\,z\,t+2\,y\,z^2}{x^3\,y^2}$.
	\end{exemplo}
	
	\begin{remark}\label{R::1}
		Realize that, from the arbitrariness of the choice of $u\in X_{\{\sigma_1,\dots,\sigma_k\}}^{\{\varepsilon_1,\dots,\varepsilon_{n-k}\}}$, if the operator $T_n[\bar M]$ verifies the property $(T_d)$ on $X_{\{\sigma_1,\dots,\sigma_k\}}^{\{\varepsilon_1,\dots,\varepsilon_{n-k}\}}$ then, for each $\ell\in\{\sigma_1,\dots,\sigma_k\}$, we have that
		\[T_\ell u(a)=\dfrac{1}{v_1(a)\dots v_\ell(a)}\,u^{(\ell)}(a)+p_{\ell_1}(a)\,u^{(\ell-1)}(a)+\cdots p_{\ell_\ell}(a)\,u(a)=0\,,\]
		implies that $p_{\ell_h}(a)=0$ for each $h\in\{1,\dots,\ell\}$ such that $\ell-h\notin\{\sigma_1,\dots,\sigma_k\}$.
		
		Analogously, for each
		$\ell\in\{\varepsilon_1,\dots,\varepsilon_{n-k}\}$
		\[T_\ell u(b)=\dfrac{1}{v_1(b)\dots v_\ell(b)}\,u^{(\ell)}(b)+p_{\ell_1}(b)\,u^{(\ell-1)}(b)+\cdots p_{\ell_\ell}(b)\,u(b)=0\,,\]
		implies that $p_{\ell_h}(b)=0$ for each $h\in\{1,\dots,\ell\}$ such that $\ell-h\notin\{\varepsilon_1,\dots,\varepsilon_{n-k}\}$.
		\end{remark}
		
		Now, we deduce two results which are a straight consequence of property $(T_d)$ and previous Remark.
		
		\begin{lemma}
			\label{L::1}
			Let $\bar M\in\mathbb{R}$ be such that $T_n[\bar M]$ satisfies the property $(T_d)$ in $X_{\{\sigma_1,\dots,\sigma_k\}}^{\{\varepsilon_1,\dots,\varepsilon_{n-k}\}}$. If $u\in C^n([a,c))$, where $c>a$, is a function that satisfies $u^{(\sigma_1)}(a)=\cdots=u^{(\sigma_{\ell-1})}(a)=0$, for $\ell=1,\dots,k$, then
			\[T_{\sigma_1}\,u(a)=\cdots=T_{\sigma_{\ell-1}}\,u(a)=0\,,\]
			and 
			\[T_{\sigma_\ell}\,u(a)=f(a)\,u^{(\sigma_\ell)}(a)\,,\]
			where 	
			\begin{equation*}
			f(t)=\dfrac{1}{v_1(t)\,\dots\,v_{\sigma_\ell}(t)}>0\,,\ t\in I\,.
			\end{equation*}
			
			In particular, $u^{(\sigma_\ell)}(a)=0$ if, and only if, $T_{\sigma_\ell}u(a)=0$.
		\end{lemma}

		\begin{proof}
			We only have to take into account   expression \eqref{Ec::Tl} and Remark \ref{R::1} to deduce the result directly.			
%			Moreover, from \eqref{Ec::Tl}, we have	
%			\[T_{\sigma_k} u(a)=\dfrac{1}{v_1(a)\dots v_{\sigma_k}(a)}\,u^{(\sigma_k)}(a)+p_{{\sigma_k}_1}(a)\,u^{({\sigma_k}-1)}(a)+\cdots p_{{\sigma_k}_{\sigma_k}}(a)\,u(a)\,.\]
%			
%			From condition $(T_d)$, we know that if $u^{(\sigma_k)}(a)=0$, then $T_{\sigma_k}\,u(a)=0$, thus	
%			\[p_{{\sigma_k}_1}(a)\,u^{({\sigma_k}-1)}(a)+\cdots p_{{\sigma_k}_{\sigma_k}}(a)\,u(a)=0\,,\]
%			and the result is true.	
		\end{proof}
		
		We have an analogous result for $t=b$.
		
		\begin{lemma}
			\label{L::2}
			Let $\bar M\in\mathbb{R}$ be such that $T_n[\bar M]$ satisfies the property $(T_d)$ in $X_{\{\sigma_1,\dots,\sigma_k\}}^{\{\varepsilon_1,\dots,\varepsilon_{n-k}\}}$. If $u\in C^n((c,b])$, where $c<b$, is a function that satisfies $u^{(\varepsilon_1)}(b)=\cdots=u^{(\varepsilon_{\ell-1})}(b)=0$, for $\ell \in\{1,\dots,n-k\}$, then
			\[T_{\varepsilon_1}\,u(b)=\cdots=T_{\varepsilon_{\ell-1}}\,u(b)=0\,,\]
			and 
			\[T_{\varepsilon_{\ell}}\,u(b)=g(b)\,u^{(\varepsilon_{\ell})}(b)\,,\]
			where 
			\begin{equation*}
			g(t)=\dfrac{1}{v_1(t)\dots v_{\varepsilon_{\ell}}(t)}>0\,,\ t\in I\,.
			\end{equation*}
			
			In particular, if $u^{(\varepsilon_{\ell})}(b)=0$, then $T_{\varepsilon_{\ell}}u(b)=0$.
		\end{lemma}
		\begin{proof}
			The proof follows, as in Lemma \ref{L::1}, from \eqref{Ec::Tl} and Remark \ref{R::1}.
		\end{proof}
	
	Now, we prove a preliminary result, which ensures that Green's function is well-defined for the operator $T_n[M]$ in $X_{\{\sigma_1,\dots,\sigma_k\}}^{\{\varepsilon_1,\dots,\varepsilon_{n-k}\}}$, provided that it verifies the property $(T_d)$ on $X_{\{\sigma_1,\dots,\sigma_k\}}^{\{\varepsilon_1,\dots,\varepsilon_{n-k}\}}$.
	
	\begin{lemma}\label{L::11}
		Let $\bar{M}\in\mathbb{R}$ be such that $T_n[\bar{M}]$ satisfies the property  $(T_d)$ in $X_{\{\sigma_1,\dots,\sigma_k\}}^{\{\varepsilon_1,\dots,\varepsilon_{n-k}\}}$. Then  $\{\sigma_1,\dots,\sigma_k\}-\{\varepsilon_1,\dots,\varepsilon_{n-k}\}$ satisfy $(N_a)$ if, and only if, $M=0$ is not an eigenvalue of $T_n[\bar{M}]$ in $X_{\{\sigma_1,\dots,\sigma_k\}}^{\{\varepsilon_1,\dots,\varepsilon_{n-k}\}}$.
	\end{lemma}
	\begin{proof}
	In order to prove the sufficient condition, let us consider $u\in X_{\{\sigma_1,\dots,\sigma_k\}}^{\{\varepsilon_1,\dots,\varepsilon_{n-k}\}}$, such that 
		\[T_n[\bar{M}]\,u(t)=0\,,t\in I\,.\]
		
		We will see that necessarily $u\equiv 0$ in $I$.
		
		Since the operator $T_n[\bar{M}]$ satisfies the property $(T_d)$ in $X_{\{\sigma_1,\dots,\sigma_k\}}^{\{\varepsilon_1,\dots,\varepsilon_{n-k}\}}$, we can use the decomposition given in \eqref{Ec::Td1}; so, we have
		\[0=T_n[\bar{M}]\,u(t)=v_1(t)\,\dots\,v_n(t)\,T_n\,u(t)\,,\ t\in I\,,\]
		which, since $v_1\,,\dots,v_n>0$, implies that
		\[T_n\,u(t)=\dfrac{d}{dt}\left( \dfrac{T_{n-1}\,u(t)}{v_n(t)}\right) =0\,, t\in I\,,\]
		hence $\dfrac{T_{n-1}\,u(t)}{v_n(t)}$ is a constant function on $I$. So, since $v_n>0$ on $I$, $T_{n-1}\,u(t)$ is of constant sign on $I$.
		
		Hence  $\dfrac{T_{n-2}\,u(t)}{v_{n-2}(t)}$ is a monotone function, with at most one zero on $I$. As before, since $v_{n-2}(t)>0$ on $I$, we can conclude that $T_{n-2}\,u(t)$ can have at most one zero on $I$.
		
		Proceeding  analogously, we conclude that $u$ can have at most $n-1$ zeros on $I$.
		
%		However, each time that $T_\ell\, u(a)=0$ or $T_\ell\,u(b)=0$ for $\ell=0,\dots,n-1$, a possible oscillation is lost. 
		
		If $T_\ell u\neq 0$ for all $\ell=1,\dots,n-1$, then each time that $T_\ell u(a)=0$ or $T_\ell u(b)=0$ a possible oscillation is lost. Indeed, if the maximum number of zeros for $T_\ell\,u$ on $I$ is $h$ and one of them is found in either $t=a$ or $t=b$, then $T_\ell\,u$ can have at most $h-1$ sign changes on $I$ ($h-2$ if both $T_\ell\,u(a)=T_\ell u(b)=0$).
		
		Since $T_n[\bar{M}]$ satisfies property $(T_d)$ in  $X_{\{\sigma_1,\dots,\sigma_k\}}^{\{\varepsilon_1,\dots,\varepsilon_{n-k}\}}$, $T_{\ell}\,u(a)=0$ or $T_\ell\,u(b)=0$, at least $n$ times.

		If $T_\ell\,u(t)\neq 0$ for all $\ell=1,\dots, n-1$, $n$ possible oscillations are lost, since $u$ can have $n-1$ zeros with maximal oscillation, this implies that necessarily $u\equiv 0$.
		
		If there exists some $\ell\in \{1,\dots,n-1\}$, such that $T_\ell\,u(t)\equiv 0$ on $I$, let us choose the least $\ell$  that satisfy this property. With the same arguments as before, we can conclude that the maximum number of zeros which $u$ can have is $\ell-1$.
		
		Using the fact that $\{\sigma_1,\dots,\sigma_k\}-\{\varepsilon_1,\dots,\varepsilon_{n-k}\}$ satisfy $(N_a)$, we know that $T_h\,u(a)=0$ or $T_h\,u(b)=0$ at least $\ell$ times from $h=0$ to $\ell$. Therefore, we loose $\ell$ possible oscillations, hence $u\equiv 0$. And, we can conclude that $0$ is not an eigenvalue of $T_n[\bar{M}]$ in $X_{\{\sigma_1,\dots,\sigma_k\}}^{\{\varepsilon_1,\dots,\varepsilon_{n-k}\}}$.	
		
		\vspace{0.5cm}
		
	Reciprocally, to prove the necessary condition, let us assume that $\{\sigma_1,\dots,\sigma_k\}-\{\varepsilon_1,\dots,\varepsilon_{n-k}\}$ do not satisfy ($N_a$). Then, there exists $h_0\in\{1,\dots,n-1\}$ such that 
	\[\sum_{\sigma_j<h_0}1+\sum_{\varepsilon_j<h_0}1<h_0\,.\]
	
	Thus, there always exists a nontrivial function verifying the boundary conditions \eqref{Ec::cfa}-\eqref{Ec::cfb} for $\sigma_\ell<h_0$ and $\varepsilon_\ell<h_0$ such that $T_hu(t)=0$.
	
	Trivially, $T_{\sigma_\ell}\,u(a)=0$ and $T_{\varepsilon_\ell}\,u(b)=0$ for either $\sigma_\ell>h_0$ or $\varepsilon_\ell>h_0$. Thus, by applying Lemmas \ref{L::1} and \ref{L::2} inductively, we conclude that  $u\in X_{\{\sigma_1,\dots,\sigma_k\}}^{\{\varepsilon_1,\dots,\varepsilon_{n-k}\}}$.
	
	As a consequence, it is obvious that $M=0$ is an eigenvalue of $T_n[\bar M]$ in  $X_{\{\sigma_1,\dots,\sigma_k\}}^{\{\varepsilon_1,\dots,\varepsilon_{n-k}\}}$.
	\end{proof}

	\chapter{Study of the adjoint operator, $T_n^*[M]$}\label{S::ad}
	
	In order to obtain the characterization of the Green's function sign, as it has been done in \cite{CabSaa} and \cite{CabSaa2}, it is necessary to study the adjoint operator, $T_n^*[M]$, defined in \eqref{EC::Ad}. So, this chapter is devoted to make an analysis of such an operator and some of its properties in relation with the hypotheses on operator $T_n[M]$ given in the previous chapter.
	
	So, we describe the space $D(T_n^*[M])$, defined in \eqref{Ec::cfad} by taking into account that, in our case, $D(T_n[M])=X_{\{\sigma_1,\dots,\sigma_k\}}^{\{\varepsilon_1,\dots,\varepsilon_{n-k}\}}$. 
	
	Let us  denote $D(T_n^*[M])=X_{\ \, \{\sigma_1,\dots,\sigma_k\}}^{*\{\varepsilon_1,\dots,\varepsilon_{n-k}\}}$.
	
	Let us consider the following sets $\{\delta_1,\dots,\delta_k\}\,,\{\tau_1,\dots,\tau_{n-k}\}\subset\{0,\dots,n-1\}$, such that
\[\begin{split}	\{\sigma_1,\dots,\sigma_k,n-1-\tau_1,\dots,\,n-1-\tau_{n-k}\}=&\{0,\dots,n-1\}\\
\{\varepsilon_1,\dots,\varepsilon_{n-k},n-1-\delta_1,\dots,\,n-1-\delta_k\}=&\{0,\dots,n-1\}
	\end{split}\]
	
	\begin{remark}\label{R::ab}
		Realize that, by the definition of $\alpha$ and $\beta$ given in \eqref{Ec::alpha} and \eqref{Ec::beta}, respectively, we have that $\alpha=n-1-\tau_{n-k}$ and $\beta=n-1-\delta_k$.
	\end{remark}
	Hence we choose $u\in X_{\{\sigma_1,\dots,\sigma_k\}}^{\{\varepsilon_1,\dots,\varepsilon_{n-k}\}}$, such that 
	\begin{eqnarray}
	\nonumber u^{(n-1-\tau_1)}(a)&=&1\,,\\\nonumber
	u^{(i)}(a)&=&0\,,\quad \forall i=0,\dots,n-1\,,\quad i \neq n-1-\tau_1\,,\\\nonumber
	u^{(i)}(b)&=&0\,,\quad \forall i=0,\dots,n-1\,.
	\end{eqnarray}
	
	Thus, from \eqref{Ec::cfad} we can conclude that every $v\in X_{\ \, \{\sigma_1,\dots,\sigma_k\}}^{*\{\varepsilon_1,\dots,\varepsilon_{n-k}\}}$, satisfies
	\[v^{(\tau_1)}(a)+\sum_{j=n-\tau_1}^{n-1}(-1)^{n-j}\,(p_{n-j}\,v)^{(\tau_1+j-n)}(a)=0\,.\]
	
	Proceeding analogously for $\tau_2,\dots,\tau_{n-k}$, we can obtain the boundary conditions for the adjoint operator at $t=a$, and working at $t=b$ for $\delta_1,\dots,\delta_k$ we are able to complete the boundary conditions related to the adjoint operator.
	
	 So, we conclude that every $v\in X_{\ \, \{\sigma_1,\dots,\sigma_k\}}^{*\{\varepsilon_1,\dots,\varepsilon_{n-k}\}}$ is a $C^n(I)$ function  that satisfies the following conditions
	\begin{eqnarray}
	\label{Cf::ad1}
	v^{(\tau_1)}(a)+\sum_{j=n-\tau_1}^{n-1}(-1)^{n-j}\,(p_{n-j}\,v)^{(\tau_1+j-n)}(a)&=&0\,,\\
	\nonumber \vdots&&\\	\label{Cf::ad11}v^{(\tau_{n-k-1})}(a)+\sum_{j=n-\tau_{n-k-1}}^{n-1}(-1)^{n-j}\,(p_{n-j}\,v)^{(\tau_{n-k-1}+j-n)}(a)&=&0\,,\end{eqnarray}
		\begin{eqnarray}
		\label{Cf::ad2}
		v^{(\tau_{n-k})}(a)+\sum_{j=n-\tau_{n-k}}^{n-1}(-1)^{n-j}\,(p_{n-j}\,v)^{(\tau_{n-k}+j-n)}(a)&=&0\,,\\
			\label{Cf::ad3}
			v^{(\delta_1)}(b)+\sum_{j=n-\delta_1}^{n-1}(-1)^{n-j}\,(p_{n-j}\,v)^{(\delta_1+j-n)}(b)&=&0\,,\\\nonumber \vdots&&\\	\label{Cf::ad31}
			v^{(\delta_{k-1})}(b)+\sum_{j=n-\delta_{k-1}}^{n-1}(-1)^{n-j}\,(p_{n-j}\,v)^{(\delta_{k-1}+j-n)}(b)&=&0\,,\\
		\label{Cf::ad4}
		v^{(\delta_k)}(b)+\sum_{j=n-\delta_k}^{n-1}(-1)^{n-j}\,(p_{n-j}\,v)^{(\delta_k+j-n)}(b)&=&0\,.	
	\end{eqnarray}
	
\begin{notation}
		Let us denote $\eta$, $\gamma\in\{0,\dots,n-1\}$ as follows
		\begin{align}
		\label{Ec::eta} &\eta\notin \{\tau_1,\dots, \tau_{n-k}\}\,,\ \text{and if } \eta\neq 0\,,\ \{0,\dots, \eta-1\}\subset\{\tau_1,\dots, \tau_{n-k}\}\,,\\\label{Ec::gamma}
		&\gamma\notin \{\delta_1,\dots, \delta_{k}\}\,,\ \text{and if } \gamma\neq 0\,,\ \{0,\dots, \gamma-1\}\subset\{\delta_1,\dots, \delta_{k}\}\,.
		\end{align}
\end{notation}

\begin{remark}
	As in Remark \ref{R::ab}, we have that $\eta=n-1-\sigma_k$ and $\gamma=n-1-\varepsilon_{n-k}$.
\end{remark}
		
		From the boundary conditions \eqref{Cf::ad1}--\eqref{Cf::ad4}, since $p_{n-j}\in C^{n-j}(I)$,  the following assertions are fulfilled:
		\begin{itemize}
			\item	If $\eta\neq 0$, for all $v\in X_{\ \, \{\sigma_1,\dots,\sigma_k\}}^{*\{\varepsilon_1,\dots,\varepsilon_{n-k}\}}$ it is satisfied $v(a)=\dots=v^{(\eta-1)}(a)=0$.
			\item	If $\gamma\neq 0$, for all $v\in X_{\ \, \{\sigma_1,\dots,\sigma_k\}}^{*\{\varepsilon_1,\dots,\varepsilon_{n-k}\}}$ it is satisfied $v(b)=\dots=v^{(\gamma-1)}(b)=0$.
			\end{itemize}
			
			\begin{exemplo}\label{Ex::2}
				Let us consider the fourth order operator $T_4[M]$ coupled with the boundary conditions
				\begin{equation}
				\label{Ec::CfEx} u(a)=u''(a)=u'(b)=u''(b)=0\,.
				\end{equation}
				
				Now, we  describe the domain of definition of the adjoint operator, $T_4^*[M]$. 
				
				In this case, $\{\tau_1,\tau_2\}=\{0,2\}$ and $\{\delta_1,\delta_2\}=\{0,3\}$. Thus, from \eqref{Cf::ad1}--\eqref{Cf::ad4}, we deduce that:
				\begin{equation}
				\label{Ec::SExAd} \begin{split} X_{\,\ \{0,2\}}^{*\{1,2\}}=\left\lbrace\right. & v\in C^4(I)\ \mid\ v(a)=v''(a)-p_1(a)\,v'(a)=v(b)=0\,, \\& v^{(3)}(b)-p_1(b)\,v''(b)+(p_2(b)-2p_1'(b))v'(b)=0\left. \right\rbrace. \end{split}
				\end{equation}
			\end{exemplo}
			
		\begin{definition}
			Let us say that the operator $T^*_n[M]$ verifies the property $(T_d^*)$ in $X_{\ \, \{\sigma_1,\dots,\sigma_k\}}^{*\{\varepsilon_1,\dots,\varepsilon_{n-k}\}}$ if, and only if, there exists a decomposition:
			\begin{equation}
			\label{Ec::Td*} T_0^*\,v(t)=w_0(t)\,v(t)\,,\quad T_k^*\,v(t)=\dfrac{-1}{w_k(t)}\dfrac{d}{dt}\left( T_{k-1}^*\,v(t)\right) \,,\ k = 1,\dots, n\,,
			\end{equation}
			where $w_k>0$, $w_k\in C^{n}(I)$ and
			\[T_n^*[M]\,v(t)=T_n^*\,v(t)\,,\ t\in I\,.\]
			
			Moreover, this decomposition satisfies that for every $v\in  X_{\ \,\{ \sigma_1,\dots,\sigma_k\}}^{*\{\varepsilon_1,\dots,\varepsilon_{n-k}\}}$:
			\begin{eqnarray}
			\label{Ec::cfT*1} T^*_{\tau_1}\,v(a)=\cdots=T_{\tau_{n-k}}^*\,v(a)&=&0\,,\\
			\label{Ec::cfT*2} T^*_{\delta_1}\,v(b)=\cdots=T^*_{\delta_{k}}\,v(b)&=&0\,.
			\end{eqnarray}
		 \end{definition}	
		 
			We have the following result.
			
			\begin{lemma}\label{L::0Ad}
				Let $\bar{M}\in \mathbb{R}$ be such that $T_n[\bar{M}]$ satisfies the property $(T_d)$ on  $X_{\{\sigma_1,\dots,\sigma_k\}}^{\{\varepsilon_1,\dots,\varepsilon_{n-k}\}}$, then the adjoint operator $T_n^*[\bar{M}]$ also satisfies the property $(T_d^*)$ on  $X_{\ \, \{\sigma_1,\dots,\sigma_k\}}^{*\{\varepsilon_1,\dots,\varepsilon_{n-k}\}}$.
			\end{lemma}
			
			\begin{proof}
				From  \cite[Chapter 3, Theorem 10]{Cop}, it is fulfilled that if $T_n[\bar M]\,v$ satisfies equation \eqref{Ec::Td2}, then $T_n^*[\bar{M}]$ can be decomposed as:
				\begin{equation}
				\label{Ec::Td*2}
				T_n^*[\bar M]\,v(t)=\dfrac{(-1)^n}{v_1(t)}\,\dfrac{d}{dt}\left( \dfrac{1}{v_2(t)}\dfrac{d}{dt}\left( \cdots\dfrac{d}{dt}\left( \dfrac{1}{v_n(t)}\dfrac{d}{dt}\left( v_1(t)\dots v_n(t)\,v(t)\right) \right) \right) \right) \,.
				\end{equation}
				
				Hence,			
			\[T_0^*\,v(t)=v_1(t)\dots v_n(t)\,v(t)\,,\quad \text{and } T_k^*v(t)=\dfrac{-1}{v_{n+1-k}(t)}\dfrac{d}{dt}\left( T_{k-1}^*\,v(t)\right) \,,\]
			so, the existence of the decomposition given in \eqref{Ec::Td*} is proved by taking $w_0(t)=v_1(t)\,\dots\,v_n(t)$ and $ w_k(t)=v_{n+1-k}(t)$ for $k=1,\dots,n$.
			
			Let us see that for every $v\in  X_{\ \, \{\sigma_1,\dots,\sigma_k\}}^{*\{\varepsilon_1,\dots,\varepsilon_{n-k}\}}$, the boundary conditions \eqref{Ec::cfT*1}-\eqref{Ec::cfT*2} are satisfied.
			
			Obviously, the expression of the $n^{\mathrm{th}}$- order scalar problem \eqref{Ec::T_n[M]}--\eqref{Ec::cfb} as a first order vectorial problem, given in \eqref{Ec::vec}, does not depend on the property $(T_d)$ of $T_n[\bar M]$.
			
			In our case, using the decomposition given by $(T_d)$, we can transform the $n^{\mathrm{th}}$-order problem $T_n[\bar M]\,u(t)=0$ into a first order vectorial problem  in an alternative way as follows			
			\begin{equation}
			\label{Ec::TVec}
			U_u'(t)=A_1(t)\,U_u(t)\,,\quad t\in I\,,\quad  B\,U_u(a)+C\,U_u(b)=0\,,
			\end{equation}
			with  $B$, $C\in \mathcal{M}_{n\times n}$ defined in \eqref{Ec::Cf} and $U_u(t)\in\mathbb{R}^n$, $A_1(t)\in \mathcal{M}_{n\times n}$, defined by
			\begin{equation}\label{Ec::UA1} U_u(t)=\left( \begin{array}{c}
			{u_1}_u(t)\\{u_2}_u(t)\\\vdots\\{u_n}_u(t)\end{array}\right) \,,\quad A_1(t)=\left( \begin{array}{ccccc}
			0&v_2(t)&0&\dots&0\\0&0&v_3(t)&\dots&0\\\vdots&\vdots&\vdots&\ddots&\vdots\\0&0&0&\dots&v_n(t)\\0&0&0&\dots&0\end{array}\right)\,,\end{equation}		 
			 where ${u_\ell}_u(t):=\dfrac{T_{\ell-1}\,u(t)}{v_\ell(t)}$ for $\ell=1,\dots, n$ and $u\in  X_{\{\sigma_1,\dots,\sigma_k\}}^{\{\varepsilon_1,\dots,\varepsilon_{n-k}\}}$.
			 
			 Indeed, if $1\leq \ell\leq n-1$:
			 \[{u_\ell}_u'(t)=\dfrac{d}{dt}\left( \dfrac{T_{\ell-1}\,u(t)}{v_\ell(t)}\right) = \dfrac{T_\ell\,u(t)}{v_{\ell+1}(t)}\,v_{\ell+1}(t)=v_{\ell+1}(t)\,{u_{\ell+1}}_u(t)\,,\]
			 and, if $\ell=n$
			 \[u_n'(t)=\dfrac{d}{dt}\left( \dfrac{T_{n-1}\,u(t)}{v_{n-1}(t)}\right) =T_n\,u(t)=\dfrac{T_n[\bar M]\,u(t)}{v_1(t)\dots v_n(t)}=0\,.\]
			 
			 Taking into account that $T_n[\bar M]$ satisfies property $(T_d)$ on $ X_{ \{\sigma_1,\dots,\sigma_k\}}^{\{\varepsilon_1,\dots,\varepsilon_{n-k}\}}$, we have that
			 \begin{equation}
			 \label{Ec::Cfvec}
			 {u_{\sigma_1+1}}_u(a)=\cdots={u_{\sigma_k+1}}_u(a)={u_{\varepsilon_1+1}}_u(b)=\cdots={u_{\varepsilon_{n-k}+1}}_u(b)=0\,.
			 \end{equation}
			 
			 Moreover, using similar arguments, by means of the decomposition \eqref{Ec::Td*2}, we can transform the $n^{\mathrm{th}}$- order scalar problem \begin{equation}
			 \label{Ec::T*}T_n^*[\bar M]\,v(t)=0\,,\quad t\in I\,,
			 \end{equation}
			 coupled with the boundary conditions \eqref{Cf::ad1}--\eqref{Cf::ad4} on the following equivalent first order vectorial problem
			 \begin{equation}
			 \label{Ec::T*vec}
			 Z_v'(t)=-A_1^T (t)\,Z_v(t)\,,\quad t\in I\,,
			 \end{equation}
			 where $A_1(t)\in \mathcal{M}_{n\times n}$ is defined in \eqref{Ec::UA1} and $Z_v(t)\in\mathbb{R}^n$ is given by
			 \[Z_v(t)=\left( \begin{array}{c}{z_1}_v(t)\\{z_2}_v(t)\\\vdots\\{z_n}_v(t)\end{array}\right)\,,\]
			 with ${z_\ell}_v(t):=T_{n-\ell}^*\,v(t)$ for $\ell =0, \dots,n-1$ and $v\in X_{\ \, \{\sigma_1,\dots,\sigma_k\}}^{*\{\varepsilon_1,\dots,\varepsilon_{n-k}\}}$. 
			 
			  Indeed, if $2\leq \ell\leq n$:
			  \[{z_\ell}_v'(t)=\dfrac{d}{dt}\left( {T_{n-\ell}^*v(t)}\right) =T_{n-\ell+1}^*v(t)\,(-v_{n+1-(n-\ell+1)}(t))=-v_{\ell}(t)\,{z_{\ell-1}}_v(t)\,,\]
			  
			  and, if $\ell=1$:
			  \[{z_1}_v'(t)=\dfrac{d}{dt}\left({ T_{n-1}^*\,v(t)}\right) =- v_1(t)\,T_n^*v(t)=-v_1(t)\,T_n^*[\bar M]v(t)=0\,.\]
			 
			Let us consider the $n^{\mathrm{th}}$-order linear differential operators $T_n[\bar M]$ and $T_n^*[\bar M]$ in a vectorial way as follows:
				\begin{align}
				\nonumber {T_n}^v[\bar M]\,U_u(t)&=U_u'(t)-A_1(t)\,U_u(t)\,,\\
				\nonumber {T_n^*}^v[\bar M]\,Z_v(t)&=-Z_v'(t)-A_1^T(t)\,Z_v(t)\,,
				\end{align}
			 with $U_u(t)$, $Z_v(t)\in\mathbb{R}$ and $A_1(t)\in\mathcal{M}_{n\times n}$ previously defined.
			 
			 As it can be seen in \cite[Section 1.3]{Cab}, ${T_n^*}^v[\bar M]$ is the adjoint operator of ${T_n}^v[\bar M]$ and vice-versa As consequence, by definition of adjoint operator, we have that for every $u \in X_{ \{\sigma_1,\dots,\sigma_k\}}^{\{\varepsilon_1,\dots,\varepsilon_{n-k}\}}$ and $v\in X_{\ \, \{\sigma_1,\dots,\sigma_k\}}^{*\{\varepsilon_1,\dots,\varepsilon_{n-k}\}}$, the following equality is fulfilled
			 \[\left\langle  {T_n}^v[\bar M]\,U_u(t), Z_v(t)\right\rangle = \left\langle U_u(t), {T_n^*}^v[\bar M]\,Z_v(t)\right\rangle \,,\]
			 where $\left\langle \cdot,\cdot \right\rangle$ is the scalar product in $\mathcal{L}^2(I,\mathbb{R}^n)$. 
			 
			 Moreover, from \cite[Section 1.3]{Cab}, we have that
			 \[\left\langle U_u(a),Z_v(a)\right\rangle =\left\langle U_u(b),Z_v(b)\right\rangle \,,\quad \forall u\in X_{ \{\sigma_1,\dots,\sigma_k\}}^{\{\varepsilon_1,\dots,\varepsilon_{n-k}\}} \text{ and } v\in X_{\ \, \{\sigma_1,\dots,\sigma_k\}}^{*\{\varepsilon_1,\dots,\varepsilon_{n-k}\}}\,.\]
			 
			 Taking into account the boundary conditions \eqref{Ec::Cfvec}, we conclude that for every   $v\in X_{\ \, \{\sigma_1,\dots,\sigma_k\}}^{*\{\varepsilon_1,\dots,\varepsilon_{n-k}\}}$ it is satisfied:
			 \[{z_{n-\tau_1}}_v(a)=\cdots={z_{n-\tau_n}}_v(a)={z_{n-\delta_1}}_v(b)=\cdots={z_{n-\delta_k}}_v(b)=0\,,\]
			 which implies that
			 \[T_{\tau_1}^*v(a)=\cdots=T_{\tau_{n-k}}^*v(a)=T_{\delta_1}^*v(b)=\cdots=T_{\delta_k}^*v(b)=0.\]
			 
			 Or, which is the same, $T_n^*[\bar M]$ satisfies the property $(T_d^*)$ on  $ X_{\ \, \{\sigma_1,\dots,\sigma_k\}}^{*\{\varepsilon_1,\dots,\varepsilon_{n-k}\}}$.
			\end{proof}
			\begin{exemplo}\label{Ex::3}
				Let us consider the fourth order operator $T_4[M]$. Moreover, let us assume that $T_4[M]$ verifies  property $(T_d)$ in $X_{\{0,2\}}^{\{1,2\}}$. That is, $T_0u(a)=T_2u(a)=T_1u(b)=T_2u(b)=0$ for all $u\in X_{\{0,2\}}^{\{1,2\}}$.
				
				From  \eqref{Ec::T1} and Example \ref{Ex::2}, taking into account the boundary conditions \eqref{Ec::CfEx},  we obtain that the following equalities are fulfilled for every $u\in X_{\{0,2\}}^{\{1,2\}}$.
				\[\begin{split}
				T_0u(a)&=0\,,\\
				T_2u(a)&=-u'(a)\dfrac{2v_2(a)v_1'(a)+v_1(a)v_2'(a)}{v_1^2(a)\,v_2^2(a)}\,,\\
				T_1u(b)&=-u(b)\dfrac{v_1'(b)}{v_1^2(b)}\,,\\
				T_2u(b)&=u(b) \dfrac{v_1(b)v_1'(b)v_2'(b)+v_2(b)\left( 2{v_1'}^2(b)-v_1(b)v_1''(b)\right) }{v_1^3(b)\,v_2^2(b)}\,.
				\end{split}\]
				
				So, $T_4[M]$ satisfies the property $(T_d)$ in $X_{\{0,2\}}^{\{1,2\}}$  if, and only if, there exists a decomposition as \eqref{Ec::Td1}-\eqref{Ec::Td2}, where $v_1\,,\ v_2\in C^4(I)$ are such that:
				\begin{align}
				\label{Ec::Ex31} \dfrac{2v_1'(a)}{v_1(a)}&=-\dfrac{v_2'(a)}{v_2(a)}\,,\\
				\label{Ec::Ex32} v_1'(b)&=v_1''(b)=0\,.
				\end{align}
				
				Let us verify that in such a case, the operator $T_4^*[M]$ satisfies the property $(T_d^*)$ in $X_{\,\ \{0,2\}}^{*\{1,2\}}$.
				
				In order to do that, we express $p_1$ and $p_2$ as functions of $v_1$, $v_2$, $v_3$ and $v_4$.
				
				Expanding the related expression \eqref{e-descomp} for $n=4$, we obtain that
				\[
				p_1\equiv -\dfrac{4v_1'}{v_1}-\dfrac{3v_2'}{v_2}-\dfrac{2v_3'}{v_3}-\dfrac{v_4'}{v_4}\,,\]
				and
				\[\begin{split}p_2\equiv&\dfrac{12 {v_1'}^2}{v_1^2}+\dfrac{6{v_2'}^2}{v_2}+\dfrac{2{v_3'}^2}{v_3^2}+\dfrac{9v_1'\,v_2'}{v_1\,v_2}+\dfrac{6v_1'\,v_3'}{v_1\,v_3}+\dfrac{4v_2'\,v_3'}{v_2\,v_3}+\dfrac{3v_1'\,v_4'}{v_1\,v_4}+\dfrac{2v_2'\,v_4'}{v_2\,v_4}\\
				&+\dfrac{v_3'\,v_4'}{v_3\,v_4}-\dfrac{6v_1''}{v_1}-\dfrac{3v_2''}{v_2}-\dfrac{v_3''}{v_3}\,.
				\end{split}\]
				
				Moreover,
				\[p_1'\equiv \dfrac{4{v_1'}^2}{v_1^2}+\dfrac{3{v_2'}^2}{v_2^2}+\dfrac{2{v_3'}^2}{v_3^2}+\dfrac{{v_4'}^2}{v_4^2}-\dfrac{4v_1''}{v_1}-\dfrac{3v_2''}{v_2}-\dfrac{2v_3''}{v_3}-\dfrac{v_4''}{v_4}\,.\]
				
				Taking into account \eqref{Ec::Ex31}-\eqref{Ec::Ex32}, the boundary conditions for the adjoint operator, given in Example \ref{Ex::2}, can be expressed in terms of $v_1$, $v_2$, $v_3$ and $v_4$ as follows:
			{\small 	\begin{align}
				\label{Ec::Ex33}v(a)=v''(a)+\left( \dfrac{v_2'(a)}{v_2(a)}+\dfrac{2 v_3'(a)}{v_3(a)}+\dfrac{v_4'(a)}{v_4(a)}\right) v'(a)=v(b)&=0,\\
				\label{Ec::Ex34} v^{(3)}(b)+\left( \dfrac{3v_2'(b)}{v_2(b)}+\dfrac{2 v_3'(b)}{v_3(b)}+\dfrac{v_4'(b)}{v_4(b)}\right) v''(b)+\left( \dfrac{4v_2'(b)v_3'(b)}{v_2(b)\,v_3(b)}-\dfrac{2{v_3'}^2(b)}{v_3^2(b)}\right.&\\\nonumber \left. +\dfrac{2\,v_2'(b)v_4'(b)}{v_2(b)\,v_4(b)}+\dfrac{v_3'(b)v_4'(b)}{v_3(b)\,v_4(b)}-\dfrac{2{v_4'}^2(b)}{v_4^2(b)}+\dfrac{3v_2''(b)}{v_2(b)}+\dfrac{3v_3''(b)}{v_3(b)}+\dfrac{2v_4''(b)}{v_4(b)}\right) v'(b)&=0\,.
				\end{align}}

			Now, let us see that $T_0^*v(a)=T_2^*v(a)=T_0^*v(b)=T_3^*v(b)=0$ for all $v\in X_{\ \,\{0,2\}}^{*\{1,2\}}$.
			
			Trivially, $T_0^*v(a)=v(a)=0$ and $T_0^*v(b)=v(b)=0$.
			
			Using the decomposition \eqref{Ec::Td*}, we have:
			\[T_2^*v(t)=-\dfrac{1}{v_3(t)}\dfrac{d}{dt}\left( \dfrac{-1}{v_4(t)}\dfrac{d}{dt}\left( v_1(t)\,v_2(t)\,v_3(t)\,v_4(t)\,v(t)\right)\right)\,,\]
			from which, considering \eqref{Ec::Ex31} and \eqref{Ec::Ex33}, we obtain
			\[T_2^*v(a)=v_1(a)\,v_2(a)\left( v''(a)+\left( \dfrac{v_2'(a)}{v_2(a)}+\dfrac{2v_3'(a)}{v_3(a)}+\dfrac{v_4'(a)}{v_4(a)}\right) v'(a)\right) =0\,.\]
			
			Finally,
			\[T_3^*v(t)=-\dfrac{1}{v_2(t)}\dfrac{d}{dt}\left( \dfrac{-1}{v_3(t)}\dfrac{d}{dt}\left( \dfrac{-1}{v_4(t)}\dfrac{d}{dt}\left( v_1(t)\,v_2(t)\,v_3(t)\,v_4(t)\,v(t)\right)\right)\right) \,,\]
			combining the previous expression with \eqref{Ec::Ex32} --\eqref{Ec::Ex34}, we obtain:
				{\footnotesize	\begin{align}\nonumber
					T_3^*v(b)=&-v_1(b)\left(  v^{(3)}(b)+\left( \dfrac{3v_2'(b)}{v_2(b)}+\dfrac{2 v_3'(b)}{v_3(b)}+\dfrac{v_4'(b)}{v_4(b)}\right) v''(b)+\left( \dfrac{4v_2'(b)v_3'(b)}{v_2(b)\,v_3(b)}-\dfrac{2{v_3'}^2(b)}{v_3^2(b)}\right.\right. \\\nonumber &\left. \left. +\dfrac{2\,v_2'(b)v_4'(b)}{v_2(b)\,v_4(b)}+\dfrac{v_3'(b)v_4'(b)}{v_3(b)\,v_4(b)}-\dfrac{2{v_4'}^2(b)}{v_4^2(b)}+\dfrac{3v_2''(b)}{v_2(b)}+\dfrac{3v_3''(b)}{v_3(b)}+\dfrac{2v_4''(b)}{v_4(b)}\right) v'(b)\right)=0\,.
					\end{align}}
				\vspace{-0.3cm}	
							
				 As a particular case of Lemma \ref{L::0Ad}, we have proved that if $T_4[M]$ satisfies the property $(T_d)$ in $X_{\{0,2\}}^{\{1,2\}}$, then $T_4^*[M]$ satisfies the property $(T_d^*)$ in $X_{\ \,\{0,2\}}^{*\{1,2\}}$.
			\end{exemplo}
			It is obvious that we can enunciate an analogous result to Lemma \ref{L::0Ad} referring to operator $\widehat T_n[(-1)^n\bar M]$ defined in \eqref{Ec::Tg}.
				\begin{lemma}\label{L::0}
					Let $\bar{M}\in \mathbb{R}$ be such that $T_n[\bar{M}]$ satisfies the property $(T_d)$ on  $X_{\{\sigma_1,\dots,\sigma_k\}}^{\{\varepsilon_1,\dots,\varepsilon_{n-k}\}}$, then the operator $\widehat T_n[(-1)^n\bar{M}]$ also satisfies the property $(T_d^*)$ on  $X_{\ \, \{\sigma_1,\dots,\sigma_k\}}^{*\{\varepsilon_1,\dots,\varepsilon_{n-k}\}}$.
				\end{lemma}
				\begin{proof}
					We only have to consider $\widehat{T}_\ell v(t)=(-1)^\ell T_\ell^*v(t)$, $\ell=0,\dots, 1$ and the result follows directly from Lemma \ref{L::0Ad}.
				\end{proof}
				
		Arguing as in \cite{CabSaa} to deduce the equality \eqref{Ec::Tl}, let us see that the expression of $\widehat T_\ell v(t)$ is given by
				\begin{equation}\label{Ec::Tgg}\widehat T_\ell v(t)=v_1(t)\,\dots\,v_{n-\ell}\,v^{(\ell)}(t)+\widehat p_{\ell_1}(t)\,v^{(\ell-1)}(t)+\cdots+\widehat p_{\ell_\ell}(t)\,v(t)\,,\end{equation}
				where $\widehat p_{\ell_i}\in C^{n-\ell}(I)$.
				
				For $\ell=0$, we have that $\widehat T_0v(t)=v_1(t)\dots v_n(t)\,v(t)$.
				
				Let us assume that \eqref{Ec::Tgg} is true for a given $\ell\geq 0$, then, by \eqref{Ec::Td*}, we have				
				\[\widehat{T}_{\ell+1}v(t)=\dfrac{1}{v_{n-\ell}(t)}\dfrac{d}{dt}\left( \widehat T_{\ell}v(t)\right)\,.\]
			
			Thus, using the induction hypothesis,
					\[\widehat{T}_{\ell+1}v(t)=\dfrac{1}{v_{n-\ell}(t)}\dfrac{d}{dt}\left( v_1(t)\,\dots\,v_{n-\ell}\,v^{(\ell)}(t)+\widehat p_{\ell_1}(t)\,v^{(\ell-1)}(t)+\cdots+\widehat p_{\ell_\ell}(t)\,v(t)\right)\,,\]
				which follows the expression \eqref{Ec::Tgg} for $\ell+1$.
				
			As a consequence of the previous results, we are able to obtain analogous results to Lemmas \ref{L::1} and \ref{L::2} for $\widehat{T}_n[(-1)^n M]$.
				 \begin{lemma}
				 	\label{L::3}
				 	Let $\bar M\in\mathbb{R}$ be such that $T_n[(-1)^n\bar M]$ satisfies the property $(T_d^*)$ in $X_{\ \, \{\sigma_1,\dots,\sigma_k\}}^{*\{\varepsilon_1,\dots,\varepsilon_{n-k}\}}$. If $v\in C^n([a,c))$, with $c>a$, is a function that satisfies \eqref{Cf::ad1}--\eqref{Cf::ad11}, then
				 	\[\widehat T_{\tau_1}\,v(a)=\cdots=\widehat T_{\tau_{n-k-1}}\,v(a)=0\,,\]
				 	and 
				 	\[\widehat T_{\tau_{n-k}}\,v(b)=\widehat f(a)\left( 	v^{(\tau_{n-k})}(a)+\sum_{j=n-\tau_{n-k}}^{n-1}(-1)^{n-j}\,(p_{n-j}\,v)^{(\tau_{n-k}+j-n)}(a)\right) \,,\]
				 	where  	$\widehat f(t)= v_1(t)\dots v_{n-{\tau_{n-k}}}(t)>0$ on $I$
				 	
				 	In particular, if $v$ satisfies \eqref{Cf::ad2}, then $\widehat T_{\tau_{n-k}}v(a)=0$.
				 \end{lemma}
				 \begin{proof}
				 	The proof is analogous to the one given in Lemma \ref{L::1}, but in this case we have that
			{\scriptsize 	 	\[ \widehat T_{\tau_{n-k}} v(a)=v_1(a)\,\dots\,v_{n-\tau_{n-k}}(a)\,v^{({\tau_{n-k}})}(a)+\widehat p_{{\tau_{n-k}}_1}(a)\,v^{(\tau_{n-k}-1)}(a)+\cdots+\widehat p_{{\tau_{n-k}}_{\tau_{n-k}}}(a)\,v(a)\,.\]}
			If \eqref{Cf::ad2} is satisfied, then $\widehat T_{\tau_{n-k}\,v(a)}=0$, and the result is true.				 	
				 \end{proof}
				
				 \begin{lemma}
				 	\label{L::4}
				 	Let $\bar M\in\mathbb{R}$ be such that $T_n[(-1)^n\bar M]$ satisfies the property $(T_d^*)$ in $X_{\ \, \{\sigma_1,\dots,\sigma_k\}}^{*\{\varepsilon_1,\dots,\varepsilon_{n-k}\}}$. If $v\in C^n((c,b])$, with $c<b$, is a function that satisfies \eqref{Cf::ad3}--\eqref{Cf::ad31}, then
				 	\[\widehat T_{\delta_1}\,v(b)=\cdots=\widehat T_{\delta_{k-1}}\,v(b)=0\,,\]
				 	and 
				 	\[\widehat T_{\delta_{k}}\,v(b)=\widehat g(a)\left( 	v^{(\delta_{k})}(b)+\sum_{j=n-\delta_{k}}^{n-1}(-1)^{n-j}\,(p_{n-j}\,v)^{(\delta_{k}+j-n)}(b)\right) \,,\]
				 	where $\widehat g(t)= v_1(t)\dots v_{n-\delta_k}(t)>0$ on $I$.
				 	
				 	In particular, if $v$ satisfies \eqref{Cf::ad4}, then $\widehat T_{\delta_{k}}v(b)=0$.
				 \end{lemma}
				 \begin{proof}
				 	The proof is analogous to the one given in Lemma \ref{L::3}.
				 \end{proof}
				
				\chapter{Strongly inverse positive (negative) character of operator $T_n[\bar M]$}
			
				In this section we  prove that if  operator $T_n[\bar M]$ satisfies the property $(T_d)$, then it is a strongly inverse positive (negative) operator on $X_{\{\sigma_1,\dots,\sigma_k\}}^{\{\varepsilon_1,\dots,\varepsilon_{n-k}\}}$. Moreover, its related Green's function satisfies a suitable condition, which allows us to apply either Theorem \ref{T::6} or Theorem \ref{T::7} and obtain one of the extremes of the interval where the related Green's function is of constant sign. The result is the following.
				
				\begin{theorem}
					\label{L::5}
					Let $\bar M\in \mathbb{R}$ be such that $T_n[\bar M]$ satisfies condition $(T_d)$ in $X_{\{\sigma_1,\dots,\sigma_k\}}^{\{\varepsilon_1,\dots,\varepsilon_{n-k}\}}$ and ${\{\sigma_1,\dots,\sigma_k\}}-{\{\varepsilon_1,\dots,\varepsilon_{n-k}\}}$ satisfy condition $(N_a)$. Then the following properties are fulfilled:
					\begin{itemize}
						\item If $n-k$ is even, then $T_n[\bar M]$ is strongly inverse positive in $X_{\{\sigma_1,\dots,\sigma_k\}}^{\{\varepsilon_1,\dots,\varepsilon_{n-k}\}}$ and, moreover, the related Green's function, $g_{\bar M}(t,s)$, satisfies $(P_g)$.
							\item If $n-k$ is odd, then $T_n[\bar M]$ is strongly inverse negative in $X_{\{\sigma_1,\dots,\sigma_k\}}^{\{\varepsilon_1,\dots,\varepsilon_{n-k}\}}$ and, moreover, the related Green's function, $g_{\bar M}(t,s)$, satisfies $(N_g)$.
					\end{itemize}
				\end{theorem}
				\begin{proof}
					Firstly, let us verify the strongly inverse positive (negative) character.
					
					To this end, we  use the decomposition of $T_n[\bar M]$ given on \eqref{Ec::Td1}.
					
					Since $v_1(t)\,\dots\,v_n(t)>0$; if $T_n[\bar M]\,u\gneqq 0$ on $I$, from \eqref{Ec::Td2}, we conclude that $T_n u\gneqq 0$ on $I$.
					
					Hence, from \eqref{Ec::Td1} we know that $\dfrac{T_{n-1}\,u}{v_n}$ is a nontrivial nondecreasing function, with at most a  sign change on $I$. Therefore, since $v_n>0$, we can affirm that $T_{n-1}u$ can have at most a sign change, being negative at $t=a$ and positive at $t=b$.
					
					Repeating this process for $T_{n-\ell}u$, with $\ell = 1,\dots n$, we can affirm that $T_0u=u$ can have at most $n$ zeros on $(a,b)$, whenever  the following inequalities are satisfied for every $\ell=1,\dots,n$:
					\begin{equation}\label{Ec::Maxos}
					\left\lbrace \begin{array}{cc}
					T_{n-\ell}\,u(a)>0\,,& \text{if $\ell$ is even,}\\\\
					T_{n-\ell}\,u(a)<0\,,& \text{if $\ell$ is odd,}\end{array}\right.\qquad\text{and}\qquad T_{n-\ell}\,u(b)>0\,.
					\end{equation}
					
					Repeating the same argument as in Lemma \ref{L::11}, we can affirm that each time that $T_{n-\ell}\,u(a)=0$ or $T_{n-\ell}\,u(b)=0$, we loose a possible oscillation and, therefore, a possible zero of $u$ in $(a,b)$.
					
					From the property $(T_d)$, we know that for all $u\in X_{\{\sigma_1,\dots,\sigma_k\}}^{\{\varepsilon_1,\dots,\varepsilon_{n-k}\}}$  \begin{equation}
					\label{Ec::CFT}
					T_{\sigma_1}\,u(a)=\cdots=T_{\sigma_k}\,u(a)=T_{\varepsilon_1}\,u(b)=\cdots=T_{\varepsilon_{n-k}}\,u(b)=0\,,\end{equation} i.e, we loose the $n$ possible zeros which $u$ could ever have. Thus, we can conclude that $u$ cannot have any zero on $(a,b)$.
					
						Let us see how is the sign of $u^{(\alpha)}(a)$ and $u^{(\beta)}(b)$ which gives the sign of $u$.
					
					Realize that, since $u(a)=\cdots=u^{(\alpha-1)}(a)=0$ and $u(b)=\cdots=u^{(\beta-1)}(b)=0$, from \eqref{Ec::Tl} we have
					\begin{equation}
					\label{Ec::AB}
					T_\alpha u(a)=\dfrac{u^{(\alpha)}(a)}{v_1(a)\dots v_\alpha(a)}\,,\quad T_\beta u(b)=\dfrac{u^{(\beta)}(b)}{v_1(b)\dots v_\beta(b)}\,,\end{equation}
					hence, $u^{(\alpha)}(a)$ and $T_\alpha u(a)$, and $u^{(\beta)}(b)$ and $T_\beta u(b)$  have the same sign, respectively.
					
					If either, $T_\ell u(a)=0$ for any $\ell \notin\{\sigma_1,\dots,\sigma_k\}$, or $T_\ell u(t)=0$ for any $\ell\notin \{\varepsilon_1,\dots,\varepsilon_{n-k}\}$, then we loose another possible oscillation and, necessarily, $u\equiv0$ on $I$ which is a contradiction with $T_n[\bar M]\,u\gneqq0$.

				Moreover, taking into account \eqref{Ec::CFT}, the sign of $T_\ell u(a)$ must allow the maximum number of oscillations for $T_\ell u$.  Otherwise  $u\equiv0$ on $I$ which is again a contradiction with $T_n[\bar M]\,u\gneqq0$.
				
					\begin{notation}\label{Not:maxos}
					Along this work, we understand for conditions of maximal oscillation those which allow $u$ to have the maximum number of zeros depending on the fixed boundary conditions without being a trivial solution.
				\end{notation}
			
			Hence $T_{n-\ell}$ must verify the conditions for maximal oscillation.	That is, $T_{n-\ell}u(a)$ must change its sign each time that it is not null, i.e., if $T_{n-\ell}u(a)>0$ for a given $\ell=1\,,\dots,\,n$, then $T_{n-\ell-1}u(a)\leq0$ and if  $T_{n-\ell-1}u(a)=0$, we consider $\tilde{\ell}\in \{\ell+1,\dots,n\}$\, such that $T_{n-\tilde{\ell}}u(a)\neq 0$ and $T_{n-h}\,u(a)=0$ for $h\in\{\ell+1,\dots,\tilde{\ell}-1\}$, then $T_{n-\tilde{\ell}}u(a)< 0$.
					
					From the property $(T_d)$, we know that $T_{n-\ell}u(a)$ vanishes $k-\alpha$ times for $\ell\in\{1,\dots,n-\alpha\}$. Hence, taking into account the previous argument and the conditions given in \eqref{Ec::Maxos}, we have
					\[\left\lbrace \begin{array}{cl}
					T_{\alpha}\,u(a)>0\,,& \text{if $n-\alpha-(k-\alpha)=n-k $ is even,}\\\\
					T_{\alpha}\,u(a)<0\,,& \text{if $n-k$ is odd.}\end{array}\right.\]
					
					Realize that, to obtain the previous inequalities, there are considered as many sign changes  for $T_h u(a)$ as times that it is non null from $h=\alpha$ to $h=n-1$. That is, the $n-\alpha$ steps minus the $k-\alpha$ zeros that are found. Thus, from \eqref{Ec::AB}
						\begin{equation}\label{Ec::ualpha}\left\lbrace \begin{array}{cc}
						u^{(\alpha)}(a)>0\,,& \text{if $n-k$ is even,}\\\\
						u^{(\alpha)}(a)<0\,,& \text{if $n-k$ is odd.}\end{array}\right.\end{equation}
						
						From this, since $u\neq 0$ on $(a,b)$, we already  conclude that
							\begin{equation}\label{Ec::uab}\left\lbrace \begin{array}{ccc}
							u(t)>0\,,&t\in(a,b)\,,& \text{if $n-k$ is even,}\\\\
							u(t)<0\,,&t\in(a,b)\,,& \text{if $n-k$ is odd.}\end{array}\right.\end{equation}
							
%							Let us see now which is the sign of $T_\beta u(b)$.
%							
%							From \eqref{Ec::Maxos}, we know that if $T_{n-\ell}\,u(b)\neq0$, it does not change its sign for the different $\ell=1,\dots,n$. However, each time that $T_{n-\ell}\,u(b)=0$, since a oscillation is lost, it has a sign change.
%							
%							 From $\ell=1$ to $n-\beta$, $T_{n-\ell}u(b)$ has $n-k-\beta$ zeros. Then, with maximal oscillation:
%							 	\[\left\lbrace \begin{array}{cl}
%							 	T_{\beta}\,u(b)>0& \text{if $n-k-\beta $ is even,}\\\\
%							 	T_{\beta}\,u(b)<0& \text{if $n-k-\beta$ is odd,}\end{array}\right.\]
%							 	since $T_\beta u(b)$ and $u^{(\beta)}(b)$ have the same sign, previous inequalities also hold for $u^{(\beta)}(b)$, and we can write them in the following way:

Taking into account that necessarily $T_\beta u(b)\neq 0$, since $\beta\notin\{\varepsilon_1,\dots,\varepsilon_{n-k}\}$, from \eqref{Ec::AB} and \eqref{Ec::uab} we have
							 	
							 	\begin{itemize}
							 		\item If $n-k$ is even
							 			\begin{equation}\label{Ec::ubetapar}\left\lbrace \begin{array}{cc}
							 			u^{(\beta)}(b)>0\,,& \text{if $\beta$ is even,}\\\\
							 			u^{(\beta)}(b)<0\,,& \text{if $\beta$ is odd.}\end{array}\right.\end{equation}
							 				\item If $n-k$ is odd
							 				\begin{equation}\label{Ec::ubetaimpar}\left\lbrace \begin{array}{cc}
							 				u^{(\beta)}(b)<0\,,& \text{if $\beta$ is even,}\\\\
							 				u^{(\beta)}(b)>0\,,& \text{if $\beta$ is odd.}\end{array}\right.\end{equation}
								\end{itemize}

							 		Hence, from \eqref{Ec::ualpha}--\eqref{Ec::ubetaimpar}, we conclude that if $n-k$ is even, then the operator  $T_n[\bar M]$ is a strongly inverse positive operator in $X_{\{\sigma_1,\dots,\sigma_k\}}^{\{\varepsilon_1,\dots,\varepsilon_{n-k}\}}$ and if $n-k$ is odd, then the operator $T_n[\bar M]$ is a strongly inverse negative operator in $X_{\{\sigma_1,\dots,\sigma_k\}}^{\{\varepsilon_1,\dots,\varepsilon_{n-k}\}}$.
							 		
							 		\vspace{0.5cm}

							 		Let us see that $g_{\bar M}(t,s)$ satisfies condition $(P_g)$ or $(N_g)$, respectively.
							 		
							 		Using  Theorem \ref{T::in2}, it is known that $(-1)^{n-k}\,g_{\bar M}(t,s)>0$ for a.e. $(t,s)\in(a,b)\times(a,b)$. Let us see that, in fact,  this inequality holds for all $(t,s)\in(a,b)\times(a,b)$.

							 		For each fixed $s\in (a,b)$, let us define $u_s(t)=(-1)^{n-k}\,g_{\bar M}(t,s)$, $u_s\in C^{n-2}(I)$ and  $u_s\in C^{n}([a,s)\cup(s,b])$.
							 		
							 		It is known that $u_s(t)\geq 0$ on $I$, and that it satisfies the boundary conditions \eqref{Ec::cfa}-\eqref{Ec::cfb}.
							 		
							 		Moreover, since $g_{\bar M}(t,s)$ is the Green's function related to the operator $T_n[\bar M]$ in $X_{\{\sigma_1,\dots,\sigma_k\}}^{\{\varepsilon_1,\dots,\varepsilon_{n-k}\}}$, we have
							 		\[T_n[\bar M]\,u_s(t)=v_1(t)\dots v_n(t)\,T_n\,u_s(t)=0\,,\quad t\neq s\,.\]
							 		
							 		Since $v_1\dots v_n>0$ on $I$, $T_n\,u_s(t) =0$ if $t\neq s$. Hence, 							 		
							 			\begin{equation}\left\lbrace \begin{array}{cc}
							 			\dfrac{1}{v_n(t)}\,T_{n-1}u_s(t)=c_1\,,& t<s\,,\\\\
							 			\dfrac{1}{v_n(t)}\,T_{n-1}u_s(t)=c_2\,,& t>s\,,\end{array}\right.\end{equation}
							 		where $c_1$, $c_2\in \mathbb{R}$ are of different sign to allow the maximal oscillation.
							 		
							 		Since $v_n>0$, $T_{n-1}\,u_s$ has the same sign as $c_1$ or $c_2$, if $t<s$ or $t>s$, respectively, i.e., in order to have maximal number of oscillations, it has two components of constant different sign.
							 		
							 		Then, since $\dfrac{1}{v_{n-1}}\,T_{n-2}\,u_s$ is a continuous function, it can have at most two sign changes and the same happens with $T_{n-2}\,u_s$.
							 		
							 		Proceeding in a similar way, we conclude that with maximal oscillation $T_{n-\ell}\,u_s$ can have at most $\ell$ zeros, for $\ell = 2,\dots,n$. In particular, $u_s$ has at most $n$ sign changes on $I$.
							 		
							 		Arguing as before, each time that $T_{n-\ell}\,u_s(a)=0$ or $T_{n-\ell}\,u_s(b)=0$ a possible oscillation is lost.

							 	 	 Taking into account that $T_n[\bar M]$ satisfies $(T_d)$ in $X_{\{\sigma_1,\dots,\sigma_k\}}^{\{\varepsilon_1,\dots,\varepsilon_{n-k}\}}$, we use  Lemmas \ref{L::1} and \ref{L::2} to affirm that $u_s$ verifies \eqref{Ec::CFT}.
%							 	 \begin{equation}
%							 	 \label{Ec::Tus}
%							 	 T_{\sigma_1}\,u_s(a)=\cdots = T_{\sigma_k}\,u_s(a)= T_{\varepsilon_1}\,u_s(b)=\cdots= T_{\varepsilon_{n-k}}\,u_s(b)=0\,,
%							 	 \end{equation}
							 	 Thus, $T_{n-\ell}\,u_s(a)$ or $T_{n-\ell}\,u_s(b)$ vanish $n$ times for $\ell =1,\dots, n$. So, we have lost the $n$ possibles zeros and we can affirm that $u_s>0$ on $(a,b)$. Or, which is the same, $(-1)^{n-k}\,g_{\bar M}(t,s)>0$ for all $(t,s)\in(a,b)\times(a,b)$.
							 	
							 		Moreover, for each $s\in (a,b)$, we obtain the following limits:
							 		\begin{eqnarray}\nonumber \ell_1(s)& = &\lim_{t\rightarrow a^+}\dfrac{(-1)^{n-k}\,g_{\bar M}(t,s)}{(t-a)^\alpha\,(b-t)^\beta} = \dfrac{(-1)^{n-k}\,\dfrac{\partial^\alpha }{\partial t^\alpha}\,g_{\bar M}(t,s)_{\mid t=a}}{\alpha!\, (b-a)^{\beta}}\,,\\\nonumber
							 		\ell_2(s)& = &\lim_{t\rightarrow b^-}\dfrac{(-1)^{n-k}\,g_{\bar M}(t,s)}{(t-a)^\alpha\,(b-t)^\beta} = \dfrac{(-1)^{n-k-\beta}\,\dfrac{\partial^\beta }{\partial t^\beta}\,g_{\bar M}(t,s)_{\mid t=a}}{\beta!\, (b-a)^{\alpha}}\,. \end{eqnarray}

							 		For each $s\in(a,b)$, let us construct the continuous extension on $I$ of $u_s$, as follows 
							 		\[\tilde{u}_s(t)=\dfrac{(-1)^{n-k}\,g_{\bar M}(t,s)}{(t-a)^\alpha\,(b-t)^\beta}\,.\]
							 		
							 		Since $u_s>0$ and $(t-a)^\alpha\,(b-t)^\beta>0$ on $(a,b)$, we have that $\tilde{u}_s>0$ on $(a,b)$. 
							 		
							 		Moreover, 	using Theorem \ref{T::in2}, we can affirm that $\ell_1(s)>0$ and $\ell_2(s)>0$ for a.e. $s\in(a,b)$. Hence, for a.e. $s\in(a,b)$, $\tilde{u}_s(a)>0$ and $\tilde{u}_s(b)>0$.
							 		
							 		Furthermore, since $g_{\bar M}(t,s)$ is the related Green's function of $T_n[\bar M]$ on $X_{\{\sigma_1,\dots,\sigma_k\}}^{\{\varepsilon_1,\dots,\varepsilon_{n-k}\}}$, we also can affirm that there exists $K>0$ such that $\tilde{u}_s\leq K$ for every $(t,s)\in I\times (a,b)$. Hence, we construct the following functions:							 		
							 		\begin{eqnarray}\nonumber 
							 		\tilde{k}_1(s)&=&\min_{t\in I} \tilde{u}_s(t)\,,\quad s\in (a,b)\,,\\\nonumber\\\nonumber
							 		\tilde{k}_2(s) &=& \max_{t\in I}\tilde{u}_s(t)\,,\quad s\in (a,b)\,, \end{eqnarray}
							 		which are continuous on $(a,b)$ and they are positive a.e. in $(a,b)$.
							 		
							 		Taking $\phi(t)=(t-a)^\alpha\,(b-t)^\beta>0$  on $(a,b)$, condition $(P_g)$ is trivially satisfied if $n-k$ is even with $k_1(s)=\tilde{k}_1(s)$ and $k_2(s)=\tilde{k}_2(s)$ and condition $(N_g)$  if $n-k$ is odd with $k_1(s)=-\tilde{k}_2(s)$ and $k_2(s)=-\tilde{k}_1(s)$.						
				\end{proof}
				
				\begin{remark}
					Realize that, from Theorem \ref{L::5}, if $T_n[\bar M]$ satisfies property $(T_d)$ on $X_{\{\sigma_1,\dots,\sigma_k\}}^{\{\varepsilon_1,\dots,\varepsilon{n-k}\}}$, then either Theorem \ref{T::6}, if $n-k$ is even, or Theorem \ref{T::7}, if $n-k$ is odd, can be applied to operator $T_n[\bar M]$ on such a space.
				\end{remark}
				 
				 \begin{exemplo}
				 	\label{Ex::4} 
				 	Let us continue the study of the fourth order operator given in Example \ref{Ex::2}. From Example \ref{Ex::3}, we can affirm that $T_4[M]$ satisfies condition $(T_d)$ if, and only if, there exists a decomposition \eqref{Ec::Td1}-\eqref{Ec::Td2} such that \eqref{Ec::Ex31}-\eqref{Ec::Ex32} are satisfied. These equalities are true, in particular, if we choose $v_1(t)=v_2(t)=v_3(t)=v_4(t)=1$ for all $t\in I$. That is, they are valid for the particular case of operator $T_4^0[0]\,u(t)=u^{(4)}(t)$. Such an choice has been done in order to simplify the calculations, the applicability of the results can be extended to a more complicated class of operators.
				 	
				 	Now, let us check directly that this operator verifies the thesis of Theorem \ref{L::5}. To do that, let us consider $I\equiv[0,1]$.
				 	
				 	In this case, $n-k=2$ is even, so let us study the strongly inverse positive character. If $u^{(4)}\gneqq0$, then $u''$ is a convex function. Since $u''(0)=u''(1)=0$, we have that $u''\lneqq0$ (if $u''\equiv0$, then $u^{(4)}\equiv0$ which is a contradiction).
				 	
				 	Hence, $u'$ is a decreasing function on $I$ verifying $u'(1)=0$, so $u'\gneqq0$. In particular, $u'(0)>0$. 
				 	
				 	Finally, taking into account that $u(0)=0$, $u$ is an increasing function on $I$ and it cannot have infinite zeros without being a trivial solution of $T_4^0[0]\,u(t)=0$, we have that $u(t)>0$ for all $t\in(0,1]$.
				 	
				 	Now, let us study the related Green's function, given by the expression:
				 	\[g_0(t,s)=\begin{cases}
				 	\dfrac{1}{6} s \left(t \left(t^2-3 t+3\right)-s^2\right)\,, & 0\leq s\leq t\leq 1\,, \\ \\
				 	\dfrac{1}{6} (s-1) t \left(t^2-3 s\right)\,, & 0<t<s\leq 1\,.
				 	\end{cases}\]
				 	
				 	Let us see that it verifies the condition $(P_g)$. 
				 	
				 First, it is obvious that $g_0(1,s)=\dfrac{1}{6}\,s\,\left( 2-s^2\right) >0$ for all $s\in (0,1)$.
				 	
				 	Moreover,
				 	\[\dfrac{\partial}{\partial t}g_0(t,s)=\begin{cases}
				 	\dfrac{1}{6} s \left(t^2+(2 t-3) t-3 t+3\right)\,, & 0\leq s\leq t\leq 1\,, \\ \\
				 	\dfrac{1}{3} (s-1) t^2+\frac{1}{6} (s-1) \left(t^2-3 s\right)\,, & 0<t<s\leq 1\,,
				 	\end{cases}\]
				 	in particular, $\dfrac{\partial}{\partial t}g_0(t,s)_{\mid t=0}=\dfrac{1}{2}\left( s-s^2\right) >0$ for all $s\in (0,1)$.
				 	
				 	Now, let us verify that $g_0(t,s)>0$ on $(0,1)\times (0,1)$.
				 	
				 	If $t<s$, we have that $s-1<0$ and $t^2-3s<-3s+s^2<0$ for all $s\in(0,1)$.
				 	
				 	If $t\geq s$, we have $t \left(t^2-3 t+3\right)-s^2\geq s \left(s^2-3 s+3\right)-s^2=3s-4s^2+s^3>0$ for all $s\in(0,1)$.
				 	
				 	Hence, $g_0(t,s)>0$ on $(0,1)\times (0,1)$.
				 	
				 	On the other hand,				 	
				 		\[\tilde{u}_s(t)=\dfrac{g_0(t,s)}{t}=\begin{cases}
				 		\dfrac{1}{6} \dfrac{s}{t} \left(t \left(t^2-3 t+3\right)-s^2\right)\,, & 0\leq s\leq t\leq 1\,, \\ \\
				 		\dfrac{1}{6} (s-1)  \left(t^2-3 s\right)\,, & 0<t<s\leq 1\,.
				 		\end{cases}\]
				 		
				 		Thus, condition $(P_g)$ is satisfied for the following functions:
				 		\[\begin{split}
				 		\phi(t)&=t\,,\\
				 		k_1(s)&=\tilde{k}_1(s)=\min_{t\in I}\tilde{u}_s(t)=\dfrac{1}{6}\,s\,(1-s^2)\,,\\
				 		k_2(s)&=\tilde{k}_2(s)=\max_{t\in I}\tilde{u}_s(t)=\dfrac{s}{2}\,(1-s)\,,\\
				 		\end{split}\]
				 		and
				 			\[t\,\dfrac{1}{6}\,s\,(1-s^2)\leq g_0(t,s)\leq t\,\dfrac{s}{2}\,(1-s)\,,\quad \text{for all } (t,s)\in I \times I \,.\]
				 \end{exemplo}
				 
				 \chapter[Existence and study of the eigenvalues]{Existence and study of the eigenvalues of operator $T_n[\bar M]$ in different spaces.}
				 In \cite{CabSaa} and \cite{CabSaa2}, the characterization of the parameters set for which the related Green's function is of constant sign has been done by means of spectral theory. In fact, the extremes of the interval are characterized by suitable eigenvalues of the operator associated to different boundary conditions. 
				 
				 The characterization here obtained follows the same structure. Thus, in this chapter we  study the existence of eigenvalues of  operator $T_n[\bar M]$ in the different spaces \[X_{\{\sigma_1,\dots,\sigma_k\}}^{\{\varepsilon_1,\dots,\varepsilon_{n-k}\}}\!,\  X_{\{\sigma_1,\dots,\sigma_k|\alpha\}}^{\{\varepsilon_1,\dots,\varepsilon_{n-k-1}\}}\!,\ X_{\{\sigma_1,\dots,\sigma_{k-1}\}}^{\{\varepsilon_1,\dots,\varepsilon_{n-k}|\beta\}}\!,\ X_{\{\sigma_1,\dots,\sigma_{k-1}|\alpha\}}^{\{\varepsilon_1,\dots,\varepsilon_{n-k}\}} \text{ and } X_{\{\sigma_1,\dots,\sigma_k\}}^{\{\varepsilon_1,\dots,\varepsilon_{n-k-1}|\beta\}}\!.\]
				 
				 Moreover, we study the constant sign of several solutions of the linear differential equation \eqref{Ec::T_n[M]} coupled with different $n-1$ additional boundary conditions.
				 
				 Firstly, let us see a result which allows us to affirm that, under the hypothesis that the property $(T_d)$ is fulfilled on $X_{\{\sigma_1,\dots,\sigma_k\}}^{\{\varepsilon_1,\dots,\varepsilon_{n-k}\}}$, the operator $T_n[\bar M]$ verifies such a property in all these spaces.
				 
				 \begin{lemma}
				 	\label{L::6}
				 	Let $\bar M\in \mathbb{R}$ be such that $T_n[\bar M]$ satisfies the property $(T_d)$ in $X_{\{\sigma_1,\dots,\sigma_k\}}^{\{\varepsilon_1,\dots,\varepsilon_{n-k}\}}$. Then the following properties are fulfilled:
				 	\begin{itemize}
				 		\item $T_n[\bar M]$ verifies the property $(T_d)$ in $X_{\{\sigma_1,\dots,\sigma_k|\alpha\}}^{\{\varepsilon_1,\dots,\varepsilon_{n-k-1}\}}.$
				 		\item $T_n[\bar M]$ verifies the property $(T_d)$ in $X_{\{\sigma_1,\dots,\sigma_{k-1}\}}^{\{\varepsilon_1,\dots,\varepsilon_{n-k}|\beta\}}.$
				 		\item If $\sigma_k\neq k-1$, $T_n[\bar M]$ verifies the property $(T_d)$ in $X_{\{\sigma_1,\dots,\sigma_{k-1}|\alpha\}}^{\{\varepsilon_1,\dots,\varepsilon_{n-k}\}}.$
				 		\item If $\varepsilon_{n-k}\neq n-k-1$, $T_n[\bar M]$ verifies the property $(T_d)$ in $X_{\{\sigma_1,\dots,\sigma_k\}}^{\{\varepsilon_1,\dots,\varepsilon_{n-k-1}|\beta\}}.$
				 	\end{itemize}
				 \end{lemma}
				 
				 \begin{proof}
				 	The proof follows trivially from Lemmas \ref{L::1} and \ref{L::2}, taking into account that under our hypothesis, from \eqref{Ec::Tl}, we have
				 	\begin{equation}\label{Ec::Tab}T_\alpha\,u(a)=\dfrac{u^{(\alpha)}(a)}{v_1(a)\dots v_\alpha(a)}\,,\quad T_\beta\,u(b)=\dfrac{u^{(\beta)}(b)}{v_1(b)\dots v_\beta(b)}\,.\end{equation}
				 				 \end{proof}
				 
				 \begin{remark}
				 	Realize that if $\sigma_k=k-1$ or $\varepsilon_{n-k}=n-k-1$, then $\alpha=k$ or $\beta=n-k$, respectively.
				 	 So, if either $u\in X_{\{\sigma_1,\dots,\sigma_{k-1}|\alpha\}}^{\{\varepsilon_1,\dots,\varepsilon_{n-k}\}}$ or $u\in X_{\{\sigma_1,\dots,\sigma_k\}}^{\{\varepsilon_1,\dots,\varepsilon_{n-k-1}|\beta\}}$, then \eqref{Ec::Tab} can be not true.
				 \end{remark}
				 \begin{exemplo}
				 	\label{Ex::5}
				 	Let us consider the fourth order operator $T_4[M]$. In Example \ref{Ex::3}, we have seen that if $T_4[M]$ verifies $(T_d)$ in $X_{\{0,2\}}^{\{1,2\}}$, then \eqref{Ec::Ex31}-\eqref{Ec::Ex32} are fulfilled. Let us see that, in such a case, $(T_d)$ also holds in $X_{\{0,1,2\}}^{\{1\}}$, $X_{\{0\}}^{\{0,1,2\}}$, $X_{\{0,1\}}^{\{1,2\}}$ and $X_{\{0,2\}}^{\{0,1\}}$.
				 	
				 	\begin{itemize}
				 		\item $X_{\{0,1,2\}}^{\{1\}}$:
				 		
				 		Trivially, since $T_\ell u(t)$ is a linear combination of $u(t)\,,\dots,u^{(\ell)}(t)$, $T_0u(a)=T_1u(a)=T_2u(a)=0$.
				 		
				 		Moreover, from \eqref{Ec::T1}, $T_1u(b)=-\dfrac{v_1'(b)}{v_1^2(b)}u(b)=0$.
				 		
				 		\item $X_{\{0\}}^{\{0,1,2\}}$:
				 		
				 		Obviously,  $T_0u(a)=T_0u(b)=T_1u(b)=T_2u(b)=0$.
				 		
				 		\item $X_{\{0,1\}}^{\{1,2\}}$:
				 		
				 		Directly, $T_0u(a)=T_1u(a)=0$.
				 		
				 		From \eqref{Ec::T1} and Example \ref{Ex::1}, $T_1u(b)=\dfrac{-v_1'(b)}{v_1^2(b)}u(b)=0$ and \[T_2u(b)=\dfrac{v_1(b)\,v_1'(b)\,v_2'(b)+v_2(b)\,\left( 2{v_1'}^2(b)-v_1(b)\,v_1''(b)\right) }{v_1^3(b)\,v_2^2(b)}u(b)=0\,.\]
				 		
				 		\item $X_{\{0,2\}}^{\{0,1\}}$:
				 		
				 		Trivially, $T_0u(a)=T_0u(b)=T_1u(b)=0$.
				 		
				 		Finally, from Example \ref{Ex::1}, $T_2u(a)=-\dfrac{2v_2(a)\,v_1'(a)+v_1(a)\,v_2'(a)}{v_1^2(a)\,v_2^2(a)}u'(a)=0$.
				 	\end{itemize}
				 \end{exemplo}
				As a consequence we can prove the following corollary.
				 
				 \begin{corollary}
							 		Let $\bar M\in \mathbb{R}$ be such that $T_n[\bar M]$ satisfies the property $(T_d)$ in $X_{\{\sigma_1,\dots,\sigma_k\}}^{\{\varepsilon_1,\dots,\varepsilon_{n-k}\}}$, and  ${\{\sigma_1,\dots,\sigma_k\}}-{\{\varepsilon_1,\dots,\varepsilon_{n-k}\}}$ satisfy $(N_a)$. Then
				 		\begin{itemize}
				 			\item If $n-k$ is even:
				 		 	\begin{itemize}
				 			\item $T_n[\bar M]$ is strongly inverse positive and verifies condition $(P_g)$ on $X_{\{\sigma_1,\dots,\sigma_k\}}^{\{\varepsilon_1,\dots,\varepsilon_{n-k}\}}$, $X_{\{\sigma_1,\dots,\sigma_{k-1}|\alpha\}}^{\{\varepsilon_1,\dots,\varepsilon_{n-k}\}}$ and $X_{\{\sigma_1,\dots,\sigma_k\}}^{\{\varepsilon_1,\dots,\varepsilon_{n-k-1}|\beta\}}$.
				 			\item $T_n[\bar M]$ is strongly inverse negative and verifies condition $(N_g)$ on  $X_{\{\sigma_1,\dots,\sigma_k|\alpha\}}^{\{\varepsilon_1,\dots,\varepsilon_{n-k-1}\}}$ and $X_{\{\sigma_1,\dots,\sigma_{k-1}\}}^{\{\varepsilon_1,\dots,\varepsilon_{n-k}|\beta\}}$.
				 		\end{itemize}
				 		
				 		\item If $n-k$ is odd:
				 		\begin{itemize}
				 			\item $T_n[\bar M]$ is strongly inverse negative and verifies condition $(N_g)$ on $X_{\{\sigma_1,\dots,\sigma_k\}}^{\{\varepsilon_1,\dots,\varepsilon_{n-k}\}}$, $X_{\{\sigma_1,\dots,\sigma_{k-1}|\alpha\}}^{\{\varepsilon_1,\dots,\varepsilon_{n-k}\}}$ and $X_{\{\sigma_1,\dots,\sigma_k\}}^{\{\varepsilon_1,\dots,\varepsilon_{n-k-1}|\beta\}}$.
				 			\item $T_n[\bar M]$ is strongly inverse positive and verifies condition $(P_g)$ on  $X_{\{\sigma_1,\dots,\sigma_k|\alpha\}}^{\{\varepsilon_1,\dots,\varepsilon_{n-k-1}\}}$ and $X_{\{\sigma_1,\dots,\sigma_{k-1}\}}^{\{\varepsilon_1,\dots,\varepsilon_{n-k}|\beta\}}$.
				 		\end{itemize}
				 		\end{itemize}
				 \end{corollary}
				 \begin{proof}
				 	It is obvious that if ${\{\sigma_1,\dots,\sigma_k\}}-{\{\varepsilon_1,\dots,\varepsilon_{n-k}\}}$ satisfy $(N_a)$, then ${\{\sigma_1,\dots,\sigma_k|\alpha\}}-{\{\varepsilon_1,\dots,\varepsilon_{n-k-1}\}}$, ${\{\sigma_1,\dots,\sigma_{k-1}\}}-{\{\varepsilon_1,\dots,\varepsilon_{n-k}|\beta\}}$ also do.
				 	
				 	Moreover, if $\sigma_k\neq k-1$, then $\alpha <\sigma_k$ and if $\varepsilon_{n-k}\neq n-k-1$, then $\beta<\varepsilon_{n-k}$. So, if ${\{\sigma_1,\dots,\sigma_k\}}-{\{\varepsilon_1,\dots,\varepsilon_{n-k}\}}$ satisfy $(N_a)$, then  ${\{\sigma_1,\dots,\sigma_{k-1}| \alpha\}}-{\{\varepsilon_1,\dots,\varepsilon_{n-k}\}}$ and ${\{\sigma_1,\dots,\sigma_k\}}-{\{\varepsilon_1,\dots,\varepsilon_{n-k-1}|\beta\}}$ also do.
				 	
				 	Thus, using Theorem \ref{L::5} and Lemma \ref{L::6},  the result is true.
					 \end{proof}
				 	
				 	 Now, from the previous Corollary and the first assertion on Theorems \ref{T::6} and \ref{T::7}, we obtain, as a direct consequence, the following result.
				 	 
				 	  \begin{corollary}
				 	 	\label{C::1}
				 	 	Let $\bar M\in \mathbb{R}$ be such that $T_n[\bar M]$ satisfies the property $(T_d)$ in $X_{\{\sigma_1,\dots,\sigma_k\}}^{\{\varepsilon_1,\dots,\varepsilon_{n-k}\}}$, and  ${\{\sigma_1,\dots,\sigma_k\}}-{\{\varepsilon_1,\dots,\varepsilon_{n-k}\}}$ satisfy $(N_a)$. Then
				 	 	\begin{itemize}
				 	 		\item If $n-k$ is even:
				 	 		\begin{itemize}
				 	 			\item There is $\lambda_1>0$, the least positive eigenvalue of $T_n[\bar M]$ in $X_{\{\sigma_1,\dots,\sigma_k\}}^{\{\varepsilon_1,\dots,\varepsilon_{n-k}\}}.$ Moreover, there exists a nontrivial constant sign eigenfunction corresponding to the eigenvalue $\lambda_1$.
				 	 			\item If $k>1$, there is $\lambda_2'<0$, the biggest negative eigenvalue of $T_n[\bar M]$ in $X_{\{\sigma_1,\dots,\sigma_{k-1}\}}^{\{\varepsilon_1,\dots,\varepsilon_{n-k}|\beta\}}.$ Moreover, there exists a nontrivial constant sign eigenfunction corresponding to the eigenvalue $\lambda_2'$.
				 	 			\item If $k<n-1$, there is $\lambda_2''<0$, the biggest negative eigenvalue of $T_n[\bar M]$ in $X_{\{\sigma_1,\dots,\sigma_k|\alpha\}}^{\{\varepsilon_1,\dots,\varepsilon_{n-k-1}\}}.$ Moreover, there exists a nontrivial constant sign eigenfunction corresponding to the eigenvalue $\lambda_2''$.
				 	 			\item If $\sigma_k\neq k-1$, there is $\lambda_3'>0$, the least positive eigenvalue of $T_n[\bar M]$ in $X_{\{\sigma_1,\dots,\sigma_{k-1}|\alpha\}}^{\{\varepsilon_1,\dots,\varepsilon_{n-k}\}}.$ Moreover, there exists a nontrivial constant sign eigenfunction corresponding to the eigenvalue $\lambda_3'$.	
				 	 			\item If $\varepsilon_{n-k}\neq n-k-1$, there is $\lambda_3''>0$, the least positive eigenvalue of $T_n[\bar M]$ in $X_{\{\sigma_1,\dots,\sigma_k\}}^{\{\varepsilon_1,\dots,\varepsilon_{n-k-1}|\beta\}}.$ Moreover, there exists a nontrivial constant sign eigenfunction corresponding to the eigenvalue $\lambda_3''$.		
				 	 		\end{itemize}
				 	 		\item If $n-k$ is odd:
				 	 		\begin{itemize}
				 	 			\item There exists $\lambda_1<0$, the biggest negative eigenvalue of $T_n[\bar M]$ in $X_{\{\sigma_1,\dots,\sigma_k\}}^{\{\varepsilon_1,\dots,\varepsilon_{n-k}\}}.$ Moreover, there is a nontrivial constant sign eigenfunction corresponding to the eigenvalue $\lambda_1$.
				 	 			\item If $k>1$, there is $\lambda_2'>0$, the least positive eigenvalue of $T_n[\bar M]$ in $X_{\{\sigma_1,\dots,\sigma_{k-1}\}}^{\{\varepsilon_1,\dots,\varepsilon_{n-k}|\beta\}}.$ Moreover, there exists a nontrivial constant sign eigenfunction corresponding to the eigenvalue $\lambda_2'$.
				 	 			\item If $k<n-1$, there is $\lambda_2''>0$, the least positive eigenvalue of $T_n[\bar M]$ in $X_{\{\sigma_1,\dots,\sigma_k|\alpha\}}^{\{\varepsilon_1,\dots,\varepsilon_{n-k-1}\}}.$ Moreover, there exists a nontrivial constant sign eigenfunction corresponding to the eigenvalue $\lambda_2''$.
				 	 			\item If $\sigma_k\neq k-1$, there is $\lambda_3'<0$, the biggest negative eigenvalue of $T_n[\bar M]$ in $X_{\{\sigma_1,\dots,\sigma_{k-1}|\alpha\}}^{\{\varepsilon_1,\dots,\varepsilon_{n-k}\}}.$ Moreover, there exists a nontrivial constant sign eigenfunction corresponding to the eigenvalue $\lambda_3'$.	
				 	 			\item If $\varepsilon_{n-k}\neq n-k-1$, there is $\lambda_3''<0$, the biggest negative eigenvalue of $T_n[\bar M]$ in $X_{\{\sigma_1,\dots,\sigma_k\}}^{\{\varepsilon_1,\dots,\varepsilon_{n-k-1}|\beta\}}.$ Moreover, there exists a nontrivial constant sign eigenfunction corresponding to the eigenvalue $\lambda_3''$.		
				 	 		\end{itemize}
				 	 	\end{itemize}
				 	 \end{corollary}
				 	 
				 	\begin{exemplo}
				 		\label{Ex::6}
				 		Continuing the study of the particular operator $T_4^0[M]u(t)=u^{(4)}+M\,u(t)$ introduced in Example \ref{Ex::4}, we can affirm the existence of the eigenvalues of $T_4^0[0]$ in the different spaces introduced in Example \ref{Ex::5} and the related constant sign eigenfunctions.
				 		
				 		In the sequel, we  obtain those eigenvalues and related eigenfunctions.
				 		
				 		\begin{itemize}
				 			\item The eigenvalues of $T_4^0[0]$ in $X_{\{0,2\}}^{\{1,2\}}$ are given by $\lambda=m^4$, where $m$ is a positive solution of the following equation:
				 			\begin{equation}
				 			\label{Ec::Ex61}
				 			\tan(m)+\tanh(m)=0\,.\end{equation}
				 			
				 			The least positive eigenvalue is $\lambda_1=m_1^4\approxeq 2,36502^4$, where $m_1$ is the least positive solution of \eqref{Ec::Ex61}. The related constant sign eigenfunctions are given by:
				 			
				 			\[u(t)=K\,\left( \dfrac{\sinh(m_1\,t)}{\cosh(m_1)}+\dfrac{\sin(m_1\,t)}{\cos(m_1)}\right) \,,\]
				 			where $K\in \mathbb{R}$.
				 			
				 			\item The biggest negative eigenvalue of $T_4^0[0]$ in $X_{\{0,1,2\}}^{\{1\}}$ is $\lambda_2''=-4\,\pi^4$. The related constant sign eigenfunctions are given by:				 			
				 			\[u(t)=K\,\left(\cosh(\pi\,t)\,\sin(\pi\,t)-\cos(\pi\,t)\,\sinh(\pi\,t)\right) \,,\]
				 			where $K\in \mathbb{R}$.
				 			
				 			\item The eigenvalues of $T_4^0[0]$ in $X_{\{0\}}^{\{0,1,2\}}$ are given by $\lambda=-m^4$, where $m$ is a positive solution of the following equation:
				 			\begin{equation}
				 			\label{Ec::Ex62}
				 			\tan\left( \dfrac{m}{\sqrt{2}}\right) -\tanh\left( \dfrac{m}{\sqrt{2}}\right)=0\,.\end{equation}
				 			
				 			The biggest negative eigenvalue is $\lambda_2'=-m_2^4\approxeq -5,550305^4$, where $m_2$ is the least positive solution of \eqref{Ec::Ex62}. The related constant sign eigenfunctions are given by:				 			
				 				\[u(t)=K\,\left(\cosh\left( \dfrac{m_2}{\sqrt{2}}\,t\right) \,\sin\left( \dfrac{m_2}{\sqrt{2}}\,t\right)-\cos\left( \dfrac{m_2}{\sqrt{2}}\,t\right)\,\sinh\left( \dfrac{m_2}{\sqrt{2}}\,t\right)\right) \,,\]
				 			where $K\in \mathbb{R}$.
				 			
				 			\item The eigenvalues of $T_4^0[0]$ in $X_{\{0,2\}}^{\{0,1\}}$ are given by $\lambda=m^4$, where $m$ is a positive solution of the following equation:
				 			\begin{equation}
				 			\label{Ec::Ex63}
				 			\tan(m)-\tanh(m)=0\,.\end{equation}
				 			
				 			The least positive eigenvalue is $\lambda_3''=m_3^4\approxeq 3,9266^4$, where $m_3$ is the least positive solution of \eqref{Ec::Ex63}. The related constant sign eigenfunctions are given by:				 			
				 			\[u(t)=K\,\left( \dfrac{\sinh(m_3\,t)}{\cosh(m_3)}-\dfrac{\sin(m_3\,t)}{\cos(m_3)}\right) \,,\]
				 			where $K\in \mathbb{R}$.
				 			
				 			\item The least positive eigenvalue of $T_4^0[0]$ in $X_{\{0,1\}}^{\{1,2\}}$ is $\lambda_3'=\pi^4$. The related constant sign eigenfunctions are given by:
				 			\[u(t)=K\,e^{-\pi  (t+1)} \left(e^{2 \pi  t}+e^{\pi  t} \left(\left(e^{\pi }-1\right) \sin (\pi  t)+\left(-1-e^{\pi }\right) \cos (\pi  t)\right)+e^{\pi }\right)\,,\]
				 			where $K\in \mathbb{R}$.
				 		\end{itemize}
				 	\end{exemplo}
			 	
				 	Now, we  introduce some results that provide sufficient conditions to ensure that suitable solutions of \eqref{Ec::T_n[M]} are of constant sign.
				 	
				 	\begin{proposition}
				 		\label{P::1}
				 		Let $\bar M\in \mathbb{R}$ be such that $T_n[\bar M]$ satisfies property $(T_d)$ on $X_{\{\sigma_1,\dots,\sigma_k\}}^{\{\varepsilon_1,\dots,\varepsilon_{n-k}\}}$  and ${\{\sigma_1,\dots,\sigma_k\}}-{\{\varepsilon_1,\dots,\varepsilon_{n-k}\}}$ satisfy $(N_a)$. If $u\in C^n(I)$ is a solution of \eqref{Ec::T_n[M]} on $(a,b)$, satisfying the boundary conditions:
				 		\begin{eqnarray}
				 		\label{Ec::cfaa} u^{(\sigma_1)}(a)=\cdots=u^{(\sigma_{k-1})}(a)&=&0\,,\\
				 		\label{Ec::cfbb} u^{(\varepsilon_1)}(b)=\cdots=u^{(\varepsilon_{n-k})}(b)&=&0\,,
				 		\end{eqnarray}
				 		then it does not have any zero on $(a,b)$ provided that one of the following assertions is satisfied:
				 		\begin{itemize}
				 			\item Let $n-k$ be even:
				 			\begin{itemize}
				 				\item If $k>1$, $\sigma_k\neq k-1$ and $M\in[\bar{M}-\lambda_3',\bar{M}-\lambda_2']$, where:
				 				\begin{itemize}
				 					\item $\lambda_3'>0$ is the least positive eigenvalue of $T_n[\bar M]$ in $X_{\{\sigma_1,\dots,\sigma_{k-1}|\alpha\}}^{\{\varepsilon_1,\dots,\varepsilon_{n-k}\}}.$
				 					\item $\lambda_2'<0$ is the biggest negative eigenvalue of $T_n[\bar M]$ in $X_{\{\sigma_1,\dots,\sigma_{k-1}\}}^{\{\varepsilon_1,\dots,\varepsilon_{n-k}|\beta\}}.$	
				 				\end{itemize}
				 				\item If $k=1$, $\sigma_1\neq 0$ and $M\in[\bar{M}-\lambda_3',+\infty)$, where:
				 				\begin{itemize}
				 					\item $\lambda_3'>0$ is the least positive eigenvalue of $T_n[\bar M]$ in $X_{\{\alpha\}}^{\{\varepsilon_1,\dots,\varepsilon_{n-k}\}}$, where $\alpha=0$.
				 				\end{itemize}
				 					\item If $k>1$, $\sigma_k = k-1$ and $M\in[\bar{M}-\lambda_1,\bar{M}-\lambda_2']$, where:
				 					\begin{itemize}
				 						\item $\lambda_1>0$ is the least positive eigenvalue of $T_n[\bar M]$ in $X_{\{0,\dots,k-1\}}^{\{\varepsilon_1,\dots,\varepsilon_{n-k}\}}.$
				 						\item $\lambda_2'<0$ is the biggest negative eigenvalue of $T_n[\bar M]$ in $X_{\{0,\dots,k-2\}}^{\{\varepsilon_1,\dots,\varepsilon_{n-k}|\beta\}}.$	
				 					\end{itemize}	
				 				\item If $k=1$ and  $\sigma_1=0$ and $M\in[\bar{M}-\lambda_1,+\infty)$, where:	
				 					\begin{itemize}
				 						\item $\lambda_1>0$ is the least positive eigenvalue of $T_n[\bar M]$ in $X_{\{0\}}^{\{\varepsilon_1,\dots,\varepsilon_{n-1}\}}.$
				 					\end{itemize}	
				 				
				 			\end{itemize}
				 				\item Let $n-k$ be odd:
				 				\begin{itemize}
				 					\item If $k>1$, $\sigma_k\neq k-1$ and $M\in[\bar{M}-\lambda_2',\bar{M}-\lambda_3']$, where:
				 					\begin{itemize}
				 						\item $\lambda_3'<0$ is the biggest negative eigenvalue of $T_n[\bar M]$ in $X_{\{\sigma_1,\dots,\sigma_{k-1}|\alpha\}}^{\{\varepsilon_1,\dots,\varepsilon_{n-k}\}}.$
				 						\item $\lambda_2'>0$ is the least positive eigenvalue of $T_n[\bar M]$ in $X_{\{\sigma_1,\dots,\sigma_{k-1}\}}^{\{\varepsilon_1,\dots,\varepsilon_{n-k}|\beta\}}.$	
				 					\end{itemize}
				 					\item If $k=1$, $\sigma_1\neq 0$ and $M\in(-\infty,\bar{M}-\lambda_3']$, where:
				 					\begin{itemize}
				 						\item $\lambda_3'<0$ is the biggest negative eigenvalue of $T_n[\bar M]$ in $X_{\{\alpha\}}^{\{\varepsilon_1,\dots,\varepsilon_{n-1}\}}$, where $\alpha=0$.
				 					\end{itemize}
				 					\item If $k>1$, $\sigma_k = k-1$ and $M\in[\bar{M}-\lambda_2',\bar{M}-\lambda_1]$, where:
				 					\begin{itemize}
				 						\item $\lambda_1<0$ is the biggest negative eigenvalue of $T_n[\bar M]$ in $X_{\{0,\dots,k-1\}}^{\{\varepsilon_1,\dots,\varepsilon_{n-k}\}}.$
				 						\item $\lambda_2'>0$ is the least positive eigenvalue of $T_n[\bar M]$ in $X_{\{0,\dots,k-2\}}^{\{\varepsilon_1,\dots,\varepsilon_{n-k}|\beta\}}.$	
				 					\end{itemize}	
				 					\item If $k=1$ and  $\sigma_1=0$ and $M\in(-\infty,\bar{M}-\lambda_1]$, where:	
				 					\begin{itemize}
				 						\item $\lambda_1<0$ is the biggest negative eigenvalue of $T_n[\bar M]$ in $X_{\{0\}}^{\{\varepsilon_1,\dots,\varepsilon_{n-k}\}}.$
				 					\end{itemize}	
				 					
				 				\end{itemize}
				 		\end{itemize}
				 	\end{proposition}
				 	\begin{proof}
				 		Firstly, let us see that for $M=\bar M$ every solution of \eqref{Ec::T_n[M]} on $(a,b)$ satisfying the boundary conditions \eqref{Ec::cfaa}-\eqref{Ec::cfbb} does not have any zero on $(a,b)$. On the proof of Lemma \ref{L::11} we have seen that, without taking into account the boundary conditions, every solution of \eqref{Ec::T_n[M]} for $M=\bar M$ has at most $n-1$ zeros on $(a,b)$. Let us prove that this $n-1$ possible oscillations are not attained because of the boundary conditions.

				 		Let us denote, $u_M\in C^n(I)$ a solution of \eqref{Ec::T_n[M]} verifying the boundary conditions \eqref{Ec::cfaa}-\eqref{Ec::cfbb}. 
				 		
				 		Each time that $T_{n-\ell}\,u_M(a)=0$ or $T_{n-\ell}\,u_M(b)=0$ for $\ell= 1,\dots n$ a possible oscillation is lost.

				 		 Since $T_n[\bar M]$ verifies property $(T_d)$ in $X_{\{\sigma_1,\dots,\sigma_k\}}^{\{\varepsilon_1,\dots,\varepsilon_{n-k}\}}$, by applying Lemmas \ref{L::1} and \ref{L::2} we conclude that for every $M\in\mathbb{R}$
				 		\begin{eqnarray}
				 		\label{Ec::cfTaa} T_{\sigma_1}\,u_{ M}(a)=\cdots=T_{\sigma_{k-1}}\,u_M(a)&=&0\,,\\
				 		\label{Ec::cfTbb} T_{\varepsilon_1}\,u_M(b)=\cdots=T_{\varepsilon_{n-k}}\,u_M(b)&=&0\,.
				 		\end{eqnarray}
				 		
				 		In particular, this property holds for $M=\bar M$. Hence, we loose the  $n-1$ possible oscillations and we can affirm that $u_{\bar M}$ does not have any zero on $(a,b)$.
				 		
				 		%\vspace{0.5cm}
				 		
				 		Now, in order to prove the result, let us move $u_M$ in a continuous way with $M$ in a neighborhood of $\bar M$.  We have that $u_M$ is a solution of \eqref{Ec::T_n[M]} on $(a,b)$, hence				 		
				 		\begin{equation}\label{Ec::TNM}T_n[\bar M]\,u_{M}(t)=(\bar M-M)\,u_M(t)\,,\quad  t\in (a,b)\,.\end{equation}
				 		
				 		First, let us see that while $u_M$ is of constant sign it cannot have any double zero on $(a,b)$.
				 		
				 		Let us assume that $u_{\bar M}>0$ on $(a,b)$ (if $u_{\bar M}<0$ on $(a,b)$ the arguments are valid by multiplying by $-1$). Thus, in equation \eqref{Ec::TNM} we have	 		
				 			\begin{equation}\left\lbrace \begin{array}{ccc}
				 			T_n[\bar M]\,u_M(t)\geq 0\,,&t\in (a,b)\,,& \text{ if $M<\bar M$,}\\&&\\
				 				T_n[\bar M]\,u_M(t)\leq 0\,,&t\in (a,b)\,,& \text{ if $M>\bar M$.}\end{array}\right.\end{equation}
				 				
				 		In both cases, $T_n[\bar M]\,u_M$ is a constant sign function. Then, since $v_1\dots v_n>0$, $T_{n-1}\,u_M$ is a monotone function with, at most, one zero.
				 		
				 		Under analogous arguments, we conclude that $T_{n-\ell}\,u_M$ has at most $\ell$ zeros, for $\ell=1,\dots,n$. In particular, $u_M$ can have $n$ zeros at most.
				 		
				 		But, $u_M$ satisfies \eqref{Ec::cfTaa}-\eqref{Ec::cfTbb}, i.e., $n-1$ possible oscillations are lost. Thus $u_M$ is only allowed to have a simple zero on $(a,b)$, but this is not possible while it is of constant sign.	
				 		
				 		Let us assume that $k>1$ and $\sigma_k\neq k-1$. In such a case, we can affirm that $u_M$ is of constant sign until one of the two following boundary conditions is satisfied:
				 		\[u^{(\alpha)}_M(a)=0\quad \text{or}\quad u^{(\beta)}_M(b)=0\,.\]
				 		
				 		Now, in order to see when the sign change begins, let us study the problem with different signs of $M$. Since we are considering $u_M\geq0$, it is obvious that
				 		\begin{equation}\label{Ec::Suab} u^{(\alpha)}_M(a)\geq0\,,\quad \text{and}\quad \left\lbrace \begin{array}{cc}
				 	u^{(\beta)}_M(b)\geq 0\,,&\text{ if $\beta$ is even,}\\&\\
				 		u^{(\beta)}_M(b)\leq 0\,,&\text{ if $\beta$ is odd.}\end{array}\right.\end{equation}

				 		Let us study the behavior of $u^{(\alpha)}_M(a)$ and $u^{(\beta)}_M(b)$, to keep the maximal oscillation, considered as in Notation \ref{Not:maxos}, in each case.  In this case, the maximum number of zeros which $u$ can have, taking into account the boundary conditions \eqref{Ec::cfaa}-\eqref{Ec::cfbb} is $1$. Then, a zero on the boundary is allowed without implying that $u\equiv 0$.  If $T_{n-\ell}\,u_M(a)=0$ for $\ell\neq n-\alpha$  and $n-\ell\notin\{\sigma_1,\dots,\sigma_{k-1}\}$ or  $T_{n-\ell}\,u_M(b)=0$ for $\ell\neq n-\beta$ and $n-\ell\notin\{\varepsilon_1,\dots,\varepsilon_{n-k}\}$, then the maximum number of zeros which $u$ can have is $0$ and we cannot have more zeros on the boundary for any  nontrivial solution of \eqref{Ec::T_n[M]}. Therefore, let us assume that the only zero which is allowed is found at $T_{\alpha}\,u_M(a)$ or $T_{\beta}\,u_M(b)$.
				 		
				 		At first, consider $M<\bar M$, we have that $T_n[\bar M]\,u_M\geq 0$, hence, with maximal oscillation, if $T_{n-\ell}u_M(a)\neq 0$ and $T_{n-\ell}u_M(b)\neq 0$ for all $\ell=1,\dots n$ \eqref{Ec::Maxos} is satisfied.
				 		
				 		However each time that  $T_{n-\ell}u_M(a)= 0$, the sign change come on the next $\tilde \ell$ for which  $T_{n-\tilde \ell}u_M(a)\neq 0$. And, if  $T_{n-\ell}u_M(b)= 0$, it changes its sign on the next $\tilde\ell$ for which $T_{n-\tilde\ell}u_M(b)\neq 0$  many times as it has vanished. From $\ell=1$ to $n-\alpha$ there are $k-1-\alpha$ zeros for $T_{n-\ell}u_M(a)$ and from $\ell=1$ to $n-\beta$ there are $n-k-\beta$ zeros for $T_{n-\ell}u_M(b)$. Hence, to allow the maximal oscillation	it is necessary that			 		
				 			\begin{equation}\label{Ec::Maxos1}
				 			\left\lbrace \begin{array}{cc}
				 			T_{\alpha}\,u_M(a)\geq0\,,& \text{if $n-\alpha-(k-\alpha-1)=n-k-1$ is even,}\\&\\
				 			T_{\alpha}\,u_M(a)\leq0\,,& \text{if $n-k-1$ is odd,}\end{array}\right. 	
				 			\end{equation}
				 		and \begin{equation}
				 		\label{Ec::Maxos2}
				 		\left\lbrace \begin{array}{cc}
				 		T_{\beta}\,u_M(b)\geq0\,,& \text{if $n-k-\beta$ is even,}\\&\\
				 		T_{\beta}\,u_M(b)\leq 0\,,& \text{if $n-k-\beta$ is odd.}\end{array}\right.
				 		\end{equation}
				 			
				 		From \eqref{Ec::Tab}, we can affirm that with maximal oscillation
				 		
				 			\begin{equation*}
				 			\left\lbrace \begin{array}{cc}
				 		u^{(\alpha)}_M(a)\geq0\,,& \text{if $n-k$ is odd,}\\&\\
				 		u^{(\alpha)}_M(a)\leq0\,,& \text{if $n-k$ is even,}\end{array}\right. 	
				 			\end{equation*}
				 			and,
				 			\begin{itemize}
				 				\item If $n-k$ is even				 			
				 		\begin{equation*}\left\lbrace \begin{array}{cc}
				 	u^{(\beta)}_M(b)\geq0\,,& \text{if $\beta$ is even,}\\&\\
				 		u^{(\beta)}_M(b)\leq 0\,,& \text{if $\beta$ is odd.}\end{array}\right.
				 		\end{equation*}
				 			\item If $n-k$ is odd				 			
				 			\begin{equation*}
				 			\left\lbrace \begin{array}{cc}
				 			u^{(\beta)}_M(b)\leq0\,,& \text{if $\beta$ is even,}\\&\\
				 			u^{(\beta)}_M(b)\geq 0\,,& \text{if $\beta$ is odd.}\end{array}\right.
				 			\end{equation*}				 		
				 		\end{itemize}
				 		
				 		Hence, we arrive at the following conclusions, taking into account \eqref{Ec::Suab}:
				 		\begin{itemize}
				 			\item If $n-k$ is even, the maximal oscillation is not allowed until $u^{(\alpha)}_M(a)=0$, which implies that $u_M>0$ on $(a,b)$ for $[\bar M-\lambda_3',\bar M]$
				 		\item If $n-k$ is odd, the maximal oscillation is not allowed until $u^{(\beta)}(b)_M=0$, which implies that $u_M>0$ on $(a,b)$ for $[\bar M -\lambda_2',\bar M]$.
				 		\end{itemize}

					\vspace{0.5 cm}
						Now, considering $M>\bar M$, we have that $T_n[\bar M]\,u_M\leq 0$, hence with maximal oscillation, if $T_{n-\ell}u_M(a)\neq 0$ and $T_{n-\ell}u_M(b)\neq 0$, for all $\ell=1,\dots n$, the following inequalities are satisfied:
						
							\begin{equation}\label{Ec::Maxos3}
							\left\lbrace \begin{array}{cc}
							T_{n-\ell}\,u_M(a)<0\,,& \text{if $\ell$ is even,}\\&\\
							T_{n-\ell}\,u_M(a)>0\,,& \text{if $\ell$ is odd,}\end{array}\right.\qquad\qquad T_{n-\ell}\,u_M(b)<0\,.
							\end{equation}
							
					In this case, since we have contrary signs from the previous case where $M<\bar M$,  to allow the maximal oscillation, the following inequalities must be satisfied:
						
						\begin{equation}\label{Ec::Maxos4}
						\left\lbrace \begin{array}{cc}
						T_{\alpha}\,u_M(a)\leq0\,,& \text{if $n-k-1$ is even,}\\&\\
						T_{\alpha}\,u_M(a)\geq0\,,& \text{if $n-k-1$ is odd,}\end{array}\right. 	
						\end{equation}
						and \begin{equation}
						\label{Ec::Maxos5}
						\left\lbrace \begin{array}{cc}
						T_{\beta}\,u_M(b)\leq0\,,& \text{if $n-k-\beta$ is even,}\\&\\
						T_{\beta}\,u_M(b)\geq 0\,,& \text{if $n-k-\beta$ is odd.}\end{array}\right.
						\end{equation}
						
						Hence, from \eqref{Ec::Tab}, we can affirm that with maximal oscillation
						
						\begin{equation*}
						\left\lbrace \begin{array}{cc}
						u^{(\alpha)}_M(a)\leq0\,,& \text{if $n-k$ is odd,}\\&\\
						u^{(\alpha)}_M(a)\geq0\,,& \text{if $n-k$ is even,}\end{array}\right. 	
						\end{equation*}
						and,
						\begin{itemize}
							\item If $n-k$ is even				 			
							\begin{equation*}\left\lbrace \begin{array}{cc}
							u^{(\beta)}_M(b)\leq0\,,& \text{if $\beta$ is even,}\\&\\
							u^{(\beta)}_M(b)\geq 0\,,& \text{if $\beta$ is odd.}\end{array}\right.
							\end{equation*}
							\item If $n-k$ is odd				 			
							\begin{equation*}
							\left\lbrace \begin{array}{cc}
							u^{(\beta)}_M(b)\geq0\,,& \text{if $\beta$ is even,}\\&\\
							u^{(\beta)}_M(b)\leq 0\,,& \text{if $\beta$ is odd.}\end{array}\right.
							\end{equation*}				 		
						\end{itemize}
						
						Hence, we arrive at the following conclusions, taking into account \eqref{Ec::Suab}:
						\begin{itemize}
							\item If $n-k$ is even, the maximal oscillation is not allowed until $u^{(\beta)}_M(b)=0$, i.e., $u_M>0$ on $(a,b)$ for $[\bar M,\bar M-\lambda_2']$
							\item If $n-k$ is odd, the maximal oscillation is not allowed until $u^{(\alpha)}_M(a)=0$, i.e. $u_M>0$ on $(a,b)$ for $[\bar M ,\bar M-\lambda_3']$.
						\end{itemize}
						
						The proof   is complete since if $k=1$, $u^{(\beta)}_M(b)\neq0$ for every $M\neq \bar M$, because the contrary will imply that $u_M$ is a nontrivial solution of the linear differential equation \eqref{Ec::T_n[M]} with a zero of multiplicity $n$ at $t=b$ and this is not possible.
						
						And, if $\sigma_k=k-1$, consider $u^{(k-1)}_M(a)$ instead of $u^{(\alpha)}_M(a)=u^{(k)}(a)$, since it is the first non null derivative at $t=a$. Since $u_M\geq 0$, then $u_M^{(k-1)}(a)\geq 0$. But, with maximal oscillation, $T_{k-1}\,u_M(a)$ follows \eqref{Ec::Maxos} if $M<\bar M$ and \eqref{Ec::Maxos3} if $M>\bar M$ for $\ell=n-k-1$. Hence, from \eqref{Ec::Tab}, we can affirm that, with maximal oscillation, the following inequalities must be fulfilled:

						\begin{itemize}
							\item If $M<\bar M$ 			 			
							\begin{equation*}\left\lbrace \begin{array}{cc}
							u^{(k-1)}_M(a)\geq0\,,& \text{if $n-k$ is odd,}\\&\\
							u^{(k-1)}_M(a)\leq 0\,,& \text{if $n-k$ is even.}\end{array}\right.
							\end{equation*}
							\item If $M>\bar M$ 				 			
								\begin{equation*}\left\lbrace \begin{array}{cc}
								u^{(k-1)}_M(a)\leq0\,,& \text{if $n-k$ is odd,}\\&\\
								u^{(k-1)}_M(a)\geq 0\,,& \text{if $n-k$ is even.}\end{array}\right.
								\end{equation*}	 		
						\end{itemize}
						 And, we can conclude  the proof:
						 \begin{itemize}
						 	\item If $n-k$ is even and $M<\bar M$, the maximal oscillation is not allowed until $u^{(k-1)}_M(a)=0$, i.e., $u_M>0$ on $(a,b)$ for $[\bar M-\lambda_1,\bar M]$
						 	\item If $n-k$ is odd and $M>\bar M$, the maximal oscillation is not allowed until $u^{(k-1)}_M(a)=0$, i.e. $u_M>0$ on $(a,b)$ for $[\bar M ,\bar M-\lambda_1]$.
						 \end{itemize}				
				 	\end{proof}
				 	
				 	\begin{exemplo}
				 		\label{Ex::7}
				 		From Proposition \ref{P::1} and Example \ref{Ex::6}, we can affirm that any nontrivial solution of $T_4^0[M]\equiv u^{(4)}(t)+M\,u(t)=0$ on $[0,1]$, verifying the boundary conditions:
				 		\[u(0)=u'(1)=u''(1)=0\,,\]
				 		does not have any zero on $(0,1)$ for $M\in[-\pi^4,m_2^4]$, where $m_2^4=-\lambda_1$ with $\lambda_1$ the first negative eigenvalue of $T_4^0[0]$ in $X_{\{0\}}^{\{0,1,2\}}$ and $m_2$ has been introduced in Example \ref{Ex::6} as the least positive solution of \eqref{Ec::Ex62}.
				 		
				 		 Such functions are given as multiples of the following  expression:
				 	{\footnotesize	\[\left\lbrace  \begin{array}{cc}
				 		\begin{split}
				 		\cos (m-m t) (\sin (m)-\sinh (m))+\sin (m-m t) (-\cos (m)-\cosh (m))\\+\sinh (m-m t) (\cos (m)+\cosh (m))+\cosh (m-m t) (\sin (m)-\sinh (m))\,,\end{split}&M=-m^4<0\,,\\\\t^3-3t^2+t\,,&M=0\,,\\\\\begin{split}
				 	e^{-\frac{m t}{\sqrt{2}}} \left(-\left(e^{\sqrt{2} m (t-1)}+e^{\sqrt{2} m t}+e^{\sqrt{2} m}+1\right) \sin \left(\frac{m t}{\sqrt{2}}\right)\right. \\\left.+\left(e^{\sqrt{2} m
				 		t}-1\right) \cos \left(\frac{m (t-2)}{\sqrt{2}}\right) +\left(e^{\sqrt{2} m t}-1\right) \cos \left(\frac{m t}{\sqrt{2}}\right)\right)\,,	
				 		\end{split}&M=m^4>0\,.\end{array}\right. \]}				
				 	\end{exemplo}
			 	
				 	Now, we enunciate a similar result, which refers to the eigenvalues in $X_{\{\sigma_1,\dots,\sigma_k|\alpha\}}^{\{\varepsilon_1,\dots,\varepsilon_{n-k-1}\}}$ and $X_{\{\sigma_1,\dots,\sigma_k\}}^{\{\varepsilon_1,\dots,\varepsilon_{n-k-1}|\beta\}}$.
				 	
				 		 	\begin{proposition}
				 		 		\label{P::2}
				 		 		Let $\bar M\in \mathbb{R}$ be such that $T_n[\bar M]$ satisfies property $(T_d)$ on $X_{\{\sigma_1,\dots,\sigma_k\}}^{\{\varepsilon_1,\dots,\varepsilon_{n-k}\}}$  and ${\{\sigma_1,\dots,\sigma_k\}}-{\{\varepsilon_1,\dots,\varepsilon_{n-k}\}}$ satisfy $(N_a)$. If $u\in C^n(I)$ is a  solution of \eqref{Ec::T_n[M]} on $(a,b)$ satisfying the boundary conditions:
				 		 		\begin{eqnarray}
				 		 		\label{Ec::cfaaa} u^{(\sigma_1)}(a)=\cdots=u^{(\sigma_{k})}(a)&=&0\,,\\
				 		 		\label{Ec::cfbbb} u^{(\varepsilon_1)}(b)=\cdots=u^{(\varepsilon_{n-k-1})}(b)&=&0\,,
				 		 		\end{eqnarray}
				 		 		then it does not have any zero on $(a,b)$ provided that one of the following assertions is satisfied:
				 		 		\begin{itemize}
				 		 			\item Let $n-k$ be even:
				 		 			\begin{itemize}
				 		 				\item If  $\varepsilon_{n-k}\neq n-k-1$ and $M\in[\bar{M}-\lambda_3'',\bar{M}-\lambda_2'']$, where:
				 		 				\begin{itemize}
				 		 					\item $\lambda_3''>0$ is the least positive eigenvalue of $T_n[\bar M]$ in $X_{\{\sigma_1,\dots,\sigma_{k}\}}^{\{\varepsilon_1,\dots,\varepsilon_{n-k-1}|\beta\}}.$
				 		 					\item $\lambda_2''<0$ is the biggest negative eigenvalue of $T_n[\bar M]$ in $X_{\{\sigma_1,\dots,\sigma_k|\alpha\}}^{\{\varepsilon_1,\dots,\varepsilon_{n-k-1}\}}.$	
				 		 				\end{itemize}
				 
				 		 				\item If  $\varepsilon_{n-k} =n- k-1$ and $M\in[\bar{M}-\lambda_1,\bar{M}-\lambda_2'']$, where:
				 		 				\begin{itemize}
				 		 					\item $\lambda_1>0$ is the least positive eigenvalue of $T_n[\bar M]$ in $X_{\{\sigma_1,\dots,\sigma_{k}\}}^{\{0,\dots,n-k-1 \}}.$
				 		 					\item $\lambda_2''<0$ is the biggest negative eigenvalue of $T_n[\bar M]$ in $X_{\{\sigma_1,\dots,\sigma_k|\alpha\}}^{\{0,\dots,k-2\}}.$	
				 		 				\end{itemize}

				 		 			\end{itemize}
				 		 			\item Let $n-k$ be odd:
				 		 			\begin{itemize}
				 		 				\item If $k<n-1$, $\varepsilon_{n-k}\neq n- k-1$ and $M\in[\bar{M}-\lambda_2'',\bar{M}-\lambda_3'']$, where:
				 		 				\begin{itemize}
				 		 					\item $\lambda_3''<0$ is the biggest negative eigenvalue of $T_n[\bar M]$ in $X_{\{\sigma_1,\dots,\sigma_{k}\}}^{\{\varepsilon_1,\dots,\varepsilon_{n-k-1}|\beta\}}.$
				 		 					\item $\lambda_2''>0$ is the least positive eigenvalue of $T_n[\bar M]$ in $X_{\{\sigma_1,\dots,\sigma_k|\alpha\}}^{\{\varepsilon_1,\dots,\varepsilon_{n-k-1}\}}.$	
				 		 				\end{itemize}
				 		 				\item If $k=n-1$, $\varepsilon_{1}\neq0$ and $M\in(-\infty,\bar{M}-\lambda_3'']$, where:
				 		 				\begin{itemize}
				 		 					\item $\lambda_3''<0$ is the biggest negative eigenvalue of $T_n[\bar M]$ in $X_{\{\sigma_1,\dots,\sigma_{n-1}\}}^{\{\beta\}}$, where $\beta=0$.
				 		 				\end{itemize}
				 		 				\item If $k<n-1$, $\varepsilon_{n-k}=n-k-1$ and $M\in[\bar{M}-\lambda_2'',\bar{M}-\lambda_1]$, where:
				 		 				\begin{itemize}
				 		 					\item $\lambda_1<0$ is the biggest negative eigenvalue of $T_n[\bar M]$ in $X_{\{\sigma_1,\dots,\sigma_{k}\}}^{\{0,1,\dots,n-k-1\}}.$
				 		 					\item $\lambda_2''>0$ is the least positive eigenvalue of $T_n[\bar M]$ in $X_{\{\sigma_1,\dots,\sigma_k|\alpha\}}^{\{\varepsilon_1,\dots,\varepsilon_{n-k-1}\}}.$
				 		 				\end{itemize}	
				 		 				\item If $k=n-1$ and  $\varepsilon_{n-k}=0$ and $M\in(-\infty,\bar{M}-\lambda_1]$, where:	
				 		 				\begin{itemize}
				 		 					\item $\lambda_1<0$ is the biggest negative eigenvalue of $T_n[\bar M]$ in $X_{\{\sigma_1,\dots,\sigma_{n-1}\}}^{\{0\}}.$
				 		 				\end{itemize}	
				 		 				
				 		 			\end{itemize}
				 		 		\end{itemize}
				 		 	\end{proposition}
				 		 	\begin{proof}
				 		 		The proof is analogous to the one of Proposition \ref{P::1}.
				 		 	\end{proof}
				 
				 \begin{exemplo}
				 	\label{Ex::8}
				 	Now, consider the fourth order differential equation $u^{(4)}(t)+M\,u(t)=0$ coupled with the boundary conditions $u(0)=u''(0)=u'(1)=0$.  Using Proposition \ref{P::2} and Example \ref{Ex::6}, we conclude that such functions do not have any zero on $(0,1)$ if $M\in[-m_3^4,4\,\pi^4]$, where $m_3^3=-\lambda_1$, with $\lambda_1$ the first negative eigenvalue of $T_4^0[0]$ in $X_{0,1,2}^{1}$ and $m_3$ has been introduced in Example \ref{Ex::6} as the least positive solution of \eqref{Ec::Ex63}.
				 	
				 	It is not difficult to verify that the solutions of this problem are given as multiples of the following expression:
				 	
				 	{\scriptsize 	\[\left\lbrace  \begin{array}{cc}
				 		\dfrac{\sin(m\,t)}{\cos(m)}-\dfrac{\sinh(m\,t)}{\cosh(m)}\,,&M=-m^4<0\,,\\\\t^3-3t\,,&M=0\,,\\\\\begin{split}
				 		e^{-\frac{m t}{\sqrt{2}}} \left(\left(e^{\sqrt{2} m (t+1)}+1\right) \sin \left(\frac{m (t-1)}{\sqrt{2}}\right)+\left(e^{\sqrt{2} m t}+e^{\sqrt{2} m}\right) \sin
				 		\left(\frac{m (t+1)}{\sqrt{2}}\right)\right. \\\left. +\left(1-e^{\sqrt{2} m (t+1)}\right) \cos \left(\frac{m (t-1)}{\sqrt{2}}\right)+\left(e^{\sqrt{2} m}-e^{\sqrt{2} m t}\right)
				 		\cos \left(\frac{m (t+1)}{\sqrt{2}}\right)\right)\,,
				 		\end{split}&M=m^4>0\,.\end{array}\right. \]}			 
				 \end{exemplo}
				 
				 To finish this section, we  show a result which gives an order on the previously obtained eigenvalues $\lambda_1$, $\lambda_3'$ and $\lambda_3''$. First, let us introduce some notation.
				 
				 	\begin{notation}\label{Not::alpha1}
				 	Let us denote  $\alpha_1\in \{1,\dots,n-1\}$  such that $\alpha_1\notin \{\sigma_1,\dots,\sigma_{k-1}|\alpha\}$ and $\{0,\dots,\alpha_1-1\}\subset\{\sigma_1,\dots,\sigma_{k-1}|\alpha\}$.
				 	
				 	 Let $\beta_1\in \{1,\dots,n-1\}$ be such that $\beta_1\notin \{\varepsilon_1,\dots,\varepsilon_{n-k-1}|\beta\}$ and $\{0,\dots,\beta_1-1\}\subset\{\varepsilon_1,\dots,\varepsilon_{n-k-1}|\beta\}$.
				 \end{notation}

				 \begin{proposition}\label{P::6.5}
				 	Let $\bar M\in \mathbb{R}$ be such that $T_n[\bar M]$ satisfies property $(T_d)$ on $X_{\{\sigma_1,\dots,\sigma_k\}}^{\{\varepsilon_1,\dots,\varepsilon_{n-k}\}}$ and  $\{\sigma_1,\dots,\sigma_k\}-\{\varepsilon_1,\dots,\varepsilon_{n-k}\}$ satisfy $(N_a)$. Then the following assertions are fulfilled:
				 	\begin{itemize}
				 		\item Let $n-k$ be even, we have:
				 		\begin{itemize}
				 			\item If $\sigma_k\neq k-1$, then  $\lambda_3'>\lambda_1>0$, where
				 			\begin{itemize}
				 				\item $\lambda_3'>0$ is the least positive eigenvalue of $T_n[\bar M]$ in $X_{\{\sigma_1,\dots,\sigma_{k-1}|\alpha\}}^{\{\varepsilon_1,\dots,\varepsilon_{n-k}\}}.$
				 				\item $\lambda_1>0$ is the least positive eigenvalue of $T_n[\bar M]$ in $X_{\{\sigma_1,\dots,\sigma_k\}}^{\{\varepsilon_1,\dots,\varepsilon_{n-k}\}}.$			 			 			
				 			\end{itemize}
				 				Moreover if there exists $\lambda_1'>\lambda_1$ another eigenvalue of  of $T_n[\bar M]$ in $X_{\{\sigma_1,\dots,\sigma_k\}}^{\{\varepsilon_1,\dots,\varepsilon_{n-k}\}}$, then $\lambda_1'>\lambda_3'$.
				 		
				 				\item If $\varepsilon_{n-k}\neq n-k-1$, then  $\lambda_3''>\lambda_1>0$, where
				 				\begin{itemize}
				 					\item $\lambda_3''>0$ is the least positive eigenvalue of $T_n[\bar M]$ in $X_{\{\sigma_1,\dots,\sigma_{k}\}}^{\{\varepsilon_1,\dots,\varepsilon_{n-k-1}|\beta\}}.$ 
				 					\item $\lambda_1>0$ is the least positive eigenvalue of $T_n[\bar M]$ in $X_{\{\sigma_1,\dots,\sigma_k\}}^{\{\varepsilon_1,\dots,\varepsilon_{n-k}\}}.$
				 				\end{itemize}
				 					Moreover if there exists $\lambda_1'>\lambda_1$ another eigenvalue of  of $T_n[\bar M]$ in $X_{\{\sigma_1,\dots,\sigma_k\}}^{\{\varepsilon_1,\dots,\varepsilon_{n-k}\}}$, then $\lambda_1'>\lambda_3''$.
				 
				 		\end{itemize}

				 			\item Let $n-k$ be odd, we have:
				 			\begin{itemize}
				 				\item If $\sigma_k\neq k-1$, then  $\lambda_3'<\lambda_1<0$, where
				 				\begin{itemize}
				 					\item $\lambda_3'<0$ is the biggest negative eigenvalue of $T_n[\bar M]$ in $X_{\{\sigma_1,\dots,\sigma_{k-1}|\alpha\}}^{\{\varepsilon_1,\dots,\varepsilon_{n-k}\}}.$
				 					\item $\lambda_1<0$ is the biggest negative eigenvalue of $T_n[\bar M]$ in $X_{\{\sigma_1,\dots,\sigma_k\}}^{\{\varepsilon_1,\dots,\varepsilon_{n-k}\}}.$
				 					\end{itemize}
				 						Moreover if there exists $\lambda_1'<\lambda_1$ another eigenvalue of  of $T_n[\bar M]$ in $X_{\{\sigma_1,\dots,\sigma_k\}}^{\{\varepsilon_1,\dots,\varepsilon_{n-k}\}}$, then $\lambda_1'<\lambda_3'$.
				 				\item If $\varepsilon_{n-k}\neq n-k-1$, then  $\lambda_3''<\lambda_1<0$, where
				 				\begin{itemize}
				 					\item $\lambda_3''<0$ is the biggest negative eigenvalue of $T_n[\bar M]$ in $X_{\{\sigma_1,\dots,\sigma_{k}\}}^{\{\varepsilon_1,\dots,\varepsilon_{n-k-1}|\beta\}}.$ 
				 					\item $\lambda_1<0$ is the biggest negative eigenvalue of $T_n[\bar M]$ in $X_{\{\sigma_1,\dots,\sigma_k\}}^{\{\varepsilon_1,\dots,\varepsilon_{n-k}\}}.$
				 				\end{itemize}
				 					Moreover if there exists $\lambda_1'<\lambda_1$ another eigenvalue of  of $T_n[\bar M]$ in $X_{\{\sigma_1,\dots,\sigma_k\}}^{\{\varepsilon_1,\dots,\varepsilon_{n-k}\}}$, then $\lambda_1'<\lambda_3''$.
				 			\end{itemize}
				 			
				 	\end{itemize} 
				 \end{proposition}
				 
				 \begin{proof}
				 	At the beginning, we  focus on the relation between $\lambda_1$ and $\lambda_3'$.
				 	
				 	We have seen in Proposition \ref{P::1} that a function $u_M$, solution of \eqref{Ec::T_n[M]}, satisfying the boundary conditions \eqref{Ec::cfaa}-\eqref{Ec::cfbb} cannot have any zero on $(a,b)$ for $M\in [\bar{M}-\lambda_3',\bar{M}]$ if $n-k$ is even and for $M\in [\bar{M},\bar{M} -\lambda_3']$ if $n-k$ is odd.

				 	Moreover, it is proved that for $M=\bar M$, without taking into account the boundary conditions, $u_{\bar M}$ has at most $n-1$ zeros, moreover, conditions \eqref{Ec::cfTaa}-\eqref{Ec::cfTbb} are satisfied by $u_{\bar M}$. Hence, we loose the $n-1$ possible oscillation. So, for $M=\bar M$ with the given boundary conditions, the maximal oscillation is achieved for the boundary conditions \eqref{Ec::cfaa}-\eqref{Ec::cfbb} .
				 		
				 		Let us assume that $u_{\bar M}\geq 0$ (if $u_{\bar M}\leq0 $ the arguments are valid by multiplying by $-1$), hence $T_\alpha\,u_{\bar M}(a)=\dfrac{u^{(\alpha)}_{\bar M}(a)}{v_1(a)\dots v_\alpha(a)}>0$. 
				 		
				 		As we have said before, $T_h\,u(a)$ changes its sign for every $h=0,\dots,n-1$ if it is non null. From $h=\alpha$ to $\sigma_k$, taking into account \eqref{Ec::cfTaa}, $k-1-\alpha$ zeros for $T_hu(a)$ are found. Hence, with maximal oscillation:				 		
				 			\begin{equation*}\left\lbrace \begin{array}{cc}
				 			T_{\sigma_k}\,u_{\bar M}(a)>0\,,& \text{if $(\sigma_k-\alpha)-(k-1-\alpha)=\sigma_k-k+1$ is even, }\\&\\
				 			T_{\sigma_k}\,u_{\bar M}(a)< 0\,,& \text{if $\sigma_k-k+1$ is odd.}\end{array}\right.
				 			\end{equation*}
				 			
				 			On the other hand, by means of Lemma \ref{L::1}, we have that:				 		
				 			\begin{equation}\label{Ec::cfsigma1}\left\lbrace \begin{array}{cc}
				 			u^{(\sigma_k)}_{\bar M}(a)>0\,,& \text{if $\sigma_k-k$ is odd, }\\&\\
				 			u^{(\sigma_k)}_{\bar M}(a)< 0\,,& \text{if $\sigma_k-k$ is even.}\end{array}\right.
				 			\end{equation}
				 			
				 			Let us move $u_M$ continuously on $M$ until $M=\bar M-\lambda_3'$. On Proposition \ref{P::1} we have proved that $u_M$ has at most $n$ zeros for every $M \in [\bar M-\lambda_3',\bar M]$ ($[\bar M, \bar M-\lambda_3]$ if $n-k$ is odd) if $u_M\geq0$, without taking into account the boundary conditions.
				 			
				 			Since $\lambda_3'$ is an eigenvalue of $T_n[\bar M]$ on $X_{\{\sigma_1,\dots,\sigma_{k-1}|\alpha\}}^{\{\varepsilon_1,\dots,\varepsilon_{n-k}\}}$, we have that $u^{(\alpha)}_{\bar{M}-\lambda_3'}(a)=0$. Thus, $T_\alpha\,u_{\bar M-\lambda_3'}(a)=0$. This fact, coupled with the boundary conditions \eqref{Ec::cfTaa}-\eqref{Ec::cfTbb}, allows us to affirm that $u_{\bar{M}-\lambda_3'}$ cannot loose more oscillations if it is a nontrivial solution. Hence, the maximal oscillation is verified.

				 				Since we have moved continuously from $\bar M$ to $\bar M-\lambda_3'$ and it was assumed $u_{\bar M}\geq 0$ on $I$, we conclude that $u_{\bar M-\lambda_3'}\geq 0$, hence $T_{\alpha_1}\,u_{\bar M}(a)=\dfrac{u^{(\alpha_1)}_{\bar M}(a)}{v_1(a)\dots v_{\alpha_1}(a)}>0$, where $\alpha_1$ has been introduced in Notation \ref{Not::alpha1}.

				 				As for $M=\bar M$, provided it is non null, $T_h\,u(a)$ changes its sign for every $h=0,\dots,n-1$ . From $h=\alpha_1$ to $\sigma_k$, taking into account \eqref{Ec::cfTaa}, $k-\alpha_1$ zeros are found. Hence, with maximal oscillation:				 				
				 				\begin{equation*}\left\lbrace \begin{array}{cc}
				 				T_{\sigma_k}\,u_{\bar M -\lambda_3'}(a)>0\,,& \text{if $(\sigma_k-\alpha_1)-(k-\alpha_1)=\sigma_k-k$ is even, }\\&\\
				 				T_{\sigma_k}\,u_{\bar M -\lambda_3'}(a)< 0\,,& \text{if $\sigma_k-k$ is odd.}\end{array}\right.
				 				\end{equation*}
				 				
				 				From Lemma \ref{L::1} again, we have:				 				
				 				\begin{equation}\label{Ec::cfsigma2}\left\lbrace \begin{array}{cc}
				 				u^{(\sigma_k)}_{\bar M -\lambda_3'}(a)>0\,,& \text{if $\sigma_k-k$ is even, }\\&\\
				 				u^{(\sigma_k)}_{\bar M -\lambda_3'}(a)< 0\,,& \text{if $\sigma_k-k$ is odd.}\end{array}\right.
				 				\end{equation}
				 			
				 			Hence, since we have been moving with continuity, from \eqref{Ec::cfsigma1} and \eqref{Ec::cfsigma2}, we can ensure the existence of a $\tilde M$ between $\bar M$ and $\bar M-\lambda_3'$ such that $u^{(\sigma_k)}_{\tilde M}(a)=0$. As consequence:
				 			\begin{itemize}
				 				\item If $n-k$ is even, $0<\lambda_1=\bar M-\tilde M<\lambda_3'$.
					 			\item If $n-k$ is odd, $0>\lambda_1=\bar M-\tilde M>\lambda_3'$.
				 			\end{itemize}
				 			
				 				The relation between $\lambda_1$ and $\lambda_3''$ is proved analogously by using Proposition \ref{P::2}.

				 				 	The assertion referring to $\lambda_1'$ is due to the fact that, if $0<\lambda_1<\lambda_1'<\lambda_3'$ on the case where $n-k$ is even, then, by Proposition \ref{P::1}, the eigenfunctions related to $\lambda_1$ and $\lambda_1'$ are of constant sign and this is not possible for an strongly inverse positive (negative) operator (see \cite[Corollary 7.27]{Zeid} and \cite[Section 1.8]{Cab}). The same happens when $n-k$ is odd and $0>\lambda_1>\lambda_1'>\lambda_3$.
				 				 	
				 				  Similarly, if either $n-k$ is even and $0<\lambda_1<\lambda_1'<\lambda_3''$ or $n-k$ is odd and $0>\lambda_1>\lambda_1'>\lambda_3''$, then, by Proposition \ref{P::2}, the eigenfunctions related to $\lambda_1$ and $\lambda_1'$ are of constant sign. 
				 				  
				 				  Thus, the result is proved.
				 					 \end{proof}
				 					 
				 					 \begin{exemplo}
				 					 	\label{Ex::9}
				 					 	Let us return to Example \ref{Ex::6}, where we have obtained the different eigenvalues for the operator $T_4^0[0]$. Let us see that the thesis of Proposition \ref{P::6.5} are fulfilled.
				 					 	
				 					 \begin{itemize}
				 					 	\item $\lambda_1=m_1^4\approxeq 2.36502^4<\lambda_3'=\pi^4$.
				 					 	\item $\lambda_1<\lambda_3''=m_3^4\approxeq 3.9266^4$.
				 					 \end{itemize}
				 					 
				 					 Moreover, we have seen in Example \ref{Ex::6} that the eigenvalues of $T_4^0[0]$ in $X_{\{0,2\}}^{\{1,2\}}$ are given as $\lambda=m^4$, where $m$ is a positive solution of \eqref{Ec::Ex61}. So $\lambda_1'\approxeq 5.497^4>\lambda_3'$ and $\lambda_1'>\lambda_3''$.
				 					 \end{exemplo}

 \chapter[Study of the eigenvalues of the  adjoint operator]{Study of the eigenvalues of the  adjoint operator $T^*_n[\bar M]$ and of $\widehat T_n[(-1)^n\bar M]$ in different spaces.}				 					 
				 					 
This chapter is devoted to the study of the eigenvalues of the adjoint operator $T^*_n[\bar M]$ in the different spaces $X_{\ \, \{\sigma_1,\dots,\sigma_k\}}^{*\{\varepsilon_1,\dots,\varepsilon_{n-k}\}}$, $X_{\ \,\{\sigma_1,\dots,\sigma_k|\alpha\}}^{*\{\varepsilon_1,\dots,\varepsilon_{n-k-1}\}}$, $X_{\ \,\{\sigma_1,\dots,\sigma_{k-1}\}}^{*\{\varepsilon_1,\dots,\varepsilon_{n-k}|\beta\}}$, $X_{\ \,\{\sigma_1,\dots,\sigma_{k-1}|\alpha\}}^{*\{\varepsilon_1,\dots,\varepsilon_{n-k}\}}$ and $X_{\ \,\{\sigma_1,\dots,\sigma_k\}}^{*\{\varepsilon_1,\dots,\varepsilon_{n-k-1}|\beta\}}$.

In Chapter \ref{S::ad} we have proved that  the boundary conditions satisfied for every $v\in X_{\ \, \{\sigma_1,\dots,\sigma_k\}}^{*\{\varepsilon_1,\dots,\varepsilon_{n-k}\}}$ are given by \eqref{Cf::ad1}--\eqref{Cf::ad4}. Proceeding analogously in the different spaces, taking into account that $\eta=n-1-\sigma_{k}$, $\gamma=n-1-\varepsilon_{n-k}$, $\alpha=n-1-\tau_{n-k}$ and $\beta=n-1-\delta_k$, we have the following assertions:
\begin{itemize}
	\item If $v\in X_{\ \,\{\sigma_1,\dots,\sigma_k|\alpha\}}^{*\{\varepsilon_1,\dots,\varepsilon_{n-k-1}\}}$, then it satisfies \eqref{Cf::ad1}--\eqref{Cf::ad11} and \eqref{Cf::ad3}--\eqref{Cf::ad4} coupled with $v^{(\gamma)}(b)=0$.
%		{\small \begin{eqnarray}
%		\label{Cf::adal1}
%		v^{(\tau_1)}(a)+\sum_{j=n-\tau_1}^{n-1}(-1)^{n-j}\,(p_{n-j}\,v)^{(\tau_1+j-n)}(a)&=&0\,,\\
%		\nonumber \vdots&&\\	\label{Cf::adal2}v^{(\tau_{n-k-1})}(a)+\sum_{j=n-\tau_{n-k-1}}^{n-1}(-1)^{n-j}\,(p_{n-j}\,v)^{(\tau_{n-k-1}+j-n)}(a)&=&0\,,\\
%		\label{Cf::adal3}
%		v^{(\delta_1)}(b)+\sum_{j=n-\delta_1}^{n-1}(-1)^{n-j}\,(p_{n-j}\,v)^{(\delta_1+j-n)}(b)&=&0\,,\\\nonumber \vdots&&\\
%		\label{Cf::adal4}
%		v^{(\delta_k)}(b)+\sum_{j=n-\delta_k}^{n-1}(-1)^{n-j}\,(p_{n-j}\,v)^{(\delta_k+j-n)}(b)&=&0\,,\\\label{Cf::adal5}v^{(\gamma)}(b)&=&0\,.	
%		\end{eqnarray}}
		\item If $v\in X_{\ \,\{\sigma_1,\dots,\sigma_{k-1}\}}^{*\{\varepsilon_1,\dots,\varepsilon_{n-k}|\beta\}}$, then it satisfies \eqref{Cf::ad1}--\eqref{Cf::ad2} and \eqref{Cf::ad3}--\eqref{Cf::ad31} coupled with $v^{(\eta)}(a)=0$.
%	{\small 	\begin{eqnarray}
%		\label{Cf::adbe1}
%		v^{(\tau_1)}(a)+\sum_{j=n-\tau_1}^{n-1}(-1)^{n-j}\,(p_{n-j}\,v)^{(\tau_1+j-n)}(a)&=&0\,,\\
%		\nonumber \vdots&&\\	\label{Cf::adbe2}v^{(\tau_{n-k})}(a)+\sum_{j=n-\tau_{n-k}}^{n-1}(-1)^{n-j}\,(p_{n-j}\,v)^{(\tau_{n-k}+j-n)}(a)&=&0\,,\\\label{Cf::adbe5}v^{(\eta)}(a)&=&0\,,\\
%		\label{Cf::adbe3}
%		v^{(\delta_1)}(b)+\sum_{j=n-\delta_1}^{n-1}(-1)^{n-j}\,(p_{n-j}\,v)^{(\delta_1+j-n)}(b)&=&0\,,\\\nonumber \vdots&&\\
%		\label{Cf::adbe4}
%		v^{(\delta_{k-1})}(b)+\sum_{j=n-\delta_{k-1}}^{n-1}(-1)^{n-j}\,(p_{n-j}\,v)^{(\delta_{k-1}+j-n)}(b)&=&0\,.	
%		\end{eqnarray}	}
			\item If $v\in X_{\ \,\{\sigma_1,\dots,\sigma_{k-1}|\alpha\}}^{*\{\varepsilon_1,\dots,\varepsilon_{n-k}\}}$, then it satisfies \eqref{Cf::ad1}--\eqref{Cf::ad11} and \eqref{Cf::ad3}--\eqref{Cf::ad4} coupled with $v^{(\eta)}(a)=0$.
%		{\small 	\begin{eqnarray}
%			\label{Cf::adal11}
%			v^{(\tau_1)}(a)+\sum_{j=n-\tau_1}^{n-1}(-1)^{n-j}\,(p_{n-j}\,v)^{(\tau_1+j-n)}(a)&=&0\,,\\
%			\nonumber \vdots&&\\	\label{Cf::adal12}v^{(\tau_{n-k-1})}(a)+\sum_{j=n-\tau_{n-k-1}}^{n-1}(-1)^{n-j}\,(p_{n-j}\,v)^{(\tau_{n-k-1}+j-n)}(a)&=&0\,,\\\label{Cf::adal15}v^{(\eta)}(a)&=&0\,,\\
%			\label{Cf::adal13}
%			v^{(\delta_1)}(b)+\sum_{j=n-\delta_1}^{n-1}(-1)^{n-j}\,(p_{n-j}\,v)^{(\delta_1+j-n)}(b)&=&0\,,\\\nonumber \vdots&&\\
%			\label{Cf::adal14}
%			v^{(\delta_k)}(b)+\sum_{j=n-\delta_k}^{n-1}(-1)^{n-j}\,(p_{n-j}\,v)^{(\delta_k+j-n)}(b)&=&0\,.
%			\end{eqnarray}}
		\item If $v\in X_{\ \,\{\sigma_1,\dots,\sigma_{k}\}}^{*\{\varepsilon_1,\dots,\varepsilon_{n-k-1}|\beta\}}$, then it satisfies \eqref{Cf::ad1}--\eqref{Cf::ad2} and \eqref{Cf::ad3}--\eqref{Cf::ad31} coupled with $v^{(\gamma)}(b)=0$.
%{\small \begin{eqnarray}
%\label{Cf::adbe11}
%v^{(\tau_1)}(a)+\sum_{j=n-\tau_1}^{n-1}(-1)^{n-j}\,(p_{n-j}\,v)^{(\tau_1+j-n)}(a)&=&0\,,\\
%\nonumber \vdots&&\\	\label{Cf::adbe12}v^{(\tau_{n-k})}(a)+\sum_{j=n-\tau_{n-k}}^{n-1}(-1)^{n-j}\,(p_{n-j}\,v)^{(\tau_{n-k}+j-n)}(a)&=&0\,,\\
%\label{Cf::adbe13}
%v^{(\delta_1)}(b)+\sum_{j=n-\delta_1}^{n-1}(-1)^{n-j}\,(p_{n-j}\,v)^{(\delta_1+j-n)}(b)&=&0\,,\\\nonumber \vdots&&\\
%\label{Cf::adbe14}
%v^{(\delta_{k-1})}(b)+\sum_{j=n-\delta_{k-1}}^{n-1}(-1)^{n-j}\,(p_{n-j}\,v)^{(\delta_{k-1}+j-n)}(b)&=&0\,,\\\label{Cf::adbe15}v^{(\gamma)}(b)&=&0\,.	
%\end{eqnarray}}
\end{itemize}

\begin{exemplo}
	\label{Ex::10}
	Arguing in an analogous way to  Example \ref{Ex::2}, we obtain
	
		\begin{equation*}
	 \begin{split} X_{\,\ \{0,1,2\}}^{*\{1\}}=&\left\lbrace v\in C^4(I)\ \mid\ v(a)=v(b)=v'(b)= v^{(3)}(b)-p_1(b)\,v''(b)=0\right\rbrace \,,\\
	 X_{\,\ \{0\}}^{*\{0,1,2\}}=&\left\lbrace v\in C^4(I)\ \mid\ v(a)=v'(a)=v''(a)= v(b)=0\right\rbrace=X_{\{0,1,2\}}^{\{0\}} \,,\\
	 X_{\,\ \{0,2\}}^{*\{0,1\}}=&\left\lbrace v\in C^4(I)\ \mid\ v(a)=v''(a)-p_1(a)\,v'(a)=v(b)=v'(b)=0\right\rbrace \,,\\
	X_{\,\ \{0,1\}}^{*\{1,2\}}=&\left\lbrace v\in C^4(I) \right.  \mid\ v^{(3)}(b)-p_1(b)\,v''(b)+(p_2(b)-2p_1'(b))v'(b)=0 \,, \\& \left. \ v(a)=v'(a)=v(b)=0\right\rbrace \,.\end{split}
		\end{equation*}
\end{exemplo}

In the sequel, we  prove analogous results to those of the previous section referring to functions defined in these spaces.

\begin{remark}
	In this case, taking into account that the eigenvalues of one operator and those of its adjoint are the same, we do not need to prove the existence of the eigenvalues. Such existence follows from the one of the eigenvalues of $T_n[\bar M]$ in the correspondent spaces, 
\end{remark}

First, we prove two results which refer to the operator $T^*_n[M]$ and then we will be able to extrapolate them for $\widehat{T}_n[(-1)^n M]$.

	\begin{proposition}
		\label{P::3}
		Let $\bar M\in \mathbb{R}$ be such that $T_n[\bar M]$ satisfies property $(T_d)$ on $X_{\{\sigma_1,\dots,\sigma_k\}}^{\{\varepsilon_1,\dots,\varepsilon_{n-k}\}}$  and ${\{\sigma_1,\dots,\sigma_k\}}-{\{\varepsilon_1,\dots,\varepsilon_{n-k}\}}$ satisfy $(N_a)$. Then every solution of $T^*_n[M]\,v(t)=0$, for $t\in (a,b)$, satisfying the boundary conditions \eqref{Cf::ad1}--\eqref{Cf::ad2} and \eqref{Cf::ad3}--\eqref{Cf::ad31}	does not have any zero on $(a,b)$ provided that one of the following assertions is fulfilled:
		\begin{itemize}
			\item Let $n-k$ be even:
			\begin{itemize}
				\item If $k>1$, $\varepsilon_{n-k}\neq n-k-1$ and $M\in[\bar{M}-\lambda_3'',\bar{M}-\lambda_2']$, where:
				\begin{itemize}
					\item $\lambda_3''>0$ is the least positive eigenvalue of $T_n[\bar M]$ in $X_{\{\sigma_1,\dots,\sigma_{k}\}}^{\{\varepsilon_1,\dots,\varepsilon_{n-k-1}|\beta\}}.$
					\item $\lambda_2'<0$ is the biggest negative eigenvalue of $T_n[\bar M]$ in $X_{\{\sigma_1,\dots,\sigma_{k-1}\}}^{\{\varepsilon_1,\dots,\varepsilon_{n-k}|\beta\}}.$	
				\end{itemize}
				\item If $k=1$, $\varepsilon_{n-1}\neq n-2$ and $M\in[\bar{M}-\lambda_3'',+\infty)$, where:
				\begin{itemize}
					\item $\lambda_3''>0$ is the least positive eigenvalue of $T_n[\bar M]$ in $X_{\{\sigma_1\}}^{\{\varepsilon_1,\dots,\varepsilon_{n-2},\beta\}}.$
				\end{itemize}
				
				\item If $k>1$, $\varepsilon_{n-k} = n-k-1$ and $M\in[\bar{M}-\lambda_1,\bar{M}-\lambda_2']$, where:
				\begin{itemize}
					\item $\lambda_1>0$ is the least positive eigenvalue of $T_n[\bar M]$ in $X_{\{\sigma_1,\dots,\sigma_k\}}^{\{0,\dots,n-k-1\}}.$
					\item $\lambda_2'<0$ is the biggest negative eigenvalue of $T_n[\bar M]$ in $X_{\{\sigma_1,\dots,\sigma_{k-1}\}}^{\{0,\dots,n-k-1,n-k\}}.$
				\end{itemize}	
				\item If $k=1$, $\varepsilon_{n-1}=n-2$ and $M\in[\bar{M}-\lambda_1,+\infty)$, where:	
				\begin{itemize}
					\item $\lambda_1>0$ is the least positive eigenvalue of $T_n[\bar M]$ in $X_{\{\sigma_1\}}^{\{0,\dots,n-2\}}.$
				\end{itemize}	
				
			\end{itemize}
			\item Let $n-k$ be odd:
			\begin{itemize}
				\item If $1<k<n-1$, $\varepsilon_{n-k}\neq n-k-1$ and $M\in[\bar{M}-\lambda_2',\bar{M}-\lambda_3'']$, where:
				\begin{itemize}
					\item $\lambda_3''<0$ is the biggest negative eigenvalue of $T_n[\bar M]$ in $X_{\{\sigma_1,\dots,\sigma_{k}\}}^{\{\varepsilon_1,\dots,\varepsilon_{n-k-1}|\beta\}}.$
					\item $\lambda_2'>0$ is the least positive eigenvalue of $T_n[\bar M]$ in $X_{\{\sigma_1,\dots,\sigma_{k-1}\}}^{\{\varepsilon_1,\dots,\varepsilon_{n-k}|\beta\}}.$	
				\end{itemize}
				\item If $1<k=n-1$, $\varepsilon_{1}\neq 0$ and $M\in[\bar{M}-\lambda_2',\bar{M}-\lambda_3'']$, where:
				\begin{itemize}
					\item $\lambda_3''<0$ is the least biggest negative eigenvalue of $T_n[\bar M]$ in $X_{\{\sigma_1,\dots,\sigma_{n-1}\}}^{\{\beta\}}$, where $\beta=0$.
					\item $\lambda_2'>0$ is the least positive eigenvalue of $T_n[\bar M]$ in $X_{\{\sigma_1,\dots,\sigma_{n-2}\}}^{\{\varepsilon_1|\beta\}}.$
				\end{itemize}
				\item If $k=1<n-1$, $\varepsilon_{n-1}\neq n-2$ and $M\in(-\infty,\bar{M}-\lambda_3'']$, where:
				\begin{itemize}
					\item $\lambda_3''<0$ is the biggest negative eigenvalue of $T_n[\bar M]$ in $X_{\{\sigma_1\}}^{\{\varepsilon_1,\dots,\varepsilon_{n-2},\beta\}}.$
				\end{itemize}
				\item If $k=1$, $n=2$, $\varepsilon_1\neq 0$ and $M\in(-\infty,\bar{M}-\lambda_3'']$, where:
				\begin{itemize}
					\item $\lambda_3''<0$ is the biggest negative eigenvalue of $T_n[\bar M]$ in $X_{\{\sigma_1\}}^{\{\beta\}}=X_{\{0\}}^{\{0\}}.$
				\end{itemize}
				\item If $1<k$, $\varepsilon_{n-k} = n-k-1$ and $M\in[\bar{M}-\lambda_2',\bar{M}-\lambda_1]$, where:
				\begin{itemize}
					\item $\lambda_1<0$ is the biggest negative eigenvalue of $T_n[\bar M]$ in $X_{\{\sigma_1,\dots,\sigma_k\}}^{\{0,\dots,n-k-1\}}.$
					\item $\lambda_2'>0$ is the least positive eigenvalue of $T_n[\bar M]$ in $X_{\{\sigma_1,\dots,\sigma_{k-1}\}}^{\{0,\dots,n-k-1,n-k\}}.$
				\end{itemize}	
				\item If $k=1$,  $\varepsilon_{n-1}=n-2$ and $M\in(-\infty,\bar{M}-\lambda_1]$, where:	
				\begin{itemize}
					\item $\lambda_1<0$ is the biggest negative eigenvalue of $T_n[\bar M]$ in $X_{\{\sigma_1\}}^{\{0,\dots,n-2\}}.$
				\end{itemize}	
				
			\end{itemize}
		\end{itemize}
	\end{proposition}
	\begin{proof}
		
The proof follows the same steps as Proposition \ref{P::1}.

Let us denote $v_M\in C^n(I)$ a solution of
\begin{equation}
\label{Ec::Ad} T^*_n[M]\,v(t)=0\,,\quad t \in (a,b)\,,
\end{equation}
satisfying the boundary conditions \eqref{Cf::ad1}--\eqref{Cf::ad2} and \eqref{Cf::ad3}--\eqref{Cf::ad31}.

At the beginning, let us see that $v_{\bar M}$ 	does not have any zero on $(a,b)$.  In order to see that, we consider the decomposition \eqref{Ec::Td*} whose existence is guaranteed by Lemma \ref{L::0Ad}.

Analogously to the proof of Lemma \ref{L::11}, since $w_0,\,\dots,\,w_n>0$, we can conclude that, without taking into account the boundary conditions, a solution of \eqref{Ec::Ad} for $M=\bar M$ can have at most $n-1$ zeros. However, as we have said before, each time that either $T^*_\ell v_M(a)=0$ or $T^*_\ell v_M(b)=0$, a possible oscillation is lost. From the boundary conditions and Lemmas \ref{L::3} and \ref{L::4}, taking into account that $T_\ell^*v_M(t)=(-1)^\ell \widehat T_\ell v_M(t)$, we can affirm that for every $M\in \mathbb{R}$:
 		\begin{eqnarray}
 		\label{Ec::cfT*aa} T^*_{\tau_1}\,v_{M}(a)=\cdots=T^*_{\tau_{n-k}}\,v_M(a)&=&0\,,\\\nonumber\\
 		\label{Ec::cfT*bb} T^*_{\delta_1}\,v_M(b)=\cdots=T^*_{\delta_{k-1}}\,v_M(b)&=&0\,.
 		\end{eqnarray}

 Thus, every nontrivial solution of \eqref{Ec::Ad} for $M=\bar M$ does not have any zero on $(a,b)$.
 
 \vspace{0.5cm}

Now, let us move $v_M$ continuously as a function of $M$ on a neighborhood of $M=\bar M$. We have that $v_M$ is a solution of \eqref{Ec::Ad}, hence:
\begin{equation}
\label{Ec::T*NM} T^*_n[\bar M]v_M(t)=(\bar M-M)\,v_M\,,\ t\in (a,b)\,.
\end{equation}

Analogously to the proof of Proposition \ref{P::1}, we will see that, while $v_M$ is of constant sign, it cannot have any double zero on $(a,b)$.

We can assume that $v_{\bar M}>0$ on $I$ (if $v_{\bar M}<0$, then the arguments are valid by multiplying by $-1$). So, in equation \eqref{Ec::T*NM} we have:		
		\begin{equation}\label{Ec::115}\left\lbrace \begin{array}{ccc}
		T_n^*[\bar M]\,v_M(t)\geq 0\,,&t\in I\,,& \text{ if $M<\bar M$,}\\&&\\
		T_n^*[\bar M]\,v_M(t)\leq 0\,,&t\in I\,,& \text{ if $M>\bar M$.}\end{array}\right.\end{equation}
		
		In both cases, since $\frac{-1}{w_n}<0$, $T_{n-1}^*v_M$ is a monotone function, with at most one zero. Studying the maximal oscillation of $T_{n-\ell}^*v_M$ for $\ell=2,\dots,n$, we conclude that $T_{n-\ell}^*v_M$ has at most $\ell$ zeros. 
		
		In particular, $T_0^*\,v_M$ has no more than $n$ zeros. Since $w_0>0$, we can affirm that $v_M$ has at most $n$ zeros.
		
		However, $v_M$ verifies \eqref{Ec::cfT*aa}-\eqref{Ec::cfT*bb}, hence $n-1$ possible oscillation are lost. Thus, $v_M$ can have at most a simple zero on $(a,b)$ which is not possible if it is of constant sign.
		
		Let us assume that $k\neq 1$ and that $\delta_k\neq k-1$ (this is equivalent to $\varepsilon_{n-k}\neq n-k-1$). Under these assumptions, we can affirm that $v_M$ is of constant sign until one of the following boundary conditions is fulfilled:
		\[v_M^{(\eta)}(a)=0\quad \text{or}\quad v_M^{(\gamma)}(b)=0\,.\]
		
		Let us study what happens by moving $M$. Since we are considering $v_M\geq 0$, we have:		
			\begin{equation}\label{Ec::Suab*} v^{(\eta)}_M(a)\geq0\,,\quad \text{and}\quad \left\lbrace \begin{array}{cc}
			v^{(\gamma)}_M(b)\geq 0\,,&\text{ if $\gamma$ is even,}\\&\\
			v^{(\gamma)}_M(b)\leq 0\,,&\text{ if $\gamma$ is odd.}\end{array}\right.\end{equation}
		
		Now, let us see how  $v_M^{(\eta)}(a)$ and $v_M^{(\gamma)}(b)$ are with maximal oscillation.
		
		As before, with maximal oscillation only one zero on the boundary is allowed. If $T^*_{\ell}v_M(a)=0$ for $\ell\notin\{\tau_1,\dots,\tau_{n-k},\eta\}$ or $T^*_{\ell}v_M(b)=0$ for $\ell\notin \{\delta_1,\dots,\delta_{k-1},\gamma\}$, we have that $T^*_{\eta}v_M(a)\neq 0$ and $T^*_{\gamma}v_M(b)\neq 0$.  Because, otherwise, $v_M\equiv 0$ on $I$ and we are looking for nontrivial solutions.
		
			From \eqref{Ec::Tgg}, taking into account that $T^*_\ell v_M(t)=(-1)^\ell\widehat{T}_\ell\,v_M(t)$, we obtain
		\begin{equation}\label{Ec::Talbe}\begin{split} T^*_\eta v_M(a)&=(-1)^\eta v_1(a)\,\dots\,v_{n-\eta}(a)\,v^{(\eta)}(a)\,,\\T^*_\gamma v_M(b)&=(-1)^\gamma v_1(b)\,\dots\,v_{n-\gamma}(b)\,v^{(\gamma)}(b)\,, \end{split}\end{equation}
		where $v_1\,\dots,v_n>0$ are given in \eqref{Ec::Td1}. 
		
		Hence, if $T^*_{\eta}v_M(a)\neq 0$ and $T^*_{\gamma}v_M(b)\neq 0$, then $v_M^{(\eta)}(a)\neq 0$ and $v_M^{(\gamma)}(b)\neq 0$, thus the function $v_M$ remains of constant sign.
		
		Thus, we can assume that the unique zero, which is allowed with maximal oscillation, is found either in $T_{\eta}^*v_M(a)$ or $T^*_{\gamma}v_M(b)$.
	
		\vspace{0.5cm}

	In this case, since $T_k^*v_M=\frac{-1}{w_k}\frac{d}{dt}\left( T_{k-1}^*v_M\right)$  with $w_k>0$, to allow the maximal oscillation, $T_{n-\ell}^*v_M(a)$ remains of constant sign, each time that it does not vanish and, if it vanishes, then it changes its sign the number of times that it has vanished on the next $\ell$ where it is non null. And $T_{n-\ell}^*v_M(b)$ changes its sign each time that it is non null.
	 
	At first, let us focus on the case $M<\bar M$, we have that $T_n^*[\bar M]v_M=T_n^*v_M\geq 0$ on $I$.

 In particular, $T_n^*v_M(a)\geq 0$ and $T_n^*v_M(b)\geq 0$. Using \eqref{Ec::cfT*aa}-\eqref{Ec::cfT*bb}, from $\ell =0$ to $n-\eta$, we have that $T_{n-\ell}v_M(a)$ vanishes $n-k-\eta$ times and from $\ell =0$ to $n-\gamma$, $T_{n-\ell}v_M(b)=0$ $k-1-\gamma$ times.
	
	Hence, to allow the maximal oscillation:
			\begin{equation}\label{Ec::Maxos*1}
			\left\lbrace \begin{array}{cc}
			T_{\eta}^*\,v_M(a)\geq0\,,& \text{if $n-k-\eta$ is even,}\\&\\
			T_{\eta}^*\,v_M(a)\leq0\,,& \text{if $n-k-\eta$ is odd,}\end{array}\right. 	
			\end{equation}
			and \begin{equation}
			\label{Ec::Maxos*2}
			\left\lbrace \begin{array}{cc}
			T_{\gamma}^*\,v_M(b)\geq0\,,& \text{if $n-\gamma-(k-1-\gamma)=n-k+1$ is even,}\\&\\
			T_{\gamma}^*\,v_M(b)\leq 0\,,& \text{if $n-k+1$ is odd.}\end{array}\right.
			\end{equation}
			
	 Using \eqref{Ec::Talbe} and \eqref{Ec::Maxos*1}-\eqref{Ec::Maxos*2}, we can affirm that to set maximal oscillation:
			
			\begin{equation}
			\left\lbrace \begin{array}{cc}
			v^{(\eta)}_M(a)\geq0\,,& \text{if $n-k$ is even,}\\&\\
			v^{(\eta)}_M(a)\leq0\,,& \text{if $n-k$ is odd,}\end{array}\right. 	
			\end{equation}
			and,
			\begin{itemize}
				\item If $n-k$ is even				 			
				\begin{equation}\left\lbrace \begin{array}{cc}
				v^{(\gamma)}_M(b)\geq0\,,& \text{if $\gamma$ is odd,}\\&\\
				v^{(\gamma)}_M(b)\leq 0\,,& \text{if $\gamma$ is even.}\end{array}\right.
				\end{equation}
				\item If $n-k$ is odd				 			
				\begin{equation}
				\left\lbrace \begin{array}{cc}
				v^{(\gamma)}_M(b)\leq0\,,& \text{if $\gamma$ is odd,}\\&\\
				v^{(\gamma)}_M(b)\geq 0\,,& \text{if $\gamma$ is even.}\end{array}\right.
				\end{equation}				 		
			\end{itemize}
			
			Hence, taking into account \eqref{Ec::Suab*}, we arrive at the following conclusions:
			\begin{itemize}
				\item If $n-k$ is even, the maximal oscillation is not allowed until $v^{(\gamma)}_M(b)=0$, i.e., $u_M>0$ on $(a,b)$ for $[\bar M-{\lambda_3^*}'',\bar M]$, where ${\lambda_3^*}''>0$ is the least positive eigenvalue of $T^*_n[\bar M]$ in $X_{\ \,\{\sigma_1,\dots,\sigma_{k}\}}^{*\{\varepsilon_1,\dots,\varepsilon_{n-k-1}|\beta\}}$.
				\item If $n-k$ is odd, the maximal oscillation is not allowed until $u^{(\eta)}(a)_M=0$, i.e. $u_M>0$ on $(a,b)$ for $[\bar M -{\lambda_2^*}',\bar M]$, where  ${\lambda_2^*}'<0$ is the biggest negative eigenvalue of $T^*_n[\bar M]$ in $X_{\ \,\{\sigma_1,\dots,\sigma_{k-1}\}}^{*\{\varepsilon_1,\dots,\varepsilon_{n-k}|\beta\}}$.
			\end{itemize}
			
			Moreover, since the eigenvalues of an operator and its adjoint are the same, we can affirm that $\lambda_3''={\lambda_3^*}''$ and $\lambda_2'={\lambda_2^*}'$.
			
			\vspace{0.5 cm}
		Consider now the other case, i.e. $M>\bar M$. From \eqref{Ec::115}, we have that $T_n^*v_M\leq0$. Thus, to obtain the maximal oscillation, the inequalities \eqref{Ec::Maxos*1}-\eqref{Ec::Maxos*2} must be reversed. So, taking into account \eqref{Ec::Talbe}, we can affirm that to get maximal oscillation:
		\begin{equation}
			\left\lbrace \begin{array}{cc}
				v^{(\eta)}_M(a)\leq0\,,& \text{if $n-k$ is even,}\\&\\
				v^{(\eta)}_M(a)\geq0\,,& \text{if $n-k$ is odd,}\end{array}\right. 	
		\end{equation}
		and,
		\begin{itemize}
			\item If $n-k$ is even				 			
			\begin{equation}\left\lbrace \begin{array}{cc}
			v^{(\gamma)}_M(b)\leq0\,,& \text{if $\gamma$ is odd,}\\&\\
			v^{(\gamma)}_M(b)\geq 0\,,& \text{if $\gamma$ is even.}\end{array}\right.
			\end{equation}
			\item If $n-k$ is odd				 			
			\begin{equation}
			\left\lbrace \begin{array}{cc}
			v^{(\gamma)}_M(b)\geq0\,,& \text{if $\gamma$ is odd,}\\&\\
			v^{(\gamma)}_M(b)\leq 0\,,& \text{if $\gamma$ is even.}\end{array}\right.
			\end{equation}				 		
		\end{itemize}
		
		Hence, we arrive at the following conclusions, taking into account \eqref{Ec::Suab*}:
		\begin{itemize}
			\item If $n-k$ is even, the maximal oscillation is not allowed until $v^{(\eta)}_M(a)=0$, i.e., $v_M>0$ on $(a,b)$ for $[\bar M,\bar M-{\lambda_2^*}']=[\bar M,\bar M-{\lambda_2}']$.
			\item If $n-k$ is odd, the maximal oscillation is not allowed until $v_M^{(\gamma)}(b)=0$, i.e. $v_M>0$ on $(a,b)$ for$[\bar M,\bar M-{\lambda_3^*}'']=[\bar M,\bar M-{\lambda_3}'']$.
		\end{itemize}
		
		\vspace{0.5cm}
		
		Now, we realize that if $k=1$, $v_M^{(\eta)}(a)\neq0$ for all $M\in \mathbb{R}$, since the contrary implies that a nontrivial solution of the homogeneous linear differential equation \eqref{Ec::Ad} has a zero at $t=a$ of multiplicity $n$, which is not possible.
		
		Finally, if $\varepsilon_{n-k}=n-k-1$ or, which is the same, $\delta_k=k-1$, we consider $v_M^{(k-1)}(b)$ instead of $v_M^{(\gamma)}(b)=v_M^{(k)}(b)$ and, taking into account that, from $\ell=0$ to $n-(k-1)$,  $T_{n-\ell}v_M(b)\neq 0$, we obtain that to allow maximal oscillation the following properties hold:
		
		\begin{itemize}
			\item If $M<\bar M$ 			 			
			\begin{equation}\left\lbrace \begin{array}{cc}
			T_{k-1}^*v_M(b)\geq0\,,& \text{if $n-k$ is odd,}\\&\\
			T_{k-1}^*v_M(b)\leq 0\,,& \text{if $n-k$ is even.}\end{array}\right.
			\end{equation}
			\item If $M>\bar M$ 				 			
			\begin{equation}\left\lbrace \begin{array}{cc}
			T_{k-1}^*v_M(b)\leq0\,,& \text{if $n-k$ is odd,}\\&\\
			T_{k-1}^*v_M(b)\geq 0\,,& \text{if $n-k$ is even.}\end{array}\right.
			\end{equation}	 		
		\end{itemize}
		
		From \eqref{Ec::Tgg}, since $T_\ell^*v_M(t)=(-1)^\ell \widehat T_\ell v_M(t)$, we have that \[T_{k-1}^*v_M(b)=(-1)^{k-1}v_1(b)\,\cdots v_{n-k-1}(b)\,v^{(k-1)}_M(b)\,.\]
		
		 So,  we obtain
		\begin{itemize}
			\item If $M<\bar M$
		
			\begin{itemize}
				\item If $n-k$ is even				 			
				\begin{equation}\left\lbrace \begin{array}{cc}
				v^{(k-1)}_M(b)\geq0\,,& \text{if $k-1$ is odd,}\\&\\
				v^{(k-1)}_M(b)\leq 0\,,& \text{if $k-1$ is even.}\end{array}\right.
				\end{equation}
				\item If $n-k$ is odd				 			
				\begin{equation}
				\left\lbrace \begin{array}{cc}
				v^{(k-1)}_M(b)\leq0\,,& \text{if $k-1$ is odd,}\\&\\
				v^{(k-1)}_M(b)\geq 0\,,& \text{if $k-1$ is even.}\end{array}\right.
				\end{equation}				 		
			\end{itemize}
			\item If $M>\bar M$
			\begin{itemize}
				\item If $n-k$ is even				 			
				\begin{equation}\left\lbrace \begin{array}{cc}
				v^{(k-1)}_M(b)\leq0\,,& \text{if $k-1$ is odd,}\\&\\
				v^{(k-1)}_M(b)\geq 0\,,& \text{if $k-1$ is even.}\end{array}\right.
				\end{equation}
				\item If $n-k$ is odd				 			
				\begin{equation}
				\left\lbrace \begin{array}{cc}
				v^{(k-1)}_M(b)\geq0\,,& \text{if $k-1$ is odd,}\\&\\
				v^{(k-1)}_M(b)\leq 0\,,& \text{if $k-1$ is even.}\end{array}\right.
				\end{equation}				 		
			\end{itemize}
			\end{itemize}	
			
		And, from \eqref{Ec::Suab*} we are able to finish the proof:
		\begin{itemize}
			\item If $n-k$ is even and $M<\bar M$, the maximal oscillation is not allowed until $v^{(k-1)}_M(b)=0$, i.e., $v_M>0$ on $(a,b)$ for $[\bar M-\lambda_1^*,\bar M]$, where $\lambda_1^*$ is the least positive eigenvalue of $T^*_n[\bar M]$ in $X_{\ \,\{\sigma_1,\dots,\sigma_k\}}^{*\{\varepsilon_1,\dots,\varepsilon_{n-k}\}}$.
			\item If $n-k$ is odd and $M>\bar M$, the maximal oscillation is not allowed until $v^{(k-1)}_M(b)=0$, i.e. $v_M>0$ on $(a,b)$ for $[\bar M ,\bar M-\lambda_1^*]$.
		\end{itemize}		
		
		Due to the coincidence of the eigenvalues of an operator and the ones of its adjoint, we can affirm that $\lambda_1=\lambda_1^*$ and the proof is complete.
	\end{proof}
	
	Now, we obtain an analogous result for different boundary conditions:
		\begin{proposition}
			\label{P::4}
			Let $\bar M\in \mathbb{R}$ be such that $T_n[\bar M]$ satisfies property $(T_d)$ on $X_{\{\sigma_1,\dots,\sigma_k\}}^{\{\varepsilon_1,\dots,\varepsilon_{n-k}\}}$  and ${\{\sigma_1,\dots,\sigma_k\}}-{\{\varepsilon_1,\dots,\varepsilon_{n-k}\}}$ satisfy $(N_a)$. Then every solution of $T^*_n[M]\,v(t)=0$ for $t\in(a,b)$, satisfying the boundary conditions \eqref{Cf::ad1}--\eqref{Cf::ad11} and \eqref{Cf::ad3}--\eqref{Cf::ad4},	does not have any zero on $(a,b)$ provided that one of the following assertions is fulfilled:
				\begin{itemize}
					\item Let $n-k$ be even:
					\begin{itemize}
						\item If $k>1$, $\sigma_k\neq k-1$ and $M\in[\bar{M}-\lambda_3',\bar{M}-\lambda_2'']$, where:
						\begin{itemize}
							\item $\lambda_3'>0$ is the least positive eigenvalue of $T_n[\bar M]$ in $X_{\{\sigma_1,\dots,\sigma_{k-1}|\alpha\}}^{\{\varepsilon_1,\dots,\varepsilon_{n-k}\}}.$
							\item $\lambda_2''<0$ is the biggest negative eigenvalue of $T_n[\bar M]$ in $X_{\{\sigma_1,\dots,\sigma_k|\alpha\}}^{\{\varepsilon_1,\dots,\varepsilon_{n-k-1}\}}.$	
						\end{itemize}
						\item If $k=1$, $\sigma_1\neq 0$ and $M\in[\bar{M}-\lambda_3',\bar{M}-\lambda_2'']$, where:
						\begin{itemize}
							\item $\lambda_3'>0$ is the least positive eigenvalue of $T_n[\bar M]$ in $X_{\{\alpha\}}^{\{\varepsilon_1,\dots,\varepsilon_{n-1}\}},$ where $\alpha=0$.
							\item $\lambda_2''<0$ is the biggest negative eigenvalue of $T_n[\bar M]$ in $X_{\{\sigma_1,0\}}^{\{\varepsilon_1,\dots,\varepsilon_{n-2}\}}.$	
						\end{itemize}
						
						\item If  $\sigma_{k} = k-1$ and $M\in[\bar{M}-\lambda_1,\bar{M}-\lambda_2'']$, where:
						\begin{itemize}
							\item $\lambda_1>0$ is the least positive eigenvalue of $T_n[\bar M]$ in $X_{\{1,\dots,k-1\}}^{\{\varepsilon_1,\dots,\varepsilon_{n-k}\}}.$
							\item $\lambda_2''<0$ is the biggest negative eigenvalue of $T_n[\bar M]$ in $X_{\{0,\dots,k-1,k\}}^{\{\varepsilon_1,\dots,\varepsilon_{n-k-1}\}}.$	
						\end{itemize}

					\end{itemize}
					\item Let $n-k$ be odd:
						\begin{itemize}
							\item If $1<k<n-1$, $\sigma_k\neq k-1$ and $M\in[\bar{M}-\lambda_2'',\bar{M}-\lambda_3']$, where:
							\begin{itemize}
								\item $\lambda_3'<0$ is the biggest negative eigenvalue of $T_n[\bar M]$ in $X_{\{\sigma_1,\dots,\sigma_{k-1}|\alpha\}}^{\{\varepsilon_1,\dots,\varepsilon_{n-k}\}}.$
								\item $\lambda_2''>0$ is the least positive eigenvalue of $T_n[\bar M]$ in $X_{\{\sigma_1,\dots,\sigma_k|\alpha\}}^{\{\varepsilon_1,\dots,\varepsilon_{n-k-1}\}}.$	
							\end{itemize}
							\item If $1=k<n-1$, $\sigma_1\neq 0$ and $M\in[\bar{M}-\lambda_2'',\bar{M}-\lambda_3']$, where:
							\begin{itemize}
								\item $\lambda_3'<0$ is the biggest negative eigenvalue of $T_n[\bar M]$ in $X_{\{\alpha\}}^{\{\varepsilon_1,\dots,\varepsilon_{n-1}\}}$, where $\alpha=0$.
								\item $\lambda_2''>0$ is the least positive eigenvalue of $T_n[\bar M]$ in $X_{\{\sigma_1,0\}}^{\{\varepsilon_1,\dots,\varepsilon_{n-2}\}}.$	
							\end{itemize}
							\item If $1<k=n-1$, $\sigma_k\neq n-2$ and $M\in(-\infty,\bar{M}-\lambda_3']$, where:
							\begin{itemize}
								\item $\lambda_3'<0$ is the biggest negative eigenvalue of $T_n[\bar M]$ in $X_{\{\sigma_1,\dots,\sigma_{n-2},\alpha\}}^{\{\varepsilon_1\}}.$
							\end{itemize}
							\item If $k=1$, $n=2$, $\sigma_{1}\neq 0$ and $M\in(-\infty,\bar{M}-\lambda_3']$, where:
							\begin{itemize}
								\item $\lambda_3'<0$ is the biggest negative eigenvalue of $T_n[\bar M]$ in $X_{\{\alpha\}}^{\{\varepsilon_1\}}=X_{\{0\}}^{\{0\}}.$
							\end{itemize}
							\item If $k<n-1$, $\sigma_{k} = k-1$ and $M\in[\bar{M}-\lambda_2'',\bar{M}-\lambda_1]$, where:
							\begin{itemize}
								\item $\lambda_1<0$ is the biggest negative eigenvalue of $T_n[\bar M]$ in $X_{\{0,\dots,k-1\}}^{\{\varepsilon_1,\dots,\varepsilon_{n-k}\}}.$
								\item $\lambda_2''>0$ is the least positive eigenvalue of $T_n[\bar M]$ in $X_{\{0,\dots,k-1,k\}}^{\{\varepsilon_1,\dots,\varepsilon_{n-k-1}\}}.$	
							\end{itemize}	
							\item If $k=n-1$ and  $\sigma_{n-1}=n-2$ and $M\in(-\infty,\bar{M}-\lambda_1]$, where:	
							\begin{itemize}
								\item $\lambda_1<0$ is the biggest negative eigenvalue of $T_n[\bar M]$ in $X_{\{0,\dots,n-2\}}^{\{\varepsilon_1\}}.$
							\end{itemize}	
							
						\end{itemize}
					\end{itemize}
		\end{proposition}

		\begin{proof}
		The proof is analogous to Proposition \ref{P::3}.
	\end{proof}
	
	\begin{exemplo}
		\label{Ex::11}
Returning to our problem, introduced in Example \ref{Ex::6}, we have that operator ${T_4^0}^*[M]\,v(t)=v^{(4)}(t)+M\,v(t)=T_4^0[M]\,v(t)$ is defined in $$X_{\ \, \{0,2\}}^{*\{1,2\}}=\left\lbrace v\in C^4([0,1])\ \mid\ v(0)=v''(0)=v(1)=v^{(3)}(1)=0\right\rbrace \,,$$  as it is proved in \eqref{Ec::SExAd} because, in this case,  $p_1(t)=p_2(t)=p_3(t)=p_4(t)=0$ for all $t\in[0,1]$.

From Proposition \ref{P::3}, we conclude that each solution of $v^{(4)}(t)+M\,v(t)=0$ on $[0,1]$ satisfying the boundary conditions $v(0)=v''(0)=v(1)=0$ does not have any zero on $(0,1)$ for $M\in[-m_3^4,m_2^4]$, where $m_3$ and $m_4$ have been introduced in Example \ref{Ex::6}.

We note that such functions follow the expressions:
	{\footnotesize \[\left\lbrace  \begin{array}{cc}
	K\left( \dfrac{\sin(m\,t)}{\cos(m)}-\dfrac{\sinh(m\,t)}{\cosh(m)}\right) \,,&M=-m^4<0\,,\\\\K\left( t-t^3\right) \,,&M=0\,,\\\\K\,e^{-\frac{m t}{\sqrt{2}}} \left(\left(e^{\sqrt{2} m (t+1)}-1\right) \sin \left(\frac{m
		(t-1)}{\sqrt{2}}\right)+\left(e^{\sqrt{2} m}-e^{\sqrt{2} m t}\right) \sin \left(\frac{m
		(t+1)}{\sqrt{2}}\right)\right)\,,&M=m^4>0\,,\end{array}\right. \]}
where $K\in\mathbb{R}$.

Moreover, from Proposition \ref{P::4}, we can affirm that any solution of $v^{(4)}(t)+M\,v(t)=0$ on $[0,1]$, satisfying the boundary conditions $v(0)=v(1)=v^{(3)}(1)=0$, does not have any zero on $(0,1)$ for $M\in[-\pi^4,4\,\pi^4]$. One can show that such solutions are given as multiples of:
{\footnotesize \[\left\lbrace  \begin{array}{cc}
\begin{split}\cos (m-m t) (\sin (m)+\sinh (m))+\sin (m- t) (\cosh (m)-\cos (m))\\+\sinh (m-m t) (\cosh
(m)-\cos (m))-\cosh (m-m t) (\sin (m)+\sinh (m))\,,\end{split}&M=-m^4<0\,,\\\\t-t^2\,,&M=0\,,\\\\\begin{split}e^{-\frac{m t}{\sqrt{2}}} \left(-\left(e^{\sqrt{2} m (t-1)}-e^{\sqrt{2} m t}+e^{\sqrt{2}
	m}-1\right) \sin \left(\frac{m t}{\sqrt{2}}\right)\right. \\\left.+\left(e^{\sqrt{2} m t}-1\right) \cos
\left(\frac{m (t-2)}{\sqrt{2}}\right) -\left(e^{\sqrt{2} m t}-1\right) \cos \left(\frac{m
	t}{\sqrt{2}}\right)\right)\,,\end{split}&M=m^4>0\,,\end{array}\right. \]}

	\end{exemplo}
		
		Taking into account that if $v_M$ is a solution of \eqref{Ec::Ad}, then $(-1)^nv_M$ is  a solution of $\widehat{T}_n[(-1)^nM]v(t)=0$ for all $t\in I$, we obtain the analogous results for $\widehat{T}_n$.
		
			\begin{proposition}
				\label{P::5}
				Let $\bar M\in \mathbb{R}$ be such that $T_n[\bar M]$ satisfies property $(T_d)$ on $X_{\{\sigma_1,\dots,\sigma_k\}}^{\{\varepsilon_1,\dots,\varepsilon_{n-k}\}}$  and ${\{\sigma_1,\dots,\sigma_k\}}-{\{\varepsilon_1,\dots,\varepsilon_{n-k}\}}$ satisfy $(N_a)$. Then every solution of $\widehat{T}_n[(-1)^n M]\,v(t)=0$ for $t\in(a,b)$, satisfying the boundary conditions \eqref{Cf::ad1}--\eqref{Cf::ad2} and \eqref{Cf::ad3}--\eqref{Cf::ad31},	does not have any zero on $(a,b)$ provided that one of the following assertions is fulfilled:
				\begin{itemize}
					\item Let $k$ be even:
					\begin{itemize}
						\item If $k<1$, $\varepsilon_{n-k}\neq n-k-1$ and $M\in[\bar{M}-\lambda_3'',\bar{M}-\lambda_2']$, where:
						\begin{itemize}
							\item $\lambda_3''>0$ is the least positive eigenvalue of $T_n[\bar M]$ in $X_{\{\sigma_1,\dots,\sigma_{k}\}}^{\{\varepsilon_1,\dots,\varepsilon_{n-k-1}|\beta\}}.$
							\item $\lambda_2'<0$ is the biggest negative eigenvalue of $T_n[\bar M]$ in $X_{\{\sigma_1,\dots,\sigma_{k-1}\}}^{\{\varepsilon_1,\dots,\varepsilon_{n-k}|\beta\}}.$	
						\end{itemize}
					
						\item If $k=1$, $\varepsilon_{n-1}\neq n-2$ and $M\in[\bar{M}-\lambda_3'',+\infty)$, where:
						\begin{itemize}
							\item $\lambda_3''>0$ is the least positive eigenvalue of $T_n[\bar M]$ in $X_{\{\sigma_1\}}^{\{\varepsilon_1,\dots,\varepsilon_{n-2},\beta\}}.$
						\end{itemize}
					
						\item If $1<k$, $\varepsilon_{n-k} = n-k-1$ and $M\in[\bar{M}-\lambda_1,\bar{M}-\lambda_2']$, where:
						\begin{itemize}
							\item $\lambda_1>0$ is the least positive eigenvalue of $T_n[\bar M]$ in $X_{\{\sigma_1,\dots,\sigma_k\}}^{\{0,\dots,n-k-1\}}.$
							\item $\lambda_2'<0$ is the biggest negative eigenvalue of $T_n[\bar M]$ in $X_{\{\sigma_1,\dots,\sigma_{k-1}\}}^{\{0,\dots,n-k-1,n-k\}}.$	
						\end{itemize}	
						\item If $k=1$ and  $\varepsilon_{n-1}=n-2$ and $M\in[\bar{M}-\lambda_1,+\infty)$, where:	
						\begin{itemize}
							\item $\lambda_1>0$ is the least positive eigenvalue of $T_n[\bar M]$ in $X_{\{\sigma_1\}}^{\{0,\dots,n-2\}}.$
						\end{itemize}	
						
					\end{itemize}
					\item Let $k$ be odd:
					\begin{itemize}
						\item If $1<k<n-1$, $\varepsilon_{n-k}\neq n-k-1$ and $M\in[\bar{M}-\lambda_2',\bar{M}-\lambda_3'']$, where:
						\begin{itemize}
							\item $\lambda_3''<0$ is the biggest negative eigenvalue of $T_n[\bar M]$ in $X_{\{\sigma_1,\dots,\sigma_{k}\}}^{\{\varepsilon_1,\dots,\varepsilon_{n-k-1}|\beta\}}.$
							\item $\lambda_2'>0$ is the least positive eigenvalue of $T_n[\bar M]$ in $X_{\{\sigma_1,\dots,\sigma_{k-1}\}}^{\{\varepsilon_1,\dots,\varepsilon_{n-k}|\beta\}}.$	
						\end{itemize}
						\item If $1<k=n-1$, $\varepsilon_{1}\neq 0$ and $M\in[\bar{M}-\lambda_2',\bar{M}-\lambda_3'']$, where:
						\begin{itemize}
							\item $\lambda_3''<0$ is the least biggest negative eigenvalue of $T_n[\bar M]$ in $X_{\{\sigma_1,\dots,\sigma_{n-1}\}}^{\{\beta\}}$, where $\beta=0$.
							\item $\lambda_2'>0$ is the least positive eigenvalue of $T_n[\bar M]$ in $X_{\{\sigma_1,\dots,\sigma_{n-2}\}}^{\{\varepsilon_1|\beta\}}.$	
						\end{itemize}
						\item If $k=1<n-1$, $\varepsilon_{n-1}\neq n-2$ and $M\in(-\infty,\bar{M}-\lambda_3'']$, where:
						\begin{itemize}
							\item $\lambda_3''<0$ is the biggest negative eigenvalue of $T_n[\bar M]$ in $X_{\{\sigma_1\}}^{\{\varepsilon_1,\dots,\varepsilon_{n-2},\beta\}}.$
						\end{itemize}
						\item If $k=1$, $n=2$, $\varepsilon_1\neq 0$ and $M\in(-\infty,\bar{M}-\lambda_3'']$, where:
						\begin{itemize}
							\item $\lambda_3''<0$ is the biggest negative eigenvalue of $T_n[\bar M]$ in $X_{\{\sigma_1\}}^{\{\beta\}}=X_{\{0\}}^{\{0\}}.$
						\end{itemize}
						\item If $1<k$, $\varepsilon_{n-k} = n-k-1$ and $M\in[\bar{M}-\lambda_2',\bar{M}-\lambda_1]$, where:
						\begin{itemize}
							\item $\lambda_1<0$ is the biggest negative eigenvalue of $T_n[\bar M]$ in $X_{\{\sigma_1,\dots,\sigma_k\}}^{\{0,\dots,n-k-1\}}.$
							\item $\lambda_2'>0$ is the least positive eigenvalue of $T_n[\bar M]$ in $X_{\{\sigma_1,\dots,\sigma_{k-1}\}}^{\{0,\dots,n-k-1,n-k\}}.$
						\end{itemize}	
						\item If $k=1$ and  $\varepsilon_{n-1}=n-2$ and $M\in(-\infty,\bar{M}-\lambda_1]$, where:	
						\begin{itemize}
							\item $\lambda_1<0$ is the biggest negative eigenvalue of $T_n[\bar M]$ in $X_{\{\sigma_1\}}^{\{0,\dots,n-2\}}.$
						\end{itemize}	
						
					\end{itemize}
				\end{itemize}
			\end{proposition}
			
				\begin{proposition}
					\label{P::6}
					Let $\bar M\in \mathbb{R}$ be such that $T_n[\bar M]$ satisfies property $(T_d)$ on $X_{\{\sigma_1,\dots,\sigma_k\}}^{\{\varepsilon_1,\dots,\varepsilon_{n-k}\}}$  and ${\{\sigma_1,\dots,\sigma_k\}}-{\{\varepsilon_1,\dots,\varepsilon_{n-k}\}}$ satisfy $(N_a)$. Then every solution of $\widehat{T}_n[(-1)^nM]\,v(t)=0$ for $t\in(a,b)$, satisfying the boundary conditions \eqref{Cf::ad1}--\eqref{Cf::ad11} and \eqref{Cf::ad3}--\eqref{Cf::ad4},	does not have any zero on $(a,b)$ provided that one of the following assertions is fulfilled:
					\begin{itemize}
						\item Let $k$ be even:
						\begin{itemize}
							\item If $k>1$, $\sigma_k\neq k-1$ and $M\in[\bar{M}-\lambda_3',\bar{M}-\lambda_2'']$, where:
							\begin{itemize}
								\item $\lambda_3'>0$ is the least positive eigenvalue of $T_n[\bar M]$ in $X_{\{\sigma_1,\dots,\sigma_{k-1}|\alpha\}}^{\{\varepsilon_1,\dots,\varepsilon_{n-k}\}}.$
								\item $\lambda_2''<0$ is the biggest negative eigenvalue of $T_n[\bar M]$ in $X_{\{\sigma_1,\dots,\sigma_k|\alpha\}}^{\{\varepsilon_1,\dots,\varepsilon_{n-k-1}\}}.$	
							\end{itemize}
							\item If $k=1$, $\sigma_1\neq 0$ and $M\in[\bar{M}-\lambda_3',\bar{M}-\lambda_2'']$, where:
							\begin{itemize}
								\item $\lambda_3'>0$ is the least positive eigenvalue of $T_n[\bar M]$ in $X_{\{\alpha\}}^{\{\varepsilon_1,\dots,\varepsilon_{n-1}\}}$, where $\alpha=0$.
								\item $\lambda_2''<0$ is the biggest negative eigenvalue of $T_n[\bar M]$ in $X_{\{\sigma_1,0\}}^{\{\varepsilon_1,\dots,\varepsilon_{n-2}\}}.$	
							\end{itemize}
							
							\item If  $\sigma_{k} = k-1$ and $M\in[\bar{M}-\lambda_1,\bar{M}-\lambda_2'']$, where:
							\begin{itemize}
								\item $\lambda_1>0$ is the least positive eigenvalue of $T_n[\bar M]$ in $X_{\{1,\dots,k-1\}}^{\{\varepsilon_1,\dots,\varepsilon_{n-k}\}}.$
								\item $\lambda_2''<0$ is the biggest negative eigenvalue of $T_n[\bar M]$ in $X_{\{0,\dots,k-1,k\}}^{\{\varepsilon_1,\dots,\varepsilon_{n-k-1}\}}.$	
							\end{itemize}

						\end{itemize}
						\item Let $k$ be odd:
						\begin{itemize}
							\item If $1<k<n-1$, $\sigma_k\neq k-1$ and $M\in[\bar{M}-\lambda_2'',\bar{M}-\lambda_3']$, where:
							\begin{itemize}
								\item $\lambda_3'<0$ is the biggest negative eigenvalue of $T_n[\bar M]$ in $X_{\{\sigma_1,\dots,\sigma_{k-1}|\alpha\}}^{\{\varepsilon_1,\dots,\varepsilon_{n-k}\}}.$
								\item $\lambda_2''>0$ is the least positive eigenvalue of $T_n[\bar M]$ in $X_{\{\sigma_1,\dots,\sigma_k|\alpha\}}^{\{\varepsilon_1,\dots,\varepsilon_{n-k-1}\}}.$	
							\end{itemize}
							\item If $1=k<n-1$, $\sigma_1\neq 0$ and $M\in[\bar{M}-\lambda_2'',\bar{M}-\lambda_3']$, where:
							\begin{itemize}
								\item $\lambda_3'<0$ is the biggest negative eigenvalue of $T_n[\bar M]$ in $X_{\{\alpha\}}^{\{\varepsilon_1,\dots,\varepsilon_{n-1}\}}$, where $\alpha=0$.
								\item $\lambda_2''>0$ is the least positive eigenvalue of $T_n[\bar M]$ in $X_{\{\sigma_1,0\}}^{\{\varepsilon_1,\dots,\varepsilon_{n-2}\}}.$	
							\end{itemize}
							\item If $1<k=n-1$, $\sigma_{n-1}\neq n-2$ and $M\in(-\infty,\bar{M}-\lambda_3']$, where:
							\begin{itemize}
								\item $\lambda_3'<0$ is the biggest negative eigenvalue of $T_n[\bar M]$ in $X_{\{\sigma_1,\dots,\sigma_{n-2},\alpha\}}^{\{\varepsilon_1\}}.$
							\end{itemize}
							\item If $k=1$, $n=2$, $\sigma_{1}\neq 0$ and $M\in(-\infty,\bar{M}-\lambda_3']$, where:
							\begin{itemize}
								\item $\lambda_3'<0$ is the biggest negative eigenvalue of $T_n[\bar M]$ in $X_{\{\alpha\}}^{\{\varepsilon_1\}}=X_{\{0\}}^{\{0\}}.$
							\end{itemize}
							\item If $k<n-1$, $\sigma_{k} = k-1$ and $M\in[\bar{M}-\lambda_2'',\bar{M}-\lambda_1]$, where:
							\begin{itemize}
								\item $\lambda_1<0$ is the biggest negative eigenvalue of $T_n[\bar M]$ in $X_{\{0,\dots,k-1\}}^{\{\varepsilon_1,\dots,\varepsilon_{n-k}\}}.$
								\item $\lambda_2''>0$ is the least positive eigenvalue of $T_n[\bar M]$ in $X_{\{0,\dots,k-1,k\}}^{\{\varepsilon_1,\dots,\varepsilon_{n-k-1}\}}.$
							\end{itemize}	
							\item If $k=n-1$ and  $\sigma_{n-1}=n-2$ and $M\in(-\infty,\bar{M}-\lambda_1]$, where:	
							\begin{itemize}
								\item $\lambda_1<0$ is the biggest negative eigenvalue of $T_n[\bar M]$ in $X_{\{0,\dots,n-2\}}^{\{\varepsilon_1\}}.$
							\end{itemize}	
							
						\end{itemize}
					\end{itemize}
				\end{proposition}
				
				\begin{remark}
					In our recurrent example, we have that $n=4$ is even, so $\tilde T_4[(-1)^4M]\equiv T_4^*[M]$. Then, Example \ref{Ex::11} is also valid to illustrate Propositions \ref{P::5} and \ref{P::6}.
				\end{remark}
	\chapter[Characterization of the strongly positive (negative) character.]{Characterization of the strongly inverse positive (negative) character of $T_n[M]$ in 	 $X_{\{\sigma_1,\dots,\sigma_{k}\}}^{\{\varepsilon_1,\dots,\varepsilon_{n-k}\}}$.}	
	
	This chapter is devoted to obtain the main result of this work, such a result gives the characterization of the parameter's set where $T_n[M]$ is either strongly inverse positive or strongly inverse negative in  $X_{\{\sigma_1,\dots,\sigma_{k}\}}^{\{\varepsilon_1,\dots,\varepsilon_{n-k}\}}$. Such a characterization is obtained under the hypotheses that there exists $\bar M\in \mathbb{R}$ such that the operator $T_n[\bar M]$ satisfies property $(T_d)$ and, moreover, ${\{\sigma_1,\dots,\sigma_k\}}-{\{\varepsilon_1,\dots,\varepsilon_{n-k}\}}$ satisfy $(N_a)$. In such a case, from Theorem \ref{L::5}, it is known that if $n-k$ is even, then $T_n[\bar M]$ is strongly inverse positive in $X_{\{\sigma_1,\dots,\sigma_{k}\}}^{\{\varepsilon_1,\dots,\varepsilon_{n-k}\}}$ and, if $n-k$ is odd, then $T_n[\bar M]$ is strongly inverse negative in $X_{\{\sigma_1,\dots,\sigma_{k}\}}^{\{\varepsilon_1,\dots,\varepsilon_{n-k}\}}$.
	
The characterization here obtained is related to the parameter's set which contains $\bar M$. That is, if $n-k$ is even we characterize the parameter's set where $T_n[M]$ is strongly inverse positive in $X_{\{\sigma_1,\dots,\sigma_{k}\}}^{\{\varepsilon_1,\dots,\varepsilon_{n-k}\}}$ and, if $n-k$ is odd we characterize the parameter's set where $T_n[M]$ is strongly inverse negative in $X_{\{\sigma_1,\dots,\sigma_{k}\}}^{\{\varepsilon_1,\dots,\varepsilon_{n-k}\}}$. In particular, $\bar M$ belongs to those intervals.

\begin{theorem}\label{T::IPN}
		Let $\bar M\in \mathbb{R}$ be such that $T_n[\bar M]$ satisfies property $(T_d)$ on $X_{\{\sigma_1,\dots,\sigma_k\}}^{\{\varepsilon_1,\dots,\varepsilon_{n-k}\}}$  and ${\{\sigma_1,\dots,\sigma_k\}}-{\{\varepsilon_1,\dots,\varepsilon_{n-k}\}}$ satisfy $(N_a)$. The following properties are fulfilled:
		\begin{itemize}
			\item If $n-k$ is even and $2\leq k\leq n-1$, then $T_n[M]$ is strongly inverse positive in $X_{\{\sigma_1,\dots,\sigma_{k}\}}^{\{\varepsilon_1,\dots,\varepsilon_{n-k}\}}$ if, and only if, $M\in(\bar M-\lambda_1,\bar M-\lambda_2]$, where
			\begin{itemize}
				\item[*] $\lambda_1>0$ is the least positive eigenvalue of $T_n[\bar M]$ in $X_{\{\sigma_1,\dots,\sigma_{k}\}}^{\{\varepsilon_1,\dots,\varepsilon_{n-k}\}}.$ 
				\item [*]$\lambda_2<0$ is the maximum between:
				\begin{itemize}
					\item [·]$\lambda_2'<0$, the biggest negative eigenvalue of $T_n[\bar M]$ in $X_{\{\sigma_1,\dots,\sigma_{k-1}\}}^{\{\varepsilon_1,\dots,\varepsilon_{n-k}|\beta\}}.$
					\item [·]$\lambda_2''<0$ is the biggest negative eigenvalue of $T_n[\bar M]$ in $X_{\{\sigma_1,\dots,\sigma_k|\alpha\}}^{\{\varepsilon_1,\dots,\varepsilon_{n-k-1}\}}.$ 
				
				\end{itemize}
			\end{itemize}
			\item If $k=1$ and $n$ is odd, then $T_n[M]$ is strongly inverse positive in $X_{\{\sigma_1\}}^{\{\varepsilon_1,\dots,\varepsilon_{n-1}\}}$ if, and only if, $M\in(\bar M-\lambda_1,\bar M-\lambda_2]$, where
			\begin{itemize}
				\item[*] $\lambda_1>0$ is the least positive eigenvalue of $T_n[\bar M]$ in $X_{\{\sigma_1\}}^{\{\varepsilon_1,\dots,\varepsilon_{n-1}\}}.$
				\item[*] $\lambda_2<0$ is the biggest negative eigenvalue of $T_n[\bar M]$ in $X_{\{\sigma_1|\alpha\}}^{\{\varepsilon_1,\dots,\varepsilon_{n-2}\}}.$  
			\end{itemize}
			\item If $n-k$ is odd and $2\leq k\leq n-2$, then $T_n[M]$ is strongly inverse negative in $X_{\{\sigma_1,\dots,\sigma_{k}\}}^{\{\varepsilon_1,\dots,\varepsilon_{n-k}\}}$ if, and only if, $M\in[\bar M-\lambda_2,\bar M-\lambda_1)$, where
			\begin{itemize}
				\item[*] $\lambda_1<0$ is the biggest negative eigenvalue of $T_n[\bar M]$ in $X_{\{\sigma_1,\dots,\sigma_{k}\}}^{\{\varepsilon_1,\dots,\varepsilon_{n-k}\}}.$ 
				\item[*] $\lambda_2>0$ is the minimum between:
				\begin{itemize}
					\item[·] $\lambda_2'>0$, the least positive eigenvalue of $T_n[\bar M]$ in $X_{\{\sigma_1,\dots,\sigma_{k-1}\}}^{\{\varepsilon_1,\dots,\varepsilon_{n-k}|\beta\}}.$
					\item[·] $\lambda_2''>0$ is the least positive eigenvalue of $T_n[\bar M]$ in $X_{\{\sigma_1,\dots,\sigma_k|\alpha\}}^{\{\varepsilon_1,\dots,\varepsilon_{n-k-1}\}}.$  
				\end{itemize}
			\end{itemize}	
				\item If $k=1$ and $n>2$ is even, then $T_n[M]$ is strongly inverse negative in $X_{\{\sigma_1\}}^{\{\varepsilon_1,\dots,\varepsilon_{n-1}\}}$ if, and only if, $M\in[\bar M-\lambda_2,\bar M-\lambda_1)$, where
				\begin{itemize}
					\item [*]$\lambda_1<0$ is the biggest negative eigenvalue of $T_n[\bar M]$ in $X_{\{\sigma_1\}}^{\{\varepsilon_1,\dots,\varepsilon_{n-1}\}}.$
					\item [*]$\lambda_2>0$ is the least positive eigenvalue of $T_n[\bar M]$ in $X_{\{\sigma_1|\alpha\}}^{\{\varepsilon_1,\dots,\varepsilon_{n-2}\}}.$
				\end{itemize}
				\item If $k=n-1$ and $n>2$, then $T_n[M]$ is strongly inverse negative in $X_{\{\sigma_1\}}^{\{\varepsilon_1,\dots,\varepsilon_{n-1}\}}$ if, and only if, $M\in[\bar M-\lambda_2,\bar M-\lambda_1)$, where
				\begin{itemize}
					\item[*]$\lambda_1<0$ is the biggest negative eigenvalue of $T_n[\bar M]$ in $X_{\{\sigma_1,\dots,\sigma_{n-1}\}}^{\{\varepsilon_1\}}.$
					\item[*] $\lambda_2>0$ is the least positive eigenvalue of $T_n[\bar M]$ in $X_{\{\sigma_1,\dots,\sigma_{n-2}\}}^{\{\varepsilon_1|\beta\}}.$  
				\end{itemize}
				\item If $n=2$, then $T_n[M]$ is strongly inverse negative in $X_{\{\sigma_1\}}^{\{\varepsilon_1\}}$ if, and only if, $M\in(-\infty,\bar M-\lambda_1)$, where
				\begin{itemize}
					\item[*] $\lambda_1<0$ is the biggest negative eigenvalue of $T_n[\bar M]$ in $X_{\{\sigma_1\}}^{\{\varepsilon_1\}}.$
				\end{itemize} 
		\end{itemize}
\end{theorem}

\begin{proof}
	From Lemma \ref{L::5}, we know that operator $T_n[\bar M]$ satisfies property $(P_g)$ and is strongly inverse positive in $X_{\{\sigma_1,\dots,\sigma_k\}}^{\{\varepsilon_1,\dots,\varepsilon_{n-k}\}}$ if $n-k$ is even. Moreover, it satisfies $(N_g)$ and is strongly inverse negative in $X_{\{\sigma_1,\dots,\sigma_k\}}^{\{\varepsilon_1,\dots,\varepsilon_{n-k}\}}$  if $n-k$ is odd.
	
	Then, using Theorems \ref{T::d1}, \ref{T::in2}, \ref{T::6} and \ref{T::7}, we conclude that
	\begin{itemize}
		\item If $n-k$ is even and $M\leq\bar M$, then $T_n[M]$ is strongly inverse positive in $X_{\{\sigma_1,\dots,\sigma_k\}}^{\{\varepsilon_1,\dots,\varepsilon_{n-k}\}}$ if, and only if, $M\in(\bar M-\lambda_1,\bar M]$.
		\item  If $n-k$ is odd and $M\geq \bar M$, then $T_n[M]$ is strongly inverse negative in $X_{\{\sigma_1,\dots,\sigma_k\}}^{\{\varepsilon_1,\dots,\varepsilon_{n-k}\}}$ if, and only if, $M\in[\bar M,\bar M-\lambda_1)$.
	\end{itemize}
	
To obtain the other extreme of the interval we use the characterization of the strongly inverse positive (negative) character given in Theorems \ref{T::in2} and \ref{T::in21}. 

The proof follows several steps. 

In order to make the paper more readable, we indicate the different steps for the case with $n-k$ even. For the case with $n-k$ odd the proof is analogous.

	\begin{itemize}
		\item[Step 1.] Study of the related Green's function at $s=a$.
		\item[Step 2.] Study of  the related Green's function at $s=b$.
		\item[Step 3.] Study of the related Green's function at $t=a$.
		\item[Step 4.] Study of  the related Green's function at $t=b$.
		\item[Step 5.] Study of  the related Green's function on $(a,b)\times(a,b)$.
	\end{itemize}
	 
	 Let us denote
	 \[g_M(t,s)=\left\lbrace \begin{array}{cc}
	 g_M^1(t,s)\,,&a\leq s\leq t\leq b\,,\\\\
	 g_M^2(t,s)\,,&a<t<s<b\,,\end{array}\right. \]
	as the related Green's function of $T_n[M]$ in $X_{\{\sigma_1,\dots,\sigma_k\}}^{\{\varepsilon_1,\dots,\varepsilon_{n-k}\}}$.
	
	\begin{itemize}
		\item[Step 1.]Study of the related Green's function at $s=a$.
	\end{itemize}
	Let us consider $w_M(t)=\dfrac{\partial^\eta}{\partial s^\eta}g_M^1(t,s)_{\mid s=a}$, where $\eta$ has been defined in \eqref{Ec::eta}.
	
	Using \eqref{Ec::gg} and the boundary conditions of the adjoint operator given in \eqref{Cf::ad1}--\eqref{Cf::ad4}, if $\eta>0$, we obtain that :
	\[g_M^1(t,a)=\dfrac{\partial}{\partial s}g_M^1(t,s)_{\mid s=a}=\cdots=\dfrac{\partial^{\eta-1}}{\partial s^{\eta-1}}g_M^1(t,s)_{\mid s=a}=0\,.\]
	
	Realize that a necessary condition to ensure the inverse positive  character is that $w_M\geq 0$. Indeed, if there exists $t^*\in [a,b]$, such that $w_M(t^*)<0$, then there exists $\rho(t^*)>0$ such that $g_M(t^*,s)<0$ for all $s\in(0,\rho(t^*))$, which contradicts the  inverse positive  character. 
	
		Hence from Lemma \ref{L::5}, we have $w_{\bar M}\geq 0$ if $n-k$ is even.
	
	Moreover, since $g_M(t,s)$ is the related Green's function of $T_n[M]$ in $X_{\{\sigma_1,\dots,\sigma_k\}}^{\{\varepsilon_1,\dots,\varepsilon_{n-k}\}}$, we have that $T_n[M]\,g_M(t,a)=0$ for all $t\in(a,b]$. Hence
	\[\dfrac{\partial^\eta}{\partial s^\eta}\left( T_n[M]\,g_M(t,s)\right)_{\mid s=a}=T_n[M]\,w_M(t)=0\,,\quad t\in(a,b]\,. \]
	
	Now, let us see which boundary conditions are satisfied by $w_M$.
	
	To this end, we  use the Green's matrix for the vectorial problem \eqref{Ec::vec}-\eqref{Ec::Cf}, introduced in \eqref{Ec:MG}, where the expression of $g_{n-j}(t,s)$ is given in \eqref{Ec::gj} for $j=1,\dots,n-1$.
	
	If $k>1$, considering the first row of \eqref{Ec::Cf}, we have 
	\[\left\lbrace \begin{array}{r}
	\dfrac{\partial^{\sigma_1}}{\partial t^{\sigma_1}}g_M^2(t,s)_{\mid t=a}=0\,,\\\\
	-\dfrac{\partial^{\sigma_1+1}}{\partial t^{\sigma_1}\partial s}g_M^2(t,s)_{\mid t=a}+\alpha_0^1(s)\dfrac{\partial^{\sigma_1}}{\partial t^{\sigma_1}}g_M^2(t,s)_{\mid t=a}=0\,,\\\\
	\vdots\hspace{0.5cm}\\\\
		(-1)^\eta\dfrac{\partial^{\sigma_1+\eta}}{\partial t^{\sigma_1}\partial s^\eta}g_M^2(t,s)_{\mid t=a}+\sum_{i=0}^{\eta-1}\alpha_i^\eta(s)\dfrac{\partial^{i+\sigma_1}}{\partial t^{\sigma_1}\partial s^i}g_M^2(t,s)_{\mid t=a}=0\,.\end{array}\right.\]
		
		This system is satisfied in particular for $s=a$. Since $\eta+\sigma_1<n-1$ we do not reach any diagonal element of $G(t,s)$, hence we obtain by continuity:		
			\[\left\lbrace \begin{array}{r}
			\dfrac{\partial^{\sigma_1}}{\partial t^{\sigma_1}}g_M^1(t,s)_{\mid (t,s)=(a,a)}=0\,,\\\\
			-\dfrac{\partial^{\sigma_1+1}}{\partial t^{\sigma_1}\partial s}g_M^1(t,s)_{\mid (t,s)=(a,a)}+\alpha_0^1(a)\dfrac{\partial^{\sigma_1}}{\partial t^{\sigma_1}}g_M^1(t,s)_{\mid (t,s)=(a,a)}=0\,,\\\\
			\vdots\hspace{0.5cm}\\\\
			(-1)^\eta\dfrac{\partial^{\sigma_1+\eta}}{\partial t^{\sigma_1}\partial s^\eta}g_M^1(t,s)_{\mid (t,s)=(a,a)}+\sum_{i=0}^{\eta-1}\alpha_i^\eta(a)\dfrac{\partial^{i+\sigma_1}}{\partial t^{\sigma_1}\partial s^i}g_M^1(t,s)_{\mid (t,s)=(a,a)}=0\,.\end{array}\right.\]
			
			Taking into account that $\alpha_i^j\in C(I)$, we have:
			\[w_M^{(\sigma_1)}(a)=\dfrac{\partial^{\sigma_1+\eta}}{\partial t^{\sigma_1}\partial s^\eta}g_M^1(t,s)_{\mid (t,s)=(a,a)}=0\,.\]
			
			Proceeding analogously for $\sigma_2,\dots,\sigma_{k-1}$, we obtain:
			\[w_M^{(\sigma_2)}(a)=\cdots=w_M^{(\sigma_{k-1})}(a)=0\,.\]
			
			Now, let us choose the row $\sigma_k$ of $G(t,s)$. From \eqref{Ec::Cf}, we have:
				\[\left\lbrace \begin{array}{r}
				\dfrac{\partial^{\sigma_k}}{\partial t^{\sigma_k}}g_M^2(t,s)_{\mid t=a}=0\,,\\\\
				-\dfrac{\partial^{\sigma_k+1}}{\partial t^{\sigma_k}\partial s}g_M^2(t,s)_{\mid t=a}+\alpha_0^1(s)\dfrac{\partial^{\sigma_k}}{\partial t^{\sigma_k}}g_M^2(t,s)_{\mid t=a}=0\,,\\\\
				\vdots\hspace{0.5cm}\\\\
				(-1)^\eta\dfrac{\partial^{\sigma_k+\eta}}{\partial t^{\sigma_k}\partial s^\eta}g_M^2(t,s)_{\mid t=a}+\sum_{i=0}^{\eta-1}\alpha_i^\eta(s)\dfrac{\partial^{i+\sigma_k}}{\partial t^{\sigma_k}\partial s^i}g_M^2(t,s)_{\mid t=a}=0\,.\end{array}\right.\]
				
				This system is satisfied in particular for $s=a$. However, since $\sigma_k+\eta=n-1$, we reach a diagonal element of $G(t,s)$. Hence, in order to express the previous system by means of $g^1_M(t,s)$, we have to take into account Remark \ref{R:2.5} to obtain:				
					\[\left\lbrace \begin{array}{r}
					\dfrac{\partial^{\sigma_k}}{\partial t^{\sigma_k}}g_M^1(t,s)_{\mid (t,s)=(a,a)}=0\,,\\\\
					-\dfrac{\partial^{\sigma_k+1}}{\partial t^{\sigma_k}\partial s}g_M^1(t,s)_{\mid (t,s)=(a,a)}+\alpha_0^1(a)\dfrac{\partial^{\sigma_k}}{\partial t^{\sigma_k}}g_M^1(t,s)_{\mid (t,s)=(a,a)}=0\,,\\\\
					\vdots\hspace{0.5cm}\\\\
					(-1)^\eta\dfrac{\partial^{\sigma_k+\eta}}{\partial t^{\sigma_k}\partial s^\eta}g_M^1(t,s)_{\mid (t,s)=(a,a)}+\sum_{i=0}^{\eta-1}\alpha_i^\eta(a)\dfrac{\partial^{i+\sigma_k}}{\partial t^{\sigma_k}\partial s^i}g_M^1(t,s)_{\mid (t,s)=(a,a)}=1\,.\end{array}\right.\]
					
					So, since $\alpha_i^j\in C(I)$, we have:
					\[w_M^{(\sigma_k)}(a)=\dfrac{\partial^{\sigma_k+\eta}}{\partial t^{\sigma_k}\partial s^\eta}g_M^1(t,s)_{\mid (t,s)=(a,a)}=(-1)^\eta=(-1)^{(n-1-\sigma_k)}\,.\]
					
					Analogously, if $k=1$, then $w_M^{(\sigma_1)}(a)=(-1)^{n-1-\sigma_1}$.
					
					Now, let us see what happens at $t=b$. If we consider the $(k+1)^\mathrm{th}$ row of \eqref{Ec::Cf}, we have					
						\[\left\lbrace \begin{array}{r}
						\dfrac{\partial^{\varepsilon_1}}{\partial t^{\varepsilon_1}}g_M^1(t,s)_{\mid t=b}=0\,,\\\\
						-\dfrac{\partial^{\varepsilon_1+1}}{\partial t^{\varepsilon_1}\partial s}g_M^1(t,s)_{\mid t=b}+\alpha_0^1(s)\dfrac{\partial^{\varepsilon_1}}{\partial t^{\varepsilon_1}}g_M^1(t,s)_{\mid t=b}=0\,,\\\\
						\vdots\hspace{0.5cm}\\\\
						(-1)^\eta\dfrac{\partial^{\varepsilon_1+\eta}}{\partial t^{\varepsilon_1}\partial s^\eta}g_M^1(t,s)_{\mid t=b}+\sum_{i=0}^{\eta-1}\alpha_i^\eta(s)\dfrac{\partial^{i+\varepsilon_1}}{\partial t^{\varepsilon_1}\partial s^i}g_M^1(t,s)_{\mid t=b}=0\,.\end{array}\right.\]
				
				Since $b\neq a$, this system is satisfied in particular at $s=a$. Thus, using that $\alpha_i^j\in C(I)$, we conclude:
				\[w_M^{(\varepsilon_1)}(b)=\dfrac{\partial^{\varepsilon_1+\eta}}{\partial t^{\varepsilon_1}\partial s^\eta}g_M^1(t,s)_{\mid (t,s)=(b,a)}=0\,.\]
				
				Proceeding analogously we obtain:
				\[w_M^{(\varepsilon_2)}(b)=\dots=w_M^{(\varepsilon_{n-k})}(b)=0\,.\]

				Hence, $w_M$ satisfies the boundary conditions \eqref{Ec::cfaa}-\eqref{Ec::cfbb}, so we can apply Proposition \ref{P::1} to affirm that:
				\begin{itemize}
					\item If $n-k$ is even and $k>1$, then $w_M> 0$ on $(a,b)$ for all $M\in [\bar M,\bar M-\lambda_2']$.
					\item If $k=1$ and $n$ is odd, then $w_M> 0$ on $(a,b)$ for all $M\geq \bar M$.
%					\item If $n-k$ is odd and $k>1$, then $w_M<0$ on $(a,b)$ for all $M\in [\bar M-\lambda_2',\bar M]$.
%					\item If $k=1$ and $n$ is even, then $w_M<0$ on $(a,b)$ for all $M\leq \bar M$.
				\end{itemize}
				
				\vspace{0.5cm}
				To finish this Step, let us see that if $n-k$ is even and $k>1$, then $T_n[M]$ cannot be inverse positive  for $M>\bar M-\lambda_2'$.

				Suppose that there exists $\widehat{M}>\bar M-\lambda_2'$ such that $T_n[\widehat{M}]$ is inverse positive in $X_{\{\sigma_1,\dots,\sigma_k\}}^{\{\varepsilon_1,\dots,\varepsilon_{n-k}\}}$, thus from Theorems \ref{T::d1} and \ref{T::int}, we can affirm that for every $M\in [\bar M-\lambda_2',\widehat{M}]$ operator $T_n[M]$ is inverse positive in $X_{\{\sigma_1,\dots,\sigma_k\}}^{\{\varepsilon_1,\dots,\varepsilon_{n-k}\}}$ and, moreover, $w_{\bar M-\lambda_2'}\geq w_M\geq w_{\widehat{M}}$.
				
				In particular, $0=w_{\bar M-\lambda_2'}^{(\beta)}(b)\leq w_M^{(\beta)}(b)\leq w_{\widehat{M}}^{(\beta)}(b)$ if $\beta$ is even and  $0=w_{\bar M-\lambda_2'}^{(\beta)}(b)\geq w_M^{(\beta)}(b)\geq w_{\widehat{M}}^{(\beta)}(b)$ if $\beta$ is odd.
				
				If $w_{\widehat{M}}^{(\beta)}(b)\neq 0$, then there exists $\rho>0$ such that $w_{\widehat{M}}(t)<0$ for all $t\in (b-\rho,b)$, which contradicts our assumption. So:
				\[0=w_{\bar M-\lambda_2'}^{(\beta)}(b)=w_M^{(\beta)}(b)= w_{\widehat{M}}^{(\beta)}(b)\,,\quad \forall M\in [\bar M-\lambda_2',\widehat{M}]\,,\]
				and this fact contradicts the discrete character of the spectrum of the operator $T_n[\bar M]$ in $X_{\{\sigma_1,\dots,\sigma_{k-1}\}}^{\{\varepsilon_1,\dots,\varepsilon_{n-k}|\beta\}}$.
				
				From this Step we obtain the following conclusions:
				\begin{itemize}
					\item If $n-k$ is even, $k>1$ and $M\in [\bar M,\bar M-\lambda_2']$:
					\[\forall t\in (a,b)\,,\quad \exists \rho(t)>0\ \mid \ g_M(t,s)>0\ \forall s\in(a,a+\rho(t))\,.\]
					
					Moreover, if $M>\bar M-\lambda_2'$; then $T_n[M]$ is not inverse positive in $X_{\{\sigma_1,\dots,\sigma_k\}}^{\{\varepsilon_1,\dots,\varepsilon_{n-k}\}}$.
					
					\item If $k=1$, $n$ is odd and $M\geq \bar M$:
					\[\forall t\in (a,b)\,,\quad \exists \rho(t)>0\ \mid \ g_M(t,s)>0\ \forall s\in(a,a+\rho(t))\,.\]					
%					\item If $n-k$ is odd, $k>1$ and  $M\in [\bar M-\lambda_2',\bar M]$:
%						\[\forall t\in (a,b)\,,\quad \exists \rho(t)>0\ \mid \ g_M(t,s)<0\ \forall s\in(a,a+\rho(t))\,.\]
%						
%						Moreover, if $M<\bar M-\lambda_2'$; then $T_n[M]$ is not inverse negative in $X_{\{\sigma_1,\dots,\sigma_k\}}^{\{\varepsilon_1,\dots,\varepsilon_{n-k}\}}$.
%					\item If $k=1$, $n$ is even and $M\leq \bar M$:
%						\[\forall t\in (a,b)\,,\quad \exists \rho(t)>0\ \mid \ g_M(t,s)<0\ \forall s\in(a,a+\rho(t))\,.\]
				\end{itemize}
				
					\begin{itemize}
						\item[Step 2.]Study of the related Green's function at $s=b$.
					\end{itemize}
					Analogously to Step 1, we consider the function
					\[y_M(t)=\dfrac{\partial ^\gamma}{\partial s^\gamma}g_M^2(t,s)_{\mid s=b}\,.\]
					
					In this case, from \eqref{Ec::gg} and the boundary conditions \eqref{Cf::ad1}--\eqref{Cf::ad4}, we obtain that if $\gamma>0$, then:
					\[g_M^2(t,b)=\dfrac{\partial}{\partial s}g_M^2(t,s)_{\mid s=b}=\cdots =\dfrac{\partial^{\gamma-1}}{\partial s^{\gamma-1}}g_M^2(t,s)_{\mid s=b}=0\,.\]
					
					We have the following assertions:
					
					\begin{itemize}
						\item If $\gamma$ is even and there exist $t^*\in(a,b)$ such that $y_M(t^*)<0$, then $T_n[M]$ cannot be inverse positive  in $X_{\{\sigma_1,\dots,\sigma_k\}}^{\{\varepsilon_1,\dots,\varepsilon_{n-k}\}}$.
						\item If $\gamma$ is odd and there exist $t^*\in(a,b)$ such that $y_M(t^*)>0$, then $T_n[M]$ cannot be inverse positive  in $X_{\{\sigma_1,\dots,\sigma_k\}}^{\{\varepsilon_1,\dots,\varepsilon_{n-k}\}}$.
					\end{itemize}
					
					The proof of these assertions is analogous to the proof in Step 1 for $w_M$.
%					 
%					 that if there exists $t^*\in(a,b)$ such that $w_M(t^*)<0$; then $T_n[M]$ cannot be inverse positive in $X_{\{\sigma_1,\dots,\sigma_k\}}^{\{\varepsilon_1,\dots,\varepsilon_{n-k}\}}$ and there exist $t^*\in(a,b)$ such that $y_M(t^*)<0$ ($y_M(t^*)>0$), then $T_n[M]$ cannot be inverse positive (negative) in $X_{\{\sigma_1,\dots,\sigma_k\}}^{\{\varepsilon_1,\dots,\varepsilon_{n-k}\}}$.
%					 
					So, from Lemma \ref{L::5}, we obtain that if $\gamma$ is even, then $y_{\bar M}\geq 0$ and if $\gamma$ is odd, then $y_{\bar M}\leq 0$.
					
					As in Step 1, it can be seen that 
					\[T_n[M]\,y_M(t)=0\,,\quad \forall\,t\in [a,b)\,.\]
					
				Moreover, we can obtain the boundary conditions which $y_M$ verifies. Studying the Green's function related to the first order vectorial problem given in \eqref{Ec:MG} and the boundary conditions \eqref{Ec::Cf}, we obtain that $y_M$ satisfies:
				\[\begin{split}
				y_M^{(\sigma_1)}(a)=\cdots=y_M^{(\sigma_k)}(a)=&0\,,\\
				y_M^{(\varepsilon_1)}(b)=\cdots=y_M^{(\varepsilon_{n-k-1})}(b)=&0\,,\\
				y_M^{(\varepsilon_{n-k})}(b)=&(-1)^{(n-\varepsilon_{n-k})}\,.
				\end{split}\]

						So, $y_M$ satisfies the boundary conditions \eqref{Ec::cfaaa}-\eqref{Ec::cfbbb}, then we can apply Proposition \ref{P::2} to conclude that 
						
							\begin{itemize}
								\item If $n-k$ is even and $k<n-1$, then $y_M> 0$ if $\gamma$ is even and $y_M<0$ if $\gamma$ is odd on $(a,b)$ for all $M\in [\bar M,\bar M-\lambda_2'']$.
%								\item If $n-k$ is odd and $k<n-1$, then $y_M<0$ if $\gamma$ is even and $y_M>0$ if $\gamma$ is odd on $(a,b)$ for all $M\in [\bar M-\lambda_2'',\bar M]$.
%								\item If $k=n-1$, then $y_M<0$  if $\gamma$ is even and $y_M>0$ if $\gamma$ is odd on $(a,b)$ for all $M\leq \bar M$.
							\end{itemize}
							
							\vspace{0.5cm}
							
							Analogously to Step 1, it can be seen that if $n-k$ is even, then $T_n[M]$ cannot be inverse positive for $M>\bar M-\lambda_2''$.
							
							So from Step 2, we obtain the following conclusions:
							
								\begin{itemize}
									\item If $n-k$ is even and $M\in [\bar M,\bar M-\lambda_2'']$:
									\[\forall t\in (a,b)\,,\quad \exists \rho(t)>0\ \mid \ g_M(t,s)>0\ \forall s\in(b-\rho(t),b)\,.\]
									
									Moreover, if $M>\bar M-\lambda_2''$; then $T_n[M]$ is not inverse positive in $X_{\{\sigma_1,\dots,\sigma_k\}}^{\{\varepsilon_1,\dots,\varepsilon_{n-k}\}}$.
									%									\item If $n-k$ is odd, $k<n-1$ and  $M\in [\bar M-\lambda_2'',\bar M]$:
%									\[\forall t\in (a,b)\,,\quad \exists \rho(t)>0\ \mid \ g_M(t,s)<0\ \forall s\in(b-\rho(t),b)\,.\]
%									
%									Moreover, if $M<\bar M-\lambda_2'$; then $T_n[M]$ is not inverse negative in $X_{\{\sigma_1,\dots,\sigma_k\}}^{\{\varepsilon_1,\dots,\varepsilon_{n-k}\}}$.
%									\item If $k=n-1$ and $M\leq \bar M$:
%									\[\forall t\in (a,b)\,,\quad \exists \rho(t)>0\ \mid \ g_M(t,s)<0\ \forall s\in(b-\rho(t),b)\,.\]
								\end{itemize}
								
								Realize that from these two Steps we can conclude that the intervals where $T_n[M]$ is  strongly inverse positive  cannot be increased. 
							
								The rest of the proof is focused into see that these intervals are the optimal ones.
								
			\begin{itemize}
				\item[Step 3.]Study of the related Green's function at $t=a$.
			\end{itemize}
			
			Let us denote 
			\[\widehat g_{(-1)^nM}(t,s)=\left\lbrace \begin{array}{cc}
			\widehat g_{(-1)^n M}^1(t,s)\,, &a\leq s\leq t\leq b\,,\\\\
			\widehat g_{(-1)^n M}^2(t,s)\,,&a<t<s<b\,,\end{array}\right. \]
			as the related Green's function of $\widehat T_n[(-1)^nM]$ in $X_{\ \,\{\sigma_1,\dots,\sigma_k\}}^{*\{\varepsilon_1,\dots,\varepsilon_{n-k}\}}$.
			
			To study the behavior at $t=a$, we consider the following function:
			\[\widehat{w}_M(t)=(-1)^n\dfrac{\partial^\alpha }{\partial\,s^\alpha}\widehat{g}_{(-1)^nM}^1(t,s)_{\mid s=a}\,,\quad t\in I\,.\]
			
			From \eqref{Ec::gg1}, it is satisfied that
			\begin{equation}\label{Ec::gs}\widehat w_M(s)=\dfrac{\partial^\alpha }{\partial\,t^\alpha}{g}_{M}^2(t,s)_{\mid t=a}\,,\quad s\in I\,, \end{equation}
			moreover, from the boundary conditions \eqref{Ec::cfa}-\eqref{Ec::cfb}, if $\alpha>0$ we obtain:
			\[g_M(a,s)=\dfrac{\partial}{\partial s}{g}_{M}(t,s)_{\mid t=a}=\cdots=\dfrac{\partial^{\alpha-1}}{\partial s^{\alpha-1}}{g}_{M}(t,s)_{\mid t=a}=0\,,\quad \forall s\in(a,b).\]
	
			Using the arguments of Step 1, we can affirm that if there exists $t^*\in(a,b)$ such that $\widehat w_M(t^*)<0$, then $T_n[M]$ is not inverse positive in $X_{\{\sigma_1,\dots,\sigma_k\}}^{\{\varepsilon_1,\dots,\varepsilon_{n-k}\}}$.
			
						Moreover, from Lemma \ref{L::5}, if $n-k$ is even, then $\widehat{w}_{\bar M}\geq 0$.
			
			From the expression of $\widehat{T}_n[(-1)^n\,M]$ given in \eqref{Ec::Tg} and \eqref{EC::Ad}, we construct the associated vectorial problem \eqref{Ec::vec} tacking, in this case			
{\footnotesize\[		\begin{split}	\nonumber	\widehat p_{n-j}(t)&= (-1)^{n+j}p_{n-j}(t)+(-1)^{(n+j+1)}(j+1)\,p_{n-j-1}'(t)+\cdots -\left(\begin{array}{c}n-1\\j\end{array} \right) \,p_1^{(n-j-1)}(t) \,,\\&j=1,\dots,n-1\,,\\		\nonumber	\widehat p_{n}(t)&=(-1)^n p_{n}(t)+(-1)^{n+1}p_{n-1}'(t)+(-1)^{n+2}\left(\begin{array}{c}2\\0\end{array} \right) \,p_{n-2}''(t)+\cdots - p_1^{(n-1)}(t)\,.\end{split}\]}

			Now, the related Green's function is given by:
		{\small\begin{equation}\label{Ec:MGh} \widehat G(t,s)=\left( \begin{array}{llll}
			\widehat g_1(t,s)&\cdots&\widehat g_{n-1}(t,s)&\widehat g_{(-1)^n\,M}(t,s)\\&&&\\
			\dfrac{\partial }{\partial t}\,\widehat g_1(t,s)& \cdots&\dfrac{\partial }{\partial t}\,\widehat g_{n-1}(t,s)& \dfrac{\partial }{\partial t}\,\widehat g_{(-1)^n\,M}(t,s)\\
		\qquad	\vdots&\cdots&\qquad\vdots&\qquad\vdots\\
			\dfrac{\partial^{n-1} }{\partial t^{n-1}}\,\widehat g_1(t,s)&\cdots&\dfrac{\partial^{n-1} }{\partial t^{n-1}}\,\widehat g_{n-1}(t,s)&\dfrac{\partial^{n-1}} {\partial t^{n-1}}\,\widehat g_{(-1)^n\,M}(t,s)\end{array} \right),\end{equation}}
			and, repeating the arguments done with $T_n[M]$, we obtain
			\begin{equation}
			\label{Ec::gjh} \widehat g_{n-j}(t,s)=(-1)^j\dfrac{\partial^j}{\partial\,s^j}\,\widehat g_{(-1)^nM}(t,s)+\sum_{i=0}^{j-1}\widehat \alpha_i^j(s)\,\dfrac{\partial^i}{\partial s^i}\widehat g_{(-1)^nM}(t,s)\,,
			\end{equation}
			where $\widehat\alpha_i^j(s)$  follow the recurrence formula \eqref{r2}--\eqref{r4} for this problem with the obvious notation.
			
		The correspondent boundary conditions \eqref{Ec::Cf} are given by the matrices $\widehat B$, $\widehat C\in \mathcal{M}_{n\times n}$, defined as follows:
		\[\begin{split}&\left. \begin{split}
				&\left( \widehat B\right)_{i\,\tau_i+1}=1\,,\\ 	&\left( \widehat B\right)_{i\,j}=0\,,\quad \tau_i+1<j\leq n\,,\quad\\	&\left( \widehat B\right)_{i\,\tau_i-h}=\widehat p_{h+1}(a)\,,\ h=0,\dots\tau_i-1\,,\quad \end{split} \right\rbrace\quad i=1,\dots,n-k\,,\\
		&\	\left( \widehat B\right)_{i\,j}=0\,,\quad  j=0,\dots,n\,,\quad i=n-k+1,\dots n\\
			&\	\left( \widehat C\right)_{i\,j}=0\,,\quad  j=0,\dots,n\,,\quad i=0,\dots,n-k\,,\\
			&\left. 	\begin{split}&
				\left( \widehat C\right)_{i\,\delta_{i-(n-k)}+1}=1\,,\\	&\left( \widehat C\right)_{i\,j}=0\,,\quad \delta_{i-(n-k)}+1<j\leq n\,,\quad\\	&\left( \widehat C\right)_{i\,\delta_{i-(n-k)}-h}=\widehat p_{h+1}(b)\,,\ h=0,\dots\delta_{i-(n-k)}-1\,,\quad \end{split}\right\rbrace\quad  i=n-k+1,\dots,n\,,
		\end{split}\]
		that is, for every $v\in C^n(I)$, we have:
		\[\widehat B\,\left( \begin{array}{c}
	v(a)\\\vdots  \\v^{(n-1)}(a) \end{array}\right) +\widehat C\,\left( \begin{array}{c}
	v(b)\\\vdots  \\v^{(n-1)}(b)\end{array}\right)  =\left( \begin{array}{c}
		W_1\\\vdots  \\W_n \end{array} \right)\,,\]
		where:		
		\begin{eqnarray}
		\nonumber W_1&=&v^{(\tau_1)}(a)+\sum_{j=n-\tau_1}^{n-1}(-1)^{n-j}\,(p_{n-j}\,v)^{(\tau_1+j-n)}(a)\,,\\
		\nonumber \vdots&&\\	
	\nonumber 	W_{n-k}&=&v^{(\tau_{n-k})}(a)+\sum_{j=n-\tau_{n-k}}^{n-1}(-1)^{n-j}\,(p_{n-j}\,v)^{(\tau_{n-k}+j-n)}(a)\,,\\
	\nonumber 	W_{n-k+1}&=&v^{(\delta_1)}(b)+\sum_{j=n-\delta_1}^{n-1}(-1)^{n-j}\,(p_{n-j}\,v)^{(\delta_1+j-n)}(b)\,,\\\nonumber \vdots&&\\	
	\nonumber 	W_n&=&v^{(\delta_k)}(b)+\sum_{j=n-\delta_k}^{n-1}(-1)^{n-j}\,(p_{n-j}\,v)^{(\delta_k+j-n)}(b)\,.	
		\end{eqnarray}	
		
		As in Steps 1 and 2, we can conclude that 
		
		\[\widehat T_n[(-1)^nM]\,\widehat{w}_M(t)=0\,,\quad t\in(a,b]\,.\]
		
	In the sequel, we obtain the boundary conditions for $\widehat{w}_M$. The used arguments are similar to the two previous steps.	
		
		By definition, $\widehat G(t,s)$ satisfies:
		\begin{equation}
		\label{Ec::CfGh} \widehat B\,\widehat G(a,s)+\widehat C\,\widehat G(b,s)=0\,,\quad  \forall s\in(a,b)\,.
		\end{equation}
		
		If $k<n-1$, let us consider the first row of \eqref{Ec::CfGh} to deduce:
		{\tiny 	\[\left\lbrace \begin{array}{rl}
			\dfrac{\partial^{\tau_1}}{\partial t^{\tau_1}}\widehat g_{(-1)^nM}^2(t,s)_{\mid t=a}+\sum_{j=n-\tau_1}^{n-1}(-1)^{n-j}	\dfrac{\partial^{\tau_1+j-n}}{\partial t^{\tau_1+j-n}}\left( p_{n-j}(t)\widehat g_{(-1)^nM}^2(t,s)\right) _{\mid t=a}&=0\,,\\\\
		-\left( \dfrac{\partial^{\tau_1+1}}{\partial t^{\tau_1}\partial s}\widehat g_{(-1)^nM}^2(t,s)_{\mid t=a}+\sum_{j=n-\tau_1}^{n-1}(-1)^{n-j}	\dfrac{\partial^{\tau_1+j-n+1}}{\partial t^{\tau_1+j-n}\partial s}\left( p_{n-j}(t)\widehat g_{(-1)^nM}^2(t,s)\right) _{\mid t=a}\right)& \\\\
		+\widehat\alpha_0^1(s)\left( \dfrac{\partial^{\tau_1}}{\partial t^{\tau_1}}\widehat g_{(-1)^nM}^2(t,s)_{\mid t=a}+\sum_{j=n-\tau_1}^{n-1}(-1)^{n-j}	\dfrac{\partial^{\tau_1+j-n}}{\partial t^{\tau_1+j-n}}\left( p_{n-j}(t)\widehat g_{(-1)^nM}^2(t,s)\right) _{\mid t=a}\right) &=0\,,\\\\
			\vdots\hspace{0.5cm}\\\\
				(-1)^\alpha\sum_{j=n-\tau_1}^{n-1}(-1)^{n-j}	\dfrac{\partial^{\tau_1+j-n+\alpha}}{\partial t^{\tau_1+j-n}\partial s^\alpha}\left( p_{n-j}(t)\widehat g_{(-1)^nM}^1(t,s)\right) _{\mid t=a}& \\\\
			+(-1)^\alpha\left( \dfrac{\partial^{\tau_1+\alpha}}{\partial t^{\tau_1}\partial s^\alpha}\widehat g_{(-1)^nM}^2(t,s)_{\mid t=a}\right) 	+\sum_{i=0}^{\alpha-1}\widehat \alpha_i^\alpha(s)\left( \dfrac{\partial^{\tau_1+i}}{\partial t^{\tau_1}\partial s^i}\widehat g_{(-1)^nM}^2(t,s)_{\mid t=a}\right. &\\\\\left. +\sum_{j=n-\tau_1}^{n-1}(-1)^{n-j}	\dfrac{\partial^{\tau_1+j-n+i}}{\partial t^{\tau_1+j-n}\partial s^i}\left( p_{n-j}(t)\widehat g_{(-1)^nM}^2(t,s)\right) _{\mid t=a}\right) &=0\,.\end{array}\right.\]}
		
	Since $\tau_1+\alpha<n-1$, we do not reach any diagonal element, hence the previous system is satisfied for $s=a$, and we obtain:
	
		{\tiny 	\[\left\lbrace \begin{array}{rl}
			\dfrac{\partial^{\tau_1}}{\partial t^{\tau_1}}\widehat g_{(-1)^nM}^1(t,s)_{\mid (t,s)=(a,a)}+\sum_{j=n-\tau_1}^{n-1}(-1)^{n-j}	\dfrac{\partial^{\tau_1+j-n}}{\partial t^{\tau_1+j-n}}\left( p_{n-j}(t)\widehat g_{(-1)^nM}^1(t,s)\right) _{\mid (t,s)=(a,a)}&=0\,,\\\\
			-\sum_{j=n-\tau_1}^{n-1}(-1)^{n-j}	\dfrac{\partial^{\tau_1+j-n+1}}{\partial t^{\tau_1+j-n}\partial s}\left( p_{n-j}(t)\widehat g_{(-1)^nM}^1(t,s)\right) _{\mid (t,s)=(a,a)}& \\\\- \dfrac{\partial^{\tau_1+1}}{\partial t^{\tau_1}\partial s}\widehat g_{(-1)^nM}^1(t,s)_{\mid (t,s)=(a,a)}
			+\widehat\alpha_0^1(a)\left( \dfrac{\partial^{\tau_1}}{\partial t^{\tau_1}}\widehat g_{(-1)^nM}^1(t,s)_{\mid (t,s)=(a,a)}\right. &\\\\\left. +\sum_{j=n-\tau_1}^{n-1}(-1)^{n-j}	\dfrac{\partial^{\tau_1+j-n}}{\partial t^{\tau_1+j-n}}\left( p_{n-j}(t)\widehat g_{(-1)^nM}^1(t,s)\right) _{\mid (t,s)=(a,a)}\right) &=0\,,\\\\
			\vdots\hspace{0.5cm}\\\\
			+(-1)^\alpha\sum_{j=n-\tau_1}^{n-1}(-1)^{n-j}	\dfrac{\partial^{\tau_1+j-n+\alpha}}{\partial t^{\tau_1+j-n}\partial s^\alpha}\left( p_{n-j}(t)\widehat g_{(-1)^nM}^1(t,s)\right) _{\mid (t,s)=(a,a)}& \\\\
			+(-1)^\alpha\left( \dfrac{\partial^{\tau_1+\alpha}}{\partial t^{\tau_1}\partial s^\alpha}\widehat g_{(-1)^nM}^1(t,s)_{\mid (t,s)=(a,a)}\right) +\sum_{i=0}^{\alpha-1}\widehat \alpha_i^\alpha(a)\left( \dfrac{\partial^{\tau_1+i}}{\partial t^{\tau_1}\partial s^i}\widehat g_{(-1)^nM}^1(t,s)_{\mid (t,s)=(a,a)}\right.& \\\\\left. +\sum_{j=n-\tau_1}^{n-1}(-1)^{n-j}	\dfrac{\partial^{\tau_1+j-n+i}}{\partial t^{\tau_1+j-n}\partial s^i}\left( p_{n-j}(t)\widehat g_{(-1)^nM}^1(t,s)\right) _{\mid (t,s)=(a,a)}\right) &=0\,.\end{array}\right.\]}
		
		Since $\widehat \alpha_i^j\in C(I)$,  we conclude that
		
		\[\widehat{w}_M^{(\tau_1)}(a)+\sum_{j=n-\tau_1}^{n-1}(-1)^{n-j}\,\left( p_{n-j}\widehat w_M\right) ^{(\tau_1+j-n)}(a)=0\,.\]
		
		Proceeding analogously with $\tau_2,\dots,\tau_{n-k-1}$, we can ensure that $\widehat w_M$ satisfies the boundary conditions \eqref{Cf::ad1}--\eqref{Cf::ad11}.
		
		Now, let us see what happens for $\tau_{n-k}$. From \eqref{Ec::CfGh}, we obtain that for all $s\in(a,b)$ the following equalities hold:
			{\tiny 	\[\left\lbrace \begin{array}{rl}
				\dfrac{\partial^{\tau_{n-k}}}{\partial t^{\tau_{n-k}}}\widehat g_{(-1)^nM}^2(t,s)_{\mid t=a}+\sum_{j=n-\tau_{n-k}}^{n-1}(-1)^{n-j}	\dfrac{\partial^{\tau_{n-k}+j-n}}{\partial t^{\tau_{n-k}+j-n}}\left( p_{n-j}(t)\widehat g_{(-1)^nM}^2(t,s)\right) _{\mid t=a}&=0\,,\\\\
		-\sum_{j=n-\tau_{n-k}}^{n-1}(-1)^{n-j}	\dfrac{\partial^{\tau_{n-k}+j-n+1}}{\partial t^{\tau_{n-k}+j-n}\partial s}\left( p_{n-j}(t)\widehat g_{(-1)^nM}^2(t,s)\right) _{\mid t=a}& \\\\-	\dfrac{\partial^{\tau_{n-k}+1}}{\partial t^{\tau_{n-k}}\partial s}\widehat g_{(-1)^nM}^2(t,s)_{\mid t=a}
				+\widehat\alpha_0^1(s)\left( \dfrac{\partial^{\tau_{n-k}}}{\partial t^{\tau_{n-k}}}\widehat g_{(-1)^nM}^2(t,s)_{\mid t=a}\right. &\\\\\left. +\sum_{j=n-\tau_{n-k}}^{n-1}(-1)^{n-j}	\dfrac{\partial^{\tau_{n-k}+j-n}}{\partial t^{\tau_{n-k}+j-n}}\left( p_{n-j}(t)\widehat g_{(-1)^nM}^2(t,s)\right) _{\mid t=a}\right) &=0\,,\\\\
				\vdots\hspace{0.5cm}\\\\
				(-1)^\alpha\sum_{j=n-\tau_{n-k}}^{n-1}(-1)^{n-j}	\dfrac{\partial^{\tau_{n-k}+j-n+\alpha}}{\partial t^{\tau_{n-k}+j-n}\partial s^\alpha}\left( p_{n-j}(t)\widehat g_{(-1)^nM}^1(t,s)\right) _{\mid t=a}& \\\\+(-1)^\alpha \dfrac{\partial^{\tau_{n-k}+\alpha}}{\partial t^{\tau_{n-k}}\partial s^\alpha}\widehat g_{(-1)^nM}^2(t,s)_{\mid t=a}
				+\sum_{i=0}^{\alpha-1}\widehat \alpha_i^\alpha(s)\left( \dfrac{\partial^{\tau_{n-k}+i}}{\partial t^{\tau_{n-k}}\partial s^i}\widehat g_{(-1)^nM}^2(t,s)_{\mid t=a}\right. &\\\\\left. +\sum_{j=n-\tau_{n-k}}^{n-1}(-1)^{n-j}	\dfrac{\partial^{\tau_{n-k}+j-n+i}}{\partial t^{\tau_{n-k}+j-n}\partial s^i}\left( p_{n-j}(t)\widehat g_{(-1)^nM}^2(t,s)\right) _{\mid t=a}\right) &=0\,.\end{array}\right.\]}
			
			In this case, since $\tau_{n-k}+\alpha=n-1$, we reach a diagonal element of $\widehat G(t,s)$, hence by Remark \ref{R:2.5}, we obtain the following system for $s=a$:
					{\tiny 	\[\left\lbrace \begin{array}{rl}
						\dfrac{\partial^{\tau_{n-k}}}{\partial t^{\tau_{n-k}}}\widehat g_{(-1)^nM}^1(t,s)_{\mid (t,s)=(a,a)}&\\\\+\sum_{j=n-\tau_{n-k}}^{n-1}(-1)^{n-j}	\dfrac{\partial^{\tau_{n-k}+j-n}}{\partial t^{\tau_{n-k}+j-n}}\left( p_{n-j}(t)\widehat g_{(-1)^nM}^1(t,s)\right) _{\mid (t,s)=(a,a)}&=0\,,\\\\
						- \sum_{j=n-\tau_{n-k}}^{n-1}(-1)^{n-j}	\dfrac{\partial^{\tau_1+j-n+1}}{\partial t^{\tau_{n-k}+j-n}\partial s}\left( p_{n-j}(t)\widehat g_{(-1)^nM}^1(t,s)\right) _{\mid (t,s)=(a,a)}& \\\\-\dfrac{\partial^{\tau_{n-k}+1}}{\partial t^{\tau_{n-k}}\partial s}\widehat g_{(-1)^nM}^1(t,s)_{\mid (t,s)=(a,a)}
						+\widehat\alpha_0^1(a)\left( \dfrac{\partial^{\tau_{n-k}}}{\partial t^{\tau_{n-k}}}\widehat g_{(-1)^nM}^1(t,s)_{\mid (t,s)=(a,a)}\right. &\\\\\left. +\sum_{j=n-\tau_{n-k}}^{n-1}(-1)^{n-j}	\dfrac{\partial^{\tau_{n-k}+j-n}}{\partial t^{\tau_{n-k}+j-n}}\left( p_{n-j}(t)\widehat g_{(-1)^nM}^1(t,s)\right) _{\mid (t,s)=(a,a)}\right) &=0\,,\\\\
						\vdots\hspace{0.5cm}\\\\
					(-1)^\alpha\sum_{j=n-\tau_{n-k}}^{n-1}(-1)^{n-j}	\dfrac{\partial^{\tau_{n-k}+j-n+\alpha}}{\partial t^{\tau_{n-k}+j-n}\partial s^\alpha}\left( p_{n-j}(t)\widehat g_{(-1)^nM}^1(t,s)\right) _{\mid (t,s)=(a,a)}& \\\\	(-1)^\alpha \dfrac{\partial^{\tau_{n-k}+\alpha}}{\partial t^{\tau_{n-k}}\partial s^\alpha}\widehat g_{(-1)^nM}^1(t,s)_{\mid (t,s)=(a,a)}
						+\sum_{i=0}^{\alpha-1}\widehat \alpha_i^\alpha(s)\left( \dfrac{\partial^{\tau_{n-k}+i}}{\partial t^{\tau_{n-k}}\partial s^i}\widehat g_{(-1)^nM}^1(t,s)_{\mid (t,s)=(a,a)}\right. &\\\\\left. +\sum_{j=n-\tau_{n-k}}^{n-1}(-1)^{n-j}	\dfrac{\partial^{\tau_{n-k}+j-n+i}}{\partial t^{\tau_{n-k}+j-n}\partial s^i}\left( p_{n-j}(t)\widehat g_{(-1)^nM}^1(t,s)\right) _{\mid (t,s)=(a,a)}\right) &=1\,.\end{array}\right.\]}
					
					Since $\widehat \alpha_i^j\in C(I)$, from the definition of $\widehat w_M$, we deduce that					
					\[\widehat{w}_M^{(\tau_{n-k})}(a)+\sum_{j=n-\tau_{n-k}}^{n-1}(-1)^{n-j}\,\left( p_{n-j}\widehat w_M\right) ^{(\tau_{n-k}+j-n)}(a)=(-1)^{n-\alpha}=(-1)^{1+\tau_{n-k}}\,.\]

					Now, let us study the behavior of $\widehat w_M$ at $t=b$. Studying the $(n-k+1)^{\mathrm{th}}$ row of \eqref{Ec::CfGh}, we have for all $s\in(a,b)$:						
							{\tiny 	\[\left\lbrace \begin{array}{rl}
								\dfrac{\partial^{\delta_1}}{\partial t^{\delta_1}}\widehat g_{(-1)^nM}^1(t,s)_{\mid t=b}+\sum_{j=n-\delta_1}^{n-1}(-1)^{n-j}	\dfrac{\partial^{\delta_1+j-n}}{\partial t^{\delta_1+j-n}}\left( p_{n-j}(t)\widehat g_{(-1)^nM}^1(t,s)\right) _{\mid t=b}&=0\,,\\\\
								-\left( \dfrac{\partial^{\delta_1+1}}{\partial t^{\delta_1}\partial s}\widehat g_{(-1)^nM}^1(t,s)_{\mid t=b}+\sum_{j=n-\delta_1}^{n-1}(-1)^{n-j}	\dfrac{\partial^{\delta_1+j-n+1}}{\partial t^{\delta_1+j-n}\partial s}\left( p_{n-j}(t)\widehat g_{(-1)^nM}^1(t,s)\right) _{\mid t=b}\right)& \\\\
								+\widehat\alpha_0^1(s)\left( \dfrac{\partial^{\delta_1}}{\partial t^{\delta_1}}\widehat g_{(-1)^nM}^1(t,s)_{\mid t=b}+\sum_{j=n-\delta_1}^{n-1}(-1)^{n-j}	\dfrac{\partial^{\delta_1+j-n}}{\partial t^{\delta_1+j-n}}\left( p_{n-j}(t)\widehat g_{(-1)^nM}^1(t,s)\right) _{\mid t=b}\right) &=0\,,\\\\
								\vdots\hspace{0.5cm}\\\\
								(-1)^\alpha\sum_{j=n-\delta_1}^{n-1}(-1)^{n-j}	\dfrac{\partial^{\delta_1+j-n+\alpha}}{\partial t^{\delta_1+j-n}\partial s^\alpha}\left( p_{n-j}(t)\widehat g_{(-1)^nM}^1(t,s)\right) _{\mid t=b}& \\\\
								(-1)^\alpha \dfrac{\partial^{\delta_1+\alpha}}{\partial t^{\delta_1}\partial s^\alpha}\widehat g_{(-1)^nM}^1(t,s)_{\mid t=b}+\sum_{i=0}^{\alpha-1}\widehat \alpha_i^\alpha(s)\left( \dfrac{\partial^{\delta_1+i}}{\partial t^{\delta_1}\partial s^i}\widehat g_{(-1)^nM}^1(t,s)_{\mid t=b}\right. &\\\\\left. +\sum_{j=n-\delta_1}^{n-1}(-1)^{n-j}	\dfrac{\partial^{\delta_1+j-n+i}}{\partial t^{\delta_1+j-n}\partial s^i}\left( p_{n-j}(t)\widehat g_{(-1)^nM}^1(t,s)\right) _{\mid t=b}\right) &=0\,.\end{array}\right.\]}
							
							Since $b\neq a$, this system is satisfied  for $s=a$. Taking into account that $\widehat \alpha_i^j\in C(I)$, we conclude:
							\[\widehat{w}_M^{(\delta_1)}(b)+\sum_{j=n-\delta_1}^{n-1}(-1)^{n-j}\,\left( p_{n-j}\widehat w_M\right) ^{(\delta_1+j-n)}(b)=0\,.\]
						
						Proceeding analogously with $\delta_2,\dots,\delta_k$, we can affirm that $\widehat{w}_M$ verifies the boundary conditions \eqref{Cf::ad3}--\eqref{Cf::ad4}.

						We have proved that $\widehat{w}_M$ satisfies the boundary conditions \eqref{Cf::ad1}--\eqref{Cf::ad11} and \eqref{Cf::ad3}--\eqref{Cf::ad4}, thus we can apply Proposition \ref{P::4} to conclude that 
							\begin{itemize}
								\item If $n-k$ is even and $k<n-1$, then $\widehat w_M> 0$ on $(a,b)$ for all $M\in [\bar M,\bar M-\lambda_2'']$.
%								\item If $n-k$ is odd and $k<n-1$, then $\widehat{w}_M<0$  on $(a,b)$ for all $M\in [\bar M-\lambda_2'',\bar M]$.
%								\item If $k=n-1$, then $\widehat{w}_M<0$ on $(a,b)$ for all $M\leq \bar M$.
							\end{itemize}
							
					From \eqref{Ec::gs}, we obtain the following conclusion:
							
							\begin{itemize}
								\item If $n-k$ is even and $M\in [\bar M,\bar M-\lambda_2'']$:
								\[\forall s\in (a,b)\,,\quad \exists \rho(s)>0\ \mid \ g_M(t,s)>0\ \forall t\in(a,a+\rho(s))\,.\]															
%								\item If $n-k$ is odd, $k<n-1$ and  $M\in [\bar M-\lambda_2'',\bar M]$:
%								\[\forall s\in (a,b)\,,\quad \exists \rho(s)>0\ \mid \ g_M(t,s)<0\ \forall t\in(a,a+\rho(s))\,.\]
%								\item If $k=n-1$ and $M\leq \bar M$:
%								\[\forall s\in (a,b)\,,\quad \exists \rho(s)>0\ \mid \ g_M(t,s)<0\ \forall t\in(a,a+\rho(s))\,.\]
							\end{itemize}

							\begin{itemize}
									\item[Step 4.] Study of  the related Green's function at $t=b$.
							\end{itemize}
							
								To study the behavior at $t=b$, we consider the following function:
								\[\widehat{y}_M(t)=(-1)^n\dfrac{\partial^\beta }{\partial\,s^\beta}\widehat{g}_{(-1)^nM}^2(t,s)_{\mid s=b}\,.\]
								
								From \eqref{Ec::gg1}, it is satisfied that
								\begin{equation}\label{Ec::gsb}\widehat y_M(s)=\dfrac{\partial^\beta }{\partial\,t^\beta}{g}_{M}^1(t,s)_{\mid t=b}\,,\end{equation}
								moreover, from the boundary conditions \eqref{Ec::cfa}-\eqref{Ec::cfb}, if $\beta>0$ we obtain:
								\[g_M(b,t)=\dfrac{\partial}{\partial s}{g}_{M}(t,s)_{\mid t=b}=\cdots=\dfrac{\partial^{\beta-1}}{\partial s^{\beta-1}}{g}_{M}(t,s)_{\mid t=b}=0\,.\]
								
								As in previous Steps, we can affirm that if there exists $t^*\in(a,b)$ such that either $\widehat y_M(t^*)<0$ and $\beta$ even or $\widehat y_M(t^*)>0$ and $\beta$ odd, then $T_n[M]$ is not inverse positive in $X_{\{\sigma_1,\dots,\sigma_k\}}^{\{\varepsilon_1,\dots,\varepsilon_{n-k}\}}$. 
								
%								Moreover, if there exists $t^*\in(a,b)$ such that either $\widehat y_M(t^*)>0$ and $\beta$ is even or $\widehat y_M(t^*)<0$ with $\beta$ odd, then $T_n[M]$ is not inverse negative in $X_{\{\sigma_1,\dots,\sigma_k\}}^{\{\varepsilon_1,\dots,\varepsilon_{n-k}\}}$.
%							

								From Lemma \ref{L::5}, if $n-k$ is even, then $\widehat{y}_{\bar M}\geq 0$ if $\beta$ is even and  $\widehat{y}_{\bar M}\leq 0$ if $\beta$ is odd. 
								
%								Moreover, if $n-k$ is odd, then $\widehat{y}_{\bar M}\leq 0$ if $\beta$ is even and $\widehat{y}_{\bar M}\geq 0$ if $\beta$ is odd.

								Furthermore, we have 								
								\[\widehat T_n[(-1)^nM]\,\widehat{y}_M(t)=0\,,\quad t\in[a,b)\,.\]
								
								Now, using similar arguments as before, we  obtain that $\widehat{y}_M$ satisfies the boundary conditions \eqref{Cf::ad1}--\eqref{Cf::ad2} and \eqref{Cf::ad3}--\eqref{Cf::ad31}. Moreover, it satisfies:	
									\[\widehat{y}_M^{(\delta_k)}(b)+\sum_{j=n-\delta_k}^{n-1}(-1)^{n-j}\,\left( p_{n-j}\widehat y_M\right) ^{(\delta_k+j-n)}(b)=(-1)^{n-\beta+1}=(-1)^{\delta_k}\,.\]							

							Thus, we can apply Proposition \ref{P::3} to conclude that 
								\begin{itemize}
									\item If $n-k$ is even and $k>1$, then $\widehat{y}_M> 0$ on $(a,b)$ if $\beta$ is even and $\widehat{y}_M< 0$ on $(a,b)$ if $\beta$ is odd  for all $M\in [\bar M,\bar M-\lambda_2']$.
									\item If $k=1$ and $n$ is odd,  then $\widehat{y}_M> 0$ on $(a,b)$ if $\beta$ is even and $\widehat{y}_M< 0$ on $(a,b)$ if $\beta$ is odd  for all $M\geq \bar M$.
%									\item If $n-k$ is odd and $k>1$,  then $\widehat{y}_M< 0$ on $(a,b)$ if $\beta$ is even and $\widehat{y}_M> 0$ on $(a,b)$ if $\beta$ is odd for all $M\in [\bar M-\lambda_2',\bar M]$.
%									\item If $k=1$ and $n$ is even,  then $\widehat{y}_M< 0$ on $(a,b)$ if $\beta$ is even and $\widehat{y}_M> 0$ on $(a,b)$ if $\beta$ is odd  for all $M\leq \bar M$.
								\end{itemize}

								So, from this Step, we obtain the following conclusions:
								\begin{itemize}
									\item If $n-k$ is even, $k>1$ and $M\in [\bar M,\bar M-\lambda_2']$:
									\[\forall s\in (a,b)\,,\quad \exists \rho(s)>0\ \mid \ g_M(t,s)>0\ \forall t\in(b-\rho(s),b)\,.\]

									\item If $k=1$, $n$ is odd and $M\geq \bar M$:
									\[\forall s\in (a,b)\,,\quad \exists \rho(s)>0\ \mid \ g_M(t,s)>0\ \forall t\in(b-\rho(s),b)\,.\]
									
%									\item If $n-k$ is odd, $k>1$ and  $M\in [\bar M-\lambda_2',\bar M]$:
%									\[\forall s\in (a,b)\,,\quad \exists \rho(s)>0\ \mid \ g_M(t,s)<0\ \forall t\in(b-\rho(t),b)\,.\]
%									
%								
%									\item If $k=1$, $n$ is even and $M\leq \bar M$:
%									\[\forall s\in (a,b)\,,\quad \exists \rho(s)>0\ \mid \ g_M(t,s)<0\ \forall t\in(b-\rho(t),b)\,.\]
								\end{itemize}
								
									\begin{itemize}
										\item[Step 5.] Study of  the related Green's function on $(a,b)\times(a,b)$.
									\end{itemize}
									To finish the proof we only need to verify that $(-1)^{n-k}\,g_M(t,s)>0$ for a.e.  $(t,s)\in I\times I$ if $M$ belongs to the given intervals.
									
									In fact, we will prove that $(-1)^{n-k}\,g_M(t,s)>0$ on $(a,b)\times(a,b)$ for those $M$. To this end, for all $s\in(a,b)$, let us denote $u_M^s(t)=g_M(t,s)$.
									
									By the definition of a Green's function it is known that for all $s\in(a,b)$:
									\begin{equation}\label{Ec::us}T_n[\bar M]\,u_M^s(t)=\left( \bar M-M\right) \,u_M^s(t)\,,\quad \forall t\neq s\,,\ t\in I\,.\end{equation}
									
									Moreover, $u_M^s\in C^{n-2}(I)$ and it satisfies the boundary conditions \eqref{Ec::cfa}-\eqref{Ec::cfb}.
									
									From Lemma \ref{L::5}, it is known that $(-1)^{n-k}u^s_{\bar M}\geq 0$ on $I$.
									
									Now, moving continuously with $M$, we will verify that while $u_M^s$ is of constant sign on $I$, it cannot have a double zero on $(a,b)$, which implies that the sign change must be either at $t=a$ or $t=b$ and then the result is proved.
									
									We study separately the cases where $n-k$ is even or odd.
									
									First, let us assume that $n-k$ is even. In this case, from Theorem \ref{T::6}, we only need to study the behavior for $M>\bar M$ and $u_M^s\geq0$. From \eqref{Ec::us}, we have that $T_n[\bar M]\,u_M^s\leq 0$; hence, since $v_1\,\dots\,v_n>0$, $\dfrac{1}{v_n}T_{n-1}\,u_M^s$ is a decreasing function, with two continuous components. Then, it has at most two zeros on $I$ (see Figure \ref{Fig::1}).										\begin{figure}[h]
											\centering
											\includegraphics[width=0.3\textwidth]{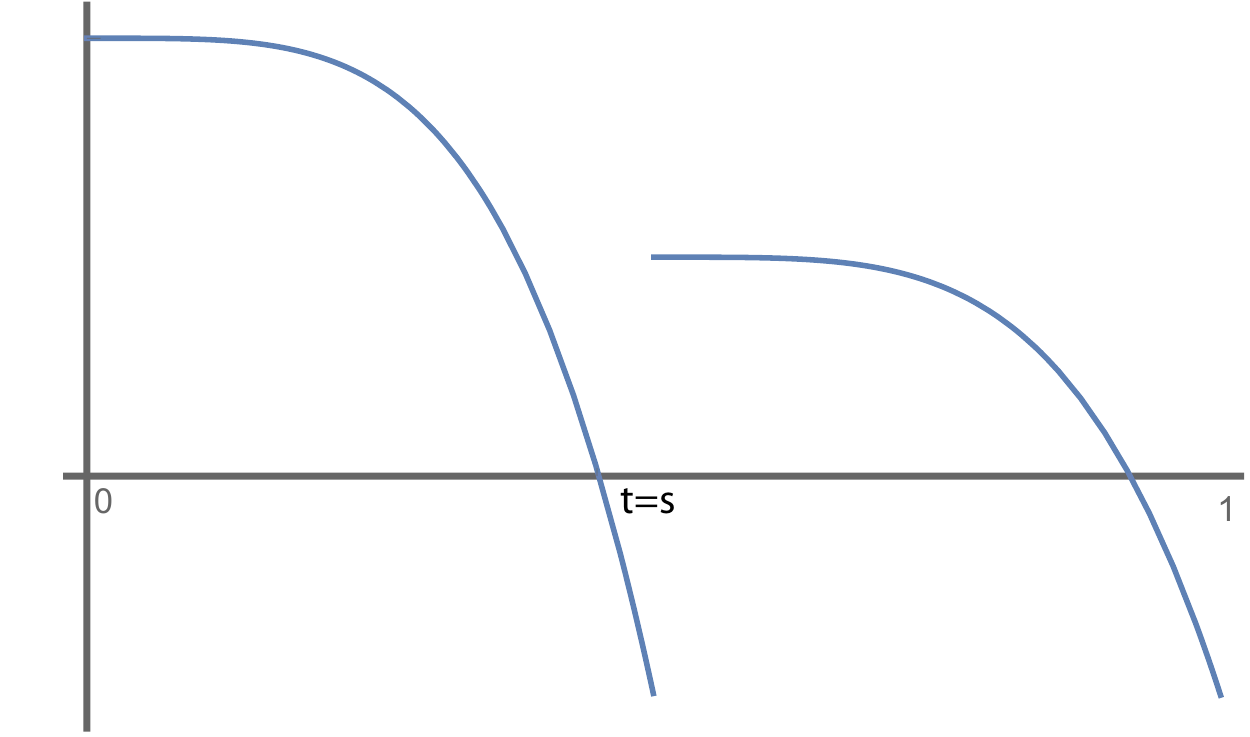}
											\caption{\scriptsize{$\dfrac{1}{v_n(t)}\,T_{n-1} u_M^s(t)$, maximal oscillation with $t\in I=[0,1]$}}\label{Fig::1}
										\end{figure}
									
									Although we cannot know the increasing or decreasing intervals of $T_{n-1}u_M^s$; since $v_n>0$, it has the same sign as  $\dfrac{1}{v_n}T_{n-1}\,u_M^s$. Thus, $T_{n-1}u_M^s$ has at most two zeros on $I$.
									
									So,  $\dfrac{1}{v_{n-1}}T_{n-2}\,u_M^s$	is a continuous function, with at most four zeros on $I$ (see Figure \ref{Fig::2}).
									
										\begin{figure}[h]
											\centering
											\includegraphics[width=0.3\textwidth]{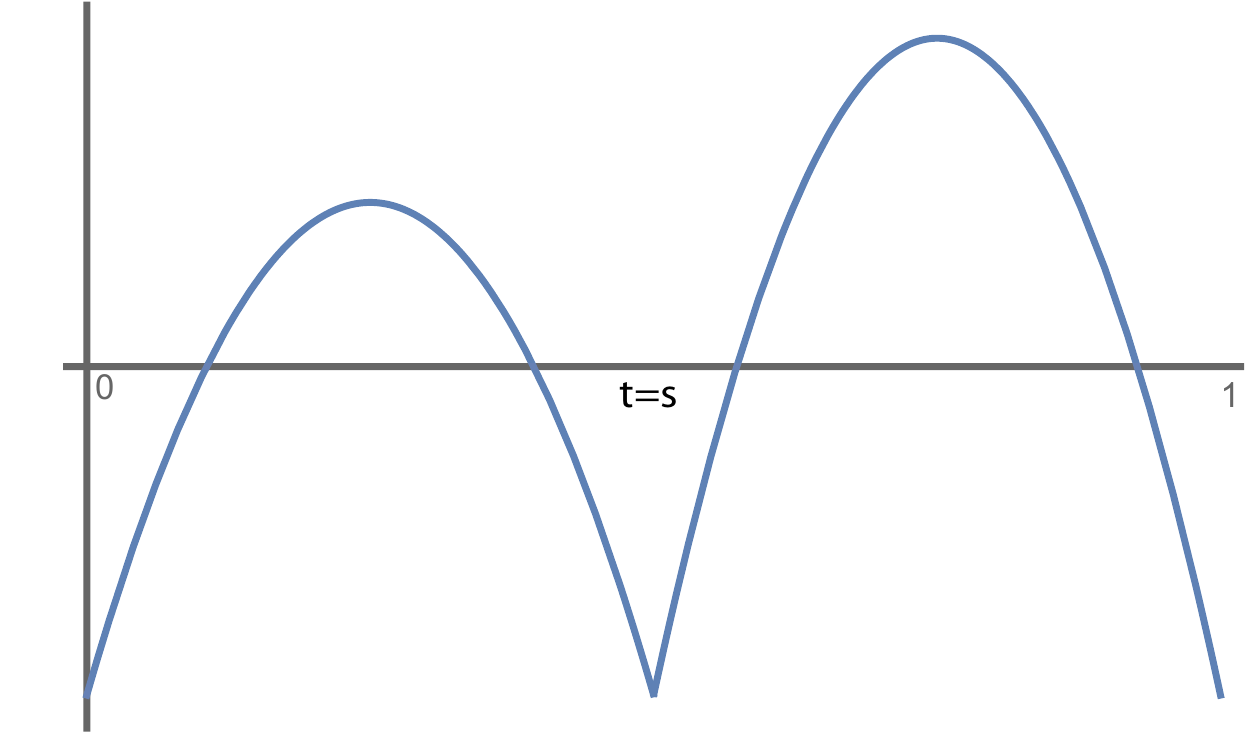}
											\caption{\scriptsize{$\dfrac{1}{v_{n-1}(t)}\,T_{n-2} u_M^s(t)$, maximal oscillation with $t\in I=[0,1]$}}\label{Fig::2}
										\end{figure}
										
										Again, since $v_{n-1}>0$, we conclude that $T_{n-2}u_M^s$ has the same sign as $\dfrac{1}{v_{n-1}}\,T_{n-2} u_M^s$. So, $T_{n-2}u_M^s$ has  at most fourth zeros on $I$. Hence,  $\dfrac{1}{v_{n-2}}\,T_{n-3} u_M^s$ has at most five zeros on $I$, the same as $T_{n-3}u_M^s$.
										
										By recurrence, we conclude that $T_{n-\ell}\,u_M^s$ has, with maximal oscillation, at most $\ell +2$ zeros on $I$.
										
										However, each time that either $T_{n-\ell}\,u_M^s(a)=0$ or $T_{n-\ell}\,u_M^s(b)=0$, a possible oscillation on $(a,b)$ is lost. From the boundary conditions \eqref{Ec::cfa}-\eqref{Ec::cfb}, coupled with Lemmas \ref{L::1} and \ref{L::2}, we can affirm that $n$ possible oscillations are lost. Hence, $u_M^s$ can have at most two zeros on $(a,b)$.

					Let us see that, despite this fact does not inhibit that, with maximal oscillation, $u_M^s$ has a double zero on $(a,b)$, this double zero is not possible. If $T_\ell u(a)=0$ for $\ell\notin \{\sigma_1,\dots,\sigma_k\}$, then $u_s$ can have only a simple zero and this is not possible while it is of constant sign.
					
					 Now, to allow this possible double zero, let us study which should be the sign of $u^{(\alpha)}(a)$. We have already said that $T_{n-\ell}u_M^s(a)$ changes its sign for two consecutive $\ell\in \{0,\dots,n\}$ if it does not vanish. Moreover, at every time that  $T_{n-\ell}u^s_M(a)= 0$ the sign change comes on the next $\tilde \ell$ for which  $T_{n-\tilde \ell}u^s_M(a)\neq 0$. Since, from $\ell=0$ to $n-\alpha$ there are $k-\alpha$ zeros of $T_{n-\ell}u^s_M(a)$, to allow the maximal oscillation it is necessary to have that						
						\begin{equation*}
						\left\lbrace \begin{array}{cc}
						T_{\alpha}\,u_M^s(a)\leq0\,,& \text{if $n-\alpha-(k-\alpha)=n-k$ is even,}\\&\\
						T_{\alpha}\,u_M^s(a)\geq0\,,& \text{if $n-k$ is odd.}\end{array}\right. 	
						\end{equation*}

						As a direct consequence of \eqref{Ec::Tab}, we can affirm that with maximal oscillation it must be verified				
						\begin{equation}\label{Ec::usal}
						\left\lbrace \begin{array}{cc}
						{u^s_M}^{(\alpha)}(a)\leq0\,,& \text{if $n-k$ is even,}\\&\\
						{u^s_M}^{(\alpha)}(a)\geq0\,,& \text{if $n-k$ is odd.}\end{array}\right. 	
						\end{equation}
										
						On another hand, since $u_M^s\geq 0$, it must be satisfied that ${u_M^s}^{(\alpha)}(a)\geq 0$. We can assume that ${u_M^s}^{(\alpha)}(a)>0$. Because, in other case, i.e. if ${u_M^s}^{(\alpha)}(a)=0$, then $T_{\alpha}\,u_M^s(a)=0$ and another possible oscillation is lost, so it only remains the possibility of having a simple zero on $(a,b)$, which is not possible when $u_M^s$ is of constant sign.
						
						If ${u_M^s}^{(\alpha)}(a)>0$, from \eqref{Ec::usal} the maximal oscillation is not allowed. So, we have again only the possibility of a simple zero on $(a,b)$. Hence we conclude:
						\begin{itemize}
							\item If  $n-k$ even and $M>\bar M$, if $u_M^s\geq 0$, then $u_M^s>0$ on $(a,b)$.							
						\end{itemize}
						
%						Now, let us see what happens if $n-k$ is odd. In this case, $u_M^s\leq 0$ on $I$ and, from Theorem \ref{T::7}, we only need to study the case $M<\bar M$. Thus from \eqref{Ec::us}, we have again that $T_n[\bar M]\,u_M^s\leq 0$. Hence, we can use previous arguments to conclude that with maximal oscillation $u_M^s$ satisfies \eqref{Ec::usal}.
%						
%						However, in this case, since $u_M^s\leq 0$, ${u_M^s}^{(\alpha)}(a)\leq 0$. Arguing as before, we can assume that ${u_M^s}^{(\alpha)}(a)<0$. Hence, the maximal oscillation is lost and we can affirm:
%							\begin{itemize}
%								\item If  $n-k$ odd and $M<\bar M$, if $u_M^s\leq 0$, then $u_M^s<0$ on $(a,b)$.							
%							\end{itemize}

					Thus, combining these assertions with the previous Steps the result is proved.						
\end{proof}
\begin{exemplo}
	\label{Ex::12}
	In Example \ref{Ex::6}  the  eigenvalues related to operator $T_4^0[0]$ the different sets, $X_{\{0,2\}}^{\{1,2\}}$, $X_{\{0,2\}}^{\{0,1\}}$ and $X_{\{0,1\}}^{\{1,2\}}$, have been obtained. They are denoted by $\lambda_1$, $\lambda_2'$ and $\lambda_2''$, respectively. We have that $\lambda_1=m_1^4$ and
	\[\lambda_2=\max\{\lambda_2',\lambda_2''\}=\max\{-m_2^4,-4\,\pi^4\}=-4\,\pi^4\,,\]
	where $m_1\approxeq 2.36502$ and $m_2\approxeq 5.550305$ have been introduced in Example \ref{Ex::6} as the least positive solutions of \eqref{Ec::Ex61} and \eqref{Ec::Ex62}, respectively.
	
	So, we can affirm that $T_4^0[M]$ is a strongly inverse positive operator in $X_{\{0,2\}}^{\{1,2\}}$ if, and only if, $M\in (-m_1^4,4\,\pi^4]$
\end{exemplo}

\begin{remark}
	Realize that in Steps 1 and 2, to obtain that $w_M$ and $y_M$ satisfy the boundary conditions \eqref{Ec::cfaa}-\eqref{Ec::cfbb} and \eqref{Ec::cfaaa}-\eqref{Ec::cfbbb}, respectively, we do not need to impose that the operator $T_n[\bar M]$ satisfies property $(T_d)$.
\end{remark}

Taking into account the previous Remark, we obtain the following result:
\begin{theorem} \label{T::cv}If either $\sigma_k=k-1$ or $\varepsilon_{n-k}=n-k-1$, we have the following properties:
	\begin{itemize}
		\item If $n-k$ is even, then there is not any $M\in \mathbb{R}$ such that $T_n[M]$ is inverse negative in $X_{\{\sigma_1,\dots,\sigma_k\}}^{\{\varepsilon_1,\dots,\varepsilon_{n-k}\}}$.
			\item If $n-k$ is odd, then there is not any $M\in \mathbb{R}$ such that $T_n[M]$ is inverse positive in $X_{\{\sigma_1,\dots,\sigma_k\}}^{\{\varepsilon_1,\dots,\varepsilon_{n-k}\}}$.
	\end{itemize}
\end{theorem}
		\begin{proof}
			If $\sigma_k=k-1$, then $\{\sigma_1,\dots,\sigma_k\}=\{0,\dots,k-1\}$.
			
			We consider
			\[w_M(t)=\dfrac{\partial^\eta}{\partial s^\eta}g_M^1(t,s)_{\mid s=a}\,,\]
			defined in Step 1 of the proof of Theorem \ref{T::IPN}.
			
			 By the calculations done in the proof of the mentioned result, we conclude that for all $M\in \mathbb{R}$, $w_M$ satisfies the following boundary conditions:		
			\[w_M(a)=\cdots=w_M^{(k-2)}(a)=0\,,\quad w_M^{(k-1)}(a)=(-1)^{n-\sigma_k-1}=(-1)^{n-k}\,.\]
			
			Hence, if $n-k$ is even, then there exists $\rho>0$, such that $w_M(t)>0$ for all $t\in(a,a+\rho)$. So, $T_n[M]$ cannot be inverse negative for any real $M$.
			
			Now, if $n-k$ is odd, then there exists $\rho>0$, such that $w_M(t)<0$ for all $t\in(a,a+\rho)$. Thus, $T_n[M]$ cannot be inverse positive for any $M\in\mathbb{R}$.
			
			Analogously, if $\varepsilon_{n-k}=n-k-1$, then $\{\varepsilon_1,\dots,\varepsilon_{n-k}\}=\{0,\dots,n-k-1\}$ and $\gamma=n-\varepsilon_{n-k}-1=k$.
			
				We consider now
				\[y_M(t)=\dfrac{\partial^\gamma}{\partial s^\gamma}g_M^2(t,s)_{\mid s=b}\,,\]
				defined in Step 2 of the proof of Theorem \ref{T::IPN}. 
				
				By the previous calculations, we conclude that for all $M\in \mathbb{R}$, $y_M$ satisfies the following boundary conditions:				
				\[y_M(b)=\cdots=y_M^{(n-k-2)}(b)=0\,,\quad y_M^{(n-k-1)}(b)=(-1)^{n-\varepsilon_{n-k}}=(-1)^{k+1}\,.\]

					Hence, if $n-k$ and $k$ are even, then there exists $\rho>0$, such that $y_M(t)>0$ for all $t\in(b-\rho,b)$. So, $T_n[M]$ cannot be inverse negative for any real $M$.
					
					 Moreover, if $n-k$ is even and $k$ odd, then there exists $\rho>0$, such that $y_M(t)<0$ for all $t\in(b-\rho,b)$. So, $T_n[M]$ cannot be inverse negative for any real $M$.
					
						Now, if $n-k$ and $k$ are odd, then there exists $\rho>0$, such that $y_M(t)>0$ for all $t\in(b-\rho,b)$. So, $T_n[M]$ cannot be inverse positive for any real $M$.
						
						Finally, if $n-k$ is odd and $k$ even, then there exists $\rho>0$, such that $y_M(t)<0$ for all $t\in(b-\rho,b)$. As consequence, $T_n[M]$ cannot be inverse positive for any real $M$. 			
		\end{proof}
		
		\section{Particular Cases}\label{SS::Ex}
		This Section is devoted to show the applicability of Theorem \ref{T::IPN} to different particular situations. Along this Section, let us consider $I=[0,1]$.
		
		We have been showing every result for the particular case where $n=4$, $\{\sigma_1,\sigma_2\}=\{0,2\}$, $\{\varepsilon_1,\varepsilon_2\}=\{1,2\}$ and $T_4^0[M]\,u(t)=u^{(4)}(t)+M\,u(t)$, which satisfies the hypotheses of Theorem \ref{T::IPN} for $\bar M=0$.
		
		If we wanted  to study the strongly inverse positive character of $T_4^0[M]$ in $X_{\{0,2\}}^{\{1,2\}}$ without taking into account Theorem \ref{T::IPN}, we would have to study the related Green's function, which is given by the following expression for $M=m^4>0$, obtained by means of the Mathematica program developed in \cite{CaCiMa}.
	{\tiny \[
	\begin{cases}\begin{split}e^{-\sqrt{2} m (s+t-2)} \left(2 e^{\frac{m (s+t-4)}{\sqrt{2}}} \sin \left(\frac{m
			(s-t)}{\sqrt{2}}\right)-e^{\frac{m (3 s+t-2)}{\sqrt{2}}} \sin \left(\frac{m
			(s-t)}{\sqrt{2}}\right)\right. \\+e^{\frac{m (s+3 t-6)}{\sqrt{2}}} \sin \left(\frac{m
			(s-t)}{\sqrt{2}}\right) -2 e^{\frac{m (3 s+3 t-4)}{\sqrt{2}}} \sin \left(\frac{m
			(s-t)}{\sqrt{2}}\right)\\+e^{\frac{m (3 s+t-4)}{\sqrt{2}}} \sin \left(\frac{m
			(s-t+2)}{\sqrt{2}}\right)-e^{\frac{m (s+3 t-4)}{\sqrt{2}}} \sin \left(\frac{m
			(s-t+2)}{\sqrt{2}}\right)\\+e^{\frac{m (s+t-4)}{\sqrt{2}}} \sin \left(\frac{m
			(s+t-2)}{\sqrt{2}}\right)-e^{\frac{m (3 s+3 t-4)}{\sqrt{2}}} \sin \left(\frac{m
			(s+t-2)}{\sqrt{2}}\right)\\+\left(-e^{\frac{m (s+t-2)}{\sqrt{2}}}+e^{\frac{3 m
				(s+t-2)}{\sqrt{2}}}+2 e^{\frac{m (3 s+t-4)}{\sqrt{2}}}-2 e^{\frac{m (s+3
				t-4)}{\sqrt{2}}}\right) \sin \left(\frac{m (s+t)}{\sqrt{2}}\right)\\+\left(e^{\frac{m (3
				s+t-2)}{\sqrt{2}}}+e^{\frac{m (s+3 t-6)}{\sqrt{2}}}\right) \cos \left(\frac{m
			(s-t)}{\sqrt{2}}\right)-e^{\frac{m (3 s+3
				t-4)}{\sqrt{2}}} \cos \left(\frac{m (s+t-2)}{\sqrt{2}}\right)\\+\left(e^{\frac{m (3 s+t-4)}{\sqrt{2}}}+e^{\frac{m (s+3
				t-4)}{\sqrt{2}}}\right) \cos \left(\frac{m (s-t+2)}{\sqrt{2}}\right) -e^{\frac{m
				(s+t-4)}{\sqrt{2}}} \cos \left(\frac{m (s+t-2)}{\sqrt{2}}\right)\\\left. -e^{\frac{m
				(s+t-2)}{\sqrt{2}}} \left(e^{\sqrt{2} m (s+t-2)}+1\right) \cos \left(\frac{m
			(s+t)}{\sqrt{2}}\right)\right)/\left( 4 \sqrt{2} m^3 \left(\sin \left(\sqrt{2} m\right)+\sinh
		\left(\sqrt{2} m\right)\right)\right) \end{split}& 0\leq s\leq t\leq 1\,, \\\\\\\\
	\begin{split}e^{-\frac{m (3 s+t-6)}{\sqrt{2}}} \left(e^{2 \sqrt{2} m (s-1)} \sin \left(\frac{m
			(s-t-2)}{\sqrt{2}}\right)-e^{\sqrt{2} m (s+t-2)} \sin \left(\frac{m
			(s-t-2)}{\sqrt{2}}\right)\right. \\+2 e^{\sqrt{2} m (s-2)} \sin \left(\frac{m
			(s-t)}{\sqrt{2}}\right)-e^{\sqrt{2} m (2 s-3)} \sin \left(\frac{m
			(s-t)}{\sqrt{2}}\right)+e^{\sqrt{2} m (s+t-1)} \sin \left(\frac{m
			(s-t)}{\sqrt{2}}\right)\\-2 e^{\sqrt{2} m (2 s+t-2)} \sin \left(\frac{m
			(s-t)}{\sqrt{2}}\right)+e^{\sqrt{2} m (s-2)} \sin \left(\frac{m
			(s+t-2)}{\sqrt{2}}\right)\\-e^{\sqrt{2} m (2 s+t-2)} \sin \left(\frac{m
			(s+t-2)}{\sqrt{2}}\right)+\left(e^{\sqrt{2} m (s+t-2)}+e^{2 \sqrt{2} m (s-1)}\right) \cos
		\left(\frac{m (s-t-2)}{\sqrt{2}}\right)\\ +\left(-2 e^{\sqrt{2} m (s+t-2)}+e^{\sqrt{2} m (2
			s+t-3)}-e^{\sqrt{2} m (s-1)}+2 e^{2 \sqrt{2} m (s-1)}\right) \sin \left(\frac{m
			(s+t)}{\sqrt{2}}\right)\\+\left(e^{\sqrt{2} m (s+t-1)}+e^{\sqrt{2} m (2
			s-3)}\right) \cos \left(\frac{m (s-t)}{\sqrt{2}}\right)\\-e^{\sqrt{2} m (2 s+t-2)} \cos \left(\frac{m
			(s+t-2)}{\sqrt{2}}\right) -\left(e^{\sqrt{2} m (2 s+t-3)}+e^{\sqrt{2} m (s-1)}\right) \cos
		\left(\frac{m (s+t)}{\sqrt{2}}\right)\\\left.-e^{\sqrt{2} m (s-2)} \cos
		\left(\frac{m (s+t-2)}{\sqrt{2}}\right)\right)/\left( 2 \sqrt{2} m^3 \left(e^{2 \sqrt{2} m}+2
		e^{\sqrt{2} m} \sin \left(\sqrt{2} m\right)-1\right)\right) \end{split} & 0<t<s\leq 1\,.\end{cases}\]}
$\!\!$which gives an example where the direct applicability of Theorem \ref{T::IPN} is posted. 

To characterize the Green's function constant sign in much more complicated problems, its expression may be inapproachable. Moreover, in some cases, for instance in problems with non constant coefficients, we cannot even obtain its expression. So, Theorem \ref{T::IPN} is very useful because it allows us to see which is the sign of the related Green's function without knowing its expression. We point out that to calculate the corresponding eigenvalues is very simple in the constant coefficient case and can be numerically approached in the non constant case.

In the sequel, we  see different examples, where the applicability of Theorem \ref{T::IPN}  is shown.
			\begin{itemize}
				\item $n^{\mathrm{th}}-$order operators with $(k,n-k)$ boundary conditions.
			\end{itemize}
			
			In \cite{CabSaa}, there are studied the operators $T_n[M]$ in the spaces $X_{\{0,\dots,k-1\}}^{\{0,\dots,n-k-1\}}$ under the hypothesis that there exists $\bar M\in\mathbb{R}$ such that $T_n[\bar M]\,u(t)=0$ is a disconjugate equation on $I$.
			
			In fact, it is proved there that in such a case $T_n[\bar M]$ verifies the property $(T_d)$. The result there obtained is a particular case of Theorem \ref{T::IPN} when $\sigma_k=k-1$ and $\varepsilon_{n-k}=n-k-1$. Moreover, since in such a case both $\sigma_k=k-1$ and $\varepsilon_{n-k}=n-k-1$, it is proved the correspondent to Theorem \ref{T::cv}.
			
			In \cite{CabSaa}, there are shown several examples, some of them with operators of non constant coefficients.
			
				\begin{itemize}
					\item Operator $T_4(p_1,p_2)[M]\,u(t)=u^{(4)}(t)+p_1(t)\,u^{(3)}(t)+p_2(t)\,u^{(2)}(t)+M\,u(t)$ in $X_{\{0,2\}}^{\{0,2\}}$.
				\end{itemize}
				
				In \cite{CabSaa2}, there is studied a particular kind of fourth order operators $T_4(p_1,p_2)[M]\,u(t)=u^{(4)}(t)+p_1(t)\,u^{(3)}(t)+p_2(t)\,u^{(2)}(t)+M\,u(t)$ in $X_{\{0,2\}}^{\{0,2\}}$, under the hypothesis that the second order equation $u''(t)+p_1(t)\,u'(t)+p_2(t)\,u(t)=0$ is disconjugate on $I$.
				
				Indeed, this allows us to prove that the operator $T_4(p_1,p_2)[0]$ satisfies the property $(T_d)$ in $X_{\{0,2\}}^{\{0,2\}}$. Hence, the result there obtained for the strongly inverse positive character is  a particular case of Theorem \ref{T::IPN}.
				
				In \cite{CabSaa2}, there are also shown several examples of this kind of operators coupled with the well-known simply supported beam boundary conditions. Again, some of the examples shown have non-constant coefficients.
		\begin{itemize}
			\item Operator $T_n^0[M]\,u(t)=u^{(n)}(t)+M\,u(t)$.
		\end{itemize}
	
	Before giving some results for this kind of operator, we take into account the following remarks:
		\begin{remark}\label{R::10}
			Realize that if we choose $\{\sigma_1,\dots,\sigma_k\}-\{\varepsilon_1,\dots,\varepsilon_{n-k}\}$ verifying $(N_a)$, then the hypotheses of Theorem \ref{T::IPN} are fulfilled for $\bar M=0$ for the operator $T_n^0[\bar M]$ by choosing $v_1(t)=\cdots=v_n(t)=1$ for all $t\in I$.
		\end{remark}
		
		\begin{remark}\label{R::11}
			Realize that for this kind of operators the behavior in  $X^{\{\sigma_1,\dots,\sigma_k\}}_{\{\varepsilon_1,\dots,\varepsilon_{n-k}\}}$ can be known by studying the behavior in $X_{\{\sigma_1,\dots,\sigma_k\}}^{\{\varepsilon_1,\dots,\varepsilon_{n-k}\}}$. This is due to the fact that the eigenvalues are the same if $n$ is even or or the opposed if $n$ is odd. 
			
			Indeed, if $u$ is a nontrivial solution of $u^{(n)}(t)+M\,u(t)=0$ on $X^{\{\sigma_1,\dots,\sigma_k\}}_{\{\varepsilon_1,\dots,\varepsilon_{n-k}\}}$, then $y(t)=u(1-t)$ is a solution of $u^{(n)}(t)+(-1)^n\,M\,u(t)=0$ on $X_{\{\sigma_1,\dots,\sigma_k\}}^{\{\varepsilon_1,\dots,\varepsilon_{n-k}\}}$.
			
			So, we do not need to study all the cases to obtain conclusions about the strongly inverse positive (negative) character.
		\end{remark}
		In the sequel, we show different examples of this kind of operators.

			\begin{itemize}
				\item[-] Second order
			\end{itemize}
			In second order, the only possibility is to consider $k=1$. There are three options for the choice of $\{\sigma_1\}-\{\varepsilon_1\}$. First of them is $\sigma_1=\varepsilon_1=0$ which correspond to the Dirichlet case, that is the boundary conditions $(1,1)$. This case has been considered in \cite{CabSaa}, where it is obtained that the operator $T_2^0[M]$ is strongly inverse negative in $X_{\{0\}}^{\{0\}}$ if, and only if $M\in(-\infty,\pi^2)$.
			
			The other two choices correspond to the mixed boundary conditions and are equivalent, $\sigma_1=0$ and $\varepsilon_1=1$ or $\sigma_1=1$ and $\varepsilon_1=0$.
			
			The biggest negative eigenvalue of $T_2^0[0]$ in $X_{\{0\}}^{\{1\}}$ is $\lambda_1=-\dfrac{\pi^2}{4}$. So, using Theorem \ref{T::IPN}, we can affirm that $T_2^0[M]$ is a strongly inverse negative operator in $X_{\{0\}}^{\{1\}}$ or in $X_{\{1\}}^{\{0\}}$ if, and only if $M\in \left( -\infty,\dfrac{\pi^2}{4}\right)$.
			
			Moreover, from Theorem \ref{T::cv}, we can conclude that there is not any $M\in\mathbb{R}$ such that $T_2^0[M]$ is strongly inverse positive either in $X_{\{0\}}^{\{1\}}$ or $X_{\{1\}}^{\{0\}}$.
			
			\begin{itemize}
				\item[-] Third  order
			\end{itemize}
			In this case the number of possible cases increases until twelve, which can be reduced to six. The cases $\{\sigma_1,\sigma_2\}=\{0,1\}$, $\{\varepsilon_1\}=\{0\}$ and $\{\sigma_1\}=\{0\}$, $\{\varepsilon_1,\varepsilon_2\}=\{0,1\}$ have been considered in \cite{CabSaa}. Let us see some of the rest.
			
			First, let us consider $\{\sigma_1,\sigma_2\}=\{1,2\}$ and $\{\varepsilon_1\}=\{0\}$.
			
			The biggest negative eigenvalue of $T_3^0[0]$ in $X_{\{1,2\}}^{\{0\}}$ is given by $\lambda_1=-m_4^3$, where $m_4\approxeq 1.85$ is the least positive solution of
			\begin{equation}\label{Ec::m4}
			e^{-m}+2\,e^{m/2}\,\cos\left( \dfrac{\sqrt{3}}{2}m\right) =0\,.
			\end{equation}
			
			In order to apply Theorem \ref{T::IPN}, we need to obtain the least positive eigenvalue of $T_3^0[0]$ in $X_{\{1\}}^{\{0,1\}}$, which is given by $\lambda_2=m_5^3$, where $m_5\approxeq 3.017$ is the least positive solution of
			\begin{equation}
			\label{Ec::m5}
			e^{-m}-e^{m/2}\left(\cos\left( \dfrac{\sqrt{3}}{2}m\right) +\sqrt{3}\sin\left( \dfrac{\sqrt{3}}{2}m\right) \right) =0\,.
			\end{equation}
			
			Since, $k=2=n-1$, we can apply Theorem \ref{T::IPN} to affirm that $T_3^0[M]$ is strongly inverse negative in $X_{\{1,2\}}^{\{0\}}$ if, and only if, $M\in [-m_5^3,m_4^3)\approxeq[-3.017^3,1.85^3)$. 
			
			Moreover, from Theorem \ref{T::cv}, we can conclude that there is not any $M\in\mathbb{R}$ such that $T_3^0[M]$ is strongly inverse positive  in $X_{\{1,2\}}^{\{0\}}$.
			
			Now, from Remark \ref{R::11}, we can affirm that $T_3^0[M]$ is strongly inverse positive in  $X^{\{1,2\}}_{\{0\}}$ if, and only if, $M\in (-m_4^3,m_5^3]\approxeq(-1.85^3,3.017^3]$. 
			
			Moreover,  we can conclude that there is not any $M\in\mathbb{R}$ such that $T_3^0[M]$ is strongly inverse negative  in $X^{\{1,2\}}_{\{0\}}$.
			
			Now, let us consider $\{\sigma_1,\sigma_2\}=\{0,2\}$ and $\{\varepsilon_1\}=\{1\}$.
			
			The biggest negative eigenvalue of $T_3^0[0]$ in $X_{\{0,2\}}^{\{1\}}$ is $\lambda_1=-m_4^3$, where $m_4$ has been defined as the least positive solution of \eqref{Ec::m4}.
			
			Moreover, the least positive eigenvalue in $X_{\{0\}}^{\{0,1\}}$ is given by $\lambda_2=m_6^3$, where $m_6\approxeq 4.223$ is the least positive solution of 
			\begin{equation}
			\label{Ec::m6}
			e^{-m}-e^{m/2}\left(\cos\left( \dfrac{\sqrt{3}}{2}m\right) -\sqrt{3}\sin\left( \dfrac{\sqrt{3}}{2}m\right) \right) =0\,.
			\end{equation}
			
			Thus, from Theorem \ref{T::IPN}, we conclude that $T_3^0[M]$ is strongly inverse negative in $X_{\{0,2\}}^{\{1\}}$ if, and only if, $M\in[-m_6^3,m_4^3)\approxeq[-4.223^3,1.85^3)$.
			
			We note that in this case $\sigma_2=2>1$ and $\varepsilon_1=1>0$, thus we cannot apply Theorem \ref{T::cv} to obtain conclusions about the strongly inverse positive character of $T_3^0[M]$ in $X_{\{0,2\}}^{\{1\}}$.
			
			However, from Remark \ref{R::11}, we can affirm that $T_3^0[M]$ is strongly inverse positive in $X_{\{1\}}^{\{0,2\}}$ if, and only if, $M\in(-m_4^3,m_6^3]\approxeq(-1.85^3,4.223^3]$.
			
			\begin{itemize}
				\item Fourth order[-]
			\end{itemize}
			There are forty possibilities, which can be decreased, by using Remark \ref{R::11}, to twenty one. There are three possibilities which have been studied in \cite{CabSaa}, they are represented on the sets $X_{\{0,1,2\}}^{\{0\}}$, $X_{\{0\}}^{\{0,1,2\}}$ and $X_{\{0,1\}}^{\{0,1\}}$. The characterization in $X_{\{0,2\}}^{\{0,2\}}$ has been obtained in \cite{CabSaa2}. Moreover, along the paper we have studied the case $X_{\{0,2\}}^{\{1,2\}}$. From Remark \ref{R::11}, the obtained characterization remains valid for the set $X_{\{1,2\}}^{\{0,2\}}$.
			
			In the sequel, let us see a pair of different cases. For instance, $X_{\{1,2,3\}}^{\{0\}}$ (which also gives the characterization in $X_{\{0\}}^{\{1,2,3\}}$) and $X_{\{1,3\}}^{\{0,2\}}$ ($X_{\{0,2\}}^{\{1,3\}}$).
			
		Let us work on the space $X_{\{1,2,3\}}^{\{0\}}$.
				
			First, we obtain the necessary eigenvalues in order to apply Theorem \ref{T::IPN} to this case:
				
				The biggest negative eigenvalue of $T_4^0[0]$ in $X_{\{1,2,3\}}^{\{0\}}$ is $\lambda_1=-\dfrac{\pi^4}{4}$.
				
				The least positive eigenvalue of $T_4^0[0]$ in $X_{\{1,2\}}^{\{0,1\}}$ is $\lambda_2=\pi^4$.
				
				Then, $T_4^0[M]$ is strongly inverse negative in $X_{\{1,2,3\}}^{\{0\}}$ if, and only if, $M\in\left[-\pi^4,\dfrac{\pi^4}{4}\right)$.
				
				Since $\varepsilon_1=0 $, we can apply Theorem \ref{T::cv} to conclude that there is not any $M\in\mathbb{R}$ such that $T_4^0[M]$ is strongly inverse positive in $X_{\{1,2,3\}}^{\{0\}}$.
				
				 Concerning to the space $X_{\{0,2\}}^{\{1,3\}}$, we have the following eigenvalues:
				
				The least positive eigenvalue of $T_4^0[0]$ in $X_{\{0,2\}}^{\{1,3\}}$ is $\lambda_1=\dfrac{\pi^4}{16}$.
				
				The biggest negative eigenvalue of $T_4^0[0]$ in $X_{\{0,1,2\}}^{\{1\}}$ is $\lambda_2'=-4\,\pi^4$.
				
				The biggest negative eigenvalue of $T_4^0[0]$  in $X_{\{0\}}^{\{0,1,3\}}$ is $\lambda_2''=-4\,\pi^4$.
				
				So, $\lambda_2=\max\{-4\,\pi^4,-4\,\pi^4\}=-4\,\pi^4$.
				
				Hence, we conclude, from Theorem \ref{T::IPN}, that $T_4^0[M]$ is strongly inverse positive in $X_{\{0,2\}}^{\{1,3\}}$ if, and only if, $M\in \left( -\dfrac{\pi^4}{16},4\,\pi^4\right] $.
				
				Since $\sigma_2=2>1$ and $\varepsilon_2=3>1$, we cannot apply Theorem \ref{T::cv} to affirm that it cannot be strongly inverse negative for any $M\in\mathbb{R}$.

		\begin{itemize}
					\item[-] Higher order
				\end{itemize}
			
	  If we increase the order of the problem, due to the fact that the related Green's function gets more complex, the usefulness of Theorem \ref{T::IPN} also increases. Even if we cannot obtain the eigenvalues analytically, we can obtain them numerically, using different methods.
	  
	  Now, let us show an example of sixth order, where we can obtain the eigenvalues analytically.
	  
	  The biggest negative eigenvalue of $T_6^0[0]$ in $X_{\{0,2,4\}}^{\{0,2,4\}}$ is $\lambda_1=-\pi^6$.
	  
	  The least positive eigenvalue of $T_6^0[0]$ in $X_{\{0,1,2,4\}}^{\{0,2\}}$ and $X_{\{0,2\}}^{\{0,1,2,4\}}$ is $\lambda_2=\lambda_2'=\lambda_2''=m_7^6$, where $m_7\approxeq5.47916$ is the least positive solution of
	  \[ \cos \left(\sqrt{3} m\right)-\cosh (m)+8 \cos \left(\frac{\sqrt{3} m}{2}\right) \sinh
	  ^2\left(\frac{m}{2}\right) \cosh \left(\frac{m}{2}\right)=0\,.\]
	  
	  Hence, from Theorem \ref{T::IPN}, we conclude that $T_6^0[M]$ is a strongly inverse negative operator in $X_{\{0,2,4\}}^{\{0,2,4\}}$ if, and only if, $M\in [-m_7^7,\pi^6)\approxeq[-5.47916^6,\pi^6)$.
	  
	  Since $\sigma_3=4>2$ and $\varepsilon_3=4>2$, we cannot apply Theorem \ref{T::cv} to obtain any conclusion about the strongly inverse positive character.		
		
		\begin{itemize}
			\item Operator $T_4^N[M]u(t)=u^{(4)}(t)+N\,u'(t)+M\,u(t)$ in $X_{\{0,2\}}^{\{1,2\}}$.
		\end{itemize}
		
	Let us denote $N=n^3$. Now, let us consider the fourth order operator $T_4^{n^3}[M]\,u(t)=u^{(4)}(t)+n^3\,u'(t)+M\,u(t)$ in $X_{\{0,2\}}^{\{1,2\}}$. Realize that for $n=0$ this operator coincides with the example that we have been considering in the different examples along the paper.
		
		Let us see that for $n\in\left(-\dfrac{4\,\pi}{3\,\sqrt{3}},\dfrac{2\,\pi}{3\,\sqrt{3}}\right) $, $T_4^{n^3}[0]$ verifies the property $(T_d)$ in $X_{\{0,2\}}^{\{1,2\}}$.
		
		In order to do that, let us consider the following fundamental system of solutions:
		\[\begin{split}
		y_1^n(t)&=1\,,\\
		y_2^n(t)&=\sqrt{3}\,e^{\frac{n\,t}{2}}\,\cos\left( \dfrac{\sqrt{3}}{2}n\,t\right) +e^{\frac{n\,t}{2}}\,\sin \left( \dfrac{\sqrt{3}}{2}n\,t\right) \,,\\
		y_3^n(t)& =e^{\frac{n\,t}{2}}\,\sin \left( \dfrac{\sqrt{3}}{2}n\,t\right)\,,\\
		y_4^n(t)&=e^{-n\,t}\,,\end{split}\]
	and the correspondent Wronskians:
	\[\begin{split}
	W_1^n(t)&=1\,,\\
	W_2^n(t)&=\dfrac{1}{2}\,e^{\frac{n\,t}{2}}\,n^2\,\left( \cos \left( \dfrac{\sqrt{3}}{2}n\,t\right)-\sin \left( \dfrac{\sqrt{3}}{2}n\,t\right)\right) \,,\\
	W_3^n(t)&=\dfrac{3}{2}\,n^3\,e^{n\,t}\,,\\
	W_4^n(t)&=-\dfrac{9\,n^6}{2}\,.
	\end{split}\]
	
	If, $n\neq 0$, then $W_1^n$, $W_3^n$ and $W_4^n$ are non null in $[0,1]$.
	
	Moreover, it can be seen that if $n\in\left( -\dfrac{4\,\pi}{3\,\sqrt{3}},\dfrac{2\,\pi}{3\,\sqrt{3}}\right)$, then $W_2^n(t)\neq 0$ for all $t\in[0,1]$. So, we can obtain the representation given in \eqref{Ec::Td1}-\eqref{Ec::Td2}.
	
	We construct $v_1,\dots,v_4$ following the proof of Theorem \ref{T::3} (\cite[Theorem 2, Chapter 3]{Cop}):
	\[\begin{split}
	v_1^n(t)&=1\,,\\
	v_2^n(t)&=W_2^n(t)=\dfrac{1}{2}\,e^{\frac{n\,t}{2}}\,n^2\,\left( \cos \left( \dfrac{\sqrt{3}}{2}n\,t\right)-\sin \left( \dfrac{\sqrt{3}}{2}n\,t\right)\right) \,,\\
	v_3^n(t)&=\dfrac{W_3^n(t)}{{W_2^n}^2(t)}\,,\\
	v_4^n(t)&=\dfrac{W_2^n(t)\,W_4^n(t)}{{W_3^n}^2(t)}\,.	
	\end{split}\]
	
	In Example \ref{Ex::3}, we have proved that a fourth order operator satisfies property $(T_d)$ in $X_{\{0,2\}}^{\{1,2\}}$ if, and only if, there exists the decomposition \eqref{Ec::Td1}-\eqref{Ec::Td2} and \eqref{Ec::Ex31}-\eqref{Ec::Ex32} are fulfilled.
	
	 Let us check it. Obviously, \eqref{Ec::Ex32} is verified. Now, since ${v_1^n}'(0)=0$ and $v_2^n(0)\neq0$, we have to verify that ${v_2^n}'(0)=0$. But, from the fact that
	\[{v_2^n}'(t)=-2\,e^{\frac{n\,t}{2}}\,\sin \left( \dfrac{\sqrt{3}}{2}n\,t\right) \,,\]
	we deduce that it is trivially satisfied that ${v_2^n}'(0)=0$. So, as consequence, \eqref{Ec::Ex31} is fulfilled and we conclude that if $n\in\left( -\dfrac{4\,\pi}{3\,\sqrt{3}},\dfrac{2\,\pi}{3\,\sqrt{3}}\right) $, then $T_4^{n^3}[0]$ verifies the property $(T_d)$ in $X_{\{0,2\}}^{\{1,2\}}$.
	
	\begin{remark}
		Realize that the interval $\left( -\dfrac{4\,\pi}{3\,\sqrt{3}},\dfrac{2\,\pi}{3\,\sqrt{3}}\right)$ is not necessarily optimal.
		
		If we study the disconjugacy set of $T_4^{n^3}[0]\,u(t)=0$ on $[0,1]$, we obtain that such equation is disconjugate if, and only if, $n\in(-n_1,n_1)$, where $n_1\approxeq5.55$ is the least positive solution of 
		\[-3+e^{-n}+2^{n/2}\,\cos\left( \dfrac{\sqrt{3}\,n}{2}\right) =0\,.\]
	
	Then, it is possible that we may find different values of $n\in(-n_1^3,n_1^3)$ such that $T_4^{n^3}[0]$ satisfies property $(T_d)$ in $X_{\{0,2\}}^{\{1,2\}}$ with a suitable choice of the fundamental system of solutions. 
	
	For instance, repeating the previous arguments for the following fundamental system of solutions:
		\[\begin{split}
	y_1^n(t)&=1\,,\\
	y_2^n(t)&=\dfrac{2}{\sqrt{3}}\,e^{\frac{n\,t}{2}}\,\sin\left( \dfrac{\sqrt{3}}{2}n\,t\right) -e^{-n\,t}\,,\\
	y_3^n(t)&=e^{-n\,t}\,,\\
	y_4^n(t)& =e^{\frac{n\,t}{2}}\,\sin \left( \dfrac{\sqrt{3}}{2}n\,t\right)\,,\end{split}\]
	we obtain a decomposition to ensure that $T_4^{n^3}[0]$ verify property $(T_d)$ for $n\in\left( -\dfrac{\pi}{\sqrt{3}},\dfrac{\pi}{\sqrt{3}}\right)$. 
	
	Thus, we can say that for $n\in\left( -\dfrac{4\,\pi}{3\,\sqrt{3}},\dfrac{2\,\pi}{3\,\sqrt{3}}\right)\cup \left( -\dfrac{\pi}{\sqrt{3}},\dfrac{\pi}{\sqrt{3}}\right)=\left( -\dfrac{4\,\pi}{3\,\sqrt{3}},\dfrac{\pi}{\sqrt{3}}\right)\subset(-n_1,n_1)$, then $T_4^{n^3}[0]$ verifies property $(T_d)$. 
	
	However, we cannot even affirm that such an interval is the optimal one.
		\end{remark}
	
	Let us choose, for instance, $n=-\dfrac{\pi}{\sqrt{3}}\in\left( -\dfrac{4\,\pi}{3\,\sqrt{3}},\dfrac{\pi}{\sqrt{3}}\right) \subset(-n_1,n_1)$ and we obtain the different eigenvalues numerically, by using Mathematica.
	
	The least positive eigenvalue of $T_4^{-\frac{\pi^3}{3\sqrt{3}}}[0]$ in $X_{\{0,2\}}^{\{1,2\}}$ is $\lambda_1\approxeq2.21152^4$.
	
	The biggest negative eigenvalue of $T_4^{-\frac{\pi^3}{3\sqrt{3}}}[0]$ in $X_{\{0,1,2\}}^{\{1\}}$ is $\lambda_2'\approxeq-4.53073^4$.
	
		The biggest negative eigenvalue of $T_4^{-\frac{\pi^3}{3\sqrt{3}}}[0]$ in $X_{\{0\}}^{\{0,1,2\}}$ is $\lambda_2''\approxeq-5.5014^4$.
		
		So, $\lambda_2=\max\{\lambda_2',\lambda_2''\}=\lambda_2'\approxeq -4.53073^4$.
		
		From Theorem \ref{T::IPN}, we conclude that $T_4^{-\frac{\pi^3}{3\sqrt{3}}}[M]$ is strongly inverse positive in $X_{\{0,2\}}^{\{1,2\}}$ if, and only if, $M\in(-\lambda_1,-\lambda_2]\approxeq(-2.21152^4,4.53073^4]$.
		
		Since $\sigma_2=\varepsilon_2=2>1$, we cannot obtain any conclusion about the strongly inverse positive character from Theorem \ref{T::cv}.
		
		\chapter[Necessary condition for the inverse negative (positive) character]{Necessary condition for the strongly inverse negative (positive) character of $T_n[M]$ in 	 $X_{\{\sigma_1,\dots,\sigma_{k}\}}^{\{\varepsilon_1,\dots,\varepsilon_{n-k}\}}$.}
		
		In the previous chapter we have obtained a characterization of  the parameter's set where the operator $T_n[M]$ is either strongly inverse positive or negative in $X_{\{\sigma_1,\dots,\sigma_{k}\}}^{\{\varepsilon_1,\dots,\varepsilon_{n-k}\}}$ if $n-k$ is even or odd, respectively.
		
		 Moreover, in some cases, we can ensure that if $n-k$ is even, then there is not any $M\in\mathbb{R}$ such that $T_n[M]$ is strongly inverse negative in $X_{\{\sigma_1,\dots,\sigma_{k}\}}^{\{\varepsilon_1,\dots,\varepsilon_{n-k}\}}$ and if $n-k$ is odd, then there is not any $M\in\mathbb{R}$ such that $T_n[M]$ is strongly inverse positive in $X_{\{\sigma_1,\dots,\sigma_{k}\}}^{\{\varepsilon_1,\dots,\varepsilon_{n-k}\}}$. However, on the cases which do not fulfill the hypotheses of Theorem \ref{T::cv}, we have not said anything about the strongly inverse negative character if $n-k$ is even or about the strongly inverse positive character if $n-k$ is odd. 
		 
		 From Theorems \ref{T::6} and \ref{T::8}, if $n-k$ is even and there exists $\bar M\in\mathbb{R}$ such that $T_n[\bar M]$ is strongly inverse negative in $X_{\{\sigma_1,\dots,\sigma_{k}\}}^{\{\varepsilon_1,\dots,\varepsilon_{n-k}\}}$, then the parameter's set, for which such a property is fulfilled, is given by an interval whose supremum is given by $\bar M-\lambda_1$.
		
		Moreover, from Theorems \ref{T::7} and \ref{T::9}, if $n-k$ is odd and there exists $\bar M\in\mathbb{R}$ such that $T_n[\bar M]$ is strongly inverse positive in $X_{\{\sigma_1,\dots,\sigma_{k}\}}^{\{\varepsilon_1,\dots,\varepsilon_{n-k}\}}$, then the parameter's set, for which such a property is fulfilled, is given by an interval whose infimum is given by $\bar M-\lambda_1$.
		
		This chapter, is devoted to obtain a bound of the other extreme of the interval. Furthermore, we will see that, in such an interval, the Green's function satisfies a suitable property which allows to prove that the obtained interval is optimal if we prove that the Green's function cannot have any zero on $(a,b)\times (a,b)$.
		
	\begin{theorem}\label{T::IPN1}
		Let $\bar M\in \mathbb{R}$ be such that $T_n[\bar M]$ satisfies property $(T_d)$ on $X_{\{\sigma_1,\dots,\sigma_k\}}^{\{\varepsilon_1,\dots,\varepsilon_{n-k}\}}$  and ${\{\sigma_1,\dots,\sigma_k\}}-{\{\varepsilon_1,\dots,\varepsilon_{n-k}\}}$ fulfill $(N_a)$. Then, the following properties are satisfied:
		\begin{itemize}
			\item If $n-k$ is even and $T_n[M]$ is  inverse negative in $X_{\{\sigma_1,\dots,\sigma_{k}\}}^{\{\varepsilon_1,\dots,\varepsilon_{n-k}\}}$, then $M\in[\bar M-\lambda_3,\bar M-\lambda_1)$, where
			\begin{itemize}
				\item [*]$\lambda_1>0$ is the least positive eigenvalue of $T_n[\bar M]$ in $X_{\{\sigma_1,\dots,\sigma_{k}\}}^{\{\varepsilon_1,\dots,\varepsilon_{n-k}\}}.$
				\item [*]$\lambda_3>0$ is the minimum between:
				\begin{itemize}
					\item [·]$\lambda_3'>0$, the least positive eigenvalue of $T_n[\bar M]$ in $X_{\{\sigma_1,\dots,\sigma_{k-1}|\alpha\}}^{\{\varepsilon_1,\dots,\varepsilon_{n-k}\}}.$
					\item [·]$\lambda_3''>0$ is the least positive eigenvalue of $T_n[\bar M]$ in $X_{\{\sigma_1,\dots,\sigma_{k}\}}^{\{\varepsilon_1,\dots,\varepsilon_{n-k-1}|\beta\}}.$  
				\end{itemize}
			\end{itemize}
%			\item If $k=1$, $n$ is odd and $T_n[M]$ is strongly inverse negative in $X_{\{\sigma_1\}}^{\{\varepsilon_1,\dots,\varepsilon_{n-1}\}}$ then $M\in[\bar M-\lambda_1,\bar M-\lambda_2]$, where
%			\begin{itemize}
%				\item $\lambda_1>0$ is the least positive eigenvalue of $T_n[\bar M]$ in $X_{\{\sigma_1\}}^{\{\varepsilon_1,\dots,\varepsilon_{n-1}\}}$.
%				\item $\lambda_2<0$ is the biggest negative eigenvalue of $T_n[\bar M]$ in $X_{\{\sigma_1|\alpha\}}^{\{\varepsilon_1,\dots,\varepsilon_{n-2}\}}$  
%			\end{itemize}
		\item If $n-k$ is odd and $T_n[M]$ is  inverse positive in $X_{\{\sigma_1,\dots,\sigma_{k}\}}^{\{\varepsilon_1,\dots,\varepsilon_{n-k}\}}$, then $M\in(\bar M-\lambda_1,\bar M-\lambda_3]$, where
		\begin{itemize}
			\item [*]$\lambda_1<0$ is the biggest negative eigenvalue of $T_n[\bar M]$ in $X_{\{\sigma_1,\dots,\sigma_{k}\}}^{\{\varepsilon_1,\dots,\varepsilon_{n-k}\}}.$ 
			\item [*]$\lambda_3<0$ is the maximum between:
			\begin{itemize}
				\item[·] $\lambda_3'<0$, the biggest negative eigenvalue of $T_n[\bar M]$ in $X_{\{\sigma_1,\dots,\sigma_{k-1}|\alpha\}}^{\{\varepsilon_1,\dots,\varepsilon_{n-k}\}}.$
				\item [·]$\lambda_3''<0$ is the biggest negative eigenvalue of $T_n[\bar M]$ in $X_{\{\sigma_1,\dots,\sigma_{k}\}}^{\{\varepsilon_1,\dots,\varepsilon_{n-k-1}|\beta\}}.$
			\end{itemize}
			\end{itemize}
		\end{itemize}
	\end{theorem}
	
	\begin{proof}
		From Theorem \ref{T::cv}, we can affirm that $\sigma_k\neq k-1$ and $\varepsilon_{n-k}\neq n-k-1$. Hence, by Corollary \ref{C::1} the existence of $\lambda_3'$ and $\lambda_3''$ is ensured.
		
		First, let us focus on the case where $n-k$ is even.
		
		Let us assume that there exists $M^*\notin [\bar M-\lambda_3, \bar M-\lambda_1)$, such that $T_n[M^*]$ is inverse negative on $X_{\{\sigma_1,\dots,\sigma_{k}\}}^{\{\varepsilon_1,\dots,\varepsilon_{n-k}\}}$. From Theorem \ref{T::6}, we know that $M^*<\bar M-\lambda_3$.
		
		Moreover, using Theorem \ref{T::8}, we can affirm that for all $M\in[M^*,\bar M-\lambda_1)$ the operator $T_n[M]$ is inverse negative on $X_{\{\sigma_1,\dots,\sigma_{k}\}}^{\{\varepsilon_1,\dots,\varepsilon_{n-k}\}}$ and, by Theorem \ref{T::d1}, that $0\geq g_{M^*}(t,s)\geq g_M(t,s)\geq g_{\bar M-\lambda_3}(t,s)$.
		
		So, in particular \[0\geq w_{M^*}(t)\geq w_M(t)\geq w_{\bar M-\lambda_3}(t)\,,\]
		and
		\[\left\lbrace \begin{array}{cc}0\geq y_{M^*}(t)\geq y_M(t)\geq y_{\bar M-\lambda_ 3}(t)\,,&\text{if } \gamma \text{ is even,}\\\\
		0\leq y_{M^*}(t)\leq y_M(t)\leq y_{\bar M-\lambda_ 3}(t)\,,& \text{if } \gamma \text{ is odd.}
		\end{array} \right. \]
		
		If $\lambda_3=\lambda_3'$, then  $w_{\bar M-\lambda_3}^{(\alpha)}(a)=0$. So, we conclude that, for all $M\in[M^*,\bar M-\lambda_3)$, $w_M^{(\alpha)}(a)=0$, which contradicts the discrete character of the spectrum of $T_n[\bar M]$ in $X_{\{\sigma_1,\dots,\sigma_{k-1}|\alpha\}}^{\{\varepsilon_1,\dots,\varepsilon_{n-k}\}}$.
		
			If $\lambda_3=\lambda_3''$, then  $y_{\bar M-\lambda_3}^{(\beta)}(b)=0$. So, we conclude that, for all $M\in[M^*,\bar M-\lambda_3)$,  $y_M^{(\beta)}(b)=0$, which contradicts the discrete character of the spectrum of $T_n[\bar M]$ in $X_{\{\sigma_1,\dots,\sigma_{k}\}}^{\{\varepsilon_1,\dots,\varepsilon_{n-k-1}|\beta\}}$.
			
	Analogously, if $n-k$ is odd and we 		assume that there exists $M^*\notin (\bar M-\lambda_1, \bar M-\lambda_3]$, such that $T_n[M^*]$ is inverse positive on $X_{\{\sigma_1,\dots,\sigma_{k}\}}^{\{\varepsilon_1,\dots,\varepsilon_{n-k}\}}$. From Theorem \ref{T::7}, we know that $M^*>\bar M-\lambda_3$.
	
	Moreover, using Theorem \ref{T::9}, we can affirm that for all $M\in(\bar M-\lambda_1,M^*]$ $T_n[M]$ is inverse positive on $X_{\{\sigma_1,\dots,\sigma_{k}\}}^{\{\varepsilon_1,\dots,\varepsilon_{n-k}\}}$ and, by Theorem \ref{T::d1}, that $  g_{\bar M-\lambda_3}(t,s)\geq g_M(t,s)\geq g_{M^*}(t,s)\geq 0$.
	
	So, in particular \[w_{\bar M-\lambda_3}(t) \geq w_M(t)\geq w_{M^*}(t)\geq 0\,,\]
	and
	\[\left\lbrace \begin{array}{cc}y_{\bar M-\lambda 3}(t) \geq y_M(t)\geq y_{M^*}(t)\geq 0\,,&\text{if } \gamma \text{ is even,}\\\\
	y_{\bar M-\lambda 3}(t)\leq y_M(t)\leq y_{M^*}(t)\leq 0\,,& \text{if } \gamma \text{ is odd.}
	\end{array} \right. \]
	
	If $\lambda_3=\lambda_3'$, then  $w_{\bar M-\lambda_3}^{(\alpha)}(a)=0$. So, we conclude that, for all $M\in(\bar M-\lambda_3, M^*]$,  $w_M^{(\alpha)}(a)=0$, which contradicts the discrete character of the spectrum of $T_n[\bar M]$ in $X_{\{\sigma_1,\dots,\sigma_{k-1}|\alpha\}}^{\{\varepsilon_1,\dots,\varepsilon_{n-k}\}}$.
	
	If $\lambda_3=\lambda_3''$, then  $y_{\bar M-\lambda_3}^{(\beta)}(b)=0$. So, we conclude that, for all $M\in(\bar M-\lambda_3,M^*]$,  $y_M^{(\beta)}(b)=0$, which contradicts the discrete character of the spectrum of $T_n[\bar M]$ in $X_{\{\sigma_1,\dots,\sigma_{k}\}}^{\{\varepsilon_1,\dots,\varepsilon_{n-k-1}|\beta\}}$.
	
	In all cases we arrive to a contradiction, thus the result is proved.	
	\end{proof}
	
	\begin{exemplo}
		\label{Ex::13}
		Now, let us focus again on our recurrent example $T_4^0[M]\,u(t)=u^{(4)}(t)+M\,u(t)$. 
		
		In Example \ref{Ex::6}, the different eigenvalues of $T_4^0[0]$ on the sets $X_{\{0,2\}}^{\{1,2\}}$, $X_{\{0,1\}}^{\{1,2\}}$ and $X_{\{0,2\}}^{\{1,2\}}$ are obtained. In particular, $\lambda_1=m_1^4$ and
		\[\lambda_3=\min\{m_3^4,\pi^4\}=\pi^4\,,\]
		where $m_1\approxeq 2.36502$ and $m_3\approxeq 3.9266$ have been introduced in Example \ref{Ex::6} as the least positive solutions of \eqref{Ec::Ex61} and \eqref{Ec::Ex63}, respectively.
		
		So, using Theorem \ref{T::IPN1}, we can affirm that if $T_4^0[M]$ is strongly inverse negative in $X_{\{0,2\}}^{\{1,2\}}$, then $M\in[-\pi^4,-m_1^4)$.
	\end{exemplo}
	
	In Theorem \ref{T::IPN1}, we have established a necessary condition on operator $T_n[M]$ to be either inverse positive or inverse negative in $X_{\{\sigma_1,\dots,\sigma_{k}\}}^{\{\varepsilon_1,\dots,\varepsilon_{n-k}\}}$. Next result, shows that this condition also ensures that the related Green's function satisfies a suitable condition on the boundary of $I\times I$.
	
		\begin{theorem}\label{T::IPN2}
			Let $\bar M\in \mathbb{R}$ be such that $T_n[\bar M]$ verifies property $(T_d)$ on $X_{\{\sigma_1,\dots,\sigma_k\}}^{\{\varepsilon_1,\dots,\varepsilon_{n-k}\}}$  and ${\{\sigma_1,\dots,\sigma_k\}}-{\{\varepsilon_1,\dots,\varepsilon_{n-k}\}}$ fulfill $(N_a)$. Moreover, $\sigma_k\neq k-1$ and $\varepsilon_{n-k}\neq n-k-1$. Then, the following properties are satisfied:
			
			\begin{itemize}
				\item If $n-k$ is even and $M\in[\bar M-\lambda_3,\bar M-\lambda_1)$, where $\lambda_1$ and $\lambda_3$ are given in Theorem \ref{T::IPN1}.
%				\begin{itemize}
%					\item $\lambda_1>0$ is the least positive eigenvalue of $T_n[\bar M]$ in $X_{\{\sigma_1,\dots,\sigma_{k}\}}^{\{\varepsilon_1,\dots,\varepsilon_{n-k}\}}.$ 
%					\item $\lambda_3>0$ is the maximum between:
%					\begin{itemize}
%						\item $\lambda_3'>0$, the least positive eigenvalue of $T_n[\bar M]$ in $X_{\{\sigma_1,\dots,\sigma_{k-1}|\alpha\}}^{\{\varepsilon_1,\dots,\varepsilon_{n-k}\}}.$
%						\item $\lambda_3''>0$ is the least positive eigenvalue of $T_n[\bar M]$ in $X_{\{\sigma_1,\dots,\sigma_{k}\}}^{\{\varepsilon_1,\dots,\varepsilon_{n-k-1}|\beta\}},$ 
%					\end{itemize}
					Then:
					\[\forall t\in (a,b)\,,\exists \,\rho_1(t)>0\ \mid \ g_M(t,s)<0\ \forall s\in (a,a+\rho_1(t))\cup (b-\rho_1(t),b)\,,\]
					and
					\[\forall s\in (a,b)\,,\exists\, \rho_2(s)>0\ \mid \ g_M(t,s)<0\ \forall s\in (a,a+\rho_2(s))\cup (b-\rho_2(s),b)\,.\]
			%	\end{itemize}
				%			\item If $k=1$, $n$ is odd and $T_n[M]$ is strongly inverse negative in $X_{\{\sigma_1\}}^{\{\varepsilon_1,\dots,\varepsilon_{n-1}\}}$ then $M\in[\bar M-\lambda_1,\bar M-\lambda_2]$, where
				%			\begin{itemize}
				%				\item $\lambda_1>0$ is the least positive eigenvalue of $T_n[\bar M]$ in $X_{\{\sigma_1\}}^{\{\varepsilon_1,\dots,\varepsilon_{n-1}\}}$.
				%				\item $\lambda_2<0$ is the biggest negative eigenvalue of $T_n[\bar M]$ in $X_{\{\sigma_1|\alpha\}}^{\{\varepsilon_1,\dots,\varepsilon_{n-2}\}}$  
				%			\end{itemize}
				\item If $n-k$ is odd and $M\in(\bar M-\lambda_1,\bar M-\lambda_3]$, where $\lambda_1$ and $\lambda_3$ are given in Theorem \ref{T::IPN1}.
%				\begin{itemize}
%					\item $\lambda_1<0$ is the biggest negative eigenvalue of $T_n[\bar M]$ in $X_{\{\sigma_1,\dots,\sigma_{k}\}}^{\{\varepsilon_1,\dots,\varepsilon_{n-k}\}}.$ 
%					\item $\lambda_3>0$ is the maximum between:
%					\begin{itemize}
%						\item $\lambda_3'<0$, the biggest negative eigenvalue of $T_n[\bar M]$ in $X_{\{\sigma_1,\dots,\sigma_{k-1}|\alpha\}}^{\{\varepsilon_1,\dots,\varepsilon_{n-k}\}}.$
%						\item $\lambda_3''<0$ is the biggest negative eigenvalue of $T_n[\bar M]$ in $X_{\{\sigma_1,\dots,\sigma_{k}\}}^{\{\varepsilon_1,\dots,\varepsilon_{n-k-1}|\beta\}},$ 
%					\end{itemize}
%				\end{itemize}
					Then:
					\[\forall t\in (a,b)\,,\exists\, \rho_1(t)>0\ \mid \ g_M(t,s)>0\ \forall s\in (a,a+\rho_1(t))\cup (b-\rho_1(t),b)\,,\]
					and
					\[\forall s\in (a,b)\,,\exists\, \rho_2(s)>0\ \mid \ g_M(t,s)>0\ \forall s\in (a,a+\rho_2(s))\cup (b-\rho_2(s),b)\,.\]
			\end{itemize}
		\end{theorem}
		\begin{proof}
	In order to prove this result we consider the following functions introduced in the proof of Theorem \ref{T::IPN}:
	\begin{eqnarray}
		\nonumber w_M(t)&=&\dfrac{\partial^{\eta}}{\partial s^\eta}g_M(t,s)_{\mid s=a}\,,\\\nonumber\\
		\nonumber y_M(t)&=&\dfrac{\partial^{\gamma}}{\partial s^\gamma}g_M(t,s)_{\mid s=b}\,,\\\nonumber\\
		\nonumber \widehat{w}_M(t)&=&(-1)^n\dfrac{\partial^{\alpha}}{\partial s^\alpha}\widehat g_M(t,s)_{\mid s=a}\,,
		\\\nonumber\\ \nonumber
		\widehat y _M(s)&=&(-1)^n\dfrac{\partial^{\beta}}{\partial s^\beta}\widehat g_M(t,s)_{\mid s=b}\,.\end{eqnarray}
		
		For these functions  we have obtained the following conclusions:
		\begin{itemize}
			\item $T_n[M]\,w_M(t)=0$ for all $t\in(a,b]$ and $w_M$ satisfies the boundary conditions \eqref{Ec::cfaa}-\eqref{Ec::cfbb}.
			\item $T_n[M]\,y_M(t)=0$ for all $t\in[a,b)$ and $y_M$ satisfies the boundary conditions \eqref{Ec::cfaaa}-\eqref{Ec::cfbbb}.
			\item $\widehat T_n[(-1)^n M]\,\widehat w_M(s)=0$ for all $s\in(a,b]$ and $\widehat w_M$ satisfies the boundary conditions \eqref{Cf::ad1}--\eqref{Cf::ad11} and \eqref{Cf::ad3}--\eqref{Cf::ad4}.
			\item $\widehat T_n[(-1)^n M]\,\widehat y_M(s)=0$ for all $s\in[a,b)$ and $\widehat y_M$ satisfies the boundary conditions \eqref{Cf::ad1}--\eqref{Cf::ad2} and \eqref{Cf::ad3}--\eqref{Cf::ad31}.
		\end{itemize}		
		
		Thus, by applying Propositions \ref{P::1}, \ref{P::2}, \ref{P::5} and \ref{P::6}, we know that $w_M$, $y_M$, $\widehat{w}_M$ and $\widehat y_M$ do not have any zero in $(a,b)$ for all $M\in[\bar M-\lambda_3,\bar M-\lambda_1)$ if $n-k$ is even and for all $M\in(\bar M-\lambda_1,\bar M-\lambda_3]$ when $n-k$ is odd.
		
		Moreover, by Proposition \ref{P::6.5},  since we do not reach any eigenvalue of $T_n[\bar M]$ in $X_{\{\sigma_1,\dots,\sigma_{k}\}}^{\{\varepsilon_1,\dots,\varepsilon_{n-k}\}}$, we have that the related Green's function is well-defined for every $M$ in those intervals. So, since we are moving continuously on $M$, we conclude that its sign is the same in all the interval.
		
		Let us, now, study the sign of these functions at a given $M$.
		
		We consider $w_M$ and $\widehat w_M$ at $M=\bar M-\lambda_3'$ and $y_M$ and $\widehat y_M$ at $M=\bar M-\lambda_3''$. As we have proved before, at this values of the real parameter the functions are of constant sign and satisfy the maximal oscillation, which means that verify the conditions at $t=a$ and $t=b$ to give the maximum number of zeros with the related boundary conditions, otherwise the function would be equivalent to zero and this is not true. Moreover, we know that they satisfy for all $M\in \mathbb{R}$ the following properties:
	{\small 	\begin{eqnarray}
	\label{Cf::wm}	w_M^{(\sigma_k)}(a)&=&(-1)^{(n-1-\sigma_k)}\,,\\\nonumber\\
		\label{Cf::ym}	y_M^{(\varepsilon_{n-k})}(b)&=&(-1)^{(n-\varepsilon_{n-k})}\,,\\\nonumber\\
		\label{Cf::wmg}	\widehat w_M^{(\tau_{n-k})}(a)+\sum_{j=n-\tau_{n-k}}^{n-1}(-1)^{n-j}\,(p_{n-j}\widehat w_M)^{(\tau_{n-k}+j-n)}(a)&=&(-1)^{1+\tau_{n-k}}\,,\\\nonumber \\
			\label{Cf::ymg}	\widehat y_M^{(\delta_{k})}(b)+\sum_{j=n-\delta_{k}}^{n-1}(-1)^{n-j}\,(p_{n-j}\widehat y_M)^{(\delta_{k}+j-n)}(b)&=&(-1)^{\delta_{k}}\,.
			\end{eqnarray}}
		
		\begin{itemize}
			\item Study of $w_{\bar M-\lambda_3'}$.
		\end{itemize}
		Let us consider $\alpha_1\in \{0,\dots,n-1\}$, previously introduced in Notation \ref{Not::alpha1}.
		
		Let us study the sign of $w_{\bar M-\lambda_3'}^{(\alpha_1)}(a)$ in order to obtain the sign of $w_{\bar M-\lambda_3'}$.
		
		From Lemma \ref{L::1}, coupled with \eqref{Cf::wm}, we conclude that
		\[\left\lbrace \begin{array}{cc}T_{\sigma_k}\,w_{\bar M-\lambda_3'}(a)>0\,,& \text{if $n-\sigma_k$ is odd,}\\\\T_{\sigma_k}\,w_{\bar M-\lambda_3'}(a)<0\,,& \text{if $n-\sigma_k$ is even.}\end{array}\right. \]
	
		As we have said before, to allow the maximal oscillation, $T_{n-\ell}w_{\bar M-\lambda_3'}$ must change its sign each time that it is different from zero. Thus, since from $\alpha_1$ to $\sigma_k$ we have $k-\alpha_1$ zeros, with maximal oscillation the following inequalities are fulfilled:
		
		If $n-\sigma_k$ is odd:
			\[\left\lbrace \begin{array}{cc}T_{\alpha_1}\,w_{\bar M-\lambda_3'}(a)>0\,,& \text{if $\sigma_k-\alpha_1-(k-\alpha_1)=\sigma_k-k$ is even,}\\\\T_{\alpha_1}\,w_{\bar M-\lambda_3'}(a)<0\,,& \text{if $\sigma_k-k$ is odd.}\end{array}\right. \]
			
			If $n-\sigma_k$ is even:
			\[\left\lbrace \begin{array}{cc}T_{\alpha_1}\,w_{\bar M-\lambda_3'}(a)<0\,,& \text{if $\sigma_k-k$ is even,}\\\\T_{\alpha_1}\,w_{\bar M-\lambda_3'}(a)>0\,,& \text{if $\sigma_k-k$ is odd.}\end{array}\right. \]
		
		From \eqref{Ec::Tl}, we conclude that
		\[T_{\alpha_1}\,w_{\bar M-\lambda_3'}(a)=\dfrac{1}{v_1(a)\dots v_{\alpha_1}(a)}w_{\bar M-\lambda_3'}^{(\alpha_1)}(a)\,,\]
		hence, we can affirm that 		
			\[\left\lbrace \begin{array}{cc}w_{\bar M-\lambda_3'}^{(\alpha_1)}(a)>0\,,& \text{if $n-k$ is odd,}\\\\w_{\bar M-\lambda_3'}^{(\alpha_1)}(a)<0º,,& \text{if $n-k$ is even.}\end{array}\right. \]
			
		Thus, we have proved that
		\begin{itemize}
			\item If $n-k$ is even and $M\in[\bar M-\lambda_3',\bar M-\lambda_1)$, then
				\[\forall t\in (a,b)\,,\exists \,{\rho_1}_1(t)>0\ \mid \ g_M(t,s)<0\ \forall s\in (a,a+{\rho_1}_1(t))\,.\]
				\item If $n-k$ is odd and $M\in(\bar M-\lambda_1,\bar M-\lambda_3']$, then
				\[\forall t\in (a,b)\,,\exists \,{\rho_1}_1(t)>0\ \mid \ g_M(t,s)>0\ \forall s\in (a,a+{\rho_1}_1(t))\,.\]
					\end{itemize}
					
					\begin{itemize}
						\item Study of $y_{\bar M-\lambda_3''}$.
					\end{itemize}
					Now, let us consider $\beta_1\in\{0,\dots,n-1\}$, introduced in Notation \ref{Not::alpha1}.
					
					 In order to obtain the sign of $y_{\bar M-\lambda_3''}$, let us study the sign of $y_{\bar M-\lambda_3''}^{(\beta_1)}(b)$.
						
						From Lemma \ref{L::2} coupled with \eqref{Cf::ym}, we conclude that
						\[\left\lbrace \begin{array}{cc}T_{\varepsilon_{n-k}}\,y_{\bar M-\lambda_3''}(b)>0\,,& \text{if $n-\varepsilon_{n-k}$ is even,}\\\\T_{\varepsilon_{n-k}}\,y_{\bar M-\lambda_3''}(b)<0\,,& \text{if $n-\varepsilon_{n-k}$ is odd.}\end{array}\right. \]
						
					In this case, as we have said on the proof of Theorem \ref{L::5},  to allow the maximal oscillation, $T_{n-\ell}w_{\bar M-\lambda_3'}(b)$ must change its sign each time that it vanishes. Thus, since from $\beta_1$ to $\varepsilon_{n-k}$ we have, with maximal oscillation, $n-k-\beta_1$ zeros, we deduce de following properties:
						
						If $n-\varepsilon_{n-k}$ is even:
						
						\[\left\lbrace \begin{array}{cc}T_{\beta_1}\,y_{\bar M-\lambda_3''}(b)>0\,,& \text{if $n-k-\beta_1$ is even,}\\\\T_{\beta_1}\,y_{\bar M-\lambda_3''}(b)<0\,,& \text{if $n-k-\beta_1$ is odd.}\end{array}\right. \]
						
						If $n-\varepsilon_{n-k}$ is odd:
						
						\[\left\lbrace \begin{array}{cc}T_{\beta_1}\,y_{\bar M-\lambda_3''}(b)<0\,,& \text{if $n-k-\beta_1$ is even,}\\\\T_{\beta_1}\,y_{\bar M-\lambda_3''}(b)>0\,,& \text{if $n-k-\beta_1$ is odd.}\end{array}\right. \]
						
						From \eqref{Ec::Tl}, we conclude that
						\[T_{\beta_1}\,y_{\bar M-\lambda_3''}(b)=\dfrac{1}{v_1(b)\dots v_{\beta_1}(b)}y_{\bar M-\lambda_3''}^{(\beta_1)}(b)\,,\]
						hence, we can affirm that 						
						\[\left\lbrace \begin{array}{cc}y_{\bar M-\lambda_3''}^{(\beta_1)}(b)>0\,,& \text{if $\varepsilon_{n-k}-k-\beta_1$ is even,}\\\\y_{\bar M-\lambda_3''}^{(\beta_1)}(b)<0\,,& \text{if $\varepsilon_{n-k}-k-\beta_1$ is odd.}\end{array}\right. \]
						
						Thus, we have proved that						
							\[\left\lbrace \begin{array}{ccc}y_{\bar M-\lambda_3''}(t)\geq 0\,,&t\in I\,,& \text{if $\varepsilon_{n-k}-k$ is even,}\\\\y_{\bar M-\lambda_3''}(t)\leq0\,,&t\in I\,,& \text{if $\varepsilon_{n-k}-k$ is odd.}\end{array}\right. \]
							
							Hence, since $\gamma=n-\varepsilon_{n-k}-1$, taking into account the fact that $y_{\bar M-\lambda_3''}$ cannot have any zero on $(a,b)$, we conclude
						\begin{itemize}
							\item If $n-k$ is even and $M\in[\bar M-\lambda_3'',\bar M-\lambda_1)$, then
							\[\forall t\in (a,b)\,,\exists \,{\rho_1}_2(t)>0\ \mid \ g_M(t,s)<0\ \forall s\in (b-{\rho_1}_2(t),b)\,.\]
							\item If $n-k$ is odd and $M\in(\bar M-\lambda_1,\bar M-\lambda_3'']$, then
							\[\forall t\in (a,b)\,,\exists \,{\rho_1}_2(t)>0\ \mid \ g_M(t,s)>0\ \forall s\in (b-{\rho_1}_2(t),b)\,.\]
						\end{itemize}
						
							\begin{itemize}
								\item Study of $\widehat w_{\bar M-\lambda_3'}$.
							\end{itemize}
						\begin{notation}
								Let us define $\eta_1\in \{0,\dots,n-1\}$ such that $\eta_1\notin \{\tau_1,\dots,\tau_{n-k-1},\eta\}$ and $\{0,\dots,\eta_1-1\}\subset\{\tau_1,\dots,\tau_{n-k-1},\eta\}$.
						\end{notation}
							
						 In order to obtain the sign of $\widehat w_{\bar M-\lambda_3'}$, let us study the sign of $\widehat w_{\bar M-\lambda_3'}^{(\eta_1)}(a)$.
							
							From Lemma \ref{L::3} coupled with \eqref{Cf::wmg}, we conclude that
							\[\left\lbrace \begin{array}{cc}\widehat T_{\tau_{n-k}}\,\widehat w_{\bar M-\lambda_3'}(a)>0\,,& \text{if $\tau_{n-k}$ is odd,}\\\\\widehat T_{\tau_{n-k}}\,\widehat{w}_{\bar M-\lambda_3'}(a)<0\,,& \text{if $\tau_{n-k}$ is even.}\end{array}\right. \]
							
							Analogously to $T_k$, to allow the maximal oscillation we conclude that $\widehat T_{n-\ell}\widehat w_{\bar M-\lambda_3'}$ must change its sign each time that it is non null. Thus, since from $\eta_1$ to $\tau_{n-k}$ we have $n-k-\eta_1$ zeros, with maximal oscillation, the following inequalities are fulfilled.
							
							If $\tau_{n-k}$ is odd:							
							\[\left\lbrace \begin{array}{cc}\widehat T_{\eta_1}\,\widehat w_{\bar M-\lambda_3'}(a)>0\,,& \text{if $\tau_{n-k}-\eta_1-(n-k-\eta_1)=\tau_{n-k}-n+k$ is even,}\\\\\widehat T_{\eta_1}\,\widehat w_{\bar M-\lambda_3'}(a)<0\,,& \text{if $\tau_{n-k}-n+k$ is odd.}\end{array}\right. \]
							
							If $\tau_{n-k}$ is even:							
								\[\left\lbrace \begin{array}{cc}\widehat T_{\eta_1}\,\widehat w_{\bar M-\lambda_3'}(a)<0\,,& \text{if $\tau_{n-k}-n+k$ is even,}\\\\\widehat T_{\eta_1}\,\widehat w_{\bar M-\lambda_3'}(a)>0\,,& \text{if $\tau_{n-k}-n+k$ is odd.}\end{array}\right. \]
							
							From \eqref{Ec::Tgg}, we conclude that
							\[\widehat T_{\eta_1}\,\widehat w_{\bar M-\lambda_3'}(a)={v_1(a)\dots v_{n-\eta_1}(a)}w_{\bar M-\lambda_3'}^{(\eta_1)}(a)\,,\]
							hence, we can affirm that 							
							\[\left\lbrace \begin{array}{cc}\widehat w_{\bar M-\lambda_3'}^{(\eta_1)}(a)>0\,,& \text{if $n-k$ is odd,}\\\\\widehat w_{\bar M-\lambda_3'}^{(\eta_1)}(a)<0\,,& \text{if $n-k$ is even.}\end{array}\right. \]

							Thus, we have proved that							
								\[\left\lbrace \begin{array}{cc}\widehat w_{\bar M-\lambda_3'}\geq 0\,,& \text{on $I$ if $n-k$ is odd,}\\\\\widehat w_{\bar M-\lambda_3'}\leq0\,,& \text{on $I$ if $n-k$ is even.}\end{array}\right. \]
								
								Hence, we conclude:
								
							\begin{itemize}
								\item If $n-k$ is even and $M\in[\bar M-\lambda_3',\bar M-\lambda_1)$, then
								\[\forall s\in (a,b)\,,\exists \,{\rho_2}_1(s)>0\ \mid \ g_M(t,s)<0\ \forall t\in (a,a+{\rho_2}_1(s))\,.\]
								\item If $n-k$ is odd and $M\in(\bar M-\lambda_1,\bar M-\lambda_3']$, then
								\[\forall s\in (a,b)\,,\exists \,{\rho_2}_2(t)>0\ \mid \ g_M(t,s)>0\ \forall t\in (a,a+{\rho_2}_1(s))\,.\]
							\end{itemize}
							
							\begin{itemize}
								\item Study of $\widehat y_{\bar M-\lambda_3''}$.
							\end{itemize}
							\begin{notation}
							Let us denote $\gamma_1\in\{0,\dots,n-1\}$, such that $\gamma_1\notin \{\delta_1,\dots,\delta_{k-1},\gamma\}$ and $\{0,\dots,\gamma_1-1\}\subset \{\delta_1,\dots,\delta_{k-1},\gamma\}$.
							\end{notation}
							
							 In order to obtain the sign of $\widehat y_{\bar M-\lambda_3''}$, let us study the sign of $\widehat y_{\bar M-\lambda_3''}^{(\beta_1)}(b)$.
							
							From Lemma \ref{L::4} coupled with \eqref{Cf::ymg}, we conclude that
							\[\left\lbrace \begin{array}{cc}\widehat T_{\delta_k}\,\widehat y_{\bar M-\lambda_3''}(b)>0\,,& \text{if $\delta_k$ is even,}\\\\\widehat T_{\delta_k}\,\widehat y_{\bar M-\lambda_3''}(b)<0\,,& \text{if $\delta_k$ is odd.}\end{array}\right. \]
							
							In this case, analogously to $T_k$, to allow the maximal oscillation $\widehat T_{n-\ell}w_{\bar M-\lambda_3'}(b)$ changes its sign each time that it vanishes  and it remains of constant sign if it does not vanish. Thus, with maximal oscillation, since from $\gamma_1$ to $\delta_k$ we have $k-\gamma_1$ zeros
							
							If $\delta_k$ is even:							
							\[\left\lbrace \begin{array}{cc}\widehat T_{\gamma_1}\,\widehat y_{\bar M-\lambda_3''}(b)>0\,,& \text{if $k-\gamma_1$ is even,}\\\\\widehat T_{\gamma_1}\,\widehat y_{\bar M-\lambda_3''}(b)<0\,,& \text{if $k-\gamma_1$ is odd.}\end{array}\right. \]
							
							If $\delta_k$ is odd:							
								\[\left\lbrace \begin{array}{cc}\widehat T_{\gamma_1}\,\widehat y_{\bar M-\lambda_3''}(b)<0\,,& \text{if $k-\gamma_1$ is even,}\\\\\widehat T_{\gamma_1}\,\widehat y_{\bar M-\lambda_3''}(b)>0\,,& \text{if $k-\gamma_1$ is odd.}\end{array}\right. \]
							
							From \eqref{Ec::Tgg}, we conclude that
							\[\widehat T_{\gamma_1}\,\widehat y_{\bar M-\lambda_3''}(b)={v_1(b)\dots v_{n-\gamma_1}(b)}\widehat y_{\bar M-\lambda_3''}^{(\gamma_1)}(b)\,,\]
							hence, we can affirm that 							
							\[\left\lbrace \begin{array}{cc}\widehat y_{\bar M-\lambda_3''}^{(\gamma_1)}(b)>0\,,& \text{if $k-\delta_k-\gamma_1$ is even,}\\\\\widehat y_{\bar M-\lambda_3'}^{(\gamma_1)}(b)<0\,,& \text{if $k-\delta_k-\gamma_1$ is odd.}\end{array}\right. \]
							
							Thus, we have proved that							
							\[\left\lbrace \begin{array}{cc}\widehat y_{\bar M-\lambda_3''}(t)\geq 0\,,& \text{if $k-\delta_k$ is even,}\\\\\widehat y_{\bar M-\lambda_3'}(t)\leq0\,,& \text{if $k-\delta_k$ is odd.}\end{array}\right. \]
							
							Hence, since $\beta=n-\delta_k-1$, we conclude
							\begin{itemize}
								\item If $n-k$ is even and $M\in[\bar M-\lambda_3'',\bar M-\lambda_1)$, then
								\[\forall s\in (a,b)\,,\exists \,{\rho_2}_2(s)>0\ \mid \ g_M(t,s)<0\ \forall t\in (b-{\rho_1}_2(s),b)\,.\]
								\item If $n-k$ is odd and $M\in(\bar M-\lambda_1,\bar M-\lambda_3'']$, then
								\[\forall t\in (a,b)\,,\exists \,{\rho_2}_2(s)>0\ \mid \ g_M(t,s)>0\ \forall t\in (b-{\rho_2}_2(s),b)\,.\]
							\end{itemize}
							
							So, the proof is complete taking $\rho_1(t)=\min\{{\rho_1}_1(t),{\rho_1}_2(t)\}$ and $\rho_2(s)=\min\{{\rho_2}_1(s),{\rho_2}_2(s)\}$.
		\end{proof}
		
		\begin{remark}
			\label{R::14}
			Realize that from Theorems \ref{T::IPN1} and \ref{T::IPN2}, if we are able to prove that the sign change of the related Green's function must begin on the boundary of $I\times I$, then the intervals obtained in Theorem \ref{T::IPN1} are optimal.
		\end{remark}
		\begin{exemplo}
			\label{Ex::14}
			Now, let us apply the Remark \ref{R::14} to our recurrent example, the operator $T_4^0[M]$.
			
			Let us assume that there exists $M^*\in[-\pi^4,m_1^4)$ such that $g_{M^*}(t,s)$ changes its sign. Then, from Theorem \ref{T::IPN2} it must exist  $s^*\in(0,1)$ such that $u^*(t)=g_{M^*}(t,s^*)$ has at least two zeros, $0<c_1<c_2<1$.
			
			By the definition of the Green's function $u^*\in C^2([0,1])$. So, there exists $c^*\in(c_1,c_2)$ such that ${u^*}'(c^*)=0$. There are two possibilities:
			
			\begin{itemize}
				\item $c^*\leq s^*$.
				
				In such a case, $u^*$ is a solution of $T_4^0[M^*]u^*(t)=0$ on $[0,c^*]$ verifying the boundary conditions $u^*(0)={u^*}''(0)={u^*}'(c^*)=0$. Moreover, it satisfies $u^*(c_1)=0$.
				
				The function $y^*(t)=u^*(c^*t)$ verifies $y^*(0)={y^*}''(0)={y^*}'(1)=0$ and $y^*\left( \frac{c_1}{c^*}\right) =0$. Moreover, it is a solution of $T_4^0[{c^*}^4M^*]y^*(t)=0$ on $[0,1]$, with $0>{c^*}^4\,M^*>M^*>-\pi^4>-m_3^4$, where $m_3^4$ has been introduced in Example \ref{Ex::6}. But, this is is a contradiction with Proposition \ref{P::2}.
			
			\item $c^*>s^*$.
			
			In this case, $u^*$ is a solution of $T_4^0[M^*]u^*(t)=0$ on $[c^*,1]$ verifying the boundary conditions ${u^*}'(c^*)={u^*}'(1)={u^*}''(1)=0$. Moreover, it satisfies $u^*(c_2)=0$.
			
				The function $y^*(t)=u^*((1-c^*)t+c^*)$ verifies 
				\begin{equation}
				\label{Ec::Ex14} {y^*}'(0)={y^*}'(1)={y^*}''(1)=0 \,.
				\end{equation}
				
				 Moreover, it is a solution of $T_4^0[(1-c^*)^4M^*]y^*(t)=0$ on $[0,1]$ and $y^*\left( \frac{c_2-c^*}{1-c^*}\right) =0$, with $0>(1-c^*)^4M^*>M^*>-\pi^4$.
			
			It can be seen that $\pi^4$ is the least positive eigenvalue of $T_4^0[0]$ in $X_{\{0,1\}}^{\{1,2\}}$. 
			
			Now, let us see that every solution of $u^{(4)}(t)+M\,u(t)=0$, verifying the given boundary conditions \eqref{Ec::Ex14}, cannot have any zero on $(0,1)$ whenever $M\in(-\pi^4,0)$. Which is a contradiction of supposing that there is a sign change on the Green's function.
			
			First, let us choose $M=-\left( \dfrac{\pi}{2}\right)^4$, the solution is given as a multiple of:
			\[u(t)=f_1(1-t)+f_2(1-t)\,,\]
			where, for $t\in[0,1]$,
			 {\footnotesize\[\qquad\begin{split} &f_1(t)=\left(1-\sinh \left(\frac{\pi }{2}\right)\right) \left(\sinh \left(\frac{\pi 
				t}{2}\right)-\sin \left(\frac{\pi  t}{2}\right)\right)\geq f_1(1)=-\left(\sinh \left(\frac{\pi }{2}\right)-1\right)^2\,,\\ &f_2(t)=\cosh \left(\frac{\pi }{2}\right)
			\left(\cos \left(\frac{\pi  t}{2}\right)+\cosh \left(\frac{\pi  t}{2}\right)\right)\geq f_2(0)=2 \cosh \left(\frac{\pi }{2}\right)\,.\end{split}\] }
		
		So, $u(t)\geq -\left(\sinh \left(\frac{\pi }{2}\right)-1\right)^2+2 \cosh \left(\frac{\pi }{2}\right)>0$ for all $t\in[0,1]$.
		
		Now, let us move continuously on $M$ to obtain the different solutions of $T_4^0[M]\,u(t)=0$, coupled with the boundary conditions \eqref{Ec::Ex14}. 
		
		Let us see that it is not possible that these solutions begin to change sign on $(0,1)$. If this was the case, we would have that there exist $\widehat{M}\in(-\pi^4,0)$ and $\widehat{t}\in(0,1)$ such that $\widehat{u}$ is a solution of $T_4^0[\widehat{M}]\widehat u(t)=0$ on $[\widehat{t},1]$, verifying $\widehat{u}(\widehat{t})=\widehat{u}'(\widehat{t})=\widehat{u}(1)=\widehat{u}''(1)=0$.
		
		Then, the function $\widehat{y}(t)=\widehat{u}((1-\widehat{t})t+\widehat{t})$ is an eigenfunction related to the eigenvalue $-(1-\widehat{t})^4\,\widehat{M}\in(0,\pi^4)$ of the operator $T_4^0[0]$ in $X_{\{0,1\}}^{\{1,2\}}$ which is a contradiction.
		
		Analogously, if there exists $\widehat{M}\in(-\pi^4,0)$, for which  there is a nontrivial solution of $T_4^0[\widehat{M}]u(t)=0$ on $[0,1]$, verifying $u(0)=0$, coupled with the boundary conditions \eqref{Ec::Ex14}, then there is an eigenvalue $-\widehat{M}\in (0,\pi^4)$ of the operator $T_4^0[0]$ in $X_{\{0,1\}}^{\{1,2\}}$, which is again a contradiction.
		
		Finally, since there is not any positive eigenvalue of $T_4^0[0]$ in $X_{\{1\}}^{\{0,1,2\}}$, we can affirm that it is not possible that the sign change  begins at $t=1$. 
		
		So, we have proved that every solution of $u^{(4)}(t)+M\,u(t)=0$ coupled with the boundary conditions \eqref{Ec::Ex14} does not not have any zero on $(0,1)$ for $M\in(-\pi^4,0)$. Thus, we also have arrived to a contradiction if $c^*>s$.
			\end{itemize}
			
			So, from Remark \ref{R::14}, coupled with Theorems \ref{T::IPN1} and \ref{T::IPN2}, we can affirm that $T_4^0[M]$ is a strongly inverse negative operator in $X_{\{0,2\}}^{\{1,2\}}$ if, and only if, $M\in[-\pi^4,-m_1^4)$, where $m_1$ has been introduced in Example \ref{Ex::6}.
		\end{exemplo}
		
		\begin{exemplo}
			Using a similar argument to Example \ref{Ex::14}, in \cite{CabSaa2} it is studied the strongly inverse negative character of the operator $T_4(p_1,p_2)[M]$ previously introduced in Section \ref{SS::Ex}. There, a characterization of the parameter's set where $T_4(p_1,p_2)$ is strongly inverse negative in $X_{\{0,2\}}^{\{0,2\}}$ is obtained and several particular examples are given.
		\end{exemplo}

		\chapter[Non homogeneous boundary conditions]{Characterization of strongly inverse positive (negative) character for non homogeneous boundary conditions.}
		This chapter is devoted to the study of the operator $T_n[M]$, coupled  with different non homogeneous boundary conditions.
		
		First, let us  consider the following set:
			{\scriptsize \begin{equation}\label{Ec::X_senh}
			\begin{split}
			\tilde X_{\{\sigma_1,\dots, \sigma_k\}}^{\{\varepsilon_1,\dots, \varepsilon_{n-k}\}}=&\left\lbrace u\in C^n(I)\ \mid\ u^{(\sigma_1)}(a)=\cdots=u^{(\sigma_k-1)}(a)=0\,,\ (-1)^{n-\sigma_k-1}u^{(\sigma_k)}(a)\geq0\,,\right. \\ &\ \left.  u^{(\varepsilon_1)}(b)=\cdots=u^{(\varepsilon_{n-k-1})}(b)=0\,,\ u^{(\varepsilon_{n-k})}(b)\leq 0\right\rbrace \,.
				\end{split}\end{equation}}
			
			That is, we consider a set where some of the boundary conditions do not have to be necessarily homogeneous. This information is very useful in order to apply the lower and upper solutions method and monotone iterative techniques for nonlinear boundary value problems, see for instance \cite{CaCiSa}.
			
			So, we are interested into characterize the parameter's set for which the operator $T_n[M]$ is strongly inverse positive or negative on $\tilde X_{\{\sigma_1,\dots, \sigma_k\}}^{\{\varepsilon_1,\dots, \varepsilon_{n-k}\}}$.
			
			We introduce the boundary conditions which a function $u\in \tilde X_{\{\sigma_1,\dots, \sigma_k\}}^{\{\varepsilon_1,\dots, \varepsilon_{n-k}\}}$ must satisfy:
			\begin{align}
			\label{Ec::cfanh} u^{(\sigma_1)}(a)=\cdots=u^{(\sigma_k-1)}(a)&=0\,,\quad u^{(\sigma_k)}(a)=c_1\,,\\			
			\label{Ec::cfbnh} u^{(\varepsilon_1)}(b)=\cdots=u^{(\varepsilon_{n-k-1})}(b)&=0\,,\quad u^{(\varepsilon_{n-k})}(b)=c_2\,,
			\end{align}
			where $(-1)^{n-\sigma_k-1}c_1\geq 0$ and $c_2\leq 0$.

			We can connect the problem \eqref{Ec::T_n[M]}, \eqref{Ec::cfanh}-\eqref{Ec::cfbnh} with the homogeneous problem \eqref{Ec::T_n[M]}--\eqref{Ec::cfb} by means of the following result.
			
			\begin{lemma}\label{L::14}
				If the problem \eqref{Ec::T_n[M]}--\eqref{Ec::cfb} has only the trivial solution.  Then the problem  $T_n[M]\,u(t)=h(t)$, $t\in I$, coupled with the boundary conditions \eqref{Ec::cfanh}-\eqref{Ec::cfbnh} has a unique solution given by:
				\begin{equation}
				\label{Ec::sol}
				u(t)=\int_a^bg_M(t,s)\,h(s)\,ds+c_1\,x_M(t)+d_1\,z_M(t)\,,
				\end{equation}
				where $g_M(t,s)$ is the related Green's function of $T_n[M]$ in $X_{\{\sigma_1,\dots, \sigma_k\}}^{\{\varepsilon_1,\dots, \varepsilon_{n-k}\}}$ and:
				\begin{itemize}
					\item $x_M$ is defined as the unique solution of
					\begin{equation}
					\label{Ec::xm}\left\lbrace \begin{array}{c}
					T_n[M]\,u(t)=0\,,\quad t\in I\,,\\\\
					u^{(\sigma_1)}(a)=\dots=u^{(\sigma_{k-1})}(a)=0\,,\quad u^{(\sigma_k)}(a)=1\,,\\\\ u^{(\varepsilon_1)}(b)=\cdots=u^{(\varepsilon_{n-k})}(b)=0\,.\end{array} \right. 
					\end{equation}
					\item $z_M$ is defined as the unique solution of
						\begin{equation}
						\label{Ec::ym}\left\lbrace \begin{array}{c}
						T_n[M]\,u(t)=0\,,\quad t\in I\\\\
						u^{(\sigma_1)}(a)=\dots=u^{(\sigma_{k})}(a)=0\,,\\\\ u^{(\varepsilon_1)}(b)=\cdots=u^{(\varepsilon_{n-k-1})}(b)=0\,,\quad u^{(\varepsilon_{n-k})}(b)=1\,.\end{array} \right. 
						\end{equation}
				\end{itemize}
			\end{lemma}
			
			Using this Lemma we can obtain the following result which characterizes the strongly inverse positive (negative) character of $T_n[M]$ in $\tilde X_{\{\sigma_1,\dots,\sigma_k\}}^{\{\varepsilon_1,\dots,\varepsilon_{n-k}\}}$.

	\begin{theorem}\label{T::IPNH}
	$T_n[M]$ is strongly inverse positive (negative) in $\tilde X_{\{\sigma_1,\dots,\sigma_k\}}^{\{\varepsilon_1,\dots,\varepsilon_{n-k}\}}$ if, and only if, it is strongly inverse positive (negative) in $X_{\{\sigma_1,\dots,\sigma_k\}}^{\{\varepsilon_1,\dots,\varepsilon_{n-k}\}}$. 	
	\end{theorem}
			\begin{proof}
				Since $X_{\{\sigma_1,\dots,\sigma_k\}}^{\{\varepsilon_1,\dots,\varepsilon_{n-k}\}}\subset \tilde X_{\{\sigma_1,\dots,\sigma_k\}}^{\{\varepsilon_1,\dots,\varepsilon_{n-k}\}}$ the necessary condition is obvious.

				Now, let us see the sufficient one. From the strongly inverse positive (negative) character of $T_n[M]$ in $X_{\{\sigma_1,\dots,\sigma_k\}}^{\{\varepsilon_1,\dots,\varepsilon_{n-k}\}}$, using Theorem \ref{T::in2}, we conclude that $g_M>0$ ($<0$) a.e. on $I\times I$. Then, from Lemma \ref{L::14}, we only need to study the sign of $x_M$ and $z_M$.
				
				In order to do that, we  establish a relationship between these functions and some derivatives of $g_M(t,s)$.
				
				Taking into account the boundary conditions, it is clear that \[x_M(t)=(-1)^{n-1-\sigma_k}\,w_M(t)\text{ and } z_M(t)=(-1)^{n-\varepsilon_{n-k}}\,y_M(t)\,,\] where $w_M$ and $y_M$ have been defined in the proof of Theorem \ref{T::IPN} as follows:				
			\[w_M(t)=\dfrac{\partial^\eta}{\partial s^\eta}g_M^1(t,s)_{\mid s=a}\,,\quad 
			y_M(t)=\dfrac{\partial ^\gamma}{\partial s^\gamma}g_M^2(t,s)_{\mid s=b}\,.\]
				
				Now, if $T_n[M]$ is strongly inverse positive (negative) in $X_{\{\sigma_1,\dots,\sigma_k\}}^{\{\varepsilon_1,\dots,\varepsilon_{n-k}\}}$, then $w_M(t)\geq 0$ ($\leq 0$) and $(-1)^{\gamma}\,y_M(t)\geq 0$, so, since $\gamma=n-1-\varepsilon_{n-k}$, we have that $(-1)^{n-\varepsilon_{n-k}}\,y_M(t)\leq 0$ in both cases. 
				
				Thus, the result is proved.
			\end{proof}
		\section{Study of particular kind of operators}
		In this section we consider a particular kind of operators which satisfy property $(T_d)$ in $X_{\{\sigma_1,\dots,\sigma_k\}}^{\{\varepsilon_1,\dots,\varepsilon_{n-k}\}}$, thus we can apply previous results to these operators. 
		
		After that, we obtain some results which characterize  either the strongly inverse positive character or the strongly inverse negative character of $T_n[M]$ if $n-k$ is even or odd, respectively, in different sets where more general non homogeneous boundary conditions are considered.
		
		First, we introduce the following notation.
		
	\begin{notation}
			First, let us denote $\alpha_2\in\{-1,0,1,\dots,n-2\}$, such that $\alpha_2\notin \{\sigma_1,\dots,\sigma_k\}$ and $\{\alpha_2+1,\alpha_2+2,\dots,\sigma_k\}\subset\{\sigma_1,\dots,\sigma_k\}$; and  $\beta_2\in\{-1,0,1,\dots,n-2\}$, such that $\beta_2\notin \{\varepsilon_1,\dots,\varepsilon_{n-k}\}$ and $\{\beta_2+1,\beta_2+2,\dots,\varepsilon_{n-k}\}\subset\{\varepsilon_1,\dots,\varepsilon_{n-k}\}$ .

		We denote $\mu=\max\{\alpha_2,\beta_2\}$.
			\end{notation}
		\begin{remark}
			Realize that if $\sigma_k=k-1$ then $\alpha_2=-1$. Otherwise, $\alpha_2\geq \alpha\geq 0$.
			 
			 Moreover,  if $\varepsilon_{n-k}=n-k-1$ then $\beta_2=-1$. Otherwise, $\beta_2\geq \beta\geq 0$.
		\end{remark}
		
		Now, we introduce the following sufficient condition for an operator to satisfy property $(T_d)$.
		\begin{proposition}
			If the linear differential equation of  $(n-\mu-1)^\mathrm{th}-$order:
			\begin{equation}
			\label{Ec::mu} L_{n-\mu-1}\,u(t)\equiv u^{(n-\mu-1)}(t)+p_1(t)\,u^{(n-\mu-2)}(t)+\cdots+p_{n-\mu-1}(t)\,u(t)=0\,,
			\end{equation}
			with $p_{j}\in C^{n-j}(I)$, is disconjugate on $I$, then the operator:
			\[\tilde T_n[0]\,u(t)=u^{(n)}(t)+p_1(t)\,u^{(n-1)}(t)+\cdots+p_{n-\mu-1}(t)\,u^{(\mu+1)}(t)\,,\]
			satisfies property $(T_d)$ in $X_{\{\sigma_1,\dots,\sigma_k\}}^{\{\varepsilon_1,\dots,\varepsilon_{n-k}\}}$.
		\end{proposition}
		\begin{proof}
			From Theorems \ref{T::4} and \ref{T::3}, since the linear differential equation \eqref{Ec::mu} is disconjugate on $I$, there exist positive functions $v_1,\,\dots\,, v_{n-\mu-1}$ such that $v_k\in C^{n-\mu-k}(I)$ for $k=1,\dots, n-\mu-1$, and
			\[L_{n-\mu-1}\,u\equiv v_1\cdots v_{n-\mu-1}\,\dfrac{d}{dt}\left( \dfrac{1}{v_{n-\mu-1}}\dfrac{d}{dt}\left( \cdots \dfrac{d}{dt}\left( \dfrac{u}{v_1}\right) \right) \right) \,.\]
			
		\begin{itemize}
			\item [Step 1.] Let us see that, in fact, $v_k\in C^n(I)$ for $k=1,\dots,n-\mu-1$.
		\end{itemize}
			
			Since, $p_{j}\in C^{n-j}(I)$ for $j\in\{ \mu+1,\dots,n-1\}$, every solution of \eqref{Ec::mu} belongs to $C^n(I)$.
			
			If we look at the proof of Theorem \ref{T::3}, given in \cite[Chapter 3, Theorem 2]{Cop}, we observe that $v_k$ is given by the following recurrence formula:
			\[v_1=y_1\,,\quad v_2=\dfrac{W(y_1,y_2)}{y_1^2}\,,\quad v_k=\dfrac{W(y_1,\dots,y_k)\,W(y_1,\dots,y_{k-2})}{W(y_1,\dots,y_k-1)^2}\,,\ \text{for } k\geq 2\,,\]
			where $\{y_1,\dots,y_{n-\mu-1}\}$ is a Markov fundamental system of solutions of \eqref{Ec::mu} and $W$ the correspondent Wronskians.
			
			Thus, taking into account that $y_1,\dots,y_{n-\mu-1}\in C^n(I)$, we conclude that $v_1\in C^n(I)$, $v_2\in C^{n-1}(I)\,,\dots,v_{n-\mu-1}\in C^{\mu+2}(I)$.
			
			\vspace{0.5cm}
			
		 Now, let us consider the expression \eqref{Ec::Tl}, with $\ell =n-\mu-1$ and ${p_\ell}_{j}=p_{j}\in C^{n-j}(I)$, $j\in\{\mu+1,\dots,n-1\}$ given by expressions \eqref{Ec::pl1}--\eqref{Ec::pll}.
		 
		 First, let us see that $v_1\in C^n(I)$, $v_2\in C^n(I)$, $v_3\in C^{n-1}(I)\,,\dots,v_{n-\mu-1}\in C^{\mu+3}(I)$.
			
		If $\mu=n-2$, then $n-\mu-1=1$ and the result is proved
		
		Otherwise, $p_1\in C^{n-1}(I)\subset C^{\mu+2}(I)$, since $v_1\,,\dots\,,v_{n-\mu-2}\in C^{\mu+3}(I)$ and $v_{n-\mu-1}\in C^{\mu+2}(I)$, from \eqref{Ec::pl1} we obtain that $v_{n-\mu-1}'\in C^{\mu+2}(I)$, then $v_{n-\mu-1}\in C^{\mu+3}(I)$.
		
		Let us assume that $v_{k+1}\in C^{n-k}(I)$, $v_{k+2}\in C^{n-k-1}(I)\,,\dots\,,v_{n-\mu-1}\in C^{\mu +3}(I)$, then since $v_{k}\in C^{n-k}(I)$ considering the expression of $p_{n-\mu-k}$, given in \eqref{Ec::pll} for $\ell_\ell=n-\mu-k$, we obtain that $v_k^{(n-\mu-k)}\in C^{\mu+1}(I)$, hence $v_k\in C^{n-k+1}(I)$.
		
		Thus, we have proved by induction that  $v_1\in C^n(I)$, $v_2\in C^n(I)$, $v_3\in C^{n-1}(I)\,,\dots,v_{n-\mu-1}\in C^{\mu+2}$.
		
		If $\mu= n-3$, then the result is proved, since $v_1$, $v_2\in C^n(I)$.
		
		\vspace{0.5cm}
		
		Now, let us assume that $\mu<n-3$. Considering the expression of $p_{n-\mu-3}\in C^{\mu +3}(I)$, given in \eqref{Ec::pll} for $\ell_\ell=n-\mu-k$. Since $v_2\in C^n(I)$, we conclude that $v_3^{(n-\mu-3)}\in C^{\mu+3}(I)$; so, $v_3\in C^n(I)$.
		
		If we suppose that $v_1,\dots,v_{k-1}\in C^n(I)$, then by considering the expression of $p_{n-\mu-k}\in C^{\mu+k}(I)$, we conclude that $v_k^{(n-\mu-k)}\in C^{\mu+k}(I)$, thus $v_k\in C^n(I)$.
		
		Then, we have proved that $v_1,\dots,v_{n-\mu-1}\in C^n(I)$.
		
		\begin{itemize}
			\item[Step 2.] Construction of the decomposition verifying the property $(T_d)$.
		\end{itemize}
	
		Now, we consider the decomposition of $\tilde{T}[0]$ as follows:		
			\[\tilde{T}[0]\,u\equiv v_1\dots v_{n-\mu-1}\,\dfrac{d}{dt}\left( \dfrac{1}{v_{n-\mu-1}}\dfrac{d}{dt}\left( \cdots \dfrac{d}{dt}\left( \dfrac{u^{(\mu+1)}}{v_1}\right) \right) \right) \,.\]
			
			Hence, if we denote $\tilde v_1=\dots=\tilde v_{\mu+1}=1$ and $\tilde v_{\mu+2}=v_1,\dots,\tilde v_n=v_{n-\mu-1}$, we can decompose $\tilde T_n[0]$ in the following sense: 
			\[\tilde T_0\, u=u\,,\quad \tilde T_k\,u=\dfrac{d}{dt}\left( \dfrac{\tilde T_{k-1}\,u}{\tilde v_k}\right) \,,\ k=1,\dots,n\,.\]
			
			Trivially $\tilde T_n[0]\,u={\tilde{v}_1\dots\tilde v_n}\tilde T_n\,u$.
			
			Now, let us see that this decomposition satisfies the property $(T_d)$.

			We have that $\tilde T_0 u=u$, $\tilde{T}_1 u=u'\,,\dots,\tilde T_{\mu+1}u=u^{(\mu+1)}$.
			
			Hence, if $\sigma_i<\alpha_2\leq \mu$ then $\tilde T_{\sigma_i}u(a)=u^{(\sigma_i)}(a)=0$.
			
			Analogously, if $\varepsilon_i<\beta_2\leq \mu$, then $\tilde T_{\varepsilon_i}u(b)=u^{(\varepsilon_i)}(b)=0$.
			
			If $h>\mu+1$, then
			\[\tilde T_{h}u=\dfrac{u^{(h)}}{v_1\dots v_h}+{p_h}_1\,u^{(h-1)}+\cdots+ {p_h}_{h-\mu-1}\,u^{(\mu+1)}\,,\]
			where ${p_h}_i$ is given by equations \eqref{Ec::pl1}--\eqref{Ec::pll}.
			
			If $\sigma_i>\mu$, then by definition of $\mu$, $u^{(\mu+1)}(a)=u^{(\mu+2)}(a)=\cdots=u^{(\sigma_i)}(a)=0$. Hence $\tilde T_{\sigma_i}u(a)=0$.
			
			Analogously, if $\varepsilon_i>\mu$, then $u^{(\mu+1)}(b)=u^{(\mu+2)}(b)=\cdots=u^{(\varepsilon_i)}(b)=0$. Hence $\tilde T_{\varepsilon_i}u(b)=0$.
			
			Thus, the result is proved.		
		\end{proof}
		
		As consequence of this result, we can apply Theorems \ref{T::IPN}, \ref{T::IPN1}, \ref{T::IPN2} and \ref{T::IPNH} to operator $\tilde T_n[M]$. 
		
		Moreover, for this particular case, we will be able to obtain a characterization of strongly inverse positive (negative) character in different spaces with inhomogeneous boundary conditions.
		
	\begin{definition}\label{Def::SE}	Let us consider $\{\sigma_{\epsilon_1},\dots,\sigma_{\epsilon_\ell}\}\subset\{\sigma_1,\dots,\sigma_k\}$ such that $\sigma_{\epsilon_1}<\sigma_{\epsilon_2}<\cdots<\sigma_{\epsilon_\ell}=\sigma_k$, with $\sigma_{\epsilon_{\ell-1}}<\mu$.
		
		And $\{\varepsilon_{\kappa_1},\dots,\varepsilon_{\kappa_h}\}\subset\{\varepsilon_1,\dots,\varepsilon_{n-k}\}$ such that $\varepsilon_{\kappa_1}<\varepsilon_{\kappa_2}<\cdots<\varepsilon_{\kappa_h}=\varepsilon_{n-k}$, with $\varepsilon_{\kappa_{h-1}}<\mu$.
		
	\end{definition}
		Let us define the set of functions $X_{\{\sigma_1,\dots,\sigma_k\}_{\{\sigma_{\epsilon_1},\dots,\sigma_{\epsilon_\ell}\}}}^{\{\varepsilon_1,\dots,\varepsilon_{n-k}\}_{\{\varepsilon_{\kappa_1},\dots,\varepsilon_{\kappa_h}\}}}$ as follows:
	{\tiny	\begin{eqnarray}
		\nonumber
	&&\hspace{-0.6cm}	\left\lbrace u\in C^n (I)\quad \mid \quad u^{(\sigma_j)}(a)=\left\lbrace \begin{array}{cc}0&j\notin \{\epsilon_1,\dots,\epsilon_\ell\}\\\\(-1)^{n-\sigma_j-(k-j)+1}\varphi_j&j\in\{\epsilon_1,\dots,\epsilon_\ell\}\end{array}\right.\text{ for some }\varphi_j\geq 0,\ j=1,\dots,k\,,\right. \\\nonumber \\\label{Ec::xNH}
&&	\left. u^{(\varepsilon_i)}(b)=\left\lbrace \begin{array}{cc}0&i\notin \{\kappa_1,\dots,\kappa_h\}\\\\(-1)^{n-k+i-1}\psi_i&i\in\{\kappa_1,\dots,\kappa_h\}\end{array}\right. \text{ for some }\psi_i\geq 0\,,\ i=1,\dots,n-k\right\rbrace\,. 
		\end{eqnarray}}

	Now, we enunciate a similar result to Lemma \ref{L::14} for this more general case
		\begin{lemma}\label{L::14-1}
			If problem \eqref{Ec::T_n[M]}--\eqref{Ec::cfb} has only the trivial solution.  Then problem  $T_n[M]\,u(t)=h(t)$, $t\in I$, coupled with the boundary conditions 
			\begin{equation}
			\label{Ec::cfnh1}u^{(\sigma_j)}(a)=\left\lbrace \begin{array}{cc}0\,,&j\notin \{\epsilon_1,\dots,\epsilon_\ell\}\,,\\\\c_j\,,&j\in\{\epsilon_1,\dots,\epsilon_\ell\}\end{array}\right.\quad j=1,\dots,k\,,
			\end{equation}
			and
			\begin{equation}
			\label{Ec::cfnh2} u^{(\varepsilon_i)}(b)=\left\lbrace \begin{array}{cc}0\,,&i\notin \{\kappa_1,\dots,\kappa_h\}\,,\\\\d_i\,,&i\in\{\kappa_1,\dots,\kappa_h\}\end{array}\right.\quad i=1,\dots,n-k\,,
			\end{equation}			
			has a unique solution given by:
			\begin{equation}
			\label{Ec::solnh}
			u(t)=\int_a^bg_M(t,s)\,h(s)\,ds+\sum_{j=1}^\ell c_{\epsilon_j}\,x_M^{\sigma_{\epsilon_j}}(t)+\sum_{i=1}^hd_{\kappa_i}\,z_M^{\varepsilon_{\kappa_i}}(t)\,,
			\end{equation}
			where $g_M(t,s)$ is the related Green's function of $T_n[M]$ in $X_{\{\sigma_1,\dots, \sigma_k\}}^{\{\varepsilon_1,\dots, \varepsilon_{n-k}\}}$ and,
			\begin{itemize}
				\item $x_M^{\sigma_{\epsilon_j}}$ is defined as the unique solution of
				\begin{equation}
				\label{Ec::xm1}\left\lbrace \begin{array}{c}
				T_n[M]\,u(t)=0\,,\quad t\in I\\\\
				u^{(\sigma_{\epsilon_j})}(a)=1\,,\\\\
				u^{(\sigma_1)}(a)=\dots=u^{(\sigma_{\epsilon_j-1})}(a)=u^{(\sigma_{\epsilon_j+1})}(a)=\cdots=u^{(\sigma_k)}(a)=0\,,\\\\ u^{(\varepsilon_1)}(b)=\cdots=u^{(\varepsilon_{n-k})}(b)=0\,.\end{array} \right. 
				\end{equation}
				\item $z_M^{\varepsilon_{\kappa_i}}$ is defined as the unique solution of
				\begin{equation}
				\label{Ec::ym1}\left\lbrace \begin{array}{c}
				T_n[M]\,u(t)=0\,,\quad t\in I\,,\\\\
				u^{(\sigma_k)}(a)=\dots=u^{(\sigma_{k})}=0\,,\\\\
				u^{(\varepsilon_{\kappa_j})}(b)=1\,,\\\\ u^{(\varepsilon_1)}(b)=\cdots=u^{(\varepsilon_{\kappa_i-1})}(b)=u^{(\varepsilon_{\kappa_i+1})}(b)=\cdots=u^{(\varepsilon_{n-k})}(b)=0\,.\end{array} \right. 
				\end{equation}
			\end{itemize}
		\end{lemma}
		
		We have the following results, which ensure the existence of the different eigenvalues.
		
		\begin{lemma}\label{L::Td}
		{\color{white}.}
			\begin{itemize}
				\item If $\sigma_{\epsilon_j}>\alpha$, then $\tilde T_n[0]$ satisfies the property $(T_d)$ in $X_{\{\sigma_1,\dots,\sigma_{\epsilon_j-1},\sigma_{\epsilon_j+1},\dots,\sigma_k|\alpha\}}^{\{\varepsilon_1,\dots,\varepsilon_{n-k}\}}$.
				\item 	Operator $\tilde T_n[0]$ satisfies the property $(T_d)$ in $X_{\{\sigma_1,\dots,\sigma_{\epsilon_j-1},\sigma_{\epsilon_j+1},\dots,\sigma_k\}}^{\{\varepsilon_1,\dots,\varepsilon_{n-k}|\beta\}}$.
				\item  If $\varepsilon_{\kappa_i}>\beta$, then $\tilde T_n[0]$ satisfies the property $(T_d)$ in $X_{\{\sigma_1,\dots,\sigma_k\}}^{\{\varepsilon_1,\dots,\varepsilon_{\kappa_i-1},\varepsilon_{\kappa_i+1},\dots,\varepsilon_{n-k}|\beta\}}.$ 
					\item Operator $\tilde T_n[0]$ satisfies the property $(T_d)$ in $X_{\{\sigma_1,\dots,\sigma_k|\alpha\}}^{\{\varepsilon_1,\dots,\varepsilon_{\kappa_i-1},\varepsilon_{\kappa_i+1},\dots,\varepsilon_{n-k}\}}$.
			\end{itemize}
		\end{lemma}
		
		\begin{proof} Let us see the different cases:
			\begin{itemize} 
				
			\item If $\sigma_j<\mu$, then $T_{\sigma_j}u(a)=u^{(\sigma_j)}(a)=0$.
			
			\item  If $\sigma_j=\sigma_k$, then $T_{\sigma_k}u(a)=0$, by the definition of $\mu$.
			
			\item If $\varepsilon_i<\mu$, then $T_{\varepsilon_i}u(b)=u^{(\varepsilon_i)}(b)=0$.
			
			\item  If $\varepsilon_i=\varepsilon_{n-k}$, then $T_{\varepsilon_{n-k}}u(b)=0$, by the definition of $\mu$.
			
			\item If $u\in X_{\{\sigma_1,\dots,\sigma_{\epsilon_j-1},\sigma_{\epsilon_j+1},\dots,\sigma_k\}}^{\{\varepsilon_1,\dots,\varepsilon_{n-k}|\beta\}}$ or $u\in X_{\{\sigma_1,\dots,\sigma_{\epsilon_j-1},\sigma_{\epsilon_j+1},\dots,\sigma_k|\alpha\}}^{\{\varepsilon_1,\dots,\varepsilon_{n-k}\}}$ and $\sigma_{\epsilon_j}>\alpha$, then $T_\alpha u(a)=\dfrac{1}{v_1(a)\,\dots\,v_\alpha(a)}u^{(\alpha)}(a)=0$.

			\item Analogously, if either $u\in X_{\{\sigma_1,\dots,\sigma_k|\alpha\}}^{\{\varepsilon_1,\dots,\varepsilon_{\kappa_i-1},\varepsilon_{\kappa_i+1},\dots,\varepsilon_{n-k}\}}$ or  $\varepsilon_{\kappa_i}>\beta$ and $u\in X_{\{\sigma_1,\dots,\sigma_k\}}^{\{\varepsilon_1,\dots,\varepsilon_{\kappa_i-1},\varepsilon_{\kappa_i+1},\dots,\varepsilon_{n-k}|\beta\}}$, then $T_\beta u(b)=\dfrac{1}{v_1(b)\,\dots\,v_\alpha(b)}u^{(\beta)}(b)=0$.		
			\end{itemize}
		\end{proof}
		
		\begin{remark}\label{R::16}
			Realize that if we can prove that either $T_{\sigma_j}u(a)=u^{(\sigma_j)}(a)$ or $T_{\varepsilon_i}u(b)=u^{(\varepsilon_i)}(b)$, we do not need the assumption that $\sigma_j<\mu$ or $\varepsilon_i<\mu$  given by the choice of $\{\epsilon_1,\dots,\epsilon_\ell\}$ and $\{\kappa_1,\dots,\kappa_h\}$ on Definition \ref{Def::SE}.
			
			This is true, in particular, if we can choose on decomposition \eqref{Ec::Td1}-\eqref{Ec::Td2}, $v_1\equiv\cdots\equiv v_{\sigma_j}\equiv 1$ or $v_1\equiv\cdots\equiv v_{\varepsilon_i}\equiv 1$. We note that such a choice is valid for the operator $T_n^0[M]=u^{(n)}(t)+M\,u(t)$, where we can choose $v_1\equiv\cdots\equiv v_n\equiv 1$.
			
			The following results are also true under the hypothesis of this remark.
		\end{remark}
		\begin{lemma}{\color{white}.}
			\begin{itemize}
				\item Let $n-k$ be even, then the following assertions are satisfied:
				\begin{itemize}
					\item If $\sigma_{\epsilon_j}>\alpha$, then there exists $\lambda_{\sigma_{\epsilon_j}}^1>0$, the least positive eigenvalue of $\tilde T_n[0]$ in $ X_{\{\sigma_1,\dots,\sigma_{\epsilon_j-1},\sigma_{\epsilon_j+1},\dots,\sigma_k|\alpha\}}^{\{\varepsilon_1,\dots,\varepsilon_{n-k}\}}$.
					\item There exists $\lambda_{\sigma_{\epsilon_j}}^2<0$ the biggest negative eigenvalue of $\tilde T_n[0]$ in $X_{\{\sigma_1,\dots,\sigma_{\epsilon_j-1},\sigma_{\epsilon_j+1},\dots,\sigma_k\}}^{\{\varepsilon_1,\dots,\varepsilon_{n-k}|\beta\}}$.
					\item If $\varepsilon_{\kappa_i}>\beta$, then there exists $\lambda_{\varepsilon_{\kappa_i}}^1>0$, the least positive eigenvalue of $\tilde T_n[0]$ in $X_{\{\sigma_1,\dots,\sigma_k\}}^{\{\varepsilon_1,\dots,\varepsilon_{\kappa_i-1},\varepsilon_{\kappa_i+1},\dots,\varepsilon_{n-k}|\beta\}}$.
					\item There exists $\lambda_{\varepsilon_{\kappa_i}}^2<0$ the biggest negative eigenvalue of $\tilde T_n[0]$ in $X_{\{\sigma_1,\dots,\sigma_k|\alpha\}}^{\{\varepsilon_1,\dots,\varepsilon_{\kappa_i-1},\varepsilon_{\kappa_i+1},\dots,\varepsilon_{n-k}\}}$.
				\end{itemize}
					\item Let $n-k$ be odd, then the following assertions are satisfied:
					\begin{itemize}
						\item If $\sigma_{\epsilon_j}>\alpha$, then there exists $\lambda_{\sigma_{\epsilon_j}}^1<0$, the biggest negative eigenvalue of $\tilde T_n[0]$ in $ X_{\{\sigma_1,\dots,\sigma_{\epsilon_j-1},\sigma_{\epsilon_j+1},\dots,\sigma_k|\alpha\}}^{\{\varepsilon_1,\dots,\varepsilon_{n-k}\}}$.
						\item There exists $\lambda_{\sigma_{\epsilon_j}}^2>0$, the least positive eigenvalue of operator $\tilde T_n[0]$ in $X_{\{\sigma_1,\dots,\sigma_{\epsilon_j-1},\sigma_{\epsilon_j+1},\dots,\sigma_k\}}^{\{\varepsilon_1,\dots,\varepsilon_{n-k}|\beta\}}$.
						\item If $\varepsilon_{\kappa_i}>\beta$, then there exists $\lambda_{\varepsilon_{\kappa_i}}^1<0$, the biggest negative eigenvalue of $\tilde T_n[0]$ in $X_{\{\sigma_1,\dots,\sigma_k\}}^{\{\varepsilon_1,\dots,\varepsilon_{\kappa_i-1},\varepsilon_{\kappa_i+1},\dots,\varepsilon_{n-k}|\beta\}}$.
						\item There exists $\lambda_{\varepsilon_{\kappa_i}}^2>0$, the least positive eigenvalue of operator $\tilde T [0]$ in $X_{\{\sigma_1,\dots,\sigma_k|\alpha\}}^{\{\varepsilon_1,\dots,\varepsilon_{\kappa_i-1},\varepsilon_{\kappa_i+1},\dots,\varepsilon_{n-k}\}}$.
					\end{itemize}
			\end{itemize}
		\end{lemma}
		\begin{proof}
			Since $\{\sigma_1,\dots,\sigma_k\}-\{\varepsilon_1,\dots,\varepsilon_{n-k}\}$ satisfy property $(N_a)$, this property is satisfied in all the spaces involved in the result. 
			
			Moreover, from Lemma \ref{L::Td}, the property $(T_d)$ is also satisfied. Then, by applying Theorems \ref{L::5}, \ref{T::6} and \ref{T::7}, the result is proved.			
		\end{proof}
		
		Now, let us see two results which allow us to ensure that functions $x_M^{\sigma_{\epsilon_j}}$ and $z_M^{\varepsilon_{\kappa_i}}$  are of constant sign for suitable values of $M$.
		
		\begin{proposition}\label{P::nh1}
			Let $u\in C^n(I)$ be a solution of $\tilde{T}[M]\,u(t)=0$ for  $t\in (a,b)$, which satisfies the boundary conditions 
				\begin{equation}
				\label{Ec::cfxm1}\left\lbrace \begin{array}{c}	
				u^{(\sigma_1)}(a)=\dots=u^{(\sigma_{\epsilon_j-1})}(a)=u^{(\sigma_{\epsilon_j+1})}(a)=\cdots=u^{(\sigma_k)}(a)=0\,,\\\\ u^{(\varepsilon_1)}(b)=\cdots=u^{(\varepsilon_{n-k})}(b)=0\,.\end{array} \right. 
				\end{equation}
				
				Then, the function $u$ does not have any zero on $(a,b)$ provided that one of the following assertions are fulfilled:
				\begin{itemize}
					\item If $n-k$ is even, $k>1$, $\sigma_{\epsilon_j}>\alpha$ and $M\in\left[-\lambda_{\sigma_{\epsilon_j}}^1,-\lambda_{\sigma_{\epsilon_j}}^2\right]$, where
					\begin{itemize}
						\item[*] $\lambda_{\sigma_{\epsilon_j}}^1>0$ is the least positive eigenvalue of $\tilde T_n[0]$ in $ X_{\{\sigma_1,\dots,\sigma_{\epsilon_j-1},\sigma_{\epsilon_j+1},\dots,\sigma_k|\alpha\}}^{\{\varepsilon_1,\dots,\varepsilon_{n-k}\}}$,
						\item[*] $\lambda_{\sigma_{\epsilon_j}}^2<0$ is the biggest negative eigenvalue of $\tilde T_n[0]$ in $X_{\{\sigma_1,\dots,\sigma_{\epsilon_j-1},\sigma_{\epsilon_j+1},\dots,\sigma_k\}}^{\{\varepsilon_1,\dots,\varepsilon_{n-k}|\beta\}}$.
					\end{itemize}
					
					\item If $n-k$ is even, $k>1$, $\sigma_{\epsilon_j}<\alpha$ and $M\in\left[-\lambda_1,-\lambda_{\sigma_{\epsilon_j}}^2\right]$, where
						\begin{itemize}
							\item[*] $\lambda_1>0$ is the least positive eigenvalue of $\tilde T_n[0]$ in $ X_{\{\sigma_1,\dots,\sigma_k\}}^{\{\varepsilon_1,\dots,\varepsilon_{n-k}\}}$,
							\item[*] $\lambda_{\sigma_{\epsilon_j}}^2<0$ is the biggest negative eigenvalue of $\tilde T_n[0]$ in $X_{\{\sigma_1,\dots,\sigma_{\epsilon_j-1},\sigma_{\epsilon_j+1},\dots,\sigma_k\}}^{\{\varepsilon_1,\dots,\varepsilon_{n-k}|\beta\}}$.
						\end{itemize}
						\item If $k=1$, $n$ odd, $\sigma_{\epsilon_j}>\alpha$ and $M\in\left[-\lambda_{\sigma_{\epsilon_j}}^1,+\infty\right)$, where
						\begin{itemize}
							\item[*] $\lambda_{\sigma_{\epsilon_j}}^1>0$ is the least positive eigenvalue of $\tilde T_n[0]$ in $ X_{\{\sigma_1,\dots,\sigma_{\epsilon_j-1},\sigma_{\epsilon_j+1},\dots,\sigma_k|\alpha\}}^{\{\varepsilon_1,\dots,\varepsilon_{n-k}\}}$.
						\end{itemize}
					\item If $k=1$, $n$ odd, $\sigma_{\epsilon_j}<\alpha$ and $M\in\left[-\lambda_1,+\infty\right)$, where
					\begin{itemize}
						\item[*] $\lambda_1>0$ is the least positive eigenvalue of $\tilde T_n[0]$ in $ X_{\{\sigma_1,\dots,\sigma_k\}}^{\{\varepsilon_1,\dots,\varepsilon_{n-k}\}}$.
				
					\end{itemize}
						\item If $n-k$ is odd, $k>1$, $\sigma_{\epsilon_j}>\alpha$ and $M\in\left[-\lambda_{\sigma_{\epsilon_j}}^2,-\lambda_{\sigma_{\epsilon_j}}^1\right]$, where
						\begin{itemize}
							\item[*] $\lambda_{\sigma_{\epsilon_j}}^2>0$ is the least positive eigenvalue of $\tilde T_n[0]$ in $X_{\{\sigma_1,\dots,\sigma_{\epsilon_j-1},\sigma_{\epsilon_j+1},\dots,\sigma_k\}}^{\{\varepsilon_1,\dots,\varepsilon_{n-k}|\beta\}}$,
							\item[*] $\lambda_{\sigma_{\epsilon_j}}^1<0$ is the biggest negative eigenvalue of $\tilde T_n[0]$ in $ X_{\{\sigma_1,\dots,\sigma_{\epsilon_j-1},\sigma_{\epsilon_j+1},\dots,\sigma_k|\alpha\}}^{\{\varepsilon_1,\dots,\varepsilon_{n-k}\}}$.
						\end{itemize}
						
						\item If $n-k$ is odd, $k>1$, $\sigma_{\epsilon_j}<\alpha$ and $M\in\left[-\lambda_{\sigma_{\epsilon_j}}^2,-\lambda_1\right]$, where
						\begin{itemize}
							\item[*] $\lambda_{\sigma_{\epsilon_j}}^2>0$ is the least positive eigenvalue of $\tilde T_n[0]$ in $X_{\{\sigma_1,\dots,\sigma_{\epsilon_j-1},\sigma_{\epsilon_j+1},\dots,\sigma_k\}}^{\{\varepsilon_1,\dots,\varepsilon_{n-k}|\beta\}}$,
							\item[*] $\lambda_1<0$ is the biggest negative eigenvalue of $\tilde T_n[0]$ in $ X_{\{\sigma_1,\dots,\sigma_k\}}^{\{\varepsilon_1,\dots,\varepsilon_{n-k}\}}$.
						\end{itemize}
						\item If $k=1$, $n$ odd, $\sigma_{\epsilon_j}>\alpha$ and $M\in\left(-\infty,-\lambda_{\sigma_{\epsilon_j}}^1\right]$, where
						\begin{itemize}
							\item[*] $\lambda_{\sigma_{\epsilon_j}}^1<0$ is the biggest negative eigenvalue of $\tilde T_n[0]$ in $ X_{\{\sigma_1,\dots,\sigma_{\epsilon_j-1},\sigma_{\epsilon_j+1},\dots,\sigma_k|\alpha\}}^{\{\varepsilon_1,\dots,\varepsilon_{n-k}\}}$.
						\end{itemize}
						\item If $k=1$, $n$ odd, $\sigma_{\epsilon_j}<\alpha$ and $M\in\left(-\infty,-\lambda_1\right]$, where
						\begin{itemize}
							\item[*] $\lambda_1<0$ is the biggest negative eigenvalue of $\tilde T_n[0]$ in $ X_{\{\sigma_1,\dots,\sigma_k\}}^{\{\varepsilon_1,\dots,\varepsilon_{n-k}\}}$.
							
						\end{itemize}
				\end{itemize}
		\end{proposition}

		\begin{proof}
			Firstly, let us see what happens for $M=0$. As we have seen in the previous results, without taking into account the boundary conditions, if $u$ is a solution of $\tilde T_n[0]\,u(t)=0$ on $(a,b)$, then $u$ has at most $n-1$ zeros.
			
			However, from the boundary conditions \eqref{Ec::cfxm1}, we conclude that $\tilde T_\ell u(a)=0$ or $\tilde T_\ell u(b)=0$ at least $n-1$ times from $\ell=0$ to $n-1$. Thus, we loose the $n-1$ possible oscillations and $u$ does not have any zero on $(a,b)$.
			
			Now, let us consider $u_M\in C^n(I)$ a solution of $\tilde T_n[0]\,u_M(t)=0$ on $(a,b)$.
			
			Assume that $u_0>0$ on $(a,b)$  (if $u_0<0$ on $(a,b)$ the arguments are valid by multiplying by $-1$) and we move continuously on $M$ to obtain $u_M$.
			
			We will see that while $u_M\geq 0$, it cannot have any double zero, which implies that it is positive on $(a,b)$.
			
			It is known that $\tilde T [0]\,u_M(t)=-M\,u_M(t)$, on $(a,b)$, hence $\tilde T_{n-1} u_M$ is a monotone function on $I$, with at most one zero. Then, arguing as before, we conclude, without taking into account the boundary conditions, that $u_M$ can have at most $n$ zeros. But, if we consider the boundary conditions \eqref{Ec::cfxm1}, we loose  $n-1$ possible oscillation and $u_M$ is only allowed to have a simple zero on $(a,b)$, which is not possible if it is of constant sign. Hence, we can affirm that $u_M>0$ on $(a,b)$ until one of the following assertions is satisfied:
				
			\begin{itemize}
				\item  $\sigma_{\epsilon_j}>\alpha$ and $u_M^{(\alpha)}(a)=0$.
				\item  $\sigma_{\epsilon_j}<\alpha$ and $u_M^{(\sigma_{\epsilon_j})}(a)=0$.
					\item  $u_M^{(\beta)}(b)=0$.
			\end{itemize}
			
			Now, let us study separately the cases where $M>0$ or $M<0$ to see with which of the previous assertions  the sign change begins in each case.
			
			\vspace{0.5cm}

If $M\geq 0$, then $\tilde{T}[0]\,u_M(t)=-M\,u_M(t)\leq 0$ for $t\in (a,b)$. Thus, $\tilde T_n u_M(a)\leq 0$ and $\tilde T_n u_M(b)\leq 0$.

With maximal oscillation $\tilde T_{n-\ell}\,u_M(a)$ changes its sign each time that it is not null and $\tilde T_{n-\ell}\,u_M(b)$ changes its sign as many times as it vanishes.	

	\begin{itemize}
		\item If $\sigma_{\epsilon_j}>\alpha$, from $\ell =0$ to $n-\alpha$,  $\tilde T_{n-\ell}u_M(a)$ vanishes $k-1-\alpha$ times. 
		
		 If $\tilde T_{n-\ell}\,u_M(a)=0$ for $\ell<n-\alpha$ and $n-\ell\notin \{\sigma_1,\dots,\sigma_{\epsilon_j-1},\sigma_{\epsilon_j+1},\dots,\sigma_k|\alpha\}$, then $\tilde T_\alpha u_M(a)\neq 0$ and $u_M$ remains positive on $(a,b)$. So, we can assume that this situation cannot be fulfilled. 
		 
		 Hence, with maximal oscillation, we have:
		\[\left\lbrace \begin{array}{cc}
		\tilde T_\alpha \,u_M(a)\geq 0\,,&\text{ if $n-\alpha-(k-1-\alpha)=n-k+1$ is odd,}\\\\
			\tilde T_\alpha \,u_M(a)\leq 0\,,&\text{ if $n-k+1$ is even.}\end{array}\right. \]
		Since  $\sigma_{\epsilon_j}>\alpha$, from \eqref{Ec::Tl}, we have that
		\[\tilde T_\alpha u_M(a)=\dfrac{u^{(\alpha)}(a)}{v_1(a)\dots v_{\alpha}(a)}\,,\]
		so, with maximal oscillation:
				\[\left\lbrace \begin{array}{cc}
				u_M^{(\alpha)}(a)\geq 0\,,&\text{ if $n-k$ is even,}\\\\
				u_M^{(\alpha)}(a)\leq 0\,,&\text{ if $n-k$ is odd.}\end{array}\right. \]
		\item If $\sigma_{\epsilon_j}<\alpha$, from $\ell =0$ to $n-\sigma_{\epsilon_j}$,  $\tilde T_{n-\ell}u_M(a)$ vanishes  $k-1-\sigma_{\epsilon_j}$ times. Again, let us assume that $\tilde T_{n-\ell}u_M(a)\neq 0$ for $\ell<n-\sigma_{\epsilon_j}$ if $n-\ell \notin\{\sigma_1,\dots,\sigma_k\}$. Then, with maximal oscillation, we have:
			\[\left\lbrace \begin{array}{cc}
			\tilde T_{\sigma_{\epsilon_j}} \,u_M(a)\geq 0&\text{ if $n-\sigma_{\epsilon_j}-(k-1-\sigma_{\epsilon_j})=n-k+1$ is odd,}\\\\
			\tilde T_{\sigma_{\epsilon_j}} \,u_M(a)\leq 0&\text{ if $n-k+1$ is even.}\end{array}\right. \]
			
			Since  $\sigma_{\epsilon_j}<\alpha$, from \eqref{Ec::Tl}, we have that
			\[\tilde T_{\sigma_{\epsilon_j}}u_M(a)=\dfrac{u^{(\sigma_{\epsilon_j})}(a)}{v_1(a)\dots v_{\sigma_{\epsilon_j}}(a)}\,.\]
			
			In particular, if $\sigma_{\epsilon_j}<\mu$, then $v_1(t)\dots v_{\sigma_{\epsilon_j}}(t)=1$.
			
		Thus, with maximal oscillation:
			\[\left\lbrace \begin{array}{cc}
			u_M^{(\sigma_{\epsilon_j})}(a)\geq 0\,,&\text{ if $n-k$ is even,}\\\\
			u_M^{(\sigma_{\epsilon_j})}(a)\leq 0\,,&\text{ if $n-k$ is odd.}\end{array}\right. \]
		
		\item On another hand, from $\ell=0$ to $n-\beta$,  $\tilde T_{n-\ell}u_M(b)$ vanishes $n-k-\beta$ times. We can also assume that $\tilde T_{n-\ell} u_M(b)\neq 0$ if $n-\ell\notin\{\varepsilon_1,\dots,\varepsilon_{n-k}|\beta\}$. Then, with maximal oscillation:
		\[\left\lbrace \begin{array}{cc}
		\tilde T_{\beta} \,u_M(b)\geq 0\,,&\text{ if $n-k-\beta$ is odd,}\\\\
		\tilde T_{\beta} \,u_M(b)\leq 0\,,&\text{ if $n-k-\beta$ is even.}\end{array}\right. \]
		
		From \eqref{Ec::Tl}, we have that
		\[\tilde T_\beta u_M(b)=\dfrac{u^{(\beta)}(b)}{v_1(b)\dots v_{\beta}(b)}\,.\]
		
		Thus:
		\begin{itemize}
			\item if $n-k$ is even, to set maximal oscillation, we need
			\[\left\lbrace \begin{array}{cc}
			u_M^{(\beta)}(b)\leq 0\,,&\text{ if $\beta$ is even,}\\\\
			u_M^{(\beta)}(b)\geq 0\,,&\text{ if $\beta$ is odd.}\end{array}\right. \]
			\item if $n-k$ is odd, to ensure maximal oscillation is necessary:
			\[\left\lbrace \begin{array}{cc}
			u_M^{(\beta)}(b)\geq 0\,,&\text{ if $\beta$ is even,}\\\\
			u_M^{(\beta)}(b)\leq 0\,,&\text{ if $\beta$ is odd.}\end{array}\right. \]
			\end{itemize}
		
	\end{itemize}	
			
	Since, we are considering $u_M\geq 0$, it is known that
	\begin{equation}\label{Ec::uap}\left\lbrace \begin{array}{cc}
	u_M^{(\alpha)}(a)\geq 0\,,&\text{ if $\sigma_{\epsilon_j}>\alpha$,}\\\\
u_M^{(\sigma_{\epsilon_j})}(a)\geq 0\,,&\text{ if $\sigma_{\epsilon_j}<\alpha$,}\end{array}\right. \end{equation}
	and
	\begin{equation}\label{Ec::ubp}\left\lbrace \begin{array}{cc}
	u_M^{(\beta)}(b)\geq 0\,,&\text{ if $\beta$ is even,}\\\\
	u_M^{(\beta)}(b)\leq 0\,,&\text{ if $\beta$ is odd,}\end{array}\right. \end{equation}
		
			Taking into account that if $k=1$, then $u^{(\beta)}_M(b)\neq 0$ for all $M\in \mathbb{R}$, we obtain the following conclusions for $M\geq 0$:
			\begin{itemize}
				\item If $n-k$ is odd and $\sigma_{\epsilon_j}>\alpha$, then $u_M\geq 0$ until $u_M^{(\alpha)}(a)=0$; i.e., until an eigenvalue of $\tilde T_n[0]$ in $X_{\{\sigma_1,\dots,\sigma_{\epsilon_j-1},\sigma_{\epsilon_j+1},\dots,\sigma_k|\alpha\}}^{\{\varepsilon_1,\dots,\varepsilon_{n-k}\}}$ is found.
				\item If $n-k$ is odd and $\sigma_{\epsilon_j}<\alpha$, then $u_M\geq 0$ until $u_M^{(\sigma_{\epsilon_j})}(a)=0$; i.e., until an eigenvalue of $\tilde T_n[0]$ in $X_{\{\sigma_1,\dots,\sigma_k\}}^{\{\varepsilon_1,\dots,\varepsilon_{n-k}\}}$ is found.
				\item If $n-k$ is even and $k>1$, then $u_M\geq 0$ until $u_M^{(\beta)}(b)=0$; i.e., until an eigenvalue of $\tilde T_n[0]$ in $X_{\{\sigma_1,\dots,\sigma_{\epsilon_j-1},\sigma_{\epsilon_j+1},\dots,\sigma_k\}}^{\{\varepsilon_1,\dots,\varepsilon_{n-k}|\beta\}}$ is found.
				\item If $k=1$ and $n$ is odd, then $u_M\geq 0$ for all $M\geq 0$.
			\end{itemize}
			
			\vspace{0.5cm}
			
			Now, let us see what happens for $M\leq 0$. In this case, we have that $\tilde T_n[0] u_M(t)=-M\,u_M(t)\geq 0$ for $t\in(a,b)$. Then, $\tilde T_nu_M(a)\geq 0$ and $\tilde T_n u_M(b)\geq0$. Hence, we conclude that with maximal oscillation, the inequalities are reversed from te case $M\geq 0$. So, we obtain that:
			
			\begin{itemize}
				\item If $\sigma_{\epsilon_j}>\alpha$, with maximal oscillation
							\[\left\lbrace \begin{array}{cc}
							u_M^{(\alpha)}(a)\leq 0\,,&\text{ if $n-k$ is even,}\\\\
							u_M^{(\alpha)}(a)\geq 0\,,&\text{ if $n-k$ is odd.}\end{array}\right. \]
				\item If $\sigma_{\epsilon_j}<\alpha$, with maximal oscillation
							\[\left\lbrace \begin{array}{cc}
							u_M^{(\sigma_{\epsilon_j})}(a)\leq 0\,,&\text{ if $n-k$ is even,}\\\\
							u_M^{(\sigma_{\epsilon_j})}(a)\geq 0\,,&\text{ if $n-k$ is odd,}\end{array}\right. \]			
			\end{itemize}
			and,
			\begin{itemize}
				\item if $n-k$ is even, with maximal oscillation:
				\[\left\lbrace \begin{array}{cc}
				u_M^{(\beta)}(b)\geq 0\,,&\text{ if $\beta$ is even,}\\\\
				u_M^{(\beta)}(b)\leq 0\,,&\text{ if $\beta$ is odd.}\end{array}\right. \]
				\item If $n-k$ is odd, with maximal oscillation:
				\[\left\lbrace \begin{array}{cc}
				u_M^{(\beta)}(b)\leq 0\,,&\text{ if $\beta$ is even,}\\\\
				u_M^{(\beta)}(b)\geq 0\,,&\text{ if $\beta$ is odd.}\end{array}\right. \]
			\end{itemize}
			Then, taking into account that $u_M\geq 0$, \eqref{Ec::uap} and \eqref{Ec::ubp} are also satisfied.
			
				Hence, using that if $k=1$, then $u^{(\beta)}_M(b)\neq 0$ for all $M\in \mathbb{R}$, we obtain the following conclusions for $M\leq 0$:
				\begin{itemize}
					\item If $n-k$ is even and $\sigma_{\epsilon_j}>\alpha$, then $u_M\geq 0$ until $u_M^{(\alpha)}(a)=0$; i.e., until an eigenvalue of $\tilde T_n[0]$ in $X_{\{\sigma_1,\dots,\sigma_{\epsilon_j-1},\sigma_{\epsilon_j+1},\dots,\sigma_k|\alpha\}}^{\{\varepsilon_1,\dots,\varepsilon_{n-k}\}}$ is found.
					\item If $n-k$ is even and $\sigma_{\epsilon_j}<\alpha$, then $u_M\geq 0$ until $u_M^{(\sigma_{\epsilon_j})}(a)=0$; i.e., until an eigenvalue of $\tilde T_n[0]$ in $X_{\{\sigma_1,\dots,\sigma_k\}}^{\{\varepsilon_1,\dots,\varepsilon_{n-k}\}}$ is found.
					\item If $n-k$ is odd and $k>1$, then $u_M\geq 0$ until $u_M^{(\beta)}(b)=0$; i.e., until an eigenvalue of $\tilde T_n[0]$ in $X_{\{\sigma_1,\dots,\sigma_{\epsilon_j-1},\sigma_{\epsilon_j+1},\dots,\sigma_k\}}^{\{\varepsilon_1,\dots,\varepsilon_{n-k}|\beta\}}$ is found.
					\item If $k=1$ and $n$ is even, then $u_M\geq 0$ for all $M\leq 0$.
				\end{itemize}
				The result is proved.
		\end{proof}

			\begin{proposition}\label{P::nh2}
				Let $u\in C^n(I)$ be a solution of $\tilde{T}[M]\,u(t)=0$ for $t\in (a,b)$, which satisfies the boundary conditions 
				\begin{equation}
				\label{Ec::cfym1}\left\lbrace \begin{array}{c}
								u^{(\sigma_1)}(a)=\dots=u^{(\sigma_{k})}(a)=0\,,\\\\
				u^{(\varepsilon_1)}(b)=\cdots=u^{(\varepsilon_{\kappa_i-1})}(b)=u^{(\varepsilon_{\kappa_i+1})}(b)=\cdots=u^{(\varepsilon_{n-k})}(b)=0\,.\end{array} \right. 
				\end{equation}
				
				Then, $u$ does not have any zero on $(a,b)$ provided that one of the following assertions is fulfilled:
				\begin{itemize}
					\item If $n-k$ is even, $\varepsilon_{\kappa_i}>\beta$ and $M\in\left[-\lambda_{\varepsilon_{\kappa_i}}^1,-\lambda_{\varepsilon_{\kappa_i}}^2\right]$, where
					\begin{itemize}
						\item[*] $\lambda_{\varepsilon_{\kappa_i}}^1>0$ is the least positive eigenvalue of $\tilde T_n[0]$ in $X_{\{\sigma_1,\dots,\sigma_k\}}^{\{\varepsilon_1,\dots,\varepsilon_{\kappa_i-1},\varepsilon_{\kappa_i+1},\dots,\varepsilon_{n-k}|\beta\}}$,
						\item[*] $\lambda_{\varepsilon_{\kappa_i}}^2<0$ is the biggest negative eigenvalue of $\tilde T_n[0]$ in $X_{\{\sigma_1,\dots,\sigma_k|\alpha\}}^{\{\varepsilon_1,\dots,\varepsilon_{\kappa_i-1},\varepsilon_{\kappa_i+1},\dots,\varepsilon_{n-k}\}}$.
					\end{itemize}
					
					\item If $n-k$ is even, ${\varepsilon_{\kappa_i}}<\alpha$ and $M\in\left[-\lambda_1,-\lambda_{\varepsilon_{\kappa_i}}^2\right]$, where
					\begin{itemize}
						\item[*] $\lambda_1>0$ is the least positive eigenvalue of $\tilde T_n[0]$ in $ X_{\{\sigma_1,\dots,\sigma_k\}}^{\{\varepsilon_1,\dots,\varepsilon_{n-k}\}}$,
						\item[*] $\lambda_{\varepsilon_{\kappa_i}}^2<0$ is the biggest negative eigenvalue of $\tilde T_n[0]$ in $X_{\{\sigma_1,\dots,\sigma_k|\alpha\}}^{\{\varepsilon_1,\dots,\varepsilon_{\kappa_i-1},\varepsilon_{\kappa_i+1},\dots,\varepsilon_{n-k}\}}$.
					\end{itemize}

						\item If $n-k$ is odd, $k<n-1$, $\varepsilon_{\kappa_i}>\beta$ and $M\in\left[-\lambda_{\varepsilon_{\kappa_i}}^2,-\lambda_{\varepsilon_{\kappa_i}}^1\right]$, where
						\begin{itemize}
							\item[*] $\lambda_{\varepsilon_{\kappa_i}}^2>0$ is the least positive eigenvalue of $\tilde T_n[0]$ in $X_{\{\sigma_1,\dots,\sigma_k|\alpha\}}^{\{\varepsilon_1,\dots,\varepsilon_{\kappa_i-1},\varepsilon_{\kappa_i+1},\dots,\varepsilon_{n-k}\}}$,
							\item[*] $\lambda_{\varepsilon_{\kappa_i}}^1<0$ is the biggest negative eigenvalue of $\tilde T_n[0]$ in $X_{\{\sigma_1,\dots,\sigma_k\}}^{\{\varepsilon_1,\dots,\varepsilon_{\kappa_i-1},\varepsilon_{\kappa_i+1},\dots,\varepsilon_{n-k}|\beta\}}$.
						\end{itemize}
						
						\item If $n-k$ is odd, $k<n-1$, ${\varepsilon_{\kappa_i}}<\alpha$ and $M\in\left[-\lambda_{\varepsilon_{\kappa_i}}^2,-\lambda_1\right]$, where
						\begin{itemize}
							\item[*] $\lambda_{\varepsilon_{\kappa_i}}^2>0$ is the least positive eigenvalue of $\tilde T_n[0]$ in $X_{\{\sigma_1,\dots,\sigma_k|\alpha\}}^{\{\varepsilon_1,\dots,\varepsilon_{\kappa_i-1},\varepsilon_{\kappa_i+1},\dots,\varepsilon_{n-k}\}}$, 
							\item[*] $\lambda_1<0$ is the biggest negative eigenvalue of $\tilde T_n[0]$ in $ X_{\{\sigma_1,\dots,\sigma_k\}}^{\{\varepsilon_1,\dots,\varepsilon_{n-k}\}}$ .
						\end{itemize}
						
						\item If $k=n-1$, ${\varepsilon_{\kappa_i}}>\alpha$ and $M\in\left(-\infty,-\lambda_{\sigma_{\epsilon_j}}^1\right]$, where
						\begin{itemize}
							\item[*] $\lambda_{\varepsilon_{\kappa_i}}^1<0$ is the biggest negative eigenvalue of $\tilde T_n[0]$ in $X_{\{\sigma_1,\dots,\sigma_k\}}^{\{\varepsilon_1,\dots,\varepsilon_{\kappa_i-1},\varepsilon_{\kappa_i+1},\dots,\varepsilon_{n-k}|\beta\}}$.
						\end{itemize}
						\item If $k=n-1$,  ${\varepsilon_{\kappa_i}}<\alpha$ and $M\in\left(-\infty,-\lambda_1\right]$, where
						\begin{itemize}
							\item[*] $\lambda_1<0$ is the biggest negative eigenvalue of $\tilde T_n[0]$ in $ X_{\{\sigma_1,\dots,\sigma_k\}}^{\{\varepsilon_1,\dots,\varepsilon_{n-k}\}}$.
							
						\end{itemize}
				\end{itemize}
			\end{proposition}
					\begin{proof}
						The proof is analogous to the proof of Proposition \ref{P::nh1}.
					\end{proof}
					
					Now, we are in a position to prove a result which gives a relationship on the eigenvalues of the different spaces $X_{\{\sigma_1,\dots,\sigma_{\epsilon_j-1},\sigma_{\epsilon_j+1},\dots,\sigma_k|\alpha\}}^{\{\varepsilon_1,\dots,\varepsilon_{n-k}\}}$  with the closest to zero eigenvalue of $\tilde T_n[0]$ in $X_{\{\sigma_1,\dots,\sigma_k\}}^{\{\varepsilon_1,\dots,\varepsilon_{n-k}\}}.$ The result is the following.
					
					\begin{proposition}\label{P::ordaut1}
					Let ${{j_1}}\in \{{\epsilon_1},\dots,{\epsilon_\ell}\}$ be such that $\alpha<\sigma_{{j_1}}$, then the following assertions are true:
					\begin{itemize}
						\item If $n-k$ is even, then $0<\lambda_1<\lambda_{\sigma_{j_1}}^1$, where
						\begin{itemize}
							\item[*] $\lambda_{\sigma_{j_1}}^1>0$ is the least positive eigenvalue of $\tilde T_n[0]$ in $X_{\{\sigma_1,\dots,\sigma_{j_1-1},\sigma_{j_1+1},\dots,\sigma_k|\alpha\}}^{\{\varepsilon_1,\dots,\varepsilon_{n-k}\}}$.
								\item[*] $\lambda_1>0$ is the least positive eigenvalue of $\tilde T_n[0]$ in $X_{\{\sigma_1,\dots,\sigma_k\}}^{\{\varepsilon_1,\dots,\varepsilon_{n-k}\}}$.
							\end{itemize}
								\item If $n-k$ is odd, then $\lambda_{\sigma_{j_1}}^1<\lambda_1<0$, where
								\begin{itemize}
									\item[*] $\lambda_{\sigma_{j_1}}^1<0$ is the biggest negative eigenvalue of $\tilde T_n[0]$ in $X_{\{\sigma_1,\dots,\sigma_{j_1-1},\sigma_{j_1+1},\dots,\sigma_k|\alpha\}}^{\{\varepsilon_1,\dots,\varepsilon_{n-k}\}}$.
									\item[*] $\lambda_1<0$ is the biggest negative eigenvalue of $\tilde T_n[0]$ in $X_{\{\sigma_1,\dots,\sigma_k\}}^{\{\varepsilon_1,\dots,\varepsilon_{n-k}\}}$.
								\end{itemize}
					\end{itemize}	
					\end{proposition}
					
					\begin{proof}
						In order to prove this result, let us denote $v_M\in C^n(I)$ as a solution of $\tilde{T} [M]\,v_M(t)=0$ on$(a,b)$, coupled with the following boundary conditions:
						\begin{equation}\label{Ec::CFAB1}\left\lbrace \begin{array}{cc}
						v_M^{(\sigma_j)}(a)=0\,,&j=0,\dots,k,\quad j\neq j_1\,,\\\\
						v_M^{(\varepsilon_i)}(b)=0,&i=0,\dots,n-k\,.\end{array}\right. 
						\end{equation}
						
						Let us study $v_0$, with the arguments done before, we know that, without taking into account the boundary conditions, $v_0$ has at most $n-1$ zeros. However, from the boundary conditions \eqref{Ec::CFAB1}, we conclude that $n-1$ possible oscillations are lost. Hence, since $v_0$ is a nontrivial function, the boundary conditions for the maximal oscillation are verified.
						
						Let us choose $v_0\geq 0$ (if $v_0\leq 0$, the arguments are valid by multiplying by $-1$), then $v_0^{(\alpha)}(a)\geq 0$. From \eqref{Ec::Tl}  $T_{\alpha}v_0(a)$ also satisfies this inequality.
						
						Let us study the sign of $v_M^{(\sigma_{j_1})}(a)$. Realize,that, to achieve the maximal oscillation, $T_\ell v_M(a)$ must change its sign each time that it is non null.
						
						From $\ell=\alpha$ to $\sigma_{j_1}$,  $T_\ell v_0(a)$ vanishes  $j_1-1-\alpha$ times, then, with maximal oscillation:
						\[\left\lbrace \begin{array}{cc}
						T_{\sigma_{j_1}}v_0(a)> 0\,,&\text{if $\sigma_{j_1}-\alpha-(j_1-1-\alpha)=\sigma_{j_1}-j_1+1$ is even,}\\\\
						T_{\sigma_{j_1}}v_0(a)< 0\,,&\text{if $\sigma_{j_1}-j_1+1$ is odd.}\end{array}\right. \]
						
						From the choice of $j_1\in\{\epsilon_1,\dots,\epsilon_\ell\}$, we can affirm that 
						\begin{equation}\label{Ec::se1}\left\lbrace \begin{array}{cc}
						v_0^{(\sigma_{j_1})}(a)< 0\,,&\text{if $\sigma_{j_1}-j_1$ is even,}\\\\
						v_0^{(\sigma_{j_1})}(a)>0\,,&\text{if $\sigma_{j_1}-j_1$ is odd.}\end{array}\right. \end{equation}
						
						Now, let us move with continuity on $M$ until $-\lambda_{\sigma_{j_1}}$ and study the sign of $v_{-\lambda_{\sigma_{j_1}}^1}^{(\sigma_{j_1})}(a)$.
						
						 From Proposition \ref{P::nh1}, it is known that $v_{-\lambda_{\sigma_{j_1}}}>0$ on $(a,b)$. Moreover, $v_{-\lambda_{\sigma_{j_1}}}^{(\alpha)}(a)=0$. Thus, with the calculations done before, we conclude that the maximal oscillation is satisfied too.
						 
						  So, we can study in this case the sign of $v_{-\lambda_{\sigma_{j_1}}^1}^{(\sigma_{j_1})}(a)$.
						
						Let us consider $\alpha_1\in\{0,\dots,n-1\}$, previously introduced in the proof of Proposition 6.5. Since $v_{-\lambda_{\sigma_{j_1}}}\ge 0$ on $I$, we can affirm that $v_{-\lambda_{\sigma_{j_1}}}^{(\alpha_1)}(a)>0$.
						
						From $\ell=\alpha_1$ to $\sigma_{j_1}$, there are $j_1-\alpha_1$ zeros for $T_\ell v_{-\lambda_{\sigma_{j_1}}^1}(a)$, then, with maximal oscillation:
						\[\left\lbrace \begin{array}{cc}
						T_{\sigma_{j_1}}v_{-\lambda_{\sigma_{j_1}}^1}(a)> 0\,,&\text{if $\sigma_{j_1}-\alpha_1-(j_1-\alpha_1)=\sigma_{j_1}-j_1$ is even,}\\\\
						T_{\sigma_{j_1}}v_{-\lambda_{\sigma_{j_1}}^1}(a)< 0\,,&\text{if $\sigma_{j_1}-j_1$ is odd.}\end{array}\right. \]
						
						From the choice of $j_1\in\{\epsilon_1,\dots,\epsilon_\ell\}$, we can affirm that 
						\begin{equation}\label{Ec::se2}\left\lbrace \begin{array}{cc}
						v_{-\lambda_{\sigma_{j_1}}^1}^{(\sigma_{j_1})}(a)> 0\,,&\text{if $\sigma_{j_1}-j_1$ is even,}\\\\
						v_{-\lambda_{\sigma_{j_1}}^1}^{(\sigma_{j_1})}(a)< 0\,,&\text{if $\sigma_{j_1}-j_1$ is odd.}\end{array}\right. \end{equation}
						
						Hence, in this case, since we have been moving continuously on $M$, we can affirm that there exist $-\tilde \lambda_1$ between $0$ and ${-\lambda_{\sigma_{j_1}}^1}$ such that $v_{-\tilde \lambda_1}^{(\sigma_{j_1})}(a)=0$, i.e. we have proved the existence on an eigenvalue of $\tilde T_n[0]$ in $X_{\{\sigma_1,\dots,\sigma_k\}}^{\{\varepsilon_1,\dots,\varepsilon_{n-k}\}}$ between $0$ and ${-\lambda_{\sigma_{j_1}}^1}$, and the result is proved.
					\end{proof}
					
				In an analogous way, we can prove the following result for the eigenvalues of $\tilde{T}[0]$ in $X_{\{\sigma_1,\dots,\sigma_k\}}^{\{\varepsilon_1,\dots,\varepsilon_{\kappa_i-1},\varepsilon_{\kappa_i+1},\dots,\varepsilon_{n-k}| \beta\}}$, comparing them with the closest to zero eigenvalue in $X_{\{\sigma_1,\dots,\sigma_k\}}^{\{\varepsilon_1,\dots,\varepsilon_{n-k}\}}$.
						\begin{proposition}\label{P::ordaut2}
							Let ${{i_1}}\in \{{\kappa_1}\,,\dots,\,{\kappa_h}\}$ be such that $\varepsilon_{{i_1}}>\beta$, then the following assertions are true:
							\begin{itemize}
							
								\item If $n-k$ is even, then $0<\lambda_1<\lambda_{\varepsilon_{i_1}}^1$, where
								\begin{itemize}
									\item[*] $\lambda_{\varepsilon_{i_1}}^1>0$ is the least positive eigenvalue of $\tilde T_n[0]$ in $X_{\{\sigma_1,\dots,\sigma_k\}}^{\{\varepsilon_1,\dots,\varepsilon_{i_1-1},\varepsilon_{i_1+1},\dots,\varepsilon_{n-k}|\beta\}}$.
									\item[*] $\lambda_1>0$ is the least positive eigenvalue of $\tilde T_n[0]$ in $X_{\{\sigma_1,\dots,\sigma_k\}}^{\{\varepsilon_1,\dots,\varepsilon_{n-k}\}}$.
								\end{itemize}

								\item If $n-k$ is odd, then $0>\lambda_1>\lambda_{\varepsilon_{i_1}}^1$, where
								\begin{itemize}
									\item[*] $\lambda_{\varepsilon_{i_1}}^1<0$ is the biggest negative eigenvalue of $\tilde T_n[0]$ in $X_{\{\sigma_1,\dots,\sigma_k\}}^{\{\varepsilon_1,\dots,\varepsilon_{i_1-1},\varepsilon_{i_1+1},\dots,\varepsilon_{n-k}|\beta\}}$.
									\item[*] $\lambda_1<0$ is the biggest negative eigenvalue of $\tilde T_n[0]$ in $X_{\{\sigma_1,\dots,\sigma_k\}}^{\{\varepsilon_1,\dots,\varepsilon_{n-k}\}}$.
								\end{itemize}
							
							\end{itemize}	
						\end{proposition}
							\begin{proof}
								The proof is analogous to the one of Proposition \ref{P::ordaut1}.
							\end{proof}
						
						Now, let us establish a comparison between the eigenvalues in the different spaces $X_{\{\sigma_1,\dots,\sigma_{\epsilon_j-1},\sigma_{\epsilon_j+1},\dots,\sigma_k\}}^{\{\varepsilon_1,\dots,\varepsilon_{n-k}| \beta\}}$.
							\begin{proposition}\label{P::ordaut3}
								Let $\sigma_{{j_1}}\,,\ \sigma_{{j_2}}\in \{\sigma_{\epsilon_1},\dots,\sigma_{\epsilon_\ell}\}$ be such that $j_1<j_2$. Then the following assertions are fulfilled:
							\begin{itemize}
									\item If $n-k$ is even and $k>1$, then
									$0>\lambda_{\sigma_{j_1}}^2>\lambda_{\sigma_{j_2}}^2$, where
									\begin{itemize}
										\item[*] $\lambda_{\sigma_{j_1}}^2<0$ is the biggest negative eigenvalue of $\tilde T_n[0]$ in $X_{\{\sigma_1,\dots,\sigma_{j_1-1},\sigma_{j_1+1},\dots,\sigma_k\}}^{\{\varepsilon_1,\dots,\varepsilon_{n-k}|\beta\}}$.
										\item[*] $\lambda_{\sigma_{j_2}}^2<0$ is the biggest negative eigenvalue of $\tilde T_n[0]$ in $X_{\{\sigma_1,\dots,\sigma_{j_2-1},\sigma_{j_2+1},\dots,\sigma_k\}}^{\{\varepsilon_1,\dots,\varepsilon_{n-k}|\beta\}}$.
									\end{itemize}
									\item If $n-k$ is odd and $k>1$, then
									$0<\lambda_{\sigma_{j_1}}^2<\lambda_{\sigma_{j_2}}^2$, where
									\begin{itemize}
										\item[*] $\lambda_{\sigma_{j_1}}^2>0$ is the least positive eigenvalue of $\tilde T_n[0]$ in $X_{\{\sigma_1,\dots,\sigma_{j_1-1},\sigma_{j_1+1},\dots,\sigma_k\}}^{\{\varepsilon_1,\dots,\varepsilon_{n-k}|\beta\}}$.
										\item[*] $\lambda_{\sigma_{j_2}}^2>0$ is the least positive eigenvalue of $\tilde T_n[0]$ in $X_{\{\sigma_1,\dots,\sigma_{j_2-1},\sigma_{j_2+1},\dots,\sigma_k\}}^{\{\varepsilon_1,\dots,\varepsilon_{n-k}|\beta\}}$.
									\end{itemize}
								\end{itemize}	
							\end{proposition}
							
							\begin{proof}	
								In order to prove this result, let us denote ${v_1}_M\in C^n(I)$ as a solution of $\tilde{T} [M]\,{v_1}_M(t)=0$ on $(a,b)$, coupled with the following boundary conditions:
								\begin{equation}\label{Ec::CFAB3}\left\lbrace \begin{array}{cc}
								{v_1}_M^{(\sigma_j)}(a)=0\,,&j=0,\dots,k,\quad \text{ if }j\neq j_1\,,\ j_2\,,\\\\
								{v_1}_M^{(\varepsilon_i)}(b)=0\,,&\text{ if } i=0,\dots,n-k\,,\\\\
								{v_1}_M^{(\beta)}(b)=0\,.\end{array}\right. 
								\end{equation}
								
								Again, from the boundary conditions \eqref{Ec::CFAB1}, to ensure that is is a nontrivial solution, ${v_1}_0$ satisfies the conditions of maximal oscillation at $t=a$ and $t=b$.
								
								\vspace{0.5cm}
								
								First, let us see what happens if $\sigma_{j_1}>\alpha$.
								
								Let us choose ${v_1}_0\geq 0$ (if ${v_1}_0\leq 0$, then the arguments are valid by multiplying by $-1$), then ${v_1}_0^{(\alpha)}(a)\geq 0$. From \eqref{Ec::Tl} we have $T_{\alpha}{v_1}_0(a)\geq0$.
								
								To study the sign of $v_0^{(\sigma_{j_2})}(a)$, realize that, to achieve the maximal oscillation, $T_\ell v_M(a)$ changes its sign each time that it is non null.
								
								From $\ell=\alpha$ to $\sigma_{j_2}$, there are $j_2-2-\alpha$ zeros for $T_\ell {v_1}_0(a)$, then, with maximal oscillation:
								\[\left\lbrace \begin{array}{cc}
								T_{\sigma_{j_2}}{v_1}_0(a)> 0\,,&\text{if $\sigma_{j_2}-\alpha-(j_2-2-\alpha)=\sigma_{j_2}-j_2+2$ is even,}\\\\
								T_{\sigma_{j_2}}{v_1}_0(a)< 0\,,&\text{if $\sigma_{j_2}-j_2+2$ is odd.}\end{array}\right. \]
								
								From the choice of $j_2\in\{\epsilon_1,\dots,\epsilon_\ell\}$, we can affirm that 
								\begin{equation}\label{Ec::se3}\left\lbrace \begin{array}{cc}
								{v_1}_0^{(\sigma_{j_2})}(a)> 0\,,&\text{if $\sigma_{j_2}-j_2$ is even,}\\\\
								{v_1}_0^{(\sigma_{j_2})}(a)<0\,,&\text{if $\sigma_{j_2}-j_2$ is odd.}\end{array}\right. \end{equation}
								
								Now, let us move with continuity on $M$ until $-\lambda_{\sigma_{j_2}}^2$ and analyze the sign of ${v_1}_{-\lambda_{\sigma_{j_2}}^2}^{(\sigma_{j_2})}(a)$. Let us denote $\bar{\lambda_2}=-\lambda_{\sigma_{j_2}}^2$, from Proposition \ref{P::nh1}, it is known that ${v_1}_{\bar{\lambda_2}}>0$ on $(a,b)$. Moreover, ${v_1}_{\bar{\lambda_2}}^{(\sigma_{j_1})}(a)=0$. Thus, since another possible zero on the boundary will imply that ${v_1}_{\bar{\lambda_2}}\equiv 0$, we conclude that the maximal oscillation is satisfied too. 
								
								So, we can study, in this case, the sign of ${v_1}_{\bar{\lambda_2}}(a)$.
								
							 Since ${v_1}_{\bar{\lambda_2}}\ge 0$ on $I$, we can affirm that, as for $M=0$, ${v_1}_{\bar{\lambda_2}}^{(\alpha)}(a)>0$.
								
								From $\ell=\alpha$ to $\sigma_{j_2}$, there are $j_2-1-\alpha$ zeros for $T_\ell {v_1}_{\bar{\lambda_2}}(a)$, then, with maximal oscillation:
								\[\left\lbrace \begin{array}{cc}
								T_{\sigma_{j_2}}{v_1}_{\bar{\lambda_2}}(a)> 0\,,&\text{if $\sigma_{j_2}-\alpha-(j_2-1-\alpha)=\sigma_{j_2}-j_2+1$ is even,}\\\\
								T_{\sigma_{j_2}}{v_1}_{\bar{\lambda_2}}(a)< 0\,,&\text{if $\sigma_{j_2}-j_2+1$ is odd.}\end{array}\right. \]
								
								From the choice of $j_2\in\{\epsilon_1,\dots,\epsilon_\ell\}$, we can affirm that 
								\begin{equation}\label{Ec::se4}\left\lbrace \begin{array}{cc}
								{v_1}_{-\lambda_{\sigma_{j_2}}^2}^{(\sigma_{j_2})}(a)< 0\,,&\text{if $\sigma_{j_2}-j_2$ is even,}\\\\
								{v_1}_{-\lambda_{\sigma_{j_2}}^2}^{(\sigma_{j_2})}(a)> 0\,,&\text{if $\sigma_{j_2}-j_2$ is odd.}\end{array}\right. \end{equation}
								
								\vspace{0.5cm}
								
								Now, let us see what happens if $\sigma_{j_1}<\alpha<\sigma_{j_2}$. In this case,  $\sigma_{j_1}=j_1-1$. 
								
								For $M=0$, since ${v_1}_0\geq 0$, we have that ${v_1}_0^{(\sigma_{j_1})}(a)\ge0$. From \eqref{Ec::Tl} we have that $T_{\sigma_{j_1}}{v_1}_0(a)\geq 0$.
								
								Let us study the sign of ${v_1}_0^{(\sigma_{j_2})}(a)$ in this case.
								
								From $\ell=\sigma_{j_1}$ to $\sigma_{j_2}$, there are $j_2-2-(j_1-1)=j_2-j_1-1$ zeros of $T_\ell {v_1}_0(a)$. Then, with maximal oscillation:
								\[\left\lbrace \begin{array}{cc}
								T_{\sigma_{j_2}}{v_1}_0(a)> 0\,,&\text{if $\sigma_{j_2}-j_1-1-(j_2-j_1-1)=\sigma_{j_2}-j_2$ is even,}\\\\
								T_{\sigma_{j_2}}{v_1}_0(a)< 0\,,&\text{if $\sigma_{j_2}-j_2$ is odd.}\end{array}\right. \]
								
								From the choice of $j_2\in\{\epsilon_1,\dots,\epsilon_\ell\}$, we can affirm that \eqref{Ec::se3} holds.
								
								Now, we study the sign of ${v_1}_{\bar{\lambda_2}}(a)$ if the conditions to allow the maximal oscillation hold.
								
								Since ${v_1}_{-\lambda_{\sigma_{j_2}}^2}\ge 0$ on $I$, we can affirm that ${v_1}_{-\lambda_{\sigma_{j_1}}}^{(\alpha)}(a)>0$.
								
								From $\ell=\alpha$ to $\sigma_{j_2}$, there are $j_2-1-\alpha$ zeros for $T_\ell {v_1}_{-\lambda_{\sigma_{j_2}}^2}(a)$, then, with maximal oscillation and repeating the previous arguments, we obtain that \eqref{Ec::se4} is satisfied.

								\vspace{0.5cm}
								
								Finally, let us study the case where $\sigma_{j_2}<\alpha$. In this situation,  $\sigma_{j_1}=j_1-1$ and $\sigma_{j_2}=j_2-1$.
								
								For $M=0$, since ${v_1}_0\geq 0$, we have that ${v_1}_0^{(\sigma_{j_1})}(a)\ge0$ and from \eqref{Ec::Tl}  $T_{\alpha}{v_1}_0(a)\geq 0$.
								
								Let us study the sign of ${v_1}_0^{(\sigma_{j_2})}(a)$ in this situation. Since $\alpha>\sigma_{j_2}$, for all $\ell =\sigma_{j_1}\,,\dots,\,,\sigma_{j_2}$, we have that $\tilde T_\ell {v_1}_0(a)=0$. So, to allow the maximal oscillation, it must be satisfied that $T_{\sigma_{j_2}}{v_1}_0(a)<0$. And this inequality also holds for ${v_1}_0^{(\sigma_{j_2})}(a)$.
								
								In this case, for $M=-\lambda_{\sigma_{j_2}}^2$, since ${v_1}_{-\lambda_{\sigma_{j_2}^2}}>0$ on $(a,b)$ and ${v_1}_{-\lambda_{\sigma_{j_2}^2}}^{(\sigma_{j_1})}(a)=0$, we have that ${v_1}_{-\lambda_{\sigma_{j_2}}^2}^{(\sigma_{j_2})}(a)>0$.
								\vspace{0.5cm}
								
								Hence, in all the cases, since we have been moving continuously on $M$, we can affirm that there exists $-\tilde \lambda_1$ lying between $0$ and ${-\lambda_{\sigma_{j_2}}^2}$, such that ${v_1}_{-\tilde \lambda_1}^{(\sigma_{j_2})}(a)=0$. As consequence, we have proved the existence on an eigenvalue of $\tilde T_n[0]$ in $X_{\{\sigma_1,\dots,\sigma_{j_1-1},\sigma_{j_1+1},\dots,\sigma_k\}}^{\{\varepsilon_1,\dots,\varepsilon_{n-k}\}}$ between $0$ and ${-\lambda_{\sigma_{j_2}}^2}$, and the result is proved.
							\end{proof}
							
					Before introducing the final result which characterizes the strongly inverse positive (negative) character in the different spaces $X_{\{\sigma_1,\dots,\sigma_k\}_{\{\sigma_{\epsilon_1},\dots,\sigma_{\epsilon_\ell}\}}}^{\{\varepsilon_1,\dots,\varepsilon_{n-k}\}_{\{\varepsilon_{\kappa_1},\dots,\varepsilon_{\kappa_h}\}}}$, we show a result which gives an order on the eigenvalues associated to different spaces $X_{\{\sigma_1,\dots,\sigma_k|\alpha\}}^{\{\varepsilon_1,\dots,\varepsilon_{\kappa_i-1},\varepsilon_{\kappa_i+1},\dots,\varepsilon_{n-k}\}}$.
					
					\begin{proposition}\label{P::ordaut4}
						Let ${{i_1}}\,,\ {{i_2}}\in \{{\kappa_1}\,,\dots,\,{\kappa_h}\}$ be such that $i_1<i_2$, then the following assertions are true:
						\begin{itemize}

							\item If $n-k$ is even, then
							$0>\lambda_{\varepsilon_{i_1}}^2>\lambda_{\varepsilon_{i_2}}^2$, where
							\begin{itemize}
								\item[*] $\lambda_{\varepsilon_{i_1}}^2<0$ is the biggest negative eigenvalue of $\tilde T_n[0]$ in $X_{\{\sigma_1,\dots,\sigma_k|\alpha\}}^{\{\varepsilon_1,\dots,\varepsilon_{i_1-1},\varepsilon_{i_1+1},\dots,\varepsilon_{n-k}\}}$.
								\item[*] $\lambda_{\varepsilon_{i_2}}^2<0$ is the biggest negative eigenvalue of $\tilde T_n[0]$ in $X_{\{\sigma_1,\dots,\sigma_k|\alpha\}}^{\{\varepsilon_1,\dots,\varepsilon_{i_2-1},\varepsilon_{i_2+1},\dots,\varepsilon_{n-k}\}}$.
							\end{itemize}
							\item If $n-k$ is odd and $k<n-1$, then
							$0<\lambda_{\varepsilon_{i_1}}^2<\lambda_{\varepsilon_{i_2}}^2$, where
							\begin{itemize}
								\item[*] $\lambda_{\varepsilon_{i_1}}^2>0$ is the least positive eigenvalue of $\tilde T_n[0]$ in $X_{\{\sigma_1,\dots,\sigma_k|\alpha\}}^{\{\varepsilon_1,\dots,\varepsilon_{i_1-1},\varepsilon_{i_1+1},\dots,\varepsilon_{n-k}\}}$.
								\item[*] $\lambda_{\varepsilon_{i_2}}^2>0$ is the least positive eigenvalue of $\tilde T_n[0]$ in $X_{\{\sigma_1,\dots,\sigma_k|\alpha\}}^{\{\varepsilon_1,\dots,\varepsilon_{i_2-1},\varepsilon_{i_2+1},\dots,\varepsilon_{n-k}\}}$.
							\end{itemize}
						\end{itemize}	
					\end{proposition}
					\begin{proof}
						The proof follows the same structure and arguments as Proposition \ref{P::ordaut3}.
					\end{proof}
					
					Once we have obtained the previous results, which allow us to characterize the constant sign of the functions $x_M^{\sigma_{\epsilon_j}}$ and $z_M^{\varepsilon_{\kappa_i}}$ for $j=1,\dots,\ell$ and $i=0,\dots,h$, respectively, we can obtain a characterization of the strongly inverse positive (negative) character of operator $\tilde T_n[M]$ in the spaces $X_{\{\sigma_1,\dots,\sigma_k\}_{\{\sigma_{\epsilon_1},\dots,\sigma_{\epsilon_\ell}\}}}^{\{\varepsilon_1,\dots,\varepsilon_{n-k}\}_{\{\varepsilon_{\kappa_1},\dots,\varepsilon_{\kappa_h}\}}}$ as follows.
					
					\begin{theorem}\label{T::IPNH2}
						If $n-k$ is even, then the operator $\tilde{T}[M]$ is strongly inverse positive in $X_{\{\sigma_1,\dots,\sigma_k\}_{\{\sigma_{\epsilon_1},\dots,\sigma_{\epsilon_\ell}\}}}^{\{\varepsilon_1,\dots,\varepsilon_{n-k}\}_{\{\varepsilon_{\kappa_1},\dots,\varepsilon_{\kappa_h}\}}}$ if, and only if, one of the following assertions is satisfied:
						\begin{itemize}
							\item If  $k>1$ and $M\in(-\lambda_1,-\lambda_2]$, where:
							\begin{itemize}
								\item[*] $\lambda_1>0$ is the least positive eigenvalue of $\tilde T_n[0]$ in $X_{\{\sigma_1,\dots,\sigma_k\}}^{\{\varepsilon_1,\dots,\varepsilon_{n-k}\}}$.
								\item[*] $\lambda_2<0$ is the maximum between,
								\begin{itemize}
								\item[·] $\lambda_{\sigma_{\epsilon_1}}^2<0$, the biggest negative eigenvalue of $\tilde T_n[0]$ in $X_{\{\sigma_1,\dots,\sigma_{\epsilon_1-1},\sigma_{\epsilon_2-1},\dots,\sigma_k\}}^{\{\varepsilon_1,\dots,\varepsilon_{n-k}|\beta\}}$.
							 \item[·] $\lambda_{\varepsilon_{\kappa_1}}^2<0$, the biggest negative eigenvalue of $\tilde T_n[0]$ in $X_{\{\sigma_1,\dots,\sigma_k|\alpha\}}^{\{\varepsilon_1,\dots,\varepsilon_{\kappa_1-1},\varepsilon_{\kappa_1+1},\dots,\varepsilon_{n-k}\}}$.
							\end{itemize}
							\end{itemize}
								\item If  $k=1$ and $M\in(-\lambda_1,-\lambda_2]$, where:
								\begin{itemize}
									\item[*] $\lambda_1>0$ is the least positive eigenvalue of $\tilde T_n[0]$ in $X_{\{\sigma_1\}}^{\{\varepsilon_1,\dots,\varepsilon_{n-1}\}}$.
								
									\item[*] $\lambda_2=\lambda_{\varepsilon_{\kappa_1}}^2<0$, the biggest negative eigenvalue of $\tilde T_n[0]$ in $X_{\{\sigma_1|\alpha\}}^{\{\varepsilon_1,\dots,\varepsilon_{\kappa_1-1},\varepsilon_{\kappa_1+1},\dots,\varepsilon_{n-1}\}}$.
								\end{itemize}	
						
						\end{itemize}
						
							If $n-k$ is odd, then the operator $\tilde{T}[M]$ is strongly inverse negative in $X_{\{\sigma_1,\dots,\sigma_k\}_{\{\sigma_{\epsilon_1},\dots,\sigma_{\epsilon_\ell}\}}}^{\{\varepsilon_1,\dots,\varepsilon_{n-k}\}_{\{\varepsilon_{\kappa_1},\dots,\varepsilon_{\kappa_h}\}}}$ if, and only if, one of the following assertions is satisfied:
							\begin{itemize}
								\item If  $1<k<n-1$ and $M\in[-\lambda_2,-\lambda_1)$, where:
								\begin{itemize}
									\item[*] $\lambda_1<0$ is the biggest negative eigenvalue of $\tilde T_n[0]$ in $X_{\{\sigma_1,\dots,\sigma_k\}}^{\{\varepsilon_1,\dots,\varepsilon_{n-k}\}}$.
									\item[*] $\lambda_2>0$ is the minimum between,
									\begin{itemize}
										\item[·] $\lambda_{\sigma_{\epsilon_1}}^2>0$, the least positive eigenvalue of $\tilde T_n[0]$ in $X_{\{\sigma_1,\dots,\sigma_{\epsilon_1-1},\sigma_{\epsilon_2-1},\dots,\sigma_k\}}^{\{\varepsilon_1,\dots,\varepsilon_{n-k}|\beta\}}$.
										\item[·] $\lambda_{\varepsilon_{\kappa_1}}^2>0$, the least positive eigenvalue of $\tilde T_n[0]$ in $X_{\{\sigma_1,\dots,\sigma_k|\alpha\}}^{\{\varepsilon_1,\dots,\varepsilon_{\kappa_1-1},\varepsilon_{\kappa_1+1},\dots,\varepsilon_{n-k}\}}$.
									\end{itemize}
								\end{itemize}
								\item If  $k=1<n-1$ and $M\in[-\lambda_2,-\lambda_1)$, where:
								\begin{itemize}
									\item[*] $\lambda_1<0$ is the biggest negative eigenvalue of $\tilde T_n[0]$ in $X_{\{\sigma_1\}}^{\{\varepsilon_1,\dots,\varepsilon_{n-1}\}}$.
									
									\item[*] $\lambda_2=\lambda_{\varepsilon_{\kappa_1}}^2>0$ is the least positive eigenvalue of $\tilde T_n[0]$ in $X_{\{\sigma_1|\alpha\}}^{\{\varepsilon_1,\dots,\varepsilon_{\kappa_1-1},\varepsilon_{\kappa_1+1},\dots,\varepsilon_{n-1}\}}$.
								\end{itemize}	
									\item If  $1<k=n-1$ and $M\in[-\lambda_2,-\lambda_1)$, where:
									\begin{itemize}
										\item[*] $\lambda_1<0$ is the biggest negative eigenvalue of $\tilde T_n[0]$ in $X_{\{\sigma_1,\dots,\sigma_{n-1}\}}^{\{\varepsilon_1\}}$.
										\item[*] $\lambda_2=\lambda_{\sigma_{\epsilon_1}}^2>0$, the least positive eigenvalue of $\tilde T_n[0]$ in $X_{\{\sigma_1,\dots,\sigma_{\epsilon_1-1},\sigma_{\epsilon_2-1},\dots,\sigma_{n-1}\}}^{\{\varepsilon_1|\beta\}}$.
										
									\end{itemize}
										\item If  $n=2$ and $M\in(-\infty,-\lambda_1)$, where:
										\begin{itemize}
											\item[*] $\lambda_1<0$ is the biggest negative eigenvalue of $\tilde T_n[0]$ in $X_{\{\sigma_1\}}^{\{\varepsilon_1\}}$.
										\end{itemize}	
							\end{itemize}
					\end{theorem}
					\begin{proof}
						From Lemma \ref{L::14-1}, we only have to study the sign of $g_M(t,s)$, $x_M^{\sigma_{\epsilon_j}}$ for $j=0,\dots,\ell$ and $z_M^{\varepsilon_{\kappa_i}}$ for $i=0,\dots,h$.
						
						First, let us see that if $M$ belongs to the given intervals, then the operator is strongly inverse positive or negative in each case. And, finally, we will see that this interval cannot be increased.
						
						Taking into account Theorem \ref{T::IPN}, $(-1)^{n-k}g_M(t,s)>0$ on the given intervals. Moreover, if either $n-k$ is even and $M<0$ or $n-k$ is odd and $M>0$, the intervals cannot be increased.
						
						\vspace{0.5cm}
						
						Now, let us study the sign of $x_0^{\sigma_{\epsilon_j}}$ and $z_0^{\sigma_{\kappa_i}}$.
						
						It is known that $x_M^{\sigma_{\epsilon_j}}$ satisfies the boundary conditions \eqref{Ec::CFAB1} introduced in the proof of Proposition \ref{P::ordaut1}. Then, for $M=0$, the maximal oscillation is satisfied. So, we can study the sign of ${x_0^{\sigma_{\epsilon_j}}}^{(\sigma_{\epsilon_j})}$ taking into account that ${x_0^{\sigma_{\epsilon_j}}}^{(\sigma_{\epsilon_j})}(a)=1$.
						 
						If $\sigma_{\epsilon_j}<\alpha$, then $x_0^{\sigma_{\epsilon_j}}>0$.
						
						If $\sigma_{\epsilon_j}>\alpha$, from $\ell = \alpha$ to $\sigma_{\epsilon_j}$, there are $\epsilon_j-1-\alpha$ zeros for $T_{\ell}x_0^{\sigma_{\epsilon_j}}(a)$.
						 
						   Realize that, from the choice of $\epsilon_j$, we have that $T_{\sigma_{\epsilon_j}}x_0^{\sigma_{\epsilon_j}}(a)>0$. So, to have maximal oscillation, we need
						   \[\left\lbrace \begin{array}{cc}
						   T_{\alpha}x_0^{\sigma_{\epsilon_j}}(a)>0\,,& \text{ if $\sigma_{\epsilon_j}-\alpha-(\epsilon_j-1-\alpha)=\sigma_{\epsilon_j}-\epsilon_j+1$ is even,}\\\\
						   	 T_{\alpha}x_0^{\sigma_{\epsilon_j}}(a)<0\,,& \text{ if $\sigma_{\epsilon_j}-\epsilon_j+1$ is odd.}\end{array}\right. \]
						   	 
						   	 These inequalities are also satisfied by ${x_0^{\sigma_{\epsilon_j}}}^{(\alpha)}(a)$, thus
						    \begin{equation}\label{Ec::x0sign}\left\lbrace \begin{array}{cc}
						    {x_0^{\sigma_{\epsilon_j}}}>0\ \text{on $I$,}& \text{ if $\sigma_{\epsilon_j}-\epsilon_j+1$ is even,}\\\\
						    {x_0^{\sigma_{\epsilon_j}}}<0\ \text{on $I$,}& \text{ if $\sigma_{\epsilon_j}-\epsilon_j+1$ is odd.}\end{array}\right. \end{equation}
						    
						    Note that if $\sigma_{\epsilon_j}<\alpha$, then $\sigma_{\epsilon_j}=\epsilon_j-1$. Hence, $\sigma_{\epsilon_j}-\epsilon_j+1=2$ is an even number. Thus, equation \eqref{Ec::x0sign} is satisfied for all $\sigma_{\epsilon_j}$, with $j=1,\dots,\ell$.
						    
						    Moreover, from Propositions \ref{P::nh1}, \ref{P::ordaut1} and \ref{P::ordaut3}, inequalities \eqref{Ec::x0sign} are satisfied on the whole intervals given in the result. Thus, for those $M$, we have						    
						     \begin{equation}\left\lbrace \begin{array}{cc}
						     (-1)^{n-\sigma_{\epsilon_j}-(k-j)+1}{x_M^{\sigma_{\epsilon_j}}}>0\ \text{on $I$,}& \text{ if $n-k$ is even,}\\\\
						     (-1)^{n-\sigma_{\epsilon_j}-(k-j)+1}{x_M^{\sigma_{\epsilon_j}}}<0\ \text{on $I$,}& \text{ if $n-k$ is odd.}\end{array}\right. \end{equation}
						 
						 \vspace{0.3cm}
						 
						 In an analogous way, we can study $z_M^{\varepsilon_{\kappa_i}}$ to conclude that for all $M$ on the intervals given on the result, it is satisfied:
						  \begin{equation}\left\lbrace \begin{array}{cc}
						 (-1)^{n-k-\kappa_i+1}{z_0^{\varepsilon_{\kappa_i}}}>0\,,& \text{ if $n-k$ is even,}\\\\
						 (-1)^{n-k-\kappa_i+1}{z_0^{\varepsilon_{\kappa_i}}}<0\,,& \text{ if $n-k$ is odd.}\end{array}\right. \end{equation}

						 So, we have proved that if $M$ belongs to those intervals, operator $\tilde T_n[M]$ is strongly inverse negative (positive). Moreover, we have also seen that if either $n-k$ is even and $M<0$ or $n-k$ is odd and $M>0$ the intervals cannot be increased, since $g_M$ is not of constant sign. So, we only need to prove that if $n-k$ is even and $M>0$ or $n-k$ is odd and $M<0$ the intervals cannot be increased too.
						 
						 \vspace{0.5cm}
						 
						 To this end, we study the functions $x_M^{\sigma_{\epsilon_1}}$ and $z_M^{\varepsilon_{\kappa_1}}$. In particular, we will verify that if either $k\neq 1$ or $k\neq n-1$, one of them must necessarily change its sign for $M>-\lambda_2$ if $n-k$ is even or for $M<-\lambda_2$ if $n-k$ is odd.
						 
						 If $\sigma_{\epsilon_1}=\sigma_k$ and $\varepsilon_{\kappa_1}=\varepsilon_{n-k}$ the result follows from Theorem \ref{T::IPNH}. Otherwise, either $\lambda_2=\lambda_{\sigma_{\epsilon_1}}$ or $\lambda_2=\lambda_{\varepsilon_{\kappa_1}}$.
						 
						 First, let us assume that  $n-k$ is even.  Suppose that there exists $M^*>-\lambda_2$ such that $\tilde T_n[M]$ is  inverse positive in $X_{\{\sigma_1,\dots,\sigma_k\}_{\{\sigma_{\epsilon_1},\dots,\sigma_{\epsilon_\ell}\}}}^{\{\varepsilon_1,\dots,\varepsilon_{n-k}\}_{\{\varepsilon_{\kappa_1},\dots,\varepsilon_{\kappa_h}\}}}$. We will arrive to a contradiction.
						 
						   If $\lambda_2=\lambda_{\sigma_{\epsilon_1}}$, let us consider the function $x_M^1(t)=(-1)^{n-\sigma_{\epsilon_j}-(k-j)+1}{x_M^{\sigma_{\epsilon_1}}}(t)$. 
						 
						 Trivially,  $x_M^1\in X_{\{\sigma_1,\dots,\sigma_k\}_{\{\sigma_{\epsilon_1},\dots,\sigma_{\epsilon_\ell}\}}}^{\{\varepsilon_1,\dots,\varepsilon_{n-k}\}_{\{\varepsilon_{\kappa_1},\dots,\varepsilon_{\kappa_h}\}}}$ and $\tilde{T} [M^*]x_{M^*}^1(t)=0$. Then, we have that $x_{M^*}^1\geq 0$ on $I$. 
						 
						 Let us see that necessarily $x_0^1\geq x_{-\lambda_2}^1\geq x_{M^*}^1$ on $I$.
						 
						 Indeed, let us construct the following sequence:						 
						 \[\alpha_0=x_0^1\,,\quad \tilde T_n[M^*]\,\alpha_{n+1}=(M^*+\lambda_2)\,\alpha_n\,,\ n\geq 0\,,\]
						 where $\alpha_n^{(\sigma_j)}(a)=0$, if $j\neq \epsilon_1$ for $j=1,\dots,k$, $\alpha_n^{(\sigma_{\epsilon_1})}(a)=(-1)^{n-\sigma_{\epsilon_1}-(k-\epsilon_1)+1}$ and $\alpha_n^{(\varepsilon_i)}(b)=0$ for $i=1,\dots,\,n-k$. In particular, $\alpha_n\in X_{\{\sigma_1,\dots,\sigma_k\}_{\{\sigma_{\epsilon_1},\dots,\sigma_{\epsilon_\ell}\}}}^{\{\varepsilon_1,\dots,\varepsilon_{n-k}\}_{\{\varepsilon_{\kappa_1},\dots,\varepsilon_{\kappa_h}\}}}$ for $n=0,1,\dots$
						 
						 Let us see that this sequence is non-increasing and bounded from below by zero clearly. \[\tilde T [M^*]\,\alpha_1=(M^*+\lambda_2)\,x_0^1\geq 0\,.\]
						 
						 Since  $\alpha_1\in X_{\{\sigma_1,\dots,\sigma_k\}_{\{\sigma_{\epsilon_1},\dots,\sigma_{\epsilon_\ell}\}}}^{\{\varepsilon_1,\dots,\varepsilon_{n-k}\}_{\{\varepsilon_{\kappa_1},\dots,\varepsilon_{\kappa_h}\}}}$ and we are working under the assumption that $\tilde T [M^*]$ is inverse positive in such set, we have that $\alpha_1\geq 0$.
						 
						 Now, $\tilde T_n[M^*](\alpha_0-\alpha_1)=-\lambda_2\,x_0^1\geq 0$. In this case $\dfrac{d^{\sigma_j}}{dt^{\sigma_j}}(\alpha_0-\alpha_1)_{\mid t=a}=0$  for $j=1,\dots,k$ and $\dfrac{d^{\varepsilon_i}}{dt^{\varepsilon_i}}(\alpha_0-\alpha_1)_{\mid t=b}=0$ for $i=1,\dots,n-k$, then $\alpha_0-\alpha_1\in X_{\{\sigma_1,\dots,\sigma_k\}_{\{\sigma_{\epsilon_1},\dots,\sigma_{\epsilon_\ell}\}}}^{\{\varepsilon_1,\dots,\varepsilon_{n-k}\}_{\{\varepsilon_{\kappa_1},\dots,\varepsilon_{\kappa_h}\}}}$. So, $\alpha_0\geq \alpha_1$.
						 
						 Proceeding analogously for $n\geq 1$, we obtain that $\{\alpha_n\}$ is a non-increasing and nonnegative sequence.
						 
						 \vspace{0.5cm}
						 
						 Now, let us consider the following sequence:						 
						 \[\beta_0=x_{M^*}^1\,,\quad \tilde T_n[M^*]\,\beta_{n+1}=(M^*+\lambda_2)\,\beta_n\,,\ n\geq 0\,,\]
						 where $\beta_n^{(\sigma_j)}(a)=0$, if $j\neq \epsilon_1$ for $j=1,\dots,\,k$, $\beta_n^{(\sigma_{\epsilon_1})}(a)=(-1)^{n-\sigma_{\epsilon_1}-(k-\epsilon_1)+1}$  and $\beta_n^{(\varepsilon_i)}(b)=0$ for $i=1,\dots,\,n-k$. As consequence, $\beta_n\in X_{\{\sigma_1,\dots,\sigma_k\}_{\{\sigma_{\epsilon_1},\dots,\sigma_{\epsilon_\ell}\}}}^{\{\varepsilon_1,\dots,\varepsilon_{n-k}\}_{\{\varepsilon_{\kappa_1},\dots,\varepsilon_{\kappa_h}\}}}$ for $n=0,1,\dots$.
						 
						 Let us see that this sequence is nondecreasing.
						 
						 By definition, $\tilde T_n[M^*](\beta_1-\beta_0)=(M^*+\lambda_2)x_{M^*}^1\geq 0$. In this case, $\dfrac{d^{\sigma_j}}{dt^{\sigma_j}}(\beta_1-\beta_0)_{\mid t=a}=0$  for $j=1,\dots,k$ and $\dfrac{d^{\varepsilon_j}}{dt^{\varepsilon_j}}(\beta_1-\beta_0)_{\mid t=b}=0$ for $i=1,\dots,n-k$, then  $\beta_1-\beta_0\in X_{\{\sigma_1,\dots,\sigma_k\}_{\{\sigma_{\epsilon_1},\dots,\sigma_{\epsilon_\ell}\}}}^{\{\varepsilon_1,\dots,\varepsilon_{n-k}\}_{\{\varepsilon_{\kappa_1},\dots,\varepsilon_{\kappa_h}\}}}$. So, $\beta_1\geq\beta_0$.
						 
						 Analogously, for $n\geq 1$, we conclude that $\{\beta_n\}$ is a nondecreasing sequence. Moreover, by properties of the related Green's function, which is  continuous on $I\times I$, it is bounded from above.
						 
						 Since $\tilde T_n[-\lambda_2]$ is strongly inverse positive in $X_{\{\sigma_1,\dots,\sigma_k\}_{\{\sigma_{\epsilon_1},\dots,\sigma_{\epsilon_\ell}\}}}^{\{\varepsilon_1,\dots,\varepsilon_{n-k}\}_{\{\varepsilon_{\kappa_1},\dots,\varepsilon_{\kappa_h}\}}}$, $x_{-\lambda_2}^1$ is the unique solution of $\tilde T_n[-\lambda_2]u(t)=0$, coupled with the boundary conditions imposed to $\alpha_n$ and $\beta_n$. Thus, we can affirm that 
						 \[\lim_{n\rightarrow \infty}\alpha_n=\lim_{n\rightarrow\infty}\beta_n=x_{-\lambda_2}^1\,,\]
						 and $\alpha_0=x_0^1\geq x_{-\lambda_2}^1\geq x_{M^*}^1= \beta_0\geq 0$ on $I$.
						 
						 Repeating the previous arguments, we can conclude that for all $M\in [-\lambda_2,M^*]$, we have:
						 \begin{equation}\label{Ec::x1d} x_{-\lambda_2}^1\geq x_M^1\geq x_{M^*}^1\geq 0\text{ on $I$}\,.\end{equation}
						 
						 On the other hand, it is known that ${x_{-\lambda_2}^1}^{(\beta)}(b)=0$. From inequality \eqref{Ec::x1d}, we have ${x_{M}^1}^{(\beta)}(b)=0$ for all $M\in [-\lambda_2,M^*]$, which contradicts the discrete character of the spectrum $\tilde T_n[0]$ in $X_{\{\sigma_1,\dots,\sigma_{\epsilon_1-1},\sigma_{\epsilon_2-1},\dots,\sigma_k\}}^{\{\varepsilon_1,\dots,\varepsilon_{n-k}|\beta\}}$. Thus, we arrive to a contradiction by supposing that there exists  $M^*>-\lambda_2$ such that $\tilde T_n[M^*]$ is inverse positive in $X_{\{\sigma_1,\dots,\sigma_k\}_{\{\sigma_{\epsilon_1},\dots,\sigma_{\epsilon_\ell}\}}}^{\{\varepsilon_1,\dots,\varepsilon_{n-k}\}_{\{\varepsilon_{\kappa_1},\dots,\varepsilon_{\kappa_h}\}}}$.

\vspace{0.5cm}

 Analogously, if $\lambda_2=\lambda_{\varepsilon_{\kappa_1}}$, it can be proved that it does not exist any $M^*>-\lambda_2$ such that $\tilde T_n[M^*]$ is inverse positive in $X_{\{\sigma_1,\dots,\sigma_k\}_{\{\sigma_{\epsilon_1},\dots,\sigma_{\epsilon_\ell}\}}}^{\{\varepsilon_1,\dots,\varepsilon_{n-k}\}_{\{\varepsilon_{\kappa_1},\dots,\varepsilon_{\kappa_h}\}}}$.

 \vspace{0.5cm}

Finally, we can proceed analogously when $n-k$ is odd to conclude that there is not any $M^*<-\lambda_2$ such that $\tilde T_n[M^*]$ is inverse negative in $X_{\{\sigma_1,\dots,\sigma_k\}_{\{\sigma_{\epsilon_1},\dots,\sigma_{\epsilon_\ell}\}}}^{\{\varepsilon_1,\dots,\varepsilon_{n-k}\}_{\{\varepsilon_{\kappa_1},\dots,\varepsilon_{\kappa_h}\}}}$.

			\end{proof}
					
					\subsection{Particular cases}
			
			This section is devoted to show the applicability of the previous results to some examples.
			
						Realize that most of the examples given in Section \ref{SS::Ex} follow the structure given on this section. So, we will be able to obtain the characterization of the strongly inverse positive (negative) character for those operators in different spaces with non homogeneous boundary conditions.

						\begin{itemize}
							\item  $n^{\mathrm{th}}-$order operators with $(k,n-k)$ boundary conditions. 
						\end{itemize}
				In this case $\mu=\max\{\alpha_2,\beta_2\}=-1$. So, since the biggest set where we can apply 	Theorem \ref{T::IPNH2} is $X_{\{0,\dots,k-1\}_{\{k-1\}}}^{\{0,\dots,n-k-1\}_{\{n-k-1\}}}$, Theorem \ref{T::IPNH2}	is equivalent to Theorem \ref{T::IPNH}. 
				
				However, in many cases,  we can be under the conditions of Remark \ref{R::16} which allows us to apply Theorem \ref{T::IPNH2} in bigger sets with more non homogeneous boundary conditions.

						\begin{itemize}
							\item Operator $T_4(p_1,p_2)[M]\,u(t)=u^{(4)}(t)+p_1(t)\,u^{(3)}(t)+p_2(t)\,u^{(2)}(t)+M\,u(t)$ in $X_{\{0,2\}}^{\{0,2\}}$.
						\end{itemize}
						The study of this kind of operators in $X_{\{0,2\}_{\{2\}}}^{\{0,2\}^{\{2\}}}$, which is obtained by applying Theorem \ref{T::IPNH}, has already been done in \cite{CabSaa3}.
								
								But, in such a case, since $\mu=\max\{\alpha_2,\beta_2\}=1$, by studying the different eigenvalues, we can characterize the strongly inverse positive character of $T_4(p_1,p_2)[M]$ in the different subsets of $X_{\{0,2\}_{\{0,2\}}}^{\{0,2\}_{\{0,2\}}}$.
								
							Let us consider, for instance, the operator
								\begin{equation*}
							T_4[p,M]\,u(t)\equiv u^{(4)}(t)-p\,u''(t)+M\,u(t)\,,\quad t\in I\equiv [a,b]\,,
							\end{equation*}	
							where $p\geq 0$.
							
							In \cite{CabSaa2}, there are obtained some of the related eigenvalues:
							\begin{itemize}
							\item	The least positive eigenvalue of $T_4[p,0]$ in $X_{\{0,2\}}^{\{0,2\}}$ is given by $\lambda_1^p=\left( \frac{\pi}{b-a}\right) ^4+p\,\left( \frac{\pi}{b-a}\right) ^2$. 
							
							\item The biggest negative eigenvalue of $T_4[p,0]$ in $X_{\{0\}}^{\{0,1,2\}}$ and in $X_{\{0,1,2\}}^{\{0\}}$ coincide and are given by $-\lambda_2^p$, where $\lambda_2^p$ is the least positive solution of						
							\[\frac{ \tan \left(\dfrac{b-a}{2}\,\sqrt{2\sqrt{\lambda}-p} \right)}{\sqrt{2\sqrt{\lambda}-p}}=\frac{ \tanh \left(\dfrac{b-a}{2}\,\sqrt{2\sqrt{\lambda}+p} \right)}{\sqrt{2\sqrt{\lambda}+p}}\,.\]
						\end{itemize}
					
					Now, let us obtain the missing eigenvalues:
					
					\begin{itemize}
						\item The biggest negative eigenvalues of $T_4[p,0]$ in $X_{\{2\}}^{\{0,1,2\}}$ and in $X_{\{0,1,2\}}^{\{2\}}$ coincide and are given by $-\lambda_{2_0}^p$, where $\lambda_{2_0}^p$ is the least positive solution of
						\[\frac{ \tan \left(\dfrac{b-a}{2}\,\sqrt{2\sqrt{\lambda}-p} \right)}{\sqrt{2\sqrt{\lambda}-p}}+\frac{ \tanh \left(\dfrac{b-a}{2}\,\sqrt{2\sqrt{\lambda}+p} \right)}{\sqrt{2\sqrt{\lambda}+p}}=0\,.\]
					\end{itemize}
				
				Thus, we obtain the following conclusions:
				\begin{itemize}
					\item $T_4[p,M]$ is strongly inverse positive  in  $X_{\{0,2\}_{\{2\}}}^{\{0,2\}_{\{2\}}}$ if, and only if, $M\in (-\lambda_1^p,\lambda_2^p]$.
					\item $T_4[p,M]$ is strongly inverse positive in $X_{\{0,2\}_{\{0,2\}}}^{\{0,2\}_{\{0,2\}}}$ if, and only if, $M\in (-\lambda_1^p,\lambda_{2_0}^p]$.
				\end{itemize}
							
						\begin{itemize}
							\item Operator $T_n^0[M]\,u(t)=u^{(n)}(t)+M\,u(t)$.
						\end{itemize}
						In the sequel, we treat some of this kind of problems which have been introduced in Section \ref{SS::Ex}.
						
						\begin{itemize}
							\item[]
							\begin{itemize}
								\item Second order
							\end{itemize}
							In second order, the only possibility is to consider $k=1$. Then, the characterization is obtained by applying Theorem \ref{T::IPNH} and the parameters set for the strongly inverse positive character is the same as in the homogeneous case which has been obtained in Section \ref{SS::Ex}.

							\begin{itemize}
								\item Third  order
							\end{itemize}
														
							Let us consider, for instance, $\{\sigma_1,\sigma_2\}=\{1,2\}$ and $\{\varepsilon_1\}=\{0\}$. In such a case, $\mu=\max\{\alpha_2,\beta_2\}=\max\{-1,0\}=0$. Then, we obtain the characterization in $X_{\{1,2\}_{\{2\}}}^{\{0\}_{\{0\}}}$ from  Theorem \ref{T::IPNH} or Theorem \ref{T::IPNH2} equivalently.
							
						But, from Remark \ref{R::16}, we are able to obtain the characterization in $X_{\{1,2\}_{\{1,2\}}}^{\{0\}^{\{0\}}}$ given as follows:
						
						$T_3^0[M]$ is strongly inverse negative in $X_{\{1,2\}_{\{1,2\}}}^{\{0\}^{\{0\}}}$ if, and only if, $M\in[-\lambda_2,-\lambda_1)$ where $\lambda_1=-m_4^3$, with $m_4\approxeq 1.85$ the least positive solution of \eqref{Ec::m4}, is the biggest negative eigenvalue of $T_3^0[0]$ in $X_{\{1,2\}}^{\{0\}}$ and $\lambda_2=m_7$, with $m_7\approxeq 1.84981$ the least positive solution of
						\[2 e^{3 m/2} \cos \left(\frac{\sqrt{3} m}{2}\right)+1=0\,,\] is the least positive eigenvalue of $T_3^0[0]$ in $X_{\{2\}}^{\{0,1\}}$.			
							\begin{itemize}
								\item Fourth order
							\end{itemize}
						Let us consider again  fourth order problems introduced in Section \ref{SS::Ex}, $X_{\{0\}}^{\{1,2,3\}}$ and $X_{\{0,2\}}^{\{1,3\}}$. In the first case we cannot apply directly Theorem \ref{T::IPNH2}, since $\mu=0$. However, with the same argument as in Remark \ref{R::16}, Theorem \ref{T::IPNH2} is still true for $\sigma_{\epsilon_{\ell-1}}\geq\mu$ or $\varepsilon_{\kappa_{h-1}}\geq\mu$.
							
								\begin{itemize}
								\item[$\bullet$]	The biggest negative eigenvalue of $T_4^0[0]$ in $X_{\{0\}}^{\{1,2,3\}}$ is $\lambda_1=-\dfrac{\pi^4}{4}$.
								
								\item[$\bullet$]	The least positive eigenvalue of $T_4^0[0]$ in $X_{\{0,1\}}^{\{0,3\}}$ is $\lambda_1^2=\pi^4$
								
								\item[$\bullet$]	The least positive eigenvalue of $T_4^0[0]$ in $X_{\{0,1\}}^{\{1,3\}}$ is $\lambda_0^2=m_1^4$, where $m_1\approxeq 2.36502$ is the least positive solution of \eqref{Ec::Ex61}.
							\end{itemize}
								
						Thus, we conclude that $T_4^0[M]$ is strongly inverse negative in $X_{\{0\}_{\{0\}}}^{\{1,2,3\}_{\{2,3\}}}$ if, and only if, $M\in \left[ -\pi^4,\dfrac{\pi^4}{4}\right) $.
						
						Moreover, $T_4^0[M]$ is strongly inverse negative in $X_{\{0\}_{\{0\}}}^{\{1,2,3\}_{\{1,2,3\}}}$ if, and only if, $M\in \left[ -m_1^4,\dfrac{\pi^4}{4}\right) $.
						
						\vspace{0.5cm}
						
						For  $X_{\{0,2\}}^{\{1,3\}}$, we have $\mu=\max\{1,2\}=2$. 	Let us study the strongly inverse positive character in $X_{\{0,2\}_{\{0,2\}}}^{\{1,3\}_{\{1,3\}}}$.
							
						\begin{itemize}
							\item[$\bullet$]	The least positive eigenvalue of $T_4^0[0]$ in $X_{\{0,2\}}^{\{1,3\}}$ is $\lambda_1=\dfrac{\pi^4}{16}$.
							
						\item[$\bullet$]	The biggest negative eigenvalue of $T_4^0[0]$ in $X_{\{2\}}^{\{0,1,3\}}$ is $\lambda_0^2=-\dfrac{\pi^4}{4}$
							
						\item[$\bullet$]	The biggest negative eigenvalue of $T_4^0[0]$ in $X_{\{0,1,2\}}^{\{1\}}$ is $\lambda_1^2=-4\,\pi^4$.
					\end{itemize}
							
							Thus, $\lambda_2=-\dfrac{\pi^4}{4}$ and we can conclude that 			
							$T_4^0[M]$ is strongly inverse positive in $X_{\{0,2\}_{\{0,2\}}}^{\{1,3\}_{\{1,3\}}}$ if, and only if, $M\in\left( -\dfrac{\pi^4}{16},\dfrac{\pi^4}{4}\right] $.
							
							\begin{itemize}
								\item []\begin{itemize}
									\item Higher order
								\end{itemize}
							\end{itemize}
							Now, let us analyze the sixth order operator given in Subsection \ref{SS::Ex}. That is, the operator $T_6^0[M]$ defined in $X_{\{0,2,4\}}^{\{0,2,4\}}$. In this case, $\mu=\max\{3,3\}=3$, so we can apply Theorem \ref{T::IPNH2} in different spaces.
							
							Let us obtain the different eigenvalues:

								\begin{itemize}
									\item[$\bullet$] The biggest negative eigenvalue of $T_6^0[0]$ in $X_{\{0,2,4\}}^{\{0,2,4\}}$ is $\lambda_1=-\pi^6$.
									\item [$\bullet$] The least positive eigenvalue of $T_6^0[0]$ in $X_{\{0,4\}}^{\{0,1,2,4\}}$ is $\lambda_2^2=m_8^6$, where $m_8\approxeq 4.14577$ is the least positive solution of
								{\small 	\[\begin{split}
									\sqrt{3} e^{m/2} \left(e^{2 m}+1\right)-3 \left(e^m+1\right)^2 \left(e^m-1\right) \sin
									\left(\frac{\sqrt{3} m}{2}\right)&\\+\sqrt{3} \left(e^m+1\right) \left(e^m-1\right)^2 \cos
									\left(\frac{\sqrt{3} m}{2}\right)-2 \sqrt{3} e^{3 m/2} \cos \left(\sqrt{3} m\right)&=0\,.
									\end{split}\]}
									\item [$\bullet$] The least positive eigenvalue of $T_6^0[0]$ in $X_{\{0,1,2,4\}}^{\{0,4\}}$ is $\lambda_2^2=m_8^6$
									\item[$\bullet$]  The least positive eigenvalue of $T_6^0[0]$ in $X_{\{2,4\}}^{\{0,1,2,4\}}$ is $\lambda_0^2=m_9^6$, where $m_9\approxeq 3.17334$ is the least positive solution of 
										{\small 	\[\begin{split}
											-\sqrt{3} e^{m/2} \left(e^{2 m}+1\right)-3 \left(e^m+1\right)^2 \left(e^m-1\right) \sin
											\left(\frac{\sqrt{3} m}{2}\right)&\\-\sqrt{3} \left(e^m+1\right) \left(e^m-1\right)^2 \cos
											\left(\frac{\sqrt{3} m}{2}\right)+2 \sqrt{3} e^{3 m/2} \cos \left(\sqrt{3} m\right)&=0\,.
											\end{split}\]}
									\item[$\bullet$]  The least positive eigenvalue of $T_6^0[0]$ in $X_{\{0,1,2,4\}}^{\{2,4\}}$ is $\lambda_0^2=m_9^6.$
								\end{itemize}
							
								Thus, we conclude that $T_6^0[M]$ is strongly inverse negative in $X_{\{0,2,4\}_{\{2,4\}}}^{\{0,2,4\}_{\{2,4\}}}$ if, and only if, $M\in [-m_8^6,\pi^6)$. Moreover, $T_6^0[M]$ is strongly inverse negative in $X_{\{0,2,4\}_{\{0,2,4\}}}^{\{0,2,4\}_{\{0,2,4\}}}$ if, and only if, $M\in [-m_9^6,\pi^6)$.
							\end{itemize}
							
							\begin{itemize}
								\item Operators with non constant coefficients.
							\end{itemize}
						To finish this work we show an example where a fourth order operator with non constant coefficients is considered.
						
						Let us define the operator
						\[T_4^{nc}[M]=u^{(4)}+e^{2\,t}\sin(2\,t)\,u'''(t)+M\,u(t)\,,\quad t\in [0,1]\]
						defined in $X_{\{0,2\}}^{\{1,2\}}$.
						
						In such a space, we have $\mu=\max\{1,0\}=1$, and the linear differential equation
						\[u''(t)+e^{2\,t}\sin(2\,t)\,u'(t)=0\,,\]
						is disconjugate on $[0,1]$, since it is a composition of two first order linear differential equations. Thus, we can apply all previous results to characterize the strongly inverse positive character of $T^{nc}_4[M]$ in $X_{\{0,2\}}^{\{1,2\}}$.
						
						First, we obtain numerically, by means of Mathematica program, the different eigenvalues of $T_4^{nc}[0]$.
						
						\begin{itemize}
							\item The least positive eigenvalue of $T_4^{nc}[0]$ in $X_{\{0,2\}}^{\{1,2\}}$ is given by $\lambda_1\approxeq 2.62355^4$.
							\item The biggest negative eigenvalue of $T_4^{nc}[0]$ in $X_{\{0,1,2\}}^{\{1\}}$ is, given by $\lambda_2''\approxeq -4.69621^4$.
							\item The biggest negative eigenvalue of $T_4^{nc}[0]$ in $X_{\{0\}}^{\{0,1,2\}}$ is, given by $\lambda_2'\approxeq -6.18170^4$.
							\item The biggest negative eigenvalue of $T_4^{nc}[0]$ in $X_{\{0,1,2\}}^{\{2\}}$ is, given by $\lambda_1^2\approxeq -3.45041^4$.
							\item The biggest negative eigenvalue of $T_4^{nc}[0]$ in $X_{\{2\}}^{\{0,1,2\}}$ is, given by $\lambda_0^2\approxeq -4.20409^4$.
						\end{itemize}
						
						Thus, by means of Theorems \ref{T::IPN} and \ref{T::IPNH2}, we conclude:
						\begin{itemize}
							\item $T_4^{nc}[M]$ is strongly inverse positive in $X_{\{0,2\}_{\{2\}}}^{\{1,2\}_{\{2\}}}$ if, and only if, $M\in \left( -2.62355^4,4.69621^4\right] $.
								\item $T_4^{nc}[M]$ is strongly inverse positive in $X_{\{0,2\}_{\{0,2\}}}^{\{1,2\}_{\{2\}}}$ if, and only if, $M\in \left( -2.62355^4,4.20409^4\right] $.
									\item $T_4^{nc}[M]$ is strongly inverse positive either in  $X_{\{0,2\}_{\{2\}}}^{\{1,2\}_{\{1,2\}}}$ or $X_{\{0,2\}_{\{0,2\}}}^{\{1,2\}_{\{1,2\}}}$ if, and only if, $M\in \left( -2.62355^4,3.45041^4\right] $.
						\end{itemize}
					
					Moreover, in order to use Theorem \ref{T::IPN1}, we can obtain the needed eigenvalues of $T_4^{nc}[0]$:
					
					\begin{itemize}
					\item The least positive eigenvalue of $T_4^{nc}[0]$ in $X_{\{0,1\}}^{\{1,2\}}$ is given by $\lambda_3'\approxeq 3.22872^4$.
					\item The least positive eigenvalue of $T_4^{nc}[0]$ in $X_{\{0,2\}}^{\{0,1\}}$ is given by $\lambda_3''\approxeq 4.33768^4$.
					\end{itemize}
						
						Thus, from Theorem \ref{T::IPN1}, if $T_4^{nc}[M]$ is a strongly inverse negative operator in $X_{\{0,2\}}^{\{1,2\}}$, then $M\in \left[-3.22872^4,-2.62355^4\right)$.

\end{document}